\documentclass[11pt]{amsart}
\usepackage[english]{babel}
\usepackage[utf8x]{inputenc}
\usepackage[T1]{fontenc}
\usepackage[]{float}
\usepackage[margin=2cm]{geometry}
\usepackage{empheq}
\usepackage{amsmath}
\usepackage{amsthm}
\usepackage{amssymb}
\usepackage{graphicx}
\usepackage[colorinlistoftodos]{todonotes}
\usepackage[colorlinks=true, allcolors=blue]{hyperref}

\newtheorem{theorem}{Theorem}[section]
\newtheorem{proposition}[theorem]{Proposition}
\newtheorem{corollary}{Corollary}[theorem]
\newtheorem{lemma}[theorem]{Lemma}
\newtheorem{remark}{Remark} [section]
\newtheorem{definition}{Definition} [section]

\DeclareMathSymbol{,}{\mathpunct}{operators}{"2C}
\DeclareMathSymbol{.}{\mathpunct}{operators}{"2E}
\numberwithin{equation}{section}

\newcommand{\ep}{\epsilon}
\newcommand{\be}{\begin{equation}}
\newcommand{\ee}{\end{equation}}
\newcommand{\p}{\partial}
\newcommand{\RP}{\text{Re}}
\newcommand{\IP}{\text{Im}}
\newcommand{\R}{\mathbb{R}}

\newcommand{\C}{\mathbb{C}}
\newcommand{\CI}{\mathcal{I}}
\newcommand{\CH}{\mathcal{H}}
\newcommand{\CR}{\mathcal{R}}
\newcommand{\CC}{\mathcal{C}}
\newcommand{\Bb}{\mathbf{b}}
\newcommand{\CD}{\mathcal{D}}
\newcommand{\BF}{\mathbf{F}}

\newcommand{\BK}{\mathbf{K}}
\newcommand{\BX}{\mathbf{X}}
\newcommand{\CE}{\mathcal{E}}
\newcommand{\CU}{\mathcal{U}}
\newcommand{\BP}{\mathbf{P}}

\begin{document}
\begingroup
\allowdisplaybreaks

\title[Water waves linearized at shear flows]{Capillary gravity water waves linearized at monotone shear flows: eigenvalues and inviscid damping}

\author[X. Liu]{Xiao Liu}
\address[X. Liu]{School of Mathematics, Georgia Institute of Technology, Atlanta, GA 30332}
\email{xliu458@gatech.edu}

\author[C. Zeng]{Chongchun Zeng}
\address[C. Zeng]{School of Mathematics, Georgia Institute of Technology, Atlanta, GA 30332}
\email{zengch@math.gatech.edu}
\thanks{CZ is supported in part by the National Science Foundation DMS-1900083.}

\begin{abstract}
We consider the 2D capillary gravity waves of finite depth $x_2 \in (-h, 0)$ linearized at a monotonic shear flow $U(x_2)$. The focuses are the eigenvalue distribution and linear inviscid damping. Unlike the 
Euler equation 
in a fixed channel 
where eigenvalues exist only in low wave numbers $k$ of the horizontal variable $x_1$, we first prove that the linearized capillary gravity wave has two branches of eigenvalues $-ik c^\pm (k)$, where the wave speeds $\R \ni c^\pm (k)  = O(\sqrt{|k|})$ for $|k|\gg1$ 
are asymptotic to those of the linear irrotational capillary gravity waves. Under the additional assumption $U''\ne 0$, we obtain the complete continuation of these two branches, which are all the eigenvalues 
 in this (and some other) case(s). In particular, $-ik c^-(k)$ could bifurcate into unstable eigenvalues at $c^-(k)=U(-h)$. 
 The bifurcation of unstable eigenvalues from inflection values of $U$ is also obtained. Assuming there are 
no embedded eigenvalues for any 
wave number $k$, 
the linearized velocity and surface profile $(v(t, x), \eta(t, x_1))$ are considered in both periodic-in-$x_1$ and $x_1\in\R$ cases. 
Each solution can be split into $(v^p, \eta^p)$ and $(v^c, \eta^c)$ whose $k$-th Fourier modes in $x_1$ correspond to the eigenvalues  and the continuous spectra of the wave number $k$, respectively. The component $(v^p, \eta^p)$ is governed by the  dispersion relation 
$\omega(k) =  - k c^\pm (k)$ 
in the case of $x_1 \in \R$. 
The other component $(v^c, \eta^c)$ satisfies the linear inviscid damping as fast as $|v_1^c|_{L_x^2}, |\eta^c|_{L_x^2} = O(\frac 1{|t|})$ and $|v_2^c|_{L_x^2}=O(\frac 1{t^2})$ as $|t| \to \infty$. Furthermore, additional $L_x^2 L_t^q$, $q\in (2, \infty]$, decay of $tv_1^c$ and $t^2 v_2^c$  is obtained after leading asymptotic terms are singled out, which are in various forms of $t$-dependent translations in $x_1$ of certain functions of $x$.  
\end{abstract}
\maketitle

\section{Introduction}

Consider the two dimensional capillary gravity water waves in the moving domain of finite depth 
\[
\CU_t=\{(x_1,x_2)\in \mathbb{T}_L \times \mathbb{R} \mid -h< x_2 < \eta(t,x)\}, \quad  \mathbb{T}_L:=\mathbb{R}/  L\mathbb{Z}, \; L>0,
\]
or 
\[
\CU_t=\{(x_1,x_2)\in \R \times \mathbb{R} \mid -h< x_2 < \eta(t,x)\}.
\]
The free surface is given by $S_t=\{(t,x) \mid x_2=\eta(t,x_1)\}$. 
For $x \in \CU_t$, let $v=(v_1(t,x), v_2(t, x)) \in \mathbb{R}^2$ denote the fluid velocity and $p=p(t,x)\in \mathbb{R}$ the pressure. They satisfy the free boundary problem of the incompressible Euler equation:
\begin{subequations} \label{E:Euler}
 \begin{empheq} 
 {align}
  & \partial_t v+(v\cdot \nabla)v+ \nabla p+g\vec{e}_2=0,  & x\in \CU_t, \label{E:Euler-1} \\
  & \nabla \cdot v=0, &x\in   \CU_t,  \label{E:Euler-2}\\
  & \p_t \eta (t, x_1)=v(t,x)\cdot(-\p_{x_1} \eta (t, x_1),1)^T, & x\in  S_t, \label{E:Euler-3} \\
  & p(t, x)=\sigma\kappa(t,x), & x\in S_t, \label{E:Euler-4}\\
  & v_2 (x_1, -h)=0, &  x_2=-h, \label{E:Euler-5}
\end{empheq}
\end{subequations}
where $\sigma > 0$, $\kappa(t,x)=-\frac{\eta_{x_1x_1}}{(1+\eta_{x_1}^2)^{\frac{3}{2}}}$ is the mean curvature of $S_t$ at $x$ which corresponds to the surface tension, $g>0$ is the gravitational acceleration,  and the constant density is normalized to be 1.

Shear flows are a fundamental class of stationary solutions in the form of 
\be \label{E:shear}
v_*:= \big(U(x_2), 0\big)^T, \quad S_*:=\{(t,x)|x_2=\eta_* (x_1) \equiv 0\}, \quad \nabla p_* = -g \vec{e}_2.
\ee 
Our primary goal in this paper is to analyze the capillary gravity water wave system linearized at a monotone shear flow satisfying 
\be \tag{{\bf H}} 
U \in C^{l_0}([-h, 0]), \;\; l_0\ge 3, \quad U'(x_2)\ne 0, \; \; \forall x_2\in [-h,0].
\ee

\begin{remark} \label{R:sign}
Due to the symmetry of horizontal reflection
\[
(x_1, x_2)\to (-x_1, x_2), \quad (v_1, v_2, \eta, p) \to (-v_1, v_2, \eta, p),  
\]
the case of $U'<0$ is completely identical except the signs of $U''$ in Theorem \ref{T:e-values}(3d) should be reversed.
\end{remark}

One of the crucial aspect of the linearized problem is the stability/instability, which is also related to the generation of surface and internal waves due to small disturbance. Mathematically, robust instability is often produced by eigenvalues of the linearized system which have positive real parts and lead to  solutions with exponential growth in $t$, while eigenvalues with negative real parts correspond to linear solutions with exponential decay. Purely imaginary eigenvalues, where there are infinitely many in the free linear capillary gravity waves (namely, linearized at zero), give periodically oscillatory linear solutions. Continuous spectra are also expected to exist, which do exist in the case of the Euler equation in a fixed channel linearized at a shear flow where certain algebraic decay of the linear solutions -- linear inviscid damping -- had been obtained under certain conditions. See Subsections \ref{SS:background} and \ref{SS:Couette} for references and the explicit example of the Couette flow. Hence  the two main aspects of the linearized capillary gravity waves that we are focusing on are the eigenvalue distribution and the linear inviscid damping.

\subsection{Linearization} \label{SS:Linearization}

We first derive the linearized system of \eqref{E:Euler} at the shear flow $(v_*=(U(x_2), 0)^T, \eta_* =0)$ given in \eqref{E:shear}  satisfied by the linearized solutions which we denote by $(v, \eta, p)$.  
Let $(S_t^\epsilon, v^\epsilon (t, x), p^\ep (t, x))$ be a one-parameter family of solutions of \eqref{E:Euler} with $(S_t^0, v^0(t,x), p^0(t, x))=(S_*, v_*, p_*)$. 
Differentiating  the Euler equation \eqref{E:Euler-1} and  \eqref{E:Euler-2} with respect to $\ep$ and then evaluating it at $\ep=0$ yield 
\begin{subequations} \label{E:LEuler}
\be\label{E:LEuler-1}
    \partial_t v+U(x_2) \p_{x_1} v+(U'(x_2) v_2 , 0)^T +\nabla p=0, \quad  \;\; \nabla\cdot v=0, \quad \; x_2 \in (-h, 0). 
\ee
Taking its divergence and also evaluating  the above linearized Euler equation at $x_2=-h$,  
we obtain
\be\label{E:LEuler-2}
 -\triangle p
 =2U'(x_2)\partial_{x_1}v_2, \quad x_2 \in (-h, 0),  \; \text{ and }  \; 
 \partial_{x_2}p|_{x_2=-h} =0. 
\ee
From the kinematic boundary condition \eqref{E:Euler-3}, we have 
\be\label{E:LEuler-4} 
    \partial_t \eta= v_2|_{x_2=0} -U(0) \partial_{x_1}\eta.
\ee
Finally differentiating  \eqref{E:Euler-4}, where the left side is $p^\ep (t, x_1, \eta^\ep(t, x_1))$, 
and using $\partial_{x_2}p_*=-g$, we obtain  
\be\label{E:LEuler-5}
    p=g\eta-\sigma \partial_{x_1}^2\eta, \quad \text{ at } \ x_2=0. 
\ee
\end{subequations}
The above (\ref{E:LEuler-1} -- \ref{E:LEuler-5}) form the linearization of the capillary gravity water wave problem \eqref{E:Euler} at the shear flow $(v_*, S_*, p_*)$ with initial values $(v_{10}(x), v_{20} (x), \eta_0(x_1))$. In fact it can be reduced to an evolutionary problem of the unknowns $(v, \eta)$, while 
$p$  can be recovered by the boundary value problem of the elliptic system \eqref{E:LEuler-2} and \eqref{E:LEuler-5}. 

\subsection{Backgrounds and motivations} \label{SS:background}

Due to its physical and mathematical significance there have been extensive studies of the Euler equation   linearized at shear currents. Many of these works were on a fixed channel with slip boundary conditions 
\be \label{E:E-channel}
\text{\eqref{E:Euler-1}--\eqref{E:Euler-2} with } g=0, \;\; \; x_2 \in (-h, 0), \quad v_2 (x_1, 0)= v_2(x_1, -h)=0, 
\ee
and some of the results have been extended to free boundary problems such as the gravity waves. The spectral analysis is naturally a crucial part of such linear systems. Eigenvalues yield linear solutions exponential in time, while the continuous spectra often lead to algebraic decay of solutions, the so-called inviscid damping due to the lack of a priori dissipation mechanism of the Euler equation.   

$\bullet$ {\it Eigenvalues.} Since the variable coefficients in the linearized Euler system depend only on $x_2$, the subspace of the $k$-th Fourier mode is invariant under the linear evolution for any $k \in \R$. Hence it is a common practice to seek eigenvalues and eigenfunctions in the form of 
\be \label{E:e-func} 
v(t, x) = e^{ik(x_1 - c t)} \big(v_{10}(x_2), v_{20} (x_2) \big), \;\; \text{ also } \; \eta(t, x_1) = e^{ik(x_1 - c t)} \eta_0 (x_1) \; \text{in the free boundary case},  
\ee
where apparently the eigenvalues take the form $\lambda =-ikc$ with the wave speed $c= c_R + i c_I \in \C$. The linear system is spectrally unstable if there exist such $c$ with $c_I>0$, which appear in conjugate pairs. Solutions in the above form with $c\in U([-h, 0])$ are in a subtle situation and are referred to as singular modes (see Definition \ref{D:modes} and Remark \ref{R:singularM} for singular and non-singular modes). In seeking solutions in the form of \eqref{E:e-func}, the wave number $k \in \R$ is often treated as a parameter. 

Classical results on the spectra of the Euler equation \eqref{E:E-channel} in a channel  linearized at a shear flow
include:
\begin{itemize}
\item Unstable eigenvalues  are isolated for any wave number $k \in \R$ and do not exist for $|k|\gg 1$.
\item Rayleigh's necessary condition of instability \cite{Ray1880}: unstable eigenvalues do not exist for any $k$ if $U'' \ne 0$ on $[-h, 0]$ (see also \cite{Fj50}).
\item Howard's Semicircle Theorem \cite{How61}: for any $k\ne 0$, eigenvalues exist only with $c$ in the disk 
\be \label{E:semi-circle}
\big(c_R-\tfrac 12 (U_{max}+U_{min})\big)^2+c_I^2\leq \tfrac 14 (U_{max}- U_{min})^2. 
\ee
\item Unstable eigenvalues may exist with $c$ near inflection values of $U$ (Tollmien \cite{To35} formally, also \cite{LinC55}). 
\end{itemize}
Many classical results can be found in books such as \cite{DR04, MP94} {\it etc.} For a class of shear flows, the rigorous bifurcation of unstable eigenvalues was proved, e.g., in \cite{FH98, Lin03}. In particular, continuation of branches of unstable eigenvalues were obtained by Lin in the latter. 

It has been extended to the linearized free boundary problem of gravity  waves (i.e. $g>0$ and $\sigma=0$ in \eqref{E:Euler}) at shear flows (see \cite{Yih72, HL08, RR13} {\it etc.}) that: a.) assuming $U'>0$ and $U''\ne 0$ on $[-h, 0]$, there are no singular neutral modes in $U\big((-h, 0)\big)$ (i.e. solution in the form of \eqref{E:e-func} with $c \in U\big((-h, 0)\big)$); b.) the semicircle theorem still holds for unstable eigenvalues; and c.) for a class of shear flows, singular neutral modes may exist at inflection values of $U$ and the bifurcation and continuation of branches of unstable eigenvalues were also obtained. Compared to channel flows with fixed boundaries, new phenomena of the linearized gravity waves include: a.) in addition, critical values of $U$, where $U'=0$, and $c=U(-h)$ may be limiting singular neutral modes; and b.) there are non-singular neutral modes, i.e. $c \in \R \setminus U([-h, 0])$. Another related result is Miles' critical layer theory \cite{Mi57, BSWZ16} on the instability of shear flows in two-phase fluid interface problem due to the resonance between the temporal frequency of the linear irrotational capillary gravity waves at  the completely stationary water and the shear flow in the air in the above.   

$\bullet$ {\it Inviscid damping.} The analysis of the inviscid damping phenomenon started with the Euler equation in a fixed periodic channel \eqref{E:E-channel} linearized at the Couette flow $U(x_2)=x_2$. In 1907, Orr \cite{Orr1907} observed that the linearized vertical velocity $v_2(t, x)$ tends to zero as $t\to \infty$. Under the assumption $\int v_{10} (x_1, x_2) dx_1=0$, $\forall x_2 \in (-h, 0)$, which removes the shear flow component of the linear solutions through an invariant splitting, explicit calculations (see, e.g., \cite{Ca60, LZ11}) yield, as $t\to \infty$, 
\be \label{E:decay-0} 
\omega_0 \in L^2 \implies |v|_{L^2} = o(1), \;\; 
\omega_0 \in H^1\implies |v|_{L^2} = O(\tfrac 1{|t|}), \;\; \omega_0 \in H^2 \implies |v_2|_{L^2} = O(\tfrac 1{|t|^2}), 
\ee
where $\omega_0$ denotes the initial vorticity. More general shear flows in a fixed channel have also been studied extensively. For a class of general stable shear flows, Bouchet and Morita \cite{BM10} predicted similar decay estimates of the linearized velocity as well as the vorticity depletion phenomenon. For monotone shear flows without infection points, an $O(|t|^{-\nu})$ decay of the stream function was proved in \cite{Ste95} and then the \eqref{E:decay-0} type decay in \cite{Zi16, Zi17} under a smallness assumption of $LU''$ (also $\omega_0|_{x_2=-h, 0}=0$ in order for the $O(t^{-2})$ decay of $v_2$). A significant contribution is \cite{WZZ18} by Wei-Zhang-Zhao where the \eqref{E:decay-0} type estimates were obtained for general monotone shear flows without singular modes. In the follow-up works \cite{WZZ19, WZZ20, WZZhu20}, vorticity depletion and velocity decay (as well as an $L_t^2$ decay if $\omega_0\in L^2$ only) were also obtained for a class of non-monotone shear flows. As the decay rates in \eqref{E:decay-0} are basically optimal, some leading order effects from both the interior and the boundary were identified  in \cite{Zi16, Jia20S}. In the absence of boundary impact, for compactly supported initial vorticity, linear inviscid damping near a class of monotone shear flows was also obtained in Gevrey spaces \cite{Jia20A}. In \cite{GNRS20}, a different approach using methods from the study of Schr\"odinger operators was successfully adopted to analyze inviscid damping. See also \cite{BCV19, IJ21} for important developments for the linear inviscid damping at circular flows in $\R^2$. 

While this paper focuses on the linearized capillary gravity waves at shear flows, among the rich literatures on the related nonlinear dynamics of the 2-d Euler equation on fixed domains we refer the readers to \cite{AK98} for nonlinear Lyapunov stability of steady states based on energy-Casimir functions by Arnold; the remarkable asymptotic stability of shear flows in Gevrey class \cite{BeMa15, IJ20a, IJ20b, MZ20} based on the linear inviscid damping; and for nonlinear instability of steady states \cite{FSV97, Gre00, Lin04, LZ13}, {\it etc.}    

$\bullet$ {\it Intuitions and goals on linearized capillary gravity waves.} 
The goal of this paper is to study thoroughly the capillary gravity water waves linearized at a shear flow $U(x_2)$ under the above monotonicity assumption ({\bf H}), focusing on the spectral distribution, stability/instability, and, if the eigenvalues are properly separated from the continuous spectra, the spectral projections and the linear inviscid damping.

For an illustration, some explicit computations of the linearized capillary gravity wave system \eqref{E:LEuler} at the Couette flow $U(x_2)=x_2$ are given in Subsection \ref{SS:Couette}. There it is easy to see that, on the one hand, the linear inviscid damping \eqref{E:decay-0} holds for the rotational part of the solutions. On the other hand, there exist two branches of neutral modes $c^\pm (k)$ (see \eqref{E:dispersion-Q}) approaching infinity at the same rate as \eqref{E:dispersion-F} of the irrotational capillary gravity waves linearized at zero. They form the two branches of the dispersion relation of the irrotational components  in the linearized water wave system at the Couette flow, which is linearly stable. 

Based on the above cited existing results on the channel flows with fixed boundaries, as well as those on the gravity water waves, and the explicit calculations of the capillary gravity waves linearized at the Couette flow, the analysis of the linearization \eqref{E:LEuler} of the capillary gravity water waves at a general monotonic shear flow $U(x_2)$ is expected to include the following. 

a.) {\it Eigenvalues for wave numbers $|k|\gg 1$.} Like matrices, eigenvalues (or equivalently, singular and non-singular modes) of \eqref{E:LEuler} correspond to roots of a ``characteristic" function (the $\BF(k, c)$ defined in \eqref{E:BF}) analytic in $k \in \R$ and $c \in \C \setminus U([-h, 0])$. In contrast to the linearized Euler equation on a fixed domain where no eigenvalues exist for any large wave number $k$, as seen in the  linearized irrotational solutions of the capillary gravity waves at both the trivial (zero) solution and the Couette flow, likely there exist two non-singular neutral modes $c^\pm(k)$ for each wave number $|k|\gg 1$. These branches behave rather differently compared to the linearized gravity waves since the surface tension is dominant for $|k|\gg1$. This part would be handled by an asymptotic analysis (Section \ref{S:e-values}).

b.) {\it Analytic continuation and bifurcation of branches of eigenvalues and the spectral stability (Section \ref{S:e-values})}. Each branch of non-singular modes could continue as long as they do not collide with each other or reach $U([-h, 0])$, the boundary of the domain of analyticity of the characteristic function. The bifurcation analysis for $c$ near $U([-h, 0])$, more specifically near inflection values of $U$ which possibly generates the instability (compare with the gravity wave case \cite{HL08, RR13, HL13}) and $U(-h)$, would require very careful study of the dependence of the solutions to the classical Rayleigh equation \eqref{E:Ray-H1-1} on the singular parameter $c$ (Section \ref{S:Ray-Homo}). Our main tool is a local transformation which isolates the singular part of the solutions. 
 
c.) {\it Spectral projections and the linear inviscid damping (Section \ref{S:Linear}).} Assuming that the eigenvalues are properly separated from the continuous spectra, a decomposition of linear solutions $(v, \eta) = (v^p, \eta^p) + (v^c, \eta^c)$ into components corresponding to the eigenvalues and the continuous spectra, respectively, is expected. However, the boundedness of this spectral projection still needs to be obtained which is conceptually related to a lower bound of the angles both the two infinite dimensional components. This is a generalization of the Hodge decomposition of the free linear capillary gravity waves (linearized at zero) into the irrotational and the rotational parts. Both this intuition and the calculations of the capillary gravity waves linearized at the Couette flow suggest that $(v^p, \eta^p)$ is mostly related to the surface motion and dispersive (possibly with some unstable modes), while the other component $(v^c, \eta^c)$ is more determined by the internal rotations and thus by the vorticity.   
Whether the Euler equation is in a fixed domain or with free boundaries, the vorticity is transported  in the same fashion by the fluid flow in the interior of the fluid domain. Hence it is natural to expect the linear inviscid damping \eqref{E:decay-0} of $(v^c, \eta^c)$. One may further ask whether \eqref{E:decay-0} is optimal in general. If so, a deeper question is whether it is possible to identify the leading order parts of $(v^c, \eta^c)$ for $|t|\gg1$? These studies would be based on the careful analysis of the  spectral contour integrals (Section \ref{S:Linear}) and the solutions to the homogeneous (Section \ref{S:Ray-Homo}) and non-homogeneous Rayleigh equation (Section \ref{S:Ray-BC}). 



\subsection{Main results} 

We first give the main theorem on the eigenvalue distribution along with its implication on the linear stability. The results on the linear inviscid damping are somewhat more technical and only roughly outlined here. Their more precise statements are given in Theorems \ref{T:decay-per} and \ref{T:decay-R} in Subsection \ref{SS:MainT}. See Definition \ref{D:modes}, Lemma \ref{L:e-v-basic-1}(5), and Remark \ref{R:singularM} for what are referred to as singular and non-singular modes. Particularly, by slightly adjusting the  same argument as in \cite{How61, Yih72},  the Semi-circle Theorem still holds for the unstable modes of the linearized system  \eqref{E:LEuler} of the capillary gravity water waves at shear flows, namely, any unstable mode satisfies \eqref{E:semi-circle}. We shall take this as granted in the rest of the paper. 


\begin{theorem} \label{T:e-values}
({\bf Eigenvalues.}) Suppose $U \in C^3$ and $U'>0$ on $[-h, 0]$, then the following hold. 
\begin{enumerate} 
\item There exists $k_0>0$ such that for any $k\in \R$ with $|k|\ge k_0$, there are no singular modes and  exactly two non-singular modes $c^+(k) \in ( U(0), +\infty)$ and $c^-(k) \in (-\infty, U(-h))$ which correspond to semi-simple eigenvalues $-ikc^\pm (k)$. Moreover,  
\begin{enumerate}
\item $c^\pm (k)$ are even and analytic in $k$ and $c^+(k)$ can be extended for all $k \in \R$ with $c^+(k) > U(0)$;  
\item $\lim_{|k| \to \infty}  c^\pm (k)/\sqrt{\sigma |k|} =\pm 1$;    
\item if $U(-h)$ is not a singular mode for any $k \in \R$, then $c^-(k)$  can also be extended to be  even and analytic in all $k\in \R$ with $c^-(k) < U(-h)$; and 
\item if singular modes do not exist  $\forall \, k\in \R$, then $(k, c^\pm(k))$ are the only non-singular modes of \eqref{E:LEuler} which is linearly stable. 
\item  for $k >0$, $c^+(k)$ (and $c^-(k) < U(-h)$ as well if it can be extended for all $k \in \R$) has either none or exactly one critical point depending on whether \eqref{E:c(k)-mono-1} holds. In the latter case, the critical point is non-degenerate. 
\end{enumerate}
\item There exists $g_\# \ge 0$ depending only on $U$ and $\sigma$ such that the following hold.
\begin{enumerate} 
\item If $g> g_\#$, then the non-singular modes $c^-(k) < U(-h)$ can also be extended to be  even and analytic in all $k\in \R$ and $\pm (c^\pm(k))' >0$ for $k>0$; 
\item $g_\#=0$ if and only if 
\be \label{E:sigma-0} 
\sigma \ge \int_{-h}^0 \big(U(x_2) - U(-h)\big)^2 dx_2.
\ee 
\end{enumerate} 
\item If $U''\ne 0$ on $[-h, 0]$ is also satisfied, then 
the following hold with the $g_\#$ given in the above statement (2).
\begin{enumerate} 
\item The only possible singular mode is $c=U(-h)$. 
\item If $g> g_\#$ then there are no singular modes,  
$c^\pm(k)$ are the only non-singular modes, and thus \eqref{E:LEuler} is linearly stable.
\item If $g=g_\#$ and $U\in C^5$, then there exists $k_\#>0$ such that 
$c^-(k)$ can be extended as an even $C^{1, \alpha}$ function (for any $\alpha \in [0, 1)$) for all $k \in \R$. Moreover $c^- (k) < U(-h)$ is analytic for all $k \ne \pm k_\#$, and $c^-(\pm k_\#) = U(-h)$. For each $k \in \R$,  $c^\pm (k)$ are the only singular or non-singular modes and thus \eqref{E:LEuler} is spectrally stable.  
\item If $g< g_\#$  and $U\in C^5$, then there exist $k_\#^+ > k_\#^->0$ such that we have the following. 
\begin{enumerate}
\item Assume $U''>0$ on $[-h, 0]$, then $c^-(k)$ can be extended as an even $C^{1, \alpha}$ function (for any $\alpha \in [0, 1)$) for all $k \in \R$ and analytic except at $k = \pm k_\#^\pm$ such that 
\[
c^-(\pm k_\#^\pm) = U(-h), \quad c^-(k) < U(-h), \;\, \forall |k| \notin [k_\#^-, k_\#^+], \quad  c_I^-(k) >0, \;\, \forall |k| \in (k_\#^-, k_\#^+).  
\] 
Moreover, for each $k$, all singular and non-singular modes are exactly $c^+ (k)$, $c^-(k)$, as well as $\overline{c^-(k)}$ if $|k| \in (k_\#^-, k_\#^+)$. Consequently, \eqref{E:LEuler} is spectrally unstable iff A.) $x_1 \in \R$ or B.) $x_1\in \mathbb{T}_L$ and there exists $m \in \mathbb{N}$ such that $\frac {2\pi m}L \in (k_\#^-, k_\#^+)$. 
\item Assume $U''<0$ on $[-h, 0]$, then $c^-(k)$ can be extended as an even $C^{1, \alpha}$ real valued function (for any $\alpha \in [0, 1)$) for $|k| \notin (k_\#^-, k_\#^+)$, analytic in $k$ if $|k| \notin [k_\#^-, k_\#^+]$, and $c^-(\pm k_\#^\pm) = U(-h)$. Moreover,  all singular and non-singular modes  are exactly $c^+ (k)$ and $c^-(k)$, where the latter is only for $|k| \notin (k_\#^-, k_\#^+)$, and \eqref{E:LEuler} is spectrally stable.    
\end{enumerate}
\end{enumerate} 
\item If $U \in C^5$ and $U'' (x_{20}) =0$ for some $x_{20} \in [-h, 0)$. Let $c_0 = U(x_{20})$. 
\begin{enumerate}
\item There exists $\sigma_0, k_0>0$ such that for any $\sigma\in (0, \sigma_0)$, there exists a unique $k> k_0$ unique in $[k_0, \infty)$  such that $c_0$ is a singular neutral mode for $\pm k$. 
\item If $x_{20} \in (-h, 0)$, $U''' (x_{20})\ne 0$, and $c_0$ is a singular neutral mode for some $k_0>0$, then, under a non-degenerate condition (verified for the one obtained in (4a) for small $\sigma$), there exist unstable modes near $c_0$ for $k$ close to $k_0$ on one side of $k_0$. 
\end{enumerate}
\end{enumerate}
\end{theorem}

\begin{remark} 
a.) 
The linear stability in (1d) and (3b) holds due to the inviscid damping in Theorems \ref{T:decay-per} and \ref{T:decay-R} and the dynamics in the directions of eigenfunctions being only oscillatory. \\
b.) If $U'' \ne 0$ and  \eqref{E:sigma-0} is satisfied, then the results in the above (3b) hold. \\
c.) Conceptually both the surface tension and the gravity have stablizing effects. For a given monotonic shear flow, assumption (1.8) is a sufficient condition to ensure that the surface tension itself is strong enough to stabilize the whole branch of the point spectra continued from $c^-(\infty)= -\infty$. \\
d.) Condition \eqref{E:c(k)-mono-1} is also directly on $U$, $g$, and $\sigma$, but less explicit, and we leave it in Subsection \ref{SS:e-v-basic}.  
\end{remark}

The existence of the unbounded branches of non-singular neutral modes $c^\pm(k)$ are in contrast to the gravity waves or the Euler equation on fixed channels. 
The geometric multiplicity of $-ik c^\pm (k)$ occurs only among different $k$. These temporal frequencies $-kc^\pm(k)$ are asymptotic to those (see \eqref{E:dispersion-F}) of the irrotational capillary gravity waves linearized at zero. Moreover, after normalizing the $L^2$ norm of the $v$ component of the eigenfunction to be $1$, the $L^2$ and $H^1$ differences in the $v$ and $\eta$ components, respectively, between the eigenfunctions of \eqref{E:LEuler} and the linearized irrotational capillary gravity waves are of the order $O(|k|^{-\frac 32})$ as $|k| \to \infty$ (see Remark \ref{R:e-func}). 
In the case (happening only if (1.8) is not satisfied) where the branch $c^-(k)$ reaches $U(-h)$, where the bifurcation equation has the worst regularity, subtle bifurcations occur. 
This had been pointed out as a possibility in the linearized gravity waves \cite{HL08, RR13, HL13}, but not analyzed. 
In particular, it runs out that the sign of $U''$ determines whether $c^- (k)$ becomes unstable or disappears at $U(-h)$. 

The spectral stability in the case $U''<0$ can also be obtained by directly modifying the usual  proof of the Rayleigh theorem 
in the fixed channel flow case, as done in \cite{Yih72} for the gravity wave. 
Our proof provides a complete picture of the eigenvalue distribution as in the above theorem, however.

While $U(0)$ is never a singular mode, just like the Rayleigh's theorem in the channel flow case the change of sign of $U''$ turns out to be necessary for the existence of {\it interior} singular modes, which is also  sufficient if $\sigma \ll1$. In the contrast this may not be sufficient if the stabilizing gravity $g$ and surface tension $\sigma$ are strong, see Remark  \ref{R:inflectionV}. 
\\ 

In the following outline of the linear inviscid damping results, $k \in \BK= \frac {2\pi}L\mathbb{N}$ if $x_1\in \mathbb{T}_L$, while $k \in \BK= \R$ if $x_1 \in \R$. The initial velocity $v_0= (v_{10}, v_{20})$ is always assumed to satisfy $\p_{x_1}^{n_1} \p_{x_2}^{n_2} v_0 \in L^2 (\mathbb{T}_L \times [-h, 0])$ or $L^2 (\R \times [-h, 0])$ for some $n_1$'s and $n_2$'s and similarly $\p_{x_1}^n \eta_0 \in L^2$. To avoid too much technicality, here we skip the detailed assumptions on their regularity  in $x_1$, but focus on that of $x_2$ only. Precise statements are given in Theorems \ref{T:decay-per} and \ref{T:decay-R}. The following $L_t^q$ is always for $t \in \R$. 

{\bf Main results on inviscid damping.} Assume ({\bf H}) and there are no singular modes for all $k\in \BK$, then any solution $(v(t, x), \eta(t, x_1))$ to \eqref{E:LEuler} can be decomposed into $v = v^p + v^c$ and $\eta = \eta^p + \eta^c$, where $(v^p, \eta^p)$ belongs to the invariant subspace generated by the eigenfunction of the non-singular modes for all $k \in \BK$, while $(v^c, \eta^c)$ and its vorticity $\omega^c$ satisfy the following estimates. 
{\it \begin{enumerate} 
\item If the initial vorticity $\omega_0 \in L_{x_2}^2$, then for any $q \in [2, \infty]$, $v^c \in L_x^2 L_t^q$ and $\eta^c \in L_t^q$.
\item If $\omega_0 \in H_{x_2}^1$, then $tv_2^c$, $t \eta^c \in L_x^2 L_t^{q_1}$ for any $q_1 \in [2, \infty]$. Moreover there exists $\Omega^c (x) \in H_{x_2}^1$ such that for any $q_2\in (2, \infty]$, 
\[
t v_1^c - U'(x_2)^{-1} \p_{x_1}^{-1}\Omega^c (x_1- U(x_2)t, x_2), \,\omega^c - \Omega^c (x_1- U(x_2)t, x_2), \, \p_{x_2}^2 v_2^c - \p_{x_1} \Omega^c (x_1- U(x_2)t, x_2) \in L_{x}^2  L_t^{q_2}. 
\]  
\item If $\omega_0 \in H_{x_2}^2$, then there exist $\Lambda_B (x), \Lambda_T(x) \in H_{x_2}^1$ such that for any $q_2\in (2, \infty]$, 
\[
t^2 v_2^c - U'(x_2)^{-2} \p_{x_1}^{-1}\Omega^c (x_1- U(x_2)t, x_2)  - \Lambda_B (x_1- U(-h)t, x_2)   
- \Lambda_T (x_1- U(0)t, x_2) \in L_{x}^2  L_t^{q_2}. 
\]  
\end{enumerate}}

In the above results, the assumption of the non-existence of singular modes, which is equivalent to the absence of embedded eigenvalues of \eqref{E:LEuler} for each wave number $k$, turns out to yield the spectral decomposition of the phase space of \eqref{E:LEuler} into the invariant subspaces corresponding to the non-singular modes/point spectra and the continuous spectra $-ik U([-h, 0])$ for each $k \in \R$. 

The component $(v^c, \eta^c)$ corresponds to the continuous spectra and enjoys temporal algebraic decay as in the case \eqref{E:E-channel} of the Euler equation in a fixed channel. 
For the case of $x_1 \in \R$, certain stronger decay in $|k|\ll1$ (for long waves) is also assumed on the initial values, see Theorem \ref{T:decay-R} and Remark \ref{R:LongW}. 
Additional to the above $L_t^q$ bounds, derivatives-in-$t$ estimates are given in Theorems \ref{T:decay-per} and \ref{T:decay-R} as well which also imply pointwise-in-$t$ decay. Compared with \eqref{E:decay-0}, these additional $L_t^q$ estimates represent an improvement of decay of roughly an order of $O(t^{-\frac 1q})$ (after appropriate $t$-dependent translations in $x_1$ of some asymptotic leading terms are identified and singled out in the cases of $tv_1^c$, $t^2 v_2^c$, {\it etc.}). 
For the Euler equation in a fixed channel \eqref{E:E-channel}, a.) when $\omega_0\in L^2$, the $|v|_{L_t^2}$ estimate was also obtained in \cite{WZZ19, WZZhu20}; b.) comparable asymptotic leading terms were identified in Lemma 3 of \cite{Zi16}  and  in Lemma 5.1 of \cite{Jia20S}. The Fourier transforms (in $x_1$) of these leading terms $\Omega^c$, $\Lambda_T$, and $\Lambda_B$ are given explicitly in \eqref{E:Omega-1}, \eqref{E:Lambda_B}, and \eqref{E:Lambda_T}, which represent the impact of the interior flow and the top and bottom boundaries, respectively.
See also \eqref{E:Lambda-BVP} for singular elliptic boundary value problems satisfied by $\Lambda_T$ and $\Lambda_B$. In particular, the free boundary effect is explicitly reflected in the boundary conditions \eqref{E:Ray-3} of the corresponding Rayleigh equation \eqref{E:Ray} and the form \eqref{E:Lambda_T} of $\Lambda_T$. The error estimates in addition to these leading asymptotic terms also justify that the estimates of $tv_1$ and $t^2 v_2$ in \eqref{E:decay-0} are optimal. The precise asymptotic leading terms could be useful for further analysis.   

The component $(v^p, \eta^p)$ are given by superpositions of the eigenfunctions of those non-singular modes, which is governed by a (possibly unstable) multi-branched dispersion relation given by $k \to -k c$ for all non-singular modes $c$ of the $k$-th Fourier modes in $x_1$. According to the above spectral analysis, this dispersion relation is asymptotic to that of the linear irrotational capillary gravity wave for $|k|\gg1$. In the case of $x_1\in \R$, in the absence of singular modes, all non-singular modes are  given by $c^\pm (k)$  which are neutral/stable. 
The dispersion of $(v^p, \eta^p)$ implies that it should decay if $x_1 \in \R$, but at a slower rate. 
Hence the dynamics of \eqref{E:LEuler} has two layers: {\it faster inviscid decay of $(v^c, \eta^c)$ leaves the remaining $(v^p, \eta^p)$ decaying at a slower rate due to the dispersion like a linear irrotational capillary gravity wave. }

In the periodic-in-$x_1$ case, as the non-existence of singular modes is assumed only for $k \in \frac {2\pi}L \mathbb{N}$, there can still be other non-singular modes besides $c^\pm (k)$ which may have bifurcated from  inflection values of $U([-h, 0])$ (as well as unstable modes from $U(-h)$) at some $k \notin \frac {2\pi}L \mathbb{N}$. In particular instability may appear in finitely many dimensions in low wave numbers.

\subsection{Outline of the proofs.} 

In the preliminary analysis in Subsection \ref{SS:Pre}, we first apply the Fourier transform in $x_1$ to \eqref{E:LEuler}, resulting in decoupled systems for each wave number $k$. The problem can be further reduced to the evolution of $\hat v_2(t, k, x_2)$, the Fourier transform of $v_2$. The Laplacian transform\footnote{Working on the Laplace transform of the unknowns is essentially equivalent to analyzing the resolvent of the linear operator in \eqref{E:LEuler}.} $V_2(k, c, x_2)$ of $\hat v_2 (t, k, x_2)$, where $s= -ikc$ is the Laplace transform variable, satisfies a non-homogeneous boundary value problem \eqref{E:Ray} of the Rayleigh equation, solutions to the associated homogeneous problem of which correspond to eigenvalues and eigenfunctions. 

A detailed analysis of the homogeneous Rayleigh equation \eqref{E:Ray-H1-1}, carried out in Section \ref{S:Ray-Homo}, lays the foundation of the study of both the eigenvalue distribution and the inviscid damping. The dependence of the estimates of the solutions on the wave number $k$ is also carefully tracked.  

We first study the Rayleigh equation away from the singularity where $|U(x_2) -c| \ge O(\mu)$, $\mu =\frac 1{\langle k \rangle} = (1+k^2)^{-\frac 12}$. Near the singularity where $|U(x_2) -c| \le O(\mu)$, different from those in, e.g., \cite{WZZ18, Jia20S}, our approach is an improved version of the technique in \cite{BSWZ16} based on the ODE blow-up and invariant manifold method. Through a transformation, solutions to the  homogeneous Rayleigh equation near $U(x_2)-c=0$ are expressed in a form involving the explicit $\log (U-c)$ and the heaviside function with coefficients smooth in $(k, c_R, x_2)$. 

We focus on a pair of fundamental solutions $y_\pm (k, c, x_2)$ to the  homogeneous Rayleigh equation which satisfy the corresponding homogeneous boundary conditions \eqref{E:Ray-2}-\eqref{E:Ray-3} in \eqref{E:Ray} at $x_2=0, -h$, respectively (boundary condition \eqref{E:Ray-3} reflects {\it the free boundary} setting). For $y_\pm$, we establish a.) their a priori bounds; b.) the convergence to their limits $y_{0\pm}(k, c_R, x_2)$ as $c_I \to 0+$; and c.) the smoothness of $y_{0\pm}$, particularly, in $c_R$. Recall $U \in C^{l_0}$, we prove $y_{0\pm}$ is $C^{l_0-2}$ in $c_R$ except at $c_R= U(-h), U(0)$. Due to the analyticity of $y_\pm$ in $c$ with $c_I>0$, the estimates of $y_{0\pm}$ also yield those of $y_\pm$ for $c_I>0$. Eventually general solutions to the non-homogeneous boundary value problem \eqref{E:Ray} of the Rayleigh equation are expressed using $y_\pm$. Finally, the quantity $Y(k, c) = \p_{x_2} y_- (k, c, 0)/y_-(k, c, 0)$ related to the Reynolds stress is carefully studied, which plays an important role in the analysis of the Rayleigh equation.  

The analysis of the Rayleigh equation near the singularity  based on a canonical form presented in Section \ref{S:Ray-Homo}  gives the most detailed information of the solutions to this singular ODE. In the representation formula \eqref{E:Ray-H0-2}, additional to the explicit form of the singular part of the solutions, the Taylor expansion of the smooth transformation $B(\cdot)_{2\times 2}$ can be carried out to an arbitrary order if needed. The section is a lengthy, but we believe this technique applied to the Rayleigh equation is widely useful for various purposes. In a forthcoming work we are studying the linearized Euler equation at a non-monotonic shear flow with a similar approach.  

In Section \ref{S:e-values} we prove the results on the eigenvalue distribution based on the detailed analysis in Section \ref{S:Ray-Homo}. We first obtain $c^\pm (k)$ for $|k|\gg1$, followed by an argument based on analytic continuation and index calculation. Bifurcations may occur at inflection values of $U$ and particularly subtle at $c= U(-h)$, which are on the boundary of the analyticity of the bifurcation equation $F(k, c)=0$. The regularity obtained in Section \ref{S:Ray-Homo} implies, when restricted to $c_I\ge 0$, $F\in C^{l_0-2}$ near $c \in U\big((-h, 0)\big)$ and $F \in C^{1, \alpha}$ near $c=U(-h)$. This makes the bifurcation analysis possible near $c=U(-h)$ and much easier even in the relatively classical case near inflection values of $U$.  

Among the results in Theorem \ref{T:e-values}, in statement (1), $c^\pm(k)$ are obtained for large $|k|$ in Lemma \ref{L:ev-large-k}(3) with more detailed estimates, the extension of $c^\pm (k)$ in Corollary \ref{C:branches-1}, and the semi-simplicity of the eigenvalues $-ikc^\pm(k)$ in Lemma \ref{L:ev-large-k}(3), Proposition \ref{P:e-v-0}, and Corollaries \ref{C:branches-1} and \ref{C:v^c}. 
Under the additional assumption of non-existence of singular modes, the non-existence of other non-singular modes is proved in Proposition \ref{P:e-v-0}. The analysis of the critical points of $c^\pm(k)$ is given in  Lemmas \ref{L:c(k)-mono}. The conjugacy to the linearized irrotational waves is proved in Proposition \ref{P:conjugacy}. See also Remark \ref{R:conjugacy}. 
The existence of $g_\#$ is  proved in Lemma \ref{L:g-thresh}, along with the existence of $k_\#$ and/or $k_\#^\pm$ in statement (3). The rest of statement (3) is proved at the end of Subsection \ref{SS:ev-details} after a series of lemmas.  
Statement (4) is proved in Subsection \ref{SS:inflectionV} with more details.

Under the assumption of the absence of singular modes, general solutions $y_B(k, c, x_2)$ to the non-homogeneous boundary value problem \eqref{E:Ray} of the Rayleigh equation are studied in Section \ref{S:Ray-BC}, which are expressed in the variation of parameter formula using $y_\pm$ obtained in Section \ref{S:Ray-Homo}. We establish the basic a priori and convergence (as $c_I \to 0+$) estimates in Subsection \ref{SS:Ray-NH-HBC}. The latter is often referred to as the limiting absorption principle (e.g. \cite{WZZ19, Jia20S}). For the inviscid damping estimates, it is crucial to obtain the smoothness of $y_B$ in $c$ (in Subsection \ref{SS:Ray-NH-CGW}). Since singularity occurs in the Rayleigh equation along $c = U(x_2)$, $\p_c^j y_B$, $j=1,2$, behaves badly there. Instead we apply a differential operator $D_c$ to the Rayleigh system \eqref{E:Ray} which differentiates along the direction of $c_R = U(x_2)$, hence $D_c^j y_B$ satisfies another boundary value problem of the Rayleigh equation in the form of \eqref{E:Ray} and enjoys reasonable estimates.  Essentially this approach is similar to those used in \cite{WZZ18, Jia20A, WZZhu20} for the Euler equation on fixed channels. The main results of Subsection \ref{SS:Ray-NH-CGW} are the estimates of $\p_c^{j_1} \p_{x_2} y_B$, $j_1=1, 2$ and $j_2=0,1$, with the most singular terms identified.     

The splitting and the linear inviscid damping estimates of solutions $(v, \eta)$ to the linearized capillary gravity waves \eqref{E:LEuler} are obtained in Section \ref{S:Linear}. While the vorticity $\omega$ is not sufficient to recover the whole solution (which is different from the fixed channel case as in e.g. \cite{Zi17, WZZ18}), the solutions are expressed in terms of the inverse Laplace transform of $V_2(k, c, x_2)$, the Laplace transform of $\hat v_2(t, k, x_2)$, which is estimated in Section \ref{S:Ray-BC}. Unlike e.g. \cite{WZZ18, Jia20S}, technically we do not immediately push the contour integral (in $c$) of the inverse Laplace transform to the limit spectra set $U([-h, 0])$, but first keep it along the boundary of a small neighborhood of it in the complex plane. This allows easy integration by parts in $c$ to establish the decay estimates in $t$ after the leading asymptotic terms are obtained by applying the Cauchy integral theorem to the most singular terms of $\p_c^j V_2$, $j=1,2$. In fact, in deriving the decay estimates of $v$, $\eta$, $tv_2$, and $t \eta$ where the leading asymptotic terms were not involved, a priori estimates, but not the limiting absorption principle, is sufficient. 

The above approach to obtain the inviscid decay also applies to the Euler equation in a fixed channel 
linearized at a shear flow $U(x_2)$. Similarly, while the asymptotic leading order terms of $tv_1$, $\omega$, and $\p_{x_2}^2 v_2$ are all generated by the asymptotic vorticity $\Omega^c$, that of $t^2 v_2$ involves two additional functions $\Omega_T$ and $\Omega_B$ due to the contributions from the top and bottom boundaries. 
We give a brief summary of the results for the channel flow in Subsection 
\ref{SS:Euler-Channel} and see also Remark \ref{R:boundryC}.\\

{\bf Notation}: Throughout the paper, $C>0$ denotes a generic constant which might change from line to line, but always independent of $k$, $c$, and $x_2$; $\delta(x)$ the delta function; $P.V.$ (or $(P.V.)_c$) the principle value (or the principle value with respect to variable $c$ {\it etc.}). The Japanese bracket $\langle k \rangle = \sqrt{k^2+1}$ is adopted. For $c= c_R + ic_I$ close to $U([-h, 0])$, $x_2^c$ denotes $U^{-1} (c_R)$ after some extension of $U$. We always denote $\mu =\langle k \rangle^{-1}$ and $' = \p_{x_2}$.

\section{Main results and preliminaries} \label{S:MainT-Pre}

In this section we give the precise statements of linear inviscid damping, along with some preliminary analysis. It is well-known that the pressure $p$ is determined by $v$ and $\eta$ (see \eqref{E:LEuler-2} and \eqref{E:LEuler-5}), so very often we shall focus only on $v$ and $\eta$.

\subsection{A brief motivational study of the Couette flow $U(x_2)=x_2$} \label{SS:Couette} 

We first describe two main relevant properties using the Couette flow as an illustration. The linearized velocity can be decomposed uniquely into the rotational and irrotational/potential parts (see e.g. \cite{SZ08a})  
\[
v = v^{ir} + v^{rot}, \ \text{ where } \, \nabla \cdot v^{ir, rot}=0, \; 
\] 
where 
\[
v^{ir} = \nabla \varphi, \quad \Delta \varphi =0, \;\; x_2 \in (-h, 0), \, \text{ and }\, \p_{x_2} \varphi|_{x_2=-h} =0,
\]
and $v^{rot}$ satisfies 
\[
\nabla \cdot v^{rot} =0, \quad v_2^{rot} |_{x_2=-h, 0}=0.  
\]
In particular, the rotational part can almost be determined by the vorticity $\omega$ in the {\it same} way as in the Euler equation \eqref{E:E-channel} in the fixed channel $x_2 \in (-h, 0)$ with slip boundary condition 
\be \label{E:vor}
v^{rot} = ( -\p_{x_2}, \p_{x_1})^T \Delta^{-1} \omega + (a, 0)^T, \, \text{ and } \, \omega =\nabla \times v= \p_{x_1} v_2 - \p_{x_2} v_1,
\ee
where $a$ is a constant and $\Delta^{-1}$ is the inverse Laplacian in the 2-d region $x_2 \in (-h, 0)$ ($L$-periodic in $x_1$ or  $x_1 \in \R$) under the zero Dirichlet boundary condition along $x_2=-h, 0$. In the periodic-in-$x_1$ case, the constant $a$ may be non-zero and is determined by the physical quantity circulation.  

{\bf I.} {\it Inviscid damping.} For the 2-d Euler equation \eqref{E:Euler-1}, one often also consider the corresponding vorticity formulation 
\be \label{E:Euler-vor}
\p_t \omega + v\cdot \omega =0. 
\ee
Linearizing it at $\omega_* = - 1$, which is the vorticity of the Couette flow, yields the linearized vorticity 
\[
\omega (t, x) = \omega_0 (x_1 -x_2t, x_2)
\] 
expressed in term of its initial  value $\omega_0$. Since $v^{rot}$ component of the linearized capillary gravity waves \eqref{E:LEuler} at the Coutte flow corresponds to the divergence free velocity field determined by its vorticity $\omega$ by \eqref{E:vor} which is the same way  as in the fixed boundary problem of the channel flow, the inviscid damping \eqref{E:decay-0} of the latter (in the periodic-in-$x_1$ case) implies 
\[
\Big|v^{rot} - \frac 1L\Big( \int_{-\frac L2}^{\frac L2} v_1 dx_1\Big) \vec{\bf e}_1  \Big|_{L^2} \le C(1+|t|)^{-1} |\omega_0|_{H^2}, \quad |v_2^{rot}|_{L^2} \le C(1+|t|)^{-2} |\omega_0|_{H^2}.
\]

{\bf II. } {\it Singular and non-singular modes.} Unlike the Euler equation in a fixed channel, there is the additional surface profile $\eta$ coupled to the irrotational part $v^{ir}$ of the velocity, which may not decay. In fact, for any $k \in \R$, let 
\[
v (t, x) = (1+ k^2)^{\frac 14} e^{ik(x_1- c^\pm (k) t) -|k|h} \big(i\cosh k(x_2+h), \sinh k (x_2+h) \big) + c.c. 
\]
\[
\eta (t, x_1) =  i (1+ k^2)^{\frac 14} e^{ik(x_1-c^\pm (k) t) -|k|h} \sinh k h/ (k c^\pm (k)) + c.c.,  
\]
\begin{align*}
p(t, x) = & i (1+ k^2)^{\frac 14} e^{ik(x_1- c^\pm (k) t) -|k|h} \Big(  (g+\sigma k^2)\frac {\sinh kh}{k c^\pm (k)} - k  \int_0^{x_2} (x_2'-c^\pm(k)) \sinh k(x_2'+h) dx_2' \Big) + c.c. \\
=& i (1+ k^2)^{\frac 14} e^{ik(x_1- c^\pm (k) t) -|k|h} \Big(  (g+\sigma k^2)\frac {\sinh kh}{k c^\pm (k)} - (x_2-c^\pm(k)) \cosh k(x_2+h) \\
& - c^\pm (k) \cosh kh +  k^{-1} ( \sinh k(x_2+h) -\sinh kh) \Big) + c.c.
\end{align*}
where ``c.c.'' denotes ``complex conjugates'' and 
\be \label{E:dispersion-Q}
c^\pm (k) = \frac {-1 \pm \sqrt{1+ 4k (g+\sigma k^2) \coth kh}}{2k \coth kh} \implies F(k, c)|_{c^\pm (k)}\triangleq (c^2 k \coth kh + c -(g+\sigma k^2))|_{c^\pm (k)}=0. 
\ee
Even though we write down these formulas based on Lemma \ref{L:e-value} in the below, it is straight forward to verify that they are solutions to (\ref{E:LEuler-1}--\ref{E:LEuler-5}) for the Couette flow. Therefore $-ikc^\pm(k)$ are eigenvalues of the linearized systems associated with the above eigenfunctions. As these solutions do not grow or decay as $t\to \infty$, $c^\pm (k)$ are 
neutral modes. 

It is worth paying slightly closer attention to the wave speed $c^\pm (k)$ and the function $F(k, c)$, all of which are even in $k$. We make the following observations.
\begin{enumerate}
\item $\lim_{k\to \infty} c^\pm (k) /(\sigma |k|)^{\frac 12} =\pm 1$, so for $|k|\gg1$ the dispersion relation $ k \to - kc^\pm(k)$ is asymptotic to those of the irrotational capillary gravity waves linearized at zero solution (system \eqref{E:LEuler} with $U\equiv 0$ and $\nabla \times v\equiv 0$) given by $-k c_{ir}^\pm$ with   
\be \label{E:dispersion-F}
c_{ir}^\pm (k) = \pm \sqrt{ k^{-1} (g+\sigma k^2) \tanh kh}, \quad C^{-1} \le |c_{ir}^\pm (k)| \le C (1+ k^2)^{\frac 14},  
\ee
which can be obtained through direct calculation based on the Fourier transform. 
\item $c^+(k)>0$ for all $k \in \R$, so it is a branch of non-singular neutral modes, namely, wave speeds outside $[-h, 0]$, the range of $U$. 
\item While $c^-(k) <-h$ in \eqref{E:dispersion-Q} as seen in the above observation (1) for large $k$, it can happen $c^-(k) \in [-h, 0]$ for $0< g, \sigma \ll 1$ and thus becomes singular modes (those in the range of $U$). 
\item Since $k \coth kh \ge h^{-1}$ with ``$=$'' achieved at $k=0$, for $g, \sigma \gg1$, $c^\pm (k) \sim c_{ir}^\pm (k)= \sqrt{\frac {g+\sigma k^2}{k \coth kh}}$ and thus both $c^\pm (k) \notin [-h, 0]$ are non-singular modes. Moreover, one may verify $\frac d{dk} |c^\pm (k)| >0$ for all $k>0$ if $\sigma \gg g \gg 1$. In particular, in the case of $x_1 \in \R$, this implies that a.) the dispersion relations $k \to - k c^\pm (k)$ determine a linear dispersive wave system formed by the superposition of these non-singular modes and b.) this dispersive system is conjugate to the irrotational capillary gravity waves linearized at zero, whose the wave speed is given by \eqref{E:dispersion-F}.  The conjugacy isomorphism can be constructed by associating the modes $k_1^\pm$ of \eqref{E:dispersion-Q} and $k_2^\pm$ of \eqref{E:dispersion-F} if they have the same temporal frequency $k_1^\pm c^\pm (k_1^\pm) = k_2^\pm c_{ir}^\pm (k_2^\pm)$. Moreover, $-ik c^\pm (k)$ would turn out to the only eigenvalues for the linearization at the Couette flow for $g, \sigma \gg1$ (see Proposition \ref{P:e-v-0}(2)). 
\end{enumerate}

{\it Generalization of the linear analysis to a general shear flow $U(x_2)$?} 
From the above discussion, one sees that solutions to  the capillary gravity water waves  linearized at the Couette flow exhibit inviscid damping in  their rotational parts while there are infinite many non-singular modes with irrotational eigenfunctions whose evolution is determined by two branches of dispersion relations. However, several complications arise in the linearization at a general shear flow $U(x_2)$ including at least the following. 
\begin{itemize} 
\item The crucial function $F(k, c)$ defined in \eqref{E:dispersion-Q} which determines the wave speed $c$ and consequently the dispersion relations, while analytic for $c \in \C \setminus U([-h, 0])$, may become rather singular for $c$ approaching $U([-h, 0])$. What regularity of $F(k, c)$ can one expect?  
\item Consequently, if a branch of non-singular modes approaches $U([-h, 0])$, possibly very subtle bifurcations may occur at the boundary of analyticity of $F$. Can instability be generated?
\item The linear inviscid damping (still of the rotational parts?) becomes much more involved, even in the case of the channel flow (see e.g. \cite{Zi16, WZZ18, Jia20S}).  
\end{itemize}    
In the rest of this paper, we address these issues, with some results even more explicit and detailed than the above, through careful analysis starting at rather fundamental level under reasonable assumptions.

\subsection{Main theorems on the invariant splitting and linear inviscid damping} \label{SS:MainT}

In this subsection,  assuming there are no singular modes, we present the theorems on the inviscid damping of linearized system  \eqref{E:LEuler} of the capillary gravity water wave problem \eqref{E:Euler} at the shear flow $(v_*, S_*, p_*)$. See Definition \ref{D:modes} Lemma \ref{L:e-v-basic-1}(5), \eqref{E:no-S-M}, and Remark \ref{R:singularM} for singular and non-singular modes. 
In this case, we shall prove that any linear solution $(v, \eta)$ to \eqref{E:LEuler} can be decomposed into the component $(v^p, \eta^p)$ corresponding to the non-singular modes and $(v^c, \eta^c)$ to the continuous spectra due to $U([-h, 0])$. This splitting is invariant under \eqref{E:LEuler} and $(v^c, \eta^c)$ is of the order $O(|t|^{-1})$ (and the vertical component $v_2^c = O(t^{-2})$) as $ |t| \to \infty$. In fact, we identify their asymptotic  leading  order terms so that the remainders decay even faster. These leading order terms are in the form of horizontal translations of three functions $\Omega^c$, $\Lambda_B$, and $\Lambda_T$, which represent the contributions from the interior vorticity and the bottom and top boundary conditions. Their Fourier transforms are given explicitly in \eqref{E:Omega-1}, \eqref{E:Lambda_B}, and \eqref{E:Lambda_T}, respectively,  using the initial vorticity $\omega_0$, the fundamental solutions $y_\pm(k,c, x_2)$ to the homogeneous Rayleigh equation, and $\Omega^c$ also by the Laplace transform of $v_2$. 
The results are stated for the cases of $x_1\in \mathbb{T}_L$ and $x_1\in \R$ separately in the following. 

\begin{theorem} \label{T:decay-per} ({\bf Inviscid damping: periodic-in-$x_1$ case})
Suppose $x_1 \in \mathbb{T}_L$. Assume $U\in C^{l_0}$, $l_0\ge 3$, $U'>0$ on $[-h, 0]$, and there are no singular modes (see \eqref{E:no-S-M}  and Lemma \ref{L:e-v-basic-1}(5)) for any $k \in \frac {2\pi}L \mathbb{N}$. For any $q_1\in [2, \infty]$, $q_2\in (2, \infty]$, and $\ep>0$, there exists $C>0$ depending only on $q_1$, $q_2$, $\ep$, and $U$, such that, for any $n_1\in \R$, integer $n_0\ge 0$, and solution $(v(t, x), \eta(t, x_1))$ of \eqref{E:LEuler} with initial value $(v_0(x), \eta_0(x_1))$ and the corresponding initial vorticity $\omega_0(x)$, there exist unique solutions $(v^\dagger (t, x), \eta^\dagger (t, x_1))$, $\dagger =p, c$, to \eqref{E:LEuler} and $L$-periodic-in-$x_1$ functions $\Omega^c(x)$, $\Lambda_B(x)$, and $\Lambda_T(x)$ determined by $(v_0, \eta_0)$ linearly  such that
\[
(v, \eta) = (v^c, \eta^c) + (v^p, \eta^p)
\]
and the following hold. 
\begin{enumerate}
\item  $(v^c, \eta^c)$ satisfy the following estimates
\[
|\p_t^{n_0} v^c|_{H_{x_1}^{n_1} L_{x_2}^2 L_t^{q_1} (\R)} \le  C  \big( |\eta_0|_{H_{x_1}^{n_0+n_1+\frac 12 -\frac 1{q_1}}} + |v_{10} (\cdot, 0)|_{H_{x_1}^{n_0+n_1-\frac 32-\frac 1{q_1}}} + |\omega_0|_{H_{x_1}^{n_0+n_1-\frac 12 -\frac 1{q_1} +\ep} L_{x_2}^2} \big),
\]
\[
|\p_t^{n_0} \eta^c|_{H_{x_1}^{n_1} L_t^{q_1} (\R)} \le C  \big( |\eta_0|_{H_{x_1}^{n_0+n_1 - 1 -\frac 1{q_1}}} + |v_{10} (\cdot, 0)|_{H_{x_1}^{n_0+n_1- 2-\frac 1{q_1}}} + |\omega_0|_{H_{x_1}^{n_0+n_1- 2 -\frac 1{q_1} +\ep} L_{x_2}^2} \big);
\]
if $U\in C^4$, 
\begin{align*}
\big| t\p_t^{n_0} v_2^c  \big|_{H_{x_1}^{n_1-\frac 32} L_{x_2}^2  L_t^{q_1} (\R)} + |t \p_t^{n_0} \eta^c|_{H_{x_1}^{n_1} L_t^{q_1} (\R)} \le & C  \big( |\eta_0|_{H_{x_1}^{n_0+n_1 - 1 -\frac 1{q_1}}} + | v_{10} (\cdot, 0)|_{H_{x_1}^{n_0+n_1- 3-\frac 1{q_1}}} \\
&+ |\omega_0|_{H_{x_1}^{n_0+n_1- 2 -\frac 1{q_1} +\ep} L_{x_2}^2} + |\p_{x_2} \omega_0|_{H_{x_1}^{n_0+n_1- 3 -\frac 1{q_1} +\ep} L_{x_2}^2} \big), 
\end{align*}
\begin{align*}
&\big|\p_t^{n_0} \big( t v_1^c - U'(x_2)^{-1} \p_{x_1}^{-1}\Omega^c (x_1- U(x_2)t, x_2)\big) \big|_{H_{x_1}^{n_1} L_{x_2}^2  L_t^{q_2} (\R)} \\
&
+ \big|\p_t^{n_0} \big( \omega^c - \Omega^c (x_1- U(x_2)t, x_2) \big) \big|_{H_{x_1}^{n_1-1} L_{x_2}^2  L_t^{q_2} (\R)} \\
&+  \big|\p_t^{n_0} \big( \p_{x_2}^2 v_2^c - \p_{x_1} \Omega^c (x_1- U(x_2)t, x_2)\big) \big|_{H_{x_1}^{n_1-2} L_{x_2}^2  L_t^{q_2} (\R)} \\
\le & C  \big(  |\eta_0|_{H_{x_1}^{n_0+n_1+\frac 12 -\frac 1{q_2}}} + |v_{10} (\cdot, 0)|_{H_{x_1}^{n_0+n_1-\frac 32-\frac 1{q_2}}}  + |\omega_0|_{H_{x_1}^{n_0+n_1-\frac 12 -\frac 1{q_2} +\ep} L_{x_2}^2}+ |\p_{x_2} \omega_0|_{H_{x_1}^{n_0+n_1-\frac 32 -\frac 1{q_2} +\ep} L_{x_2}^2} \big); 
\end{align*}
and if, in addition, $U \in C^5$, then 
\begin{align*}
&\big|\p_t^{n_0} \big( t^2 v_2^c - U'(x_2)^{-2} \p_{x_1}^{-1}\Omega^c (x_1- U(x_2)t, x_2) 
 - \Lambda_B (x_1- U(-h)t, x_2)  
- \Lambda_T (x_1- U(0)t, x_2) \big) \big|_{H_{x_1}^{n_1} L_{x_2}^2  L_t^{q_2} (\R)}  \\ 
\le & C  \big(  |\eta_0|_{H_{x_1}^{n_0+n_1+\frac 12 -\frac 1{q_2}}} + |v_{10} (\cdot, 0)|_{H_{x_1}^{n_0+n_1-\frac 32-\frac 1{q_2}}}  + |\omega_0|_{H_{x_1}^{n_0+n_1-\frac 12 -\frac 1{q_2} +\ep} L_{x_2}^2}\\
&  \qquad \qquad \qquad \qquad \qquad \qquad + |\p_{x_2} \omega_0|_{H_{x_1}^{n_0+n_1-\frac 32 -\frac 1{q_2} +\ep} L_{x_2}^2} + |\p_{x_2}^2 \omega_0|_{H_{x_1}^{n_0+n_1-\frac 52 -\frac 1{q_2} +\ep} L_{x_2}^2} \big).
\end{align*}
\item  
 $\Omega^c$ and $\Lambda_\dagger$, $\dagger=B, T$, satisfy
\[
 |\Omega^c -  \omega_0|_{H_{x_1}^{n_1} L_{x_2}^2} \le C  \big(  |\eta_0|_{H_{x_1}^{n_1 }} + |v_{10} (\cdot, 0)|_{H_{x_1}^{n_1-2}}  + |\omega_0|_{H_{x_1}^{n_1-1 +\ep} L_{x_2}^2} \big),
\]
\[
|k \hat \Lambda_B(k, \cdot)|_{L_{x_2}^q} \le C \langle k\rangle^{-\frac 1q} |\hat \omega_0(k, -h)|, 
\quad |k \hat \Lambda_T(k, \cdot)|_{L_{x_2}^q} \le C \langle k\rangle^{-\frac 1q} ( |\hat \omega_0(k, 0)| + |\hat \eta_0(k)|), \quad \; \forall q\in [1, \infty],
\]
\[
|k\p_{x_2}  \hat \Lambda_B(k, \cdot)|_{L_{x_2}^q} \le C \langle k\rangle^{1-\frac 1q} |\hat \omega_0(k, -h)|, 
\;\; |k \p_{x_2} \hat \Lambda_T(k, \cdot)|_{L_{x_2}^q} \le C \langle k\rangle^{1-\frac 1q} ( |\hat \omega_0(k, 0)| + |\hat \eta_0(k)|), \;\; \forall q\in [1, \infty),
\]
where $\hat f (k, x_2)$ denotes the Fourier transform of a function $f(x_1, x_2)$ with respect to $x_1$.  Moreover, $\Lambda_\dagger$, $\dagger=B, T$, satisfy $\hat \Lambda_\dagger (k=0, x_2) =0$ and  
\begin{subequations} \label{E:Lambda-BVP}
\be \label{E:LambdaT-BVP} \begin{cases}  
- (U- U(0)) \Delta \Lambda_T + U'' \Lambda_T =0, \qquad \qquad \qquad \qquad \qquad \qquad \qquad x_2 \in (-h, 0), \\ 
\Lambda_T (x_1, -h) = 0, \qquad U'(0)^{2}\p_{x_1} \Lambda_T (x_1, 0) =  U''(0) \eta_0 (x_1, 0) - \omega_0(x_1, 0); 
\end{cases}\ee
\be \label{E:LambdaB-BVP} \begin{cases}  
- (U- U(-h)) \Delta \Lambda_B + U'' \Lambda_B =0, \qquad \qquad \qquad \qquad \qquad \qquad \qquad x_2 \in (-h, 0), \\ 
U'(-h)^{2} \p_{x_1} \Lambda_B (\cdot, -h) = - \omega_0(x_1, -h), \\
\big(U(0)-U(-h)\big)^2 \p_{x_2} \Lambda_B (x_1, 0) - \big(U'(0) (U(0)-U(-h)) + g - \sigma \p_{x_1}^2\big) \Lambda_B (x_1, 0) =0. 
\end{cases}\ee
If $U\in C^4$, then 
\begin{align*}
|\p_{x_2} \Omega^c -  \p_{x_2} \omega_0|_{H_{x_1}^{n_1} L_{x_2}^2} 
\le C  \big(  |\eta_0|_{H_{x_1}^{n_1+1 }} + |v_{10} (\cdot, 0)|_{H_{x_1}^{n_1-1}}  + |\omega_0|_{H_{x_1}^{n_1 +\ep} L_{x_2}^2}+ |\p_{x_2} \omega_0|_{H_{x_1}^{n_1-1 +\ep} L_{x_2}^2} \big).
\end{align*}
\end{subequations}
\item
There exist $\lambda_0 \ge 0$ and integer $N \ge 0$ (given in \eqref{E:lambda&N})
such that, for any $n_1\in \R$ and integer $n_2\in [1, l_0]$, 
\begin{align*}
& \big|\p_{x_1}^{n_1+1} \big(v_1^p(t, \cdot) - \hat v_{10} (k=0, \cdot)\big)\big|_{L_{x_1}^2 H_{x_2}^{n_2-1}} 
+ |\p_{x_1}^{n_1} v_2^p(t, \cdot)|_{L_{x_1}^2 H_{x_2}^{n_2}} \\
\le & Ce^{\lambda_0 |t|} (1+ |t|^{N-1}) \big(|\eta_0|_{H_{x_1}^{n_1+n_2+1}} + |v_{10} (\cdot, 0)|_{H_{x_1}^{n_1+n_2-\frac 12}} + |\omega_0|_{H_{x_1}^{n_1+n_2-1}L_{x_2}^2}\big), 
\end{align*}
\[
|\eta^p(t, \cdot) - \eta_0(0)|_{H_{x_1}^{n_1}} \le C e^{\lambda_0 |t|} (1+ |t|^{N-1})  \big(|\eta_0|_{H_{x_1}^{n_1}} + |v_{10} (\cdot, 0)|_{H_{x_1}^{n_1-\frac 32}} + |\omega_0|_{H_{x_1}^{n_1-2}L_{x_2}^2}\big). 
\]
\item Let 
\[
\BX^\dagger = \{ (v^\dagger, \eta^\dagger)|_{t=0} \mid \text{ all } (v_0, \eta_0)\} \subset H^1 \big(\mathbb{T}_L \times (-h, 0)\big) \times H^2 (\mathbb{T}_L ), \quad \dagger = c, p,
\]
then they are closed invariant subspaces of $H^1 \big(\mathbb{T}_L \times (-h, 0)\big) \times H^2 (\mathbb{T}_L )$ under \eqref{E:LEuler}. Moreover \eqref{E:LEuler} is also well-posed in the $L^2 \times H^1$ completion of $\BX^p$.
\end{enumerate}
\end{theorem}

\begin{remark}
1.) Observe that, for any compactly supported smooth function $f(t, x)$, $q\ge1$, it holds  
\[
|f(t_0, x)|^q  \le q \int_{t_0}^{+\infty} |f(t', x)|^{q-1} |\p_t f(t', x)| dt' \le q |f(\cdot, x)|_{L_t^q ([t_0, +\infty))}^{q-1}  |\p_t f(\cdot, x)|_{L_t^q ([t_0, +\infty))},  
\]
which implies 
\begin{align*}
|f(t_0, \cdot)|_{L_x^2} = &  \Big(\int |f(t_0, x)|^2 dx\Big)^{\frac 12} \le q^{\frac 1q} \Big( \int |f(\cdot, x)|_{L_t^q ([t_0, +\infty))}^{\frac {2(q-1)}q}  |\p_t f(\cdot, x)|_{L_t^q ([t_0, +\infty))}^{\frac 2q}  dx\Big)^{\frac 12} \\
\le & q^{\frac 1q}  |f|_{L_x^2 L_t^q ([t_0, +\infty))}^{\frac {q-1}q}  |\p_t f|_{L_x^2 L_t^q ([t_0, +\infty))}^{\frac 1q}. 
\end{align*}
By the standard density argument, this inequality also holds for any function $f \in L_x^2 W_t^{1,q}$.
Hence the above estimates in statement (1) also imply various pointwise-in-$t$ decay of $v^c, \eta^c \in L_x^2$ as $t \to \infty$. \\
2.) The function $\Omega^c (x)$ is referred to as the scattering limit of the vorticity in \cite{Zi17, WZZ18, Jia20S}. \\
3.) The assumption of non-existence of singular modes is satisfied if the horizontal period $L$ is small (by Theorem \ref{T:e-values}(1) as $\frac {2\pi}L$ is large) or if $U'' \ne 0$ and \eqref{E:sigma-0} hold (by Theorem \ref{T:e-values}(2b)).
\end{remark}

The proof of this theorem is completed in Subsection \ref{SS:per-case}. 

From \eqref{E:Lambda_B} and \eqref{E:Lambda_T}, \eqref{E:BF}, and Lemma \ref{L:y-lower-b}(2), $\hat \Lambda_{B, T}(0, x_2)=0$ and the elliptic boundary value problem \eqref{E:LambdaB-BVP} has a unique solution $\Lambda_B$, while \eqref{E:LambdaT-BVP} has a unique solution $\Lambda_T$ under the assumption of the non-existence of singular modes. Moreover, according to the definitions \eqref{E:Lambda_T}, \eqref{E:BF}, \eqref{E:y-pm}, \eqref{E:y0}, and Lemma \ref{L:B}, $\p_{x_2} \Lambda_B$ and $\p_{x_2}\Lambda_T$ exhibit logarithmic singularity at $x_2 = -h$ and $0$, respectively. In particular, $\Lambda_B=0$ vanishes if the initial vorticity $\omega_0|_{x_2=-h}=0$, while $\Lambda_T =0$ if $U''(0) \eta_0- \omega_0|_{x_2=0}=0$. 
%
In this paper as we focus on the damping estimates with additional $L_t^q$ decay of $(v, \eta)$ after the leading order terms are singled out, we adopted $L_x^2$ based norms to somewhat simplify the calculations. If the decay in other $L_x^r$ or $L_x^\infty$ based norms is necessary, some basic estimates in these norms are also given in Subsection \ref{SS:Ray-NH-HBC} and one may make an attempt  following the procedure as in Sections \ref{S:Ray-BC} and \ref{S:Linear}. To avoid more technicality, the assumptions on the regularity of $\omega_0$ in $x_1$ in the theorem may not be close to optimal, particularly when $q_1$ and $q_2$ are away from $2$, see Remark \ref{R:contour}(b) as well as Remark \ref{R:pcy2}. Moreover, the small $\ep$ may not be necessary, see e.g. \cite{WZZ19, WZZhu20} in the fixed boundary case. The assumptions on the more essential regularity of $\omega_0$ in $x_2$ 
are optimal even in the existing results in the fixed boundary case.  

In the estimates of the component $(v^p, \eta^p)$ which are superpositions of eigenfunctions, the possible exponential growth (if $\lambda_0>0$) is caused by unstable modes, where $\lambda_0$ is the maximum real parts of the eigenvalues and $N$ is the maximum multiplicity of those eigenvalues of the maximal real parts. Due to Theorem \ref{T:e-values}(1), growth does not occur for $|k|\gg1$. It is also worth pointing out that, taking $n_2=0$, the the regularity of $\eta^p$  is $\frac 32$ order better than that of $v^p$ restricted to the surface $x_2=0$, which is consistent with the regularity results of nonlinear capillary gravity waves in the existing literature. In the contrast, the regularity requirement on $\omega_0$ in the damping estimates of $(v^c, \eta^c)$ is stronger than that of $(v^p, \eta^p)$.  Compared with the above example of the linearization at the Couette flow, conceptually these phenomena is due to the fact that the component $(v^c, \eta^c)$ is mainly the rotational part of the solution which depends on the vorticity more heavily, while $(v^p, \eta^p)$ are more like the irrotational part. 

The estimate in statement (3) at $t=0$ implies the  boundedness of the projection onto $\BX^p$, whose kernal is $\BX^c$. Some more detailed information of this projection can be found in Lemma \ref{L:Bb} and \ref{L:inv-decomp-per}. In fact the subspace $\BX^p$ is generated by the eigenfunction of all non-singular modes for all $k \in \R$. 

The inviscid decay estimates in the case of $x_1\in \R$ is slightly subtle due to the presence of small wave number $|k| \ll 1$. We use similar notations in the following theorem.

\begin{theorem} \label{T:decay-R} ({\bf Inviscid damping: $x_1 \in \R$ case})
Suppose $x_1 \in \R$. Assume $U\in C^{l_0}$, $l_0\ge 3$, $U'>0$ on $[-h, 0]$, and there are no singular modes (see \eqref{E:no-S-M}  and Lemma \ref{L:e-v-basic-1}(5)) for any $k \in \R$. For any $q_1\in [2, \infty]$, $q_2\in (2, \infty]$, and $\ep>0$, there exists $C>0$ depending only on $q_1$, $q_2$, $\ep$, and $U$, such that, for any $n_1\in \R$, integers $n_0\ge 0$, and solution $(v(t, x), \eta(t, x_1))$ of \eqref{E:LEuler} with initial value $(v_0(x), \eta_0(x_1))$, there exist solutions $(v^\dagger (t, x), \eta^\dagger (t, x_1))$, $\dagger =p, c$, to \eqref{E:LEuler} and functions $\Omega^c(x)$, $\Lambda_B(x)$, and $\Lambda_T(x)$ determined by $(v_0, \eta_0)$ linearly 
such that
\[
(v, \eta) = (v^c, \eta^c) + (v^p, \eta^p)
\]
and the following hold. 
\begin{enumerate}
\item $(v^c, \eta^c)$ satisfy the following estimates
\begin{align*}
|\p_t^{n_0} \p_{x_1}^{n_1} v_1^c|_{L_x^2 L_t^{q_1} (\R)} + |\p_t^{n_0} \p_{x_1}^{n_1-1}  (1-\p_{x_1}^2)^{\frac 12}& v_2^c|_{L_x^2 L_t^{q_1} (\R)}  \le C  \big( \big| |\p_{x_1}|^{n_0+n_1-\frac 1{q_1}} \eta_0\big|_{H_{x_1}^{\frac 12 }} \\
&+\big| |\p_{x_1}|^{n_0+n_1-\frac 1{q_1}} v_{10} (\cdot, 0)\big|_{H_{x_1}^{-\frac 32}} + \big| |\p_{x_1}|^{n_0+n_1-\frac 1{q_1}}\omega_0\big|_{H_{x_1}^{\ep-\frac 12} L_{x_2}^2} \big),
\end{align*}
\[
|\p_t^{n_0} \p_{x_1}^{n_1} \eta^c|_{L_{x_1}^2 L_t^{q_1} (\R)} \le C  \big( \big| |\p_{x_1}|^{n_0+n_1-\frac 1{q_1}} \eta_0\big|_{H_{x_1}^{-1}} +\big| |\p_{x_1}|^{n_0+n_1-\frac 1{q_1}} v_{10} (\cdot, 0)\big|_{H_{x_1}^{-2}} + \big| |\p_{x_1}|^{n_0+n_1-\frac 1{q_1}}\omega_0\big|_{H_{x_1}^{\ep-2} L_{x_2}^2} \big);
\]
if $U \in C^4$, then 
\begin{align*}
\big| t\p_t^{n_0} \p_{x_1}^{n_1} (1-&\p_{x_1}^2)^{-\frac 14} v_2^c \big|_{L_{x}^2  L_t^{q_1} (\R)}  + |t \p_t^{n_0} \p_{x_1}^{n_1+1}  
\eta^c|_{L_{x_1}^2 L_t^{q_1} (\R)} 
\le  C  \big(  | 
\eta_0|_{\dot H_{x_1}^{n_0+n_1-\frac 1{q_1}}} \\
&+\big| |\p_{x_1}|^{n_0+n_1-\frac 1{q_1}} v_{10} (\cdot, 0)\big|_{H_{x_1}^{-2}} + \big| |\p_{x_1}|^{n_0+n_1-\frac 1{q_1}}\omega_0\big|_{H_{x_1}^{\ep -1} L_{x_2}^2} 
+  \big| |\p_{x_1}|^{n_0+n_1-\frac 1{q_1}} \p_{x_2} \omega_0\big|_{H_{x_1}^{\ep-2} L_{x_2}^2}\big)
\end{align*}
\begin{align*}
&\big|\p_t^{n_0} \p_{x_1}^{n_1+1} \big( t v_1^c - U'(x_2)^{-1} \p_{x_1}^{-1}\Omega^c (x_1- U(x_2)t, x_2) \big) \big|_{L_{x}^2  L_t^{q_2} (\R)}\\
&+ \big|\p_t^{n_0} \p_{x_1}^{n_1} \big( \omega^c - \Omega^c (x_1- U(x_2)t, x_2) \big) \big|_{L_{x}^2  L_t^{q_2} (\R)} \\
&+  \big|\p_t^{n_0}\p_{x_1}^{n_1-1} \big( \p_{x_2}^2 v_2^c - \p_{x_1} \Omega^c (x_1- U(x_2)t, x_2) \big) \big|_{L_{x}^2  L_t^{q_2} (\R)} \\
\le & C  \big( \big| |\p_{x_1}|^{n_0+n_1-\frac 1{q_2}} \eta_0\big|_{H_{x_1}^{\frac 32 }}+\big| |\p_{x_1}|^{n_0+n_1-\frac 1{q_2}} v_{10} (\cdot, 0)\big|_{H_{x_1}^{-\frac 12}} + \big| |\p_{x_1}|^{n_0+n_1-\frac 1{q_2}}\omega_0\big|_{H_{x_1}^{\ep+\frac 12} L_{x_2}^2} \\
& \qquad +  \big| |\p_{x_1}|^{n_0+n_1-\frac 1{q_2}} \p_{x_2} \omega_0\big|_{H_{x_1}^{\ep-\frac 12} L_{x_2}^2}\big);
\end{align*}
and if, in addition, $U \in C^5$, then 
\begin{align*}
&\big|\p_t^{n_0}\p_{x_1}^{n_1} \big( t^2 v_2^c - U'(x_2)^{-2} \p_{x_1}^{-1}\Omega^c (x_1- U(x_2)t, x_2) 
 - \Lambda_B (x_1- U(-h)t, x_2)  \\
 & \qquad \qquad \qquad \qquad \qquad \qquad\qquad  \qquad \qquad\qquad  \qquad 
- \Lambda_T (x_1- U(0)t, x_2) \big) \big|_{L_{x}^2  L_t^{q_2} (\R)}  \\ 
\le & C  \big( \big| |\p_{x_1}|^{n_0+n_1-1-\frac 1{q_2}} \eta_0\big|_{H_{x_1}^{\frac 32 }}+\big| |\p_{x_1}|^{n_0+n_1-1-\frac 1{q_2}} v_{10} (\cdot, 0)\big|_{H_{x_1}^{-\frac 12}} + \big| |\p_{x_1}|^{n_0+n_1-1-\frac 1{q_2}}\omega_0\big|_{H_{x_1}^{\ep+\frac 12} L_{x_2}^2} \\
& \qquad +  \big| |\p_{x_1}|^{n_0+n_1-1-\frac 1{q_2}} \p_{x_2} \omega_0\big|_{H_{x_1}^{\ep-\frac 12} L_{x_2}^2}+  \big| |\p_{x_1}|^{n_0+n_1-1-\frac 1{q_2}} \p_{x_2}^2 \omega_0\big|_{H_{x_1}^{\ep-\frac 32} L_{x_2}^2} \big).
\end{align*}
\item $\Lambda_T$ and $\Lambda_B$ satisfy \eqref{E:Lambda-BVP} and the same estimates as in Theorem \ref{T:decay-per}(2). Moreover, for any $q \in [1, \infty)$, it holds  
\[
|\Omega^c -\omega_0|_{H_{x_1}^{n_1} L_{x_2}^2} \le C \big(|\eta_0|_{H_{x_1}^{n_1}} + |v_{10}(\cdot, 0)|_{H_{x_1}^{n_1-2} } + |\omega_0|_{H_{x_1}^{n_1-1+\ep}L_{x_2}^2} \big),
\]
and if $U\in C^4$, then 
\[
|\p_{x_2} \Omega^c -  \p_{x_2} \omega_0|_{H_{x_1}^{n_1} L_{x_2}^2} 
\le C  \big(  |\eta_0|_{H_{x_1}^{n_1+1 }} + |v_{10} (\cdot, 0)|_{H_{x_1}^{n_1-1}}  + |\omega_0|_{H_{x_1}^{n_1 +\ep} L_{x_2}^2}+ |\p_{x_2} \omega_0|_{H_{x_1}^{n_1-1 +\ep} L_{x_2}^2} \big).
\]
\item For any $n_1\in \R$ and the following integer $n_2$, 
\[
|\p_{x_1}^{n_1} \p_{x_2}^{n_2} v_1^p(t, \cdot) |_{L_{x}^2}^2  \le  C \big(|\p_{x_1}^{n_1}\eta_0|_{H_{x_1}^{n_2+1}}^2 + |\p_{x_1}^{n_1} v_{10} (\cdot, 0)|_{H_{x_1}^{n_2-\frac 12}}^2 + |\p_{x_1}^{n_1} \omega_0|_{H_{x_1}^{n_2-1}L_{x_2}^2}^2\big), \quad \forall n_2\in [0, l_0-1],
\]
\[
|\p_{x_1}^{n_1} \p_{x_2}^{n_2} v_2^p(t, \cdot)|_{L_{x}^2 }^2 \le C \big(|\p_{x_1}^{n_1+1} \eta_0|_{H_{x_1}^{n_2}}^2 + |\p_{x_1}^{n_1+1} v_{10} (\cdot, 0)|_{H_{x_1}^{n_2-\frac 32}}^2 + |\p_{x_1}^{n_1+1} \omega_0|_{H_{x_1}^{n_2-2}L_{x_2}^2}^2\big), \quad \forall n_2\in [0, l_0],
\]
\[
|\eta^p(t, \cdot)|_{\dot H_{x_1}^{n_1}}^2\le C \big(|\eta_0|_{\dot H_{x_1}^{n_1}}^2 + |\p_{x_1}^{n_1} v_{10} (\cdot, 0)|_{H_{x_1}^{-\frac 32}}^2 + |\p_{x_1}^{n_1}\omega_0|_{H_{x_1}^{-2}L_{x_2}^2}^2\big). 
\]
\item Let 
\[
\BX^\dagger = \{ (v^\dagger, \eta^\dagger)|_{t=0} \mid \text{ all } (v_0, \eta_0)\} \subset H^1 \big(\R \times (-h, 0)\big) \times H^2 (\R ), \quad \dagger = c, p,
\]
then they are invariant closed subspaces of $H^1 \big(\R \times (-h, 0)\big) \times H^2 (\R )$ under \eqref{E:LEuler}. Moreover \eqref{E:LEuler} is also well-posed in the $L^2 \times H^1$ completion of $\BX^p$. If, in addition, \eqref{E:c(k)-mono-1} holds for both $\pm$ and $0 \in U([-h, 0])$, then \eqref{E:LEuler} restricted to the $L^2 \times H^1$ completion of $\BX^p$, or $\BX^p \cap (H^n \times H^{n+1})$ with $n \le l_0-1$, is conjugate through an isomorphism to the irrotational capillary gravity waves  linearized at zero (characterized by its wave speed \eqref{E:dispersion-F}). 
\end{enumerate}\end{theorem}

\begin{remark} \label{R:LongW}
In the above estimates, for some $n_0$ and $n_1$, the $|\p_{x_1}|^{-s}$, $s>0$, applied to the initial values indicates some stronger decay assumptions for wave number $|k|\ll1$. 
\end{remark}

The proof of this theorem is completed in Subsection \ref{SS:R-case}. 
Most of the remarks after Theorem \ref{T:decay-per} are also valid. In particular, there are only two branches of non-singular modes corresponding to eigenvalues $ik c^\pm (k)$ of both algebraic and geometric multiplicity two, hence there is no growth at all. The conjugacy of the dynamics of $(v^p, \eta^p)$ to the linear irrotational capillary gravity waves is basically a restatement of Theorem \ref{T:e-values}(2b). 

\subsection{Preliminary linear analysis} \label{SS:Pre}

To analyze the linear system \eqref{E:LEuler}, we first reduce it to an evolution problem of the Fourier transform of $v_2$ in $x_1$, which in turn determines $v_1$, $\eta$, and $p$. We then apply the Laplace transform in $t$ to obtain a non-homogeneous boundary value problem of the well-known Rayleigh equation in $x_2 \in (-h, 0)$ with a non-homogeneous Robin type boundary condition at $x_2=0$ due the boundary conditions at the free boundary. The main analysis will focus on the Rayleigh equation.  

Consider the Fourier transforms of the unknowns $(v(t, x), \eta(t, x_1), p(t, x))$ in $x_1$ 
\[
v (x) =  \sum_{k \in \tfrac {2\pi}L \mathbb{Z}} \hat v (k, x_2) e^{ik x_1}, \;\; \eta (x_1) = \sum_{k \in \tfrac {2\pi}L \mathbb{Z}} \hat \eta (k) e^{ik x_1}, \;\; p(x)= \sum_{k \in \tfrac {2\pi}L \mathbb{Z}} \hat p (k, x_2) e^{ik x_1},  
\] 
in the case of $x_1 \in \mathbb{T}_L$ and 
\[
v(x)=\frac 1{2\pi} \int_\R \hat v (k, x_2) e^{ik x_1} dk, \;\; \eta (x_1)=\frac 1{2\pi} \int_\R \hat \eta (k) e^{ik x_1} dk, \;\; p(x)=\frac 1{2\pi} \int_\R \hat p (k, x_2) e^{ik x_1} dk,  
\]
in the case of $x_1\in \R$, where we skipped the variable $t$. The Fourier transform of the linearized system \eqref{E:LEuler} takes the form 
\be \label{E:LE-F}
    \begin{cases}
    \partial_t\hat{v} +ikU(x_2)\hat{v} + (U' (x_2) \hat v_2, 0)^T + (ik \hat p, \hat p')^T =0, \qquad ik \hat v_1 + \hat v_2' =0, \qquad \qquad & x_2 \in (-h,0)\\
    (k^2-\partial_{x_2}^2)p=2ikU'(x_2)\hat{v}_2, \qquad \qquad \qquad &x_2 \in (-h,0)\\
    \partial_t\hat{\eta}=-ikU (0)\hat{\eta}+\hat{v}_2 (t, k, x_2=0),   \\   
     \hat p (t, k, 0)=   (g+\sigma k^2)\hat{\eta}, \\
\hat{v}_2 (t, k, -h) =0,  \quad \hat p' (t, k, -h)=0, \\
    \end{cases}\ee
where $'$ denotes the derivative with respect to $x_2$ as in the rest of the paper. Due to the divergence free condition on $v$ and the boundary conditions, it is easy to see 
\be \label{E:0-mean}
\hat v_2 (t, 0, x_2)=0, \quad \hat p(t, 0, x_2)=g,  \quad \hat v_1(t, 0, x_2) = v_{10} (0, x_2), \quad \hat \eta(t, 0)= \hat \eta_0 (0). 
\ee
For $k \ne 0$, $\hat v_1$ can also be determined by $\hat v_2$ using  the divergence free condition, $\hat \eta$ by the third equation of \eqref{E:LE-F}, while $\hat p$ by $\hat v_2$ and $\hat \eta$ by solving the elliptic boundary value problem. So we shall mainly focus on $\hat v_2$. 

Combining the equation of $\hat v_2$ acted by $k^2 - \p_{x_2}^2$ and  the one of $\hat p$ acted by $\p_{x_2}$, we obtain 
\begin{subequations} \label{E:LEuler-v2} 
\be \label{E:LEuler-v2-1}
(\p_t + ik U) (k^2 - \p_{x_2}^2) \hat v_2  + ik U'' \hat v_2 =0, \quad x_2 \in (-h, 0),   
\ee
which is actually the linearized transport equation of the vorticity (as defined in \eqref{E:vor}) 
\[
\hat \omega = ik \hat v_2 - \hat v_1' = \tfrac ik (k^2 - \p_{x_2}^2) \hat v_2 
\]
 in its Fourier transform. In addition to the above equation, we need its boundary information to completely determine $\hat v_2$. Applying $\p_{x_2}$ to the first equation of \eqref{E:LE-F}, then evaluating  at $x_2=0$, and using the equation of $\hat p$, we have 
\[
(\p_t+ ik U(0)) \hat v_2' (t, k, 0) - ik U'(0) \hat v_2 (t, k, 0) +  k^2 (g + \sigma k^2) \hat \eta(t, k) =0. 
\]
Finally applying $\p_t + ik U(0)$ to the above equation and using the third equation of \eqref{E:LE-F}, we obtain  
\be \label{E:LEuler-v2-2} 
\big( (\p_t+ ik U)^2 \hat v_2'  - ik U' (\p_t+ ik U) \hat v_2  +  k^2 (g + \sigma k^2) \hat v_2 \big)\big|_{x_2=0}  =0, \quad \hat v_2 |_{x_2= -h}=0, 
\ee 
\end{subequations}  
where we also included the boundary value of $\hat v_2$ at $x_2=-h$. 

To analyze the evolutionary problem, we apply the Laplace transform $\mathcal{L}$ to the unknowns 
\be \label{E:LT-1}
V(s) = (V_1 (s), V_2(s)):=\mathcal{L}\{\hat{v}\}(s), \quad P(s):=\mathcal{L}\{p\}(s), \quad \tilde{\eta}:=\mathcal{L}\{\hat{\eta}\}(s). 
\ee
An often used change of variable for $k \ne 0$ is  
\be \label{E:s-c}
c:= is/k  =c_R+ic_I
\ee
with $c_R$ and $c_I$ being the real and imaginary parts. From \eqref{E:LEuler-v2}, our main unknown $V_2(k, c, x_2)$ satisfies the following non-homogeneous Rayleigh equation 
\begin{subequations} \label{E:Ray} 
\be \label{E:Ray-1}
-V_2'' +(k^2+\frac{U''}{U-c})V_2 =\frac{(k^2-\partial_{x_2}^2)\hat{v}_{20}}{ik(U-c)}= - \frac { \hat \omega_0 }{U-c},  \qquad \qquad x_2\in (-h,0),
\ee
where $\hat \omega = \hat \omega_0 (k, x_2)$ is the Fourier transform of the initial vorticity, 
with the obvious boundary condition 
\be \label{E:Ray-2} 
V_2 (-h)=0.
\ee
Here we skipped the $k$ and $c$ variables of $V_2$. Similarly, the Laplace transform applied to the boundary equation \eqref{E:LEuler-v2-2} 
implies 
\begin{align*}
& \big((U-c)^2V_2'-(U'(U-c)+(g+\sigma k^2))V_2\big)\big|_{x_2=0} 
= -\tfrac 1{k^2} \big( \p_t \hat v_2' - ick \hat v_2' + 2ik U \hat v_2' - ik U' \hat v_2 \big)\big|_{t=x_2=0} \\
=& -\tfrac 1{k^2} \big( (\p_t  + ik U) \hat v_2' - ik U' \hat v_2  +ik (U-c) \hat v_2'  \big)\big|_{t=x_2=0}. 
\end{align*}
Therefore we obtain 
\be \label{E:Ray-3} 
\big((U-c)^2V_2'-(U'(U-c)+(g+\sigma k^2))V_2\big)\big|_{x_2=0} =(g+\sigma k^2)\hat\eta_0 -\tfrac{i}{k}(U(0)-c)\hat{v}_{20}' (0).
\ee
\end{subequations} 
The last boundary condition can be viewed as determining the dispersion relation which is highly nonlocal. The Laplace transforms of $V_1$ and $\tilde \eta$ of $\hat v_1$ and $\hat \eta$ can be recovered from the divergence free condition and the third equation of \eqref{E:LE-F} 
\be \label{E:t-eta}
V_1 = \tfrac ik V_2', \quad  \tilde \eta (c, k) = \frac {V_2 (c, k, 0) + \hat \eta_0(k)}{ik\big( U(0)-c\big)}.
\ee
Hence in most of the paper we shall focus on the non-homogeneous boundary value problem \eqref{E:Ray} of the Rayleigh equation and then use it to obtain the eigenvalue distribution of  \eqref{E:LEuler} and the inviscid damping of its solutions. 

System \eqref{E:Ray} is a boundary value problem of a non-homogeneous second order ODE with  coefficients  analytic in $k\in \R$ and $c  \in \C \setminus U([-h, 0])$, so it has a unique solution analytic in $k$ and $c$ except for those $(k, c)$ for which the corresponding homogeneous system of \eqref{E:Ray}, where $\hat v_{20}=0$ and $\hat \eta_0 =0$, has non-trivial solutions. Such singular $(k, c)$ also give the eigenvalues of \eqref{E:Ray} in the form of $-ick$. In fact we have the following lemma. 

\begin{lemma} \label{L:e-value}
For $k \in \mathbb{R}\backslash \{0\}$, there exists a non-trivial solution $\big(c, V_2(x_2)\big)$ with  $c  \notin U([-h, 0])$ to the  corresponding  homogeneous problem of \eqref{E:Ray} (namely, with $\hat v_{20}=0$ and $\hat \eta_0 =0$)  if and only if $-ikc$ is an eigenvalue of the linearized capillary-gravity wave system \eqref{E:LEuler} associated with the linear solution in the form of \eqref{E:e-func} given by 
\be \label{E:e-function} \begin{split}
&v_1(t, x) = \frac ik e^{ik(x_1- c t)} V_2' (x_2), \quad v_2 (t, x) = e^{ik(x_1-c t)} V_2 (x_2), \quad \eta(t, x_1) = e^{ik(x_1- c t)} \frac { V_2(0)}{ik (U(0)-c)},\\ 
&p(t, x)= e^{ik(x_1- c t)} \Big(  \frac {g+\sigma k^2 }{ik (U(0)- c)}V_2(0) -ik  \int_0^{x_2} (U- c)  V_2 dx_2' \Big).   
\end{split} \ee
\end{lemma} 

\begin{proof} 
On the one hand, it is straight forward to verify that the above $v$, $\eta$, and $p$ satisfy \eqref{E:LEuler-4}, \eqref{E:LEuler-5}, $\p_{x_2} p|_{x_2=-h}=0$, and $\nabla \cdot v=0$. The Poisson equation of $p$ in \eqref{E:LEuler-2} is a consequence of the linearized Euler equation in \eqref{E:LEuler-1}, the $v_2$ equation of which is also easily verified. Hence we only need to consider the $v_1$ equation in \eqref{E:LEuler-1}. In fact, that equation holds for the above $(v, \eta, p)$ if 
\[
- (U- c) V_2' + U'  V_2 + \frac {g+\sigma k^2}{U(0)- c}  V_2(0) + k^2 \int_0^{x_2} (U- c) V_2 dx_2'  =0.
\]
The $x_2$-derivative of this function is equal to $0$ due to the Rayleigh equation \eqref{E:Ray-1} and its boundary value equal is to $0$ at $x_2=0$ due to the boundary condition \eqref{E:Ray-3}.

On the other hand, suppose $\big(k, c, v_2(t, x), \eta(t, x_1), p(t, x)\big)$ is a solution to \eqref{E:LEuler} in the form of \eqref{E:e-func} with $k \ne 0$ and $c \notin U([-h, 0])$. 
Equation \eqref{E:LEuler-v2-1} implies that $V_2$ must be a solution to the corresponding homogeneous equation of \eqref{E:Ray-1}, while \eqref{E:LEuler-v2-2} yields the homogeneous boundary conditions of the types of (\ref{E:Ray-2}-\ref{E:Ray-3}).  Therefore $(c, V_2(x_2))$ have to be homogeneous solutions to \eqref{E:Ray}. Subsequently, $v_1$ is obtained from $\nabla \cdot v=0$, $\eta$ from the third equation in \eqref{E:LE-F}, and $p$ from the $v_2$ equation in \eqref{E:LE-F} along with its boundary value at $x_2=0$. 
\end{proof}



\begin{definition} \label{D:modes}
$(k, c)$ is a non-singular mode if $c \in \C \setminus U([-h, 0])$ and there exists a non-trivial solution $V_2(x_2)$ to the corresponding homogeneous problem of \eqref{E:Ray} (thus also yields a solution to \eqref{E:LEuler} in the form of \eqref{E:e-func}). $(k, c)$  is a singular mode if  $c \in U([-h, 0])$ and there exists a $H_{x_2}^2$ solution $y(x_2)$  to 
\be \label{E:Ray-reg}
(U-c) (-y'' + k^2y) + U'' y =0
\ee
along with the corresponding homogeneous   boundary conditions of (\ref{E:Ray-2}--\ref{E:Ray-3}). (See also Remark \ref{R:singularM}.)
\end{definition}

After acquiring good understanding on the homogeneous problem of the Rayleigh equation \eqref{E:Ray} (Section \ref{S:Ray-Homo}) and its eigenvalues (Section \ref{S:e-values}), we proceed to analyze the general non-homogeneous problem of \eqref{E:Ray} (Section \ref{S:Ray-BC}), in particular, the dependence of solutions on $c$. Finally in Section \ref{S:Linear} we apply the inverse Laplace transform to estimate the solution to the linear system \eqref{E:LEuler}. Recall the inverse Laplace transform 
\be \label{E:ILT}
f(t) = \frac 1{2\pi i} \int_{\gamma - i\infty}^{\gamma +i\infty}  e^{st} F(s)  ds = \frac {|k|}{2\pi} \int_{-\infty + \frac {i\gamma}k}^{+\infty + \frac {i\gamma}k} e^{-ikc t} F(-ikc) dc, 
\ee
where $\gamma$ is a real number so that $F(s)$ is analytic in the region $\RP\, s>\gamma$ and the change of variable \eqref{E:s-c} was used in the second equality. Due to the analyticity, the integral can be eventually carried out along contours enclosing $U([-h, 0])\subset \C$ and the non-singular modes of \eqref{E:LEuler}. Assuming there is no singular modes in $U([-h, 0])$, we shall eventually obtain the decay in $t$ of the component of the linear solution corresponding to the integral along the contour surrounding  $U([-h, 0])$ by integration by parts in $c$.

\section{Analysis of the Rayleigh equation} \label{S:Ray-Homo}

In this section, we shall thoroughly analyze the homogeneous Rayleigh equation
\be \label{E:Ray-H1-1}
-y''(x_2)+ \big(k^2 + \tfrac{U''(x_2)}{U(x_2)-c}\big)y(x_2)=0, \quad x_2\in [-h,0],
\ee
where
\[
k\in \R, \quad c=c_R + i c_I \in \C, \quad '= \p_{x_2}.
\]
Throughout this section (except for some lemmas in Subsection \ref{SS:Y}), we assume 
\be \label{E:U'}
U'(x_2)>0, \quad \forall x_2 \in [-h, 0]. 
\ee
As pointed out in the introduction, due to the symmetry of the reflection in $x_1$ variable, the case of $U'<0$ can be reduced to the above one. Hence all results under \eqref{E:U'} hold for all uniformly monotonic $U(x_2)$, namely those $U$ satisfying $U'\ne 0$ on $[-h, 0]$. 

To some extent, we will also consider the non-homogeneous Rayleigh equation 
\be \label{E:Ray-NH-1}
-y''(x_2)+\big(k^2 + \tfrac{U''(x_2)}{U(x_2)-c}\big)y(x_2)=\phi \big(k, c, x_2\big), \quad x_2\in [-h,0].
\ee
More detailed forms and conditions of $\phi(k, c, x_2)$ will be specified when we obtained detailed estimates in Sections \ref{S:Ray-BC} and \ref{S:Linear}. 
As in typical problems of linear estimates based on density argument, we shall mostly work on $\phi$ with sufficient regularity, but carefully tracking its norms involved in the estimates. 

The solutions to the Rayleigh equation \eqref{E:Ray-H1-1} 
are obviously even in $k$ and thus $k\ge 0$ will be assumed mostly. Similarly complex conjugate of solutions also solve \eqref{E:Ray-H1-1}  with $c$ 
replaced by $\bar c$, 
so we will restrict our consideration to $c_I\ge 0$. We have to consider the cases of $c \in \C$ away from $U([-h, 0])$, near $U([-h, 0])$, and then finally $c \in U([-h, 0])$, separately. Due to small  scales in $x_2$ created by $k\gg1$, 
the dependence of the estimates on $k\gg1$ will be carefully tracked. 

Recall $U \in C^{l_0}$. For technical convenience we extend $U$ to be a $C^{l_0}$ function on a neighborhood $[-h_0-h, h_0]$ of $[-h, 0]$, where 
\be \label{E:h0}
h_0 = \min\Big\{\frac h2, \,  \frac {\inf_{[-h, 0]} U'}{4 |U''|_{C^0([-h, 0])}}
\Big\} >0,  
\ee
such that, on $[-h_0-h, h_0]$,  
\be \label{E:U-ext}
U' \ge \tfrac 12 \inf_{[-h, 0]} U'(x_2), \quad |U'|_{C^{l_0-1}([-h_0-h, h_0])} \le 2|U'|_{C^{l_0-1} ([-h, 0])}.
\ee
In the analysis of the most singular case of $c$ close to the range $U([-h, 0])$, we let $x_2^c$ be such that 
\be \label{E:x2c}
c_R = U(x_{2}^c), \; \text{ if } c_R \in U([-h_0-h, h_0]).
\ee
We also extend the non-homogeneous term $\phi(k, c, x_2)$ for $x_2 \in [-h_0-h, h_0]$ while keeping its relevant bounds comparable.


\subsection{Rayleigh equation in the regular region}

In the initial step we consider the rather regular case where $k^2 |U-c|$ is bounded from below. For not so small $k$, we first transform the homogeneous Rayleigh equation \eqref{E:Ray-H1-1} into a system of first order (complex valued) ODEs. Let 
\[
z_\pm = y' \pm |k| y, 
\]
and then \eqref{E:Ray-H1-1} takes the form of the coupled equations 
\be \label{E:Ray-z-1} 
z_\pm' = \pm |k| z_\pm + \tfrac 12 \beta(k, c, x_2) (z_+ - z_-), \qquad \beta(k, c, x_2) = \tfrac {U''}{|k| (U-c)}.
\ee

\begin{lemma} \label{L:Ray-regular}  
There exists $C>0$ depending only on $|U'|_{C^1}$, and $|(U')^{-1}|_{C^0}$,  such that for any $\rho \in (0, 1]$, $k\ne 0$, and $\mathcal{I} = [x_{2l}, x_{2r}] \subset [-h_0-h, h_0]$ satisfying 
\be \label{E:kc-away-1}
\left| \frac {1}{U-c} \right| \le  \rho k^2 (1+ |U''|_{C^0([-h_0-h, h_0])})^{-1}, 
\quad \forall \, x_2 \in \mathcal{I},  
\ee
and any solution $z=(z_+, z_-)^T$ to \eqref{E:Ray-z-1} with 
\be \label{E:iniC-L}
|z_+(x_{2l})| \ge |z_- (x_{2l})|, 
\ee
it holds, for $x_2 \in \mathcal{I}$, $|z_+(x_{2})| \ge |z_- (x_{2})|$ and 
\be \label{E:z_pm-l} \begin{split} 
& \left| z_+(x_2) - e^{|k| (x_2 - x_{2l})}z_+(x_{2l}) \right| + \left| z_-(x_2) - e^{- |k| (x_2 - x_{2l})}z_-(x_{2l}) \right|  \\
\le & C|k|^{-1}  \log \big(1+ C\rho k^2(x_2 -x_{2l})\big) e^{|k|(x_2-x_{2l})} |z_+(x_{2l})|. 
\end{split} \ee
Moreover, for any solution with 
\be \label{E:iniC-R}
|z_+(x_{2r})| \le |z_- (x_{2r})|,
\ee
we have, for $x_2 \in \mathcal{I}$, $|z_+(x_{2})| \le |z_- (x_{2})|$ and   
\be \label{E:z_pm-r} \begin{split} 
& \left| z_+(x_2) - e^{|k| (x_2 - x_{2r})}z_+(x_{2r}) \right| + \left| z_-(x_2) - e^{- |k| (x_2 - x_{2r})}z_-(x_{2r}) \right|  \\
\le & C|k|^{-1} \log \big(1+ C\rho k^2(x_{2r} -x_{2})\big) e^{|k|(x_{2r}-x_{2})}   |z_-(x_{2r})|.
\end{split} \ee
\end{lemma}

While \eqref{E:iniC-L} provides some technical convenience, indeed some assumption of this type on the initial values is needed to ensure estimates of solutions such as  \eqref{E:z_pm-l}. For example, if $|\beta|\ll k$, the standard ODE theory implies that there are two solutions behaving like $e^{\pm k (x_2-x_{2l})}$ corresponding to the Lyapunov exponents close to $\pm k$, then the decaying solution may not satisfy \eqref{E:z_pm-l} with $C$ uniform in $k\gg1$.  

\begin{proof}
We start with the observation of a simple consequence of \eqref{E:kc-away-1}. Namely, 
one may  compute straight forwardly  
\be \label{E:temp-mono-1}
(|z_+|^2 - |z_-|^2 )' =2 |k| (|z_+|^2 + |z_-|^2 ) +  \RP \beta |z_+-z_-|^2 \ge 0. 
\ee
This monotonicity along with boundary conditions yields an order relation between $|z_\pm|$ which can be used to control terms in \eqref{E:Ray-z-1}. 

We shall focus on the case under assumption \eqref{E:iniC-L}, which ensures 
\be \label{E:temp-order}
|z_+| \ge |z_-|, \quad \forall \, x_2 \in \mathcal{I}.
\ee
By factorizing $z_+$ on the right side of \eqref{E:Ray-z-1}, its solutions satisfy   
\be \label{E:z+}
z_+(x_2) - e^{|k|(x_2 - x_{2l})} z_+ (x_{2l})=  \big(e^{ \frac12 \int_{x_{2l}}^{x_2} \beta (k, c, x_2') \big(1- \frac {z_-(x_2')}{z_+(x_2')}\big) dx_2'} -1\big) e^{|k| (x_2 - x_{2l})} z_+ (x_{2l}).  
\ee
If $c_R \in U([-h_0-h, h_0])$, let $x_2^c$ be defined as in \eqref{E:x2c} and we use \eqref{E:kc-away-1} to estimate  
\begin{align*}
\int_{x_{2l}}^{x_2} |\beta (k, c, x_2')| 
dx_2' \le& \frac C{|k|} \int_{x_{2l}}^{x_2} (|x_2' - x_2^c|^2 + c_I^2)^{-\frac 12} dx_2' 
=  \frac C{|k|} \left| \log \frac {x_2- x_{2}^c + \sqrt{(x_2 - x_2^c)^2 + c_I^2}}{x_{2l}- x_{2}^c + \sqrt{(x_{2l} - x_2^c)^2 + c_I^2}  }  \right|,
\end{align*}
where the last equality is the exact integral. If $x_2^c \le x_{2l} \le x_2$, then the numerator in the logarithm is greater than the denominator. Applying the triangle inequality to $x_2$, $x_{2l}$ and $c$, we obtain 
\begin{align*}
\left| \log \frac {x_2- x_{2}^c + \sqrt{(x_2 - x_2^c)^2 + c_I^2}}{x_{2l}- x_{2}^c + \sqrt{(x_{2l} - x_2^c)^2 + c_I^2}  }  \right| \le & \left| \log \left(1+ \frac {C(x_2-x_{2l})}{x_{2l}- x_{2}^c + |U(x_{2l}) - c|}\right) \right| 
\le  \log\big(1+ C\rho k^2 (x_2 -x_{2l})\big).
\end{align*}
If $x_{2l} \le x_{2} \le x_2^c$, multiplying the top and bottom of the quotient by their conjugates and proceeding much as in the previous case, we have 
\begin{align*}
\left| \log \frac {x_2- x_{2}^c + \sqrt{(x_2 - x_2^c)^2 + c_I^2}}{x_{2l}- x_{2}^c + \sqrt{(x_{2l} - x_2^c)^2 + c_I^2}  }  \right|= & \left| \log \frac {x_2^c- x_{2l} + \sqrt{(x_{2l} - x_2^c)^2 + c_I^2}}{x_{2}^c- x_{2} + \sqrt{(x_{2} - x_2^c)^2 + c_I^2}  }  \right|  
\le  \log\big(1+ C \rho k^2 (x_2 -x_{2l})\big).
\end{align*}
Finally, in the case $x_{2l} < x_2^c < x_2$,  by splitting the interval at $x_2^c$ and applying the above estimates on the two subintervals, we obtain 
\begin{align*}
& \left| \log \frac {x_2- x_{2}^c + \sqrt{(x_2 - x_2^c)^2 + c_I^2}}{x_{2l}- x_{2}^c + \sqrt{(x_{2l} - x_2^c)^2 + c_I^2}  }  \right| \\
=&  \left| \log \frac {x_2- x_{2}^c + \sqrt{(x_2 - x_2^c)^2 + c_I^2}}{|c_I|} + \log \frac {|c_I|}{x_{2l}- x_{2}^c + \sqrt{(x_{2l} - x_2^c)^2 + c_I^2}  }  \right|\\
\le& \log\big(1+ C \rho k^2 (x_2 -x_{2}^c)\big) + \log\big(1+ C \rho k^2 (x_2^c -x_{2l})\big) 
\le 2\log\big(1+ C \rho k^2 (x_2 -x_{2l})\big).
\end{align*}
Therefore the desired estimate \eqref{E:z_pm-l}  on $z_+$ follows from \eqref{E:z+} and \eqref{E:temp-order} and 
\begin{align*}
&\left| e^{ \frac12 \int_{x_{2l}}^{x_2} \beta (k, c, x_2') \big(1- \frac {z_-(x_2')}{z_+(x_2')}\big) dx_2'} -1 \right| \le C \int_{x_{2l}}^{x_2} |\beta (k, c, x_2') | dx_2'  
\le C|k|^{-1} \log\big(1+ C \rho k^2(x_2 -x_{2l})\big), 
\end{align*}
as $C|k|^{-1} \log\big(1+ C \rho k^2(x_2 -x_{2l})\big)$ is bounded uniformly in all $k\ne 0$. 
If $c_R \notin U([-h_0-h, h_0])$, one can bound $|\beta|$ by $\tfrac C{|k|} \min\{1, \rho k^2\}$ which is also bounded for all $k\ne 0$. If $\rho k^2 \le 1$, then $\rho k^2 (x_2 - x_{2l})$ is bounded by $C \log\big(1+ \rho k^2 (x_2 -x_{2l})\big)$. If $1 \le \rho k^2$, then   
\[
x_2 - x_{2l} \le C \log\big(1+x_2 -x_{2l}\big) \le C \log\big(1+ \rho k^{2} (x_2 -x_{2l})\big).
\]
Therefore in both cases we have  
\[
\int_{x_{2l}}^{x_2} |\beta (k, c, x_2')| dx_2' \le \tfrac C{|k|} \min\{1, \rho k^2 \} (x_2-x_{2l}) \le \tfrac C{|k|} \log\big(1+ \rho k^2(x_2 -x_{2l})\big)
\]
and thus \eqref{E:z_pm-l} for $z_+$ follows from \eqref{E:z+} and \eqref{E:temp-order}. 

Turning attention to $z_-$, from the variation of parameter formula, we have 
\be \label{E:z-}
z_-(x_2) - e^{-|k| (x_2 - x_{2l})} z_-(x_{2l}) = \frac12 \int_{x_{2l}}^{x_2} e^{-|k| (x_2 - x_2')} \beta (k, c, x_2') \big(z_+ (x_2') - z_-(x_2')\big) dx_2',   
\ee
which along with \eqref{E:kc-away-1}, \eqref{E:z_pm-l} for $z_+$, and \eqref{E:temp-order}, implies 
\begin{align*}
&|z_-(x_2) - e^{ -|k| (x_2 - x_{2l})} z_-(x_{2l}) | 
\le Ce^{|k| (x_2-x_{2l})}  |z_+(x_{2l})|\int_{x_{2l}}^{x_2} |\beta(k, c, x_2')| dx_2'. 
\end{align*}
The desired estimate on $z_-$ follows from the above inequality on $\int |\beta|$. 
The estimates on $z_\pm(x_2)$ with initial condition $z_\pm (x_{2r})$ satisfying \eqref{E:iniC-R} can be  derived in exactly the same fashion. 
\end{proof}

In the following we use the above lemma to analyze some solutions to the homogeneous and non-homogeneous Rayleigh equations \eqref{E:Ray-H1-1} and \eqref{E:Ray-NH-1}. 

\begin{lemma} \label{L:Ray-regular-1}
Consider 
\[
(\Theta_1, \Theta_2) \in \{\sinh, \cosh\}^2 \setminus \{(\cosh, \sinh)\}.
\]
There exists $C >0$ depending only on $|U'|_{C^1}$ and $|(U')^{-1}|_{C^0}$, such that, for any $k\ne 0$, $\rho \in (0, 1]$, $C_0\ge 0$, and interval $\CI=[x_{2l}, x_{2r}] \subset [-h, 0]$ satisfying \eqref{E:kc-away-1}, 
\begin{enumerate}
\item if a solution $y(x_2)$ to \eqref{E:Ray-H1-1}   
satisfies 
\be \label{E:y-ini-1}
\big||k|  y (x_{2l}) - \sinh |k|s\big|\le C_0 \Theta_1 (|k|s), \;\; |y'(x_{2l}) - \cosh ks|\le C_0 \Theta_2 (|k|s),  \quad s\ge 0, 
\ee
then it holds that, for all $x_2 \in \CI$,  
\[
\big||k| y (x_{2}) - \sinh  |k| (x_{2} - x_{2l}+s)| \le   C \big( C_0 + (1+C_0)\big( \rho + |k|^{-1} \log (1+C  \rho k^2)   \big) \big) \Theta_1 (|k|(x_2 -x_{2l}+s)),
\]
\[
|y'(x_{2}) - \cosh k (x_2-x_{2l}+s)| \le  C \big( C_0 + (1+C_0)\big( \rho + |k|^{-1}\log (1+C  \rho k^2)   \big) \big) \Theta_2 (|k|(x_2 -x_{2l}+s));
\]
\item if a solution $y(x_2)$ to \eqref{E:Ray-H1-1}   
satisfies 
\be \label{E:y-ini-2}
\big||k|  y (x_{2r}) - \sinh |k|s\big| \le C_0 \Theta_1 (|ks|),  \;\; |y'(x_{2r}) - \cosh ks| \le C_0 \Theta_2 (|ks|), \quad s\le 0, 
\ee
then it holds that, for all $x_2 \in \CI$, 
\[
\big||k| y (x_{2}) - \sinh  |k| (x_2 - x_{2r} +s)\big| \le C \big( C_0 + (1+C_0)\big( \rho + |k|^{-1}\log (1+C  \rho k^2)   \big) \big)  \Theta_1 (|k (x_2 - x_{2r} +s)|),
\]
\[
|y'(x_{2}) - \cosh k (x_2 - x_{2r}+s)| \le  C \big( C_0 + (1+C_0)\big( \rho + |k|^{-1}\log (1+C  \rho k^2)   \big) \big)  \Theta_2 (|k (x_2 - x_{2r} +s)|). 
\]
\item Moreover, the solution $y(x_2)$ to \eqref{E:Ray-NH-1} with $y(x_{20}) = y'(x_{20})=0$ for some $x_{20} \in \CI$ satisfies 
\be \label{E:y-nh-1} \begin{split}
&\Big| |k|y (x_{2})- \int_{x_{20}}^{x_2} \phi(k, c, x_2') \sinh |k(x_2 - x_2')| dx_2' \Big| \\
&+ \Big|y' (x_{2})- \int_{x_{20}}^{x_2} \phi(k, c, x_2') \cosh k(x_2 - x_2') dx_2' \Big| \\
\le & C\big( \rho + |k|^{-1} \log (1+ C \rho k^2)   \big)  \Big|\int_{x_{20}}^{x_2} \phi\big(k,  c, x_2'\big) \sinh |k(x_2 - x_2')| dx_2'\Big|.
\end{split} \ee 
\end{enumerate}
\end{lemma}

\begin{proof}
We first consider the special solution $y (x_2)$ to the homogeneous 
\eqref{E:Ray-H1-1} satisfying \eqref{E:y-ini-1} with $C_0=0$, namely,  with the initial values
\[
y(x_{2l})= |k|^{-1} \sinh |k|s, \;\; y'(x_{2l}) = \cosh |k|s, \quad s\ge 0, 
\]
whose corresponding form in terms of $z_\pm$ with initial values $z_\pm (x_{2l})= e^{\pm |k|s}$ satisfies the assumptions of Lemma \ref{L:Ray-regular}.
On the one hand, for $|k| ( x_2 - x_{2l}) \le 1$, it holds 
\begin{align*}
& |k|^{-1} e^{|k|(x_{2}-x_{2l})} \log \big(1+ C\rho k^2(x_2 -x_{2l})\big) \le C \rho |k|(x_2 -x_{2l})  \le C \rho\sinh |k| (x_2 -x_{2l}),  
\end{align*}
while, for $|k| (x_2 - x_{2l})\ge 1$, we have 
\begin{align*}
&|k|^{-1} e^{|k|(x_{2}-x_{2l})} \log \big(1+ C\rho k^2(x_2 -x_{2l})\big) \le  C|k|^{-1}  \log (1+ C\rho k^2) \sinh |k|(x_{2}-x_{2l}).  
\end{align*}
Therefore Lemma \ref{L:Ray-regular}  and $\rho \in (0, 1]$ imply  
\begin{align*}
|z_+(x_{2}) - e^{|k|(x_{2} -x_{2l}+s)}| & +|z_-(x_{2})  - e^{-|k|(x_{2}-x_{2l}+s)}|\\
\le & C \big( \rho + |k|^{-1} \log (1+ C \rho k^2)   \big) e^{|k|s} \sinh |k|( x_2 -x_{2l} ). 
\end{align*}
Recovering $y(x_2)$ and $y'(x_2)$ from $z_\pm(x_2)$, we obtain the desired estimates in the case of $\Theta_1=\Theta_2=\sinh$  
under the additional assumption $C_0=0$. 

In the following we prove the estimates for a homogeneous solution $y(x_2)$ to \eqref{E:Ray-H1-1}  under \eqref{E:y-ini-1} with general $C_0\ge 0$. Let $Y_1(x_2)$ and $Y_2(x_2)$ be solution to \eqref{E:Ray-H1-1} with initial values 
\[
Y_1(x_{2l}) = |k|^{-1} \sinh 1, \;\;  Y_1'(x_{2l})=\cosh1; \quad Y_2(x_{2l}) =0, \;\; Y_2'(x_{2l})=1. 
\]  
Clearly $Y_1$ and $Y_2$ satisfy the above estimates with $s=|k|^{-1}$ and $s=0$, respectively, and  
\[
y(x_2) = |k| ( \sinh 1)^{-1} y(x_{2l}) Y_1(x_2) + \big(y' (x_{2l}) - |k| (\coth 1) y(x_{2l})\big) Y_2(x_2).   
\]
Therefore, for $x_2 \in \CI$,  
\begin{align*}
\big| |k| y& (x_2)  - \sinh |k| (x_2 -x_{2l} +s)\big| =   \Big| ( \sinh 1)^{-1} \Big(|k| y(x_{2l}) \big( |k| Y_1(x_2) - \sinh( |k|(x_2 - x_{2l})+ 1) \big)  \\
&+ ( |k| y(x_{2l}) - \sinh |k|s) \sinh( |k|(x_2 - x_{2l})+ 1) \Big) +(\sinh 1)^{-1}  \sinh |k|s \sinh( |k|(x_2 - x_{2l})+ 1) \\
&+ \big(y' (x_{2l}) -  (\coth 1) |k|y(x_{2l})\big) \big(|k| Y_2(x_2) - \sinh  |k|(x_2 - x_{2l}) \big) \\
& +  \big(y' (x_{2l}) - \cosh |k|s -  (\coth 1) (|k|y(x_{2l}) - \sinh |k|s) \big) \sinh |k|(x_2 - x_{2l}) \\
&+ (\cosh |k|s - (\coth1) \sinh |k|s) \sinh  |k|(x_2 - x_{2l}) - \sinh |k| (x_2 -x_{2l} +s)\Big|.
\end{align*}
In the above summation, all the hyperbolic trigonometric combinations without $y(x_{2l})$ or $Y_{1,2}(x_2)$ are eventually cancelled and the remaining terms can be estimated by the using the assumptions on the initial values and the already obtained estimates on $Y_1$ and $Y_2$. We have 
\begin{align*}
&\big| |k| y (x_2)  - \sinh |k| (x_2 -x_{2l} +s)\big| \le \big( (1+C_0) \big( \rho + |k|^{-1}\log (1+C  \rho k^2)   \big) +C_0\big) \\
&\qquad \qquad \qquad \qquad \times  \big(\Theta_1( |k|s) \sinh( |k|(x_2 - x_{2l})+ 1) + \cosh |k|s \sinh  |k|(x_2 - x_{2l}) \big)\\
\le & \big( (1+C_0) \big( \rho + |k|^{-1}\log (1+C  \rho k^2)   \big) +C_0\big) \Theta_1 |k|(x_2 - x_{2l}+ s), 
\end{align*}
where the last inequality was obtained by considering the two possible cases of $\Theta_1$ spearately.  The inequality on $y'(x_2)$ can be obtained similarly as 
\begin{align*}
\big| y' &(x_2)  - \cosh |k| (x_2 -x_{2l} +s)\big| \le   \Big| ( \sinh 1)^{-1} \Big(|k| y(x_{2l}) \big( |k| Y_1'(x_2) - \cosh( |k|(x_2 - x_{2l})+ 1) \big)  \\
&+ ( |k| y(x_{2l}) - \sinh |k|s) \cosh( |k|(x_2 - x_{2l})+ 1) \Big) +(\sinh 1)^{-1}  \sinh |k|s \cosh( |k|(x_2 - x_{2l})+ 1) \\
&+ \big(y' (x_{2l}) -  (\coth 1) |k|y(x_{2l})\big) \big(|k| Y_2'(x_2) - \cosh  |k|(x_2 - x_{2l}) \big) \\
& +  \big(y' (x_{2l}) - \cosh |k|s -  (\coth 1) (|k|y(x_{2l}) - \sinh |k|s) \big) \cosh |k|(x_2 - x_{2l}) \\
&+ (\cosh |k|s - (\coth1) \sinh |k|s) \cosh  |k|(x_2 - x_{2l}) - \cosh |k| (x_2 -x_{2l} +s)\Big|
\end{align*}
and thus 
\begin{align*}
&\big| y' (x_2)  - \cosh |k| (x_2 -x_{2l} +s)\big| \le \big( (1+C_0) \big( \rho + |k|^{-1}\log (1+C  \rho k^2)   \big) +C_0\big) \\
&\quad  \times  \big(\Theta_1(|k|s) \cosh( |k|(x_2 - x_{2l})+ 1) + \cosh |k|s \sinh  |k|(x_2 - x_{2l}) + \Theta_2 (|k|s) \cosh  |k|(x_2 - x_{2l})  \big)\\
\le & \big( (1+C_0) \big( \rho + |k|^{-1}\log (1+C  \rho k^2)   \big) +C_0\big) \Theta_2 |k|(x_2 - x_{2l}+ s). 
\end{align*}
This proves the desired estimates under the assumption \eqref{E:y-ini-1}. The proofs of the inequalities under assumption \eqref{E:y-ini-2} are similar and we omit the details.  

Using the variation of parameter formula, we can write the solution $y(x_2)$ with $y(x_{20}) = y'(x_{20})=0$ to the non-homogeneous Rayleigh equation \eqref{E:Ray-NH-1} as  
\[
\begin{pmatrix} y \\ y' \end{pmatrix} (x_2) = 
\int_{x_{20}}^{x_2} \phi \big(k, c, x_2'\big) S(x_2, x_{2}') \begin{pmatrix} 0 \\ 1 \end{pmatrix}dx_2'
\]
where $S(x_2, x_{2}')$ is the $2\times 2$ fundamental matrix of the homogeneous equation \eqref{E:Ray-H1-1} with initial value $S(x_2', x_2') =I$. Therefore, 
\[
S(x_2, x_{2}') \begin{pmatrix} 0 \\ 1 \end{pmatrix}=\begin{pmatrix} \tilde y(x_2, x_2')  \\ \tilde y' (x_2, x_2') \end{pmatrix} 
\]
where $\tilde y(\cdot, x_2')$ is the solution to \eqref{E:Ray-H1-1} whose initial value is given by $\tilde y(x_2', x_2') = 0$ and $\tilde y(x_2', x_2') =1$. The desired estimates follow from applying the above estimates in the homogeneous case with $s=0=C_0$ and $\Theta_1=\Theta_2=\sinh$.  
\end{proof}

Practically the above estimates are more effective for $k$ bounded from below. To end this subsection, we give the following simple estimate of the Rayleigh equation for $k$ bounded from above,
which compares $y(x_2)$ to the free solution (where the $U$ term is removed)
\[
y_F(x_2) = \big(\cosh k(x_2 - x_{20}) \big) y(x_{20})+ k^{-1} \big(\sinh k(x_2 -x_{20})\big) y'(x_{20}).
\]
Here $k^{-1} \sinh ks|_{k=0} = s$ is understood. 

\begin{lemma} \label{L:regular-small-k}
For any $k^*, M>0$, there exists $C>0$ depending only on $h$, $k^*$, and $M$ such that for any 
$|k|\le k^*$, $C_0>0$, $x_{20} \in \mathcal{I} = [x_{2l}, x_{2r}] \subset [-h, 0]$ satisfying 
\[
\left| \frac {1}{U-c}\right| \le C_0\le M, \quad \forall \, x_2 \in \mathcal{I},  
\]
and any solution $y(x_2)$ to \eqref{E:Ray-NH-1}, it holds 
\begin{align*}
| y(x_2)- y_F(x_2)| + |y'(x_2)- y_F' (x_2)|\le  C \Big( & C_0\big( | y(x_{20})| |x_2-x_{20}|+ |y'(x_{20})| |x_2-x_{20}|^2 \big) \\
&+ \Big|\int_{x_{20}}^{x_2} | \phi (k, c, x_2')| 
dx_2' \Big| \Big). 
\end{align*}
\end{lemma}

\begin{proof} 
The proof is based on some straight forward elementary argument and we shall only outline it. Let $\tilde y = y - y_F$. We can write the solution $y(x_2)$ using the variation of constant formula
\begin{align*}
\begin{pmatrix} \tilde y (x_2) \\ \tilde y'(x_2) \end{pmatrix} = & \int_{x_{20}}^{x_2} \left( \frac {U'' y_F}{U-c} -\phi   \right) (x_2') \begin{pmatrix} k^{-1} \sinh k(x_2 -x_{2}') \\ \cosh k(x_2 - x_{2}') \end{pmatrix} dx_2' \\
&+ \int_{x_{20}}^{x_2} \left( \frac {U'' \tilde y}{U-c} \right) (x_2') \begin{pmatrix} k^{-1} \sinh k(x_2 -x_{2}') \\ \cosh k(x_2 - x_{2}') \end{pmatrix} dx_2'. 
\end{align*}
It implies 
\begin{align*}
| \tilde y(x_2)| + |\tilde y'(x_2)| \le & C \Big(C_0\big( | y(x_{20})| |x_2-x_{20}|+ |y'(x_{20})| |x_2-x_{20}|^2 \big)  + \left| \int_{x_{20}}^{x_2} |\phi(k, c, x_2')| dx_2' \right|\Big) \\
&+ CC_0 \left|\int_{x_{20}}^{x_2} |\tilde y(x_2')|dx_2' \right|
\end{align*}
and the estimates on $y-y_F$ and $y'-y_F'$ 
follow immediately from the Gronwall inequality. 
\end{proof}

\subsection{Rayleigh equation near singularity and its convergence as $c_I \to 0+$} \label{SS:a-esti} 

In the rest of the section, we shall mostly focus on the case when $(1+k^2)^{\frac 12} |U-c|$ is small, so 
\be \label{E:x2c-1}
c_R = U(x_2^c), \quad x_2^c \in [-\tfrac 12 h_0-h, \tfrac 12 h_0],
\ee
will always be assumed, while the domains of $U$ and $\phi$ have been extended to $[-h_0-h, h_0]$. Due to complex conjugacy, we only need to consider $c_I \ge 0$. In particular,  if $x_2^c \in (-h, x_2)$, the strong singularity in \eqref{E:Ray-H1-1} will lead to $y\big(c_R + i(0+), k, x_2\big)  \notin \R$ even if $y(-h), y'(-h) \in \R$. 
Even though some estimates are stated for $c_I>0$, most of the inequalities are mostly uniform as $c_I \to 0+$ and thus hold for the limits. 

In order to obtain estimates uniform in $k \in \R$, 
rescale 
\be \label{E:tau}
\mu = \langle k \rangle^{-1} = \tfrac 1{\sqrt{k^2+1}}, \quad x_2 = x_2^c + \mu \tau, \quad c_I= \mu \ep, \quad  w = (w_1, w_2)^T = (\mu^{-1} y, y')^T \in \C^2, 
\ee 
where $x_2^c$ satisfies \eqref{E:x2c-1} as well as in the above. Equation \eqref{E:Ray-H1-1} becomes 
\be \label{E:Ray-H2}
w_\tau = \begin{pmatrix} 0 & 1 \\ 1- \mu^2 + \frac {\mu^2 U'' ( x_2^c + \mu \tau )}{U( x_2^c + \mu \tau ) -c} & 0 \end{pmatrix} w -  \begin{pmatrix} 0 \\ \tilde \phi(\mu, c,\tau) \end{pmatrix},  
\ee
where 
\[
\tilde \phi (\mu, c, \tau) = \mu \phi\big(k, c, x_2^c + \mu \tau \big).
\]
We shall consider this ODE on intervals $\tau \in [-M, M]$ such that 
\be \label{E:M}
[x_2^c - \mu M, x_2^c + \mu M] \subset [-h_0-h, h_0], 
\ee
is that $U$ is well-defined when $|\tau|\le M$. 
As $c_I\to 0+$, one would naturally expect $w(\tau)$ to converge to solutions to 
\be \label{E:Ray-H0-1}
W_\tau = \begin{pmatrix} 0 & 1 \\  1- \mu^2 + \frac {\mu^2 U'' ( x_2^c + \mu \tau )}{U( x_2^c + \mu \tau ) -c_R} & 0 \end{pmatrix} W -  \begin{pmatrix} 0 \\ \tilde \phi\big(\mu, c_R,\tau\big) \end{pmatrix}.
\ee
However, this  limit equation becomes singular at $\tau=0$ and conditions have to be specified there. \\

\noindent $\bullet$ {\bf Fundamental matrix of the homogeneous Rayleigh equation.} 
Its construction is adapted from the one used in \cite{BSWZ16}. Let    
\be \label{E:Gamma-1}
\Gamma (\mu, c_R, \ep, \tau)
= (1-\mu^2) \tau + \frac {\mu U''(x_2^c)}{2 U'(x_2^c)} \log (\tilde U^2 + \ep^2 ) + \gamma (\mu, c_R, \ep, \tau) + \int_{-M}^\tau \frac {i\mu \ep U_2(\tau')}{\tilde U(\tau')^2 +\ep^2} d\tau', 
\ee
where, for $j=1,2,3$,   
\be \label{E:tU}
U_j (c_R, \mu, \tau) = (\tfrac {d^j}{d x_2^j} U) ( x_2^c + \mu \tau), \;\; \tilde U(c_R, \mu, \tau) = \tfrac 1\mu \big(U(x_2^c + \mu \tau) - c_R\big) =  \tfrac 1\mu \big(U(x_2^c + \mu \tau) - U(x_2^c) \big),  
\ee
and the remainder $\gamma$ of $\Gamma$ is given by 
\be \label{E:gamma-1}
\gamma(\mu, c_R, \ep, 0)=0, \quad \gamma_\tau = \frac{\mu \big(U_1(0) U_2 - U_2(0) U_1\big)\tilde U  }{U_1(0) (\tilde U^2 +\ep^2)}, 
\; \Longrightarrow \; \Gamma_\tau = 1 -\mu^2+ \frac {\mu U_2}{\tilde U -i \ep}.  
\ee
It is not hard to see that $\gamma (\mu, c_R, 0, \tau)$ is $C^{l_0-2}$ in $\tau$ and $\mu$ and $C^{l_0-3}$ in $c_R$. 
We often skip writing the explicit dependence on those variables other than $\tau$. Denote 
\be \label{E:Gamma0} \begin{split}
\Gamma_0 (\mu, c_R, \tau) =& \lim_{\ep \to 0+} \Gamma (\mu, c_R, \ep, \tau) \\
=& (1-\mu^2) \tau + \frac {\mu U''(x_2^c)}{U'(x_2^c)} \log |\tilde U(\tau)| + \gamma(\mu, c_R, 0, \tau) + \frac {i\pi \mu U''(x_2^c)}{2 U'(x_2^c)} \big(sgn(\tau)+1\big), 
\end{split} \ee
where we note that the integrand of the imaginary  part of $\Gamma$ converges to a delta mass as $\ep \to 0+$ and produces a jump in $\Gamma_0$ at $\tau=0$ (see Lemma \ref{L:Gamma-B} in the below). Let $\tilde B (\mu, c_R, \ep, \tau)$ be a $2\times 2$ matrix given by 
\be \label{E:tB-1}
\tilde B_\tau = \begin{pmatrix} \Gamma(\mu, c_R, \ep, \tau) & 1 \\ -\Gamma(\mu, c_R, \ep, \tau)^2 & -\Gamma(\mu, c_R, \ep, \tau) \end{pmatrix} \tilde B, \quad \tilde B (\mu, c_R, \ep, 0) =I_{2\times 2},  
\ee
and 
\be \label{E:tPhi-0}
\tilde \Phi (\mu, c, \tau) = \begin{pmatrix}\tilde \Phi_1 (\mu, c, \tau) \\ \tilde  \Phi_2 (\mu, c, \tau) \end{pmatrix} =\int_{-M}^\tau \tilde \phi(\mu, c, \tau') \tilde B(\mu, c_R, \tfrac {c_I}\mu, \tau')^{-1}\begin{pmatrix} 0 \\  1 \end{pmatrix}   d\tau'.  
\ee
It is worth pointing out that $\Gamma_0$ is real for $\tau<0$ and imaginary for $\tau>0$. To keep the notations simple we often skip the arguments other than $\tau$.  
In the following lemma we collect some basic estimates of $\Gamma$ and $\tilde B$ where we often bound the $\log |\tau|$ singularity in $\Gamma$ by $|\tau|^{-\alpha}$, $\alpha >0$, for simplicity. 

\begin{lemma} \label{L:Gamma-B}
For any $M> 0$ satisfying \eqref{E:M} and $\alpha, \alpha'  \in (0, 1)$ with $\alpha + \alpha' <1$, there exists $C >0$ depending only on $M$, $\alpha$, $\alpha'$, $|U'|_{C^2}$, and $|(U')^{-1}|_{C^0}$, such that, for any $0< \ep < M$, the following hold for $|\tau|\le M$,  
\be \label{E:tB-2} 
 \det \tilde B =1, \quad|\tilde B - I | \le e^{|\tau| + C ( |\tau|^3 + \mu^2|\tau|^{\alpha})} -1, \quad |\tilde B^{-1} -I | \le 4 (e^{|\tau| + C( |\tau|^3 + \mu^2|\tau|^{\alpha})} -1).
\ee
\be \label{E:Gamma-d} 
|\Gamma (\ep, \tau) - \Gamma_0 (\tau) | \le  C\mu \big(\mu \ep | \log\ep| + \tfrac \ep{\ep +|\tau|} + \log (1+ \tfrac {C\ep^2}{\tau^2}) \big)
\ee
\be \label{E:temp-tB-d}
\left| \tilde B(\ep, \tau) -\tilde B_0(\tau) \right| \le C \mu \min\{  \ep^\alpha (|\tau|^{1-\alpha} + \mu |\tau|^{\alpha'}), \  \ep (1+ |\log \ep| +\mu \log^2 \ep)\} 
\ee
where $\tilde B_0(\mu, c_R, \tau) = \lim_{\ep \to 0+} \tilde B(\mu, c_R, \ep, \tau)$. Moreover, general solutions of \eqref{E:Ray-H2} with $c_I>0$ is given by 
\be \label{E:tw-1}
 \begin{pmatrix} \mu^{-1} y(x_2) \\ y'(x_2) \end{pmatrix} =w(\tau)= \begin{pmatrix} 1 & 0 \\ \Gamma(\mu, c_R, \ep, \tau) &1 \end{pmatrix} \tilde B(\mu, c_R, \ep, \tau) \big(b - \tilde \Phi (\mu, c, \tau) \big), \; b=  \begin{pmatrix} b_1 \\ b_2 \end{pmatrix}  \in \C^2.
\ee
\end{lemma} 

\begin{remark} \label{R:uniformity}
Even though $\ep>0$ is assumed in the above and the remaining statements in this and the next subsections, as $C>0$ is independent of $\ep=\langle k \rangle c_I \in (0, M]$ in a priori estimates and thus they 
hold even as $\ep \to 0+$.
\end{remark}

Expression \eqref{E:tw-1} essentially is the variation of parameter formula including the fundamental matrix of the Rayleigh equation. Due to $\det \tilde B=1$, it is possible to extend the definition of $\tilde B$ to include all $x_2 \in [-h_0-h, h_0]$, but its bound would be non-uniform in $k\gg1$ for $|x_2 -x_2^c| \gg \mu$. 

\begin{proof} 
Since $\Gamma$ has a logarithmic singularity at the worst (even for $\ep=0$), 
$\tilde B$ is obviously well-defined. 
The zero trace value of the coefficient matrix in \eqref{E:tB-1} yields $\det B=1$. The form \eqref{E:tw-1} of general solutions of \eqref{E:Ray-H2} for $c_I>0$ follows from 
straightforward verifications. 

Equation \eqref{E:tB-1} implies 
\[
|(\tilde B -I)_\tau | \le 1+\Gamma^2 + (1+\Gamma^2) |\tilde B-I|, 
\]
where $1 +\Gamma^2$ is the operator norm of the coefficient matrix. From Gronwall inequality, we obtain 
\[
|\tilde B - I | \le e^{|\tau + \int_0^\tau \Gamma^2 d\tau'|} -1.
\]
It is clear from the definition of $\gamma$ that 
\[
|\gamma_\tau| \le C\mu^2, \quad \left| \int_{-M}^\tau \frac {i\mu \ep U_2(\tau')}{\tilde U(\tau')^2 +\ep^2} d\tau' \right| \le C \mu.
\]
The definition of $\Gamma$, the boundedness of $|\tilde U|$, and the estimate on $\gamma_\tau$ imply, for $\tau\in  [-M,M]$,   
\begin{align*}
\left| \int_0^\tau \Gamma^2 d\tau' \right| \le C(|\tau|^2 + \mu^2) |\tau| + C \mu^2 \left|\int_0^\tau \log^2 (| \tau'| + \ep) d\tau' \right| \le C (|\tau|^3 +  \mu^2|\tau|^{\alpha}), 
\end{align*}
where $C $ is a generic constant determined by $M$ and $k_*$ and the H\"older inequality was used to obtain $|\tau|^{\alpha}$, for any $\alpha \in (0, 1)$. 
The desired estimate in \eqref{E:tB-2} on $\tilde B -I$ follows immediately which along with $\det B=1$ in turn yields the estimate on $\tilde B^{-1} -I$. 

The definition of $\gamma$ implies 
\[
|\gamma (\ep, \tau) - \gamma (0, \tau) | \le \left|\int_0^\tau \frac {C \mu^2 \ep^2}{(\tau')^2 +\ep^2} d\tau' \right| = C \mu^2 \ep \tan^{-1} \frac {|\tau|}\ep. 
\]
Regarding the imaginary part of $\Gamma$, we observe 
\begin{align*}
&\int_{-M}^\tau \left| \frac {U_2(\tau')}{\tilde U(\tau')^2 +\ep^2} - \frac {U_2(0)}{U_1(0)^2 (\tau')^2 +\ep^2} \right| d\tau' \\
\le & \int_{-M}^\tau \frac { \left| U_2(\tau')\big(U_1(0)^2 (\tau')^2 +\ep^2 \big)- U_2(0)\big(\tilde U(\tau')^2 +\ep^2\big)\right|}{\big(\tilde U(\tau')^2 +\ep^2\big)\big(U_1(0)^2 (\tau')^2 +\ep^2 \big) }  d\tau'\\
\le & C\mu \int_{-M}^\tau \frac {|\tau'|}{(\tau')^2 +\ep^2}  d\tau' \le C\mu (1+| \log\ep|), 
\end{align*}
where we used the smoothness of $U_1$ and $U_2$ in $\mu \tau$. It implies 
\be \label{E:Gamma-temp-1}
\left| \int_{-M}^\tau \frac {\mu \ep U_2(\tau')}{\tilde U(\tau')^2 +\ep^2} d\tau' - \frac {\mu U_2(0)}{U_1(0)} \Big(\tan^{-1} \frac {U_1(0)\tau}\ep +\frac \pi2\Big) \right| \le C \mu \ep (1+\mu | \log\ep|), 
\ee
and thus
\begin{align*}
\left| \int_{-M}^\tau \frac {\mu \ep U_2(\tau')}{\tilde U(\tau')^2 +\ep^2} d\tau' - \frac {\pi \mu U_2(0)}{2U_1(0)} \big(sgn(\tau) +1\big) \right| 
\le  & C\mu \big(\mu \ep | \log\ep| + \min\{1, \, \ep |\tau|^{-1}\}\big) \\
\le & C\mu \big(\mu \ep | \log\ep| + \frac 1{1 + \tfrac {|\tau|}\ep}\big).
\end{align*}
The error estimate \eqref{E:Gamma-d} follows consequently. 

Proceeding to consider $\tilde B(\ep, \tau) - \tilde B_0(\tau)$ where $\tilde B_0(\mu, c_R, \tau) =\tilde B(\mu, c_R, 0, \tau)$, we have 
\begin{align*}
&\p_\tau \big(\tilde B(\ep, \tau) -\tilde B_0(\tau)\big) -\begin{pmatrix} \Gamma_0 (\tau) & 1 \\ -\Gamma_0 (\tau)^2 & -\Gamma_0 (\tau) \end{pmatrix}\big(\tilde B(\ep, \tau) -\tilde B_0 (\tau)\big) \\
= & \begin{pmatrix} \Gamma (\ep, \tau) -  \Gamma_0 (\tau) & 0 \\ -\Gamma (\ep, \tau)^2 + \Gamma_0 (\tau)^2 & -\Gamma (\ep, \tau) +\Gamma_0(\tau) \end{pmatrix} \tilde B(\ep, \tau).
\end{align*}
Recalling that $\tilde B_0 (\tau)$ is the elementary fundamental matrix of the above corresponding homogeneous ODE system, the variation of parameter formula implies  
\begin{align*}
&\left| \tilde B(\ep, \tau) -\tilde B_0(\tau) \right| 
= \left| \int_0^\tau \tilde B_0(\tau) \tilde B_0(\tau')^{-1} \begin{pmatrix} \Gamma (\ep, \tau') -  \Gamma_0 (\tau') & 0 \\ -\Gamma (\ep, \tau')^2 + \Gamma_0(\tau')^2 & -\Gamma (\ep, \tau') +\Gamma_0(\tau') \end{pmatrix} \tilde B(\ep, \tau')d\tau' \right| \\
\le & C \left| \int_0^\tau \big( 1+ |\Gamma (\ep, \tau')| +|\Gamma_0 (\tau')| \big) |\Gamma (\ep, \tau') -\Gamma_0 (\tau')| d\tau' \right| \\
\le & C \left|\int_0^\tau \big( 1+\mu \big|\log |\tau'|\big| \big) |\Gamma (\ep, \tau') -\Gamma_0 (\tau')| 
d\tau' \right| \le C\big| 1+\mu |\log (\cdot)|\big|_{L^{\frac 1{1-\alpha}}} |\Gamma(\ep, \cdot) - \Gamma_0(\cdot)|_{L^{\frac 1\alpha}}. 
\end{align*}
The second desired upper bound in \eqref{E:temp-tB-d} of $\tilde B- \tilde B_0$ follows from direct estimating the above integral without using the H\"older inequality. For the first upper bound there we use, for any $|\tau_1|, |\tau_2| \le M$,  
\be \label{E:dGamma-p} 
|\Gamma(\ep, \cdot) - \Gamma_0(\cdot)|_{L^\rho [\tau_1, \tau_2]} \le C\mu \ep^{\frac 1\rho}, \; \rho \in (1, +\infty); \quad |\Gamma(\ep, \cdot) - \Gamma_0(\cdot)|_{L^1 [\tau_1, \tau_2]} \le C\mu \ep (1+|\log \ep|), 
\ee
which can be verified by straight forward computation. 
The proof of the lemma is complete.
\end{proof}

\noindent $\bullet$ {\bf A priori estimates.} A direct corollary of the form \eqref{E:tw-1} of the general solution to the Rayleigh equation \eqref{E:Ray-H2} is an estimate of $w(\tau)$ in terms of $b$ and $\tilde \Phi$. Let $\tilde \Gamma(\tau)$ denote 
\[
\tilde \Gamma (\tau) = \frac {\mu U''(x_2^c)}{U'(x_2^c)}\big(\frac 12 \log (\tilde U(\tau)^2 +\ep^2) +i \big( \tan^{-1} \frac {U'(x_2^c) \tau}\ep +\frac \pi2 \big)\big).
\]

\begin{corollary} \label{C:A-esti-w} 
For $b \in \C^2$ and $|\tau| \le M$, let 
\[
\tilde b (\tau)= \begin{pmatrix} 1 & 0 \\ \Gamma(\tau) &1 \end{pmatrix} \tilde B (\tau) b, \quad \tilde b_0 (\tau)= \begin{pmatrix} 1 & 0 \\ \Gamma_0(\tau) &1 \end{pmatrix} \tilde B_0 (\tau) b,
\]
then under the same assumptions of Lemma \ref{L:Gamma-B}, it holds, for any $\alpha_1 \in [0, 1-\alpha)$, 
\begin{align*}
&|\tilde b_1 (\tau) - b_1 | \le C (|\tau|+ \mu^{2}|\tau|^\alpha ) |b|, \quad \big|\tilde b_2 (\tau) - \big(b_2 + b_1 \tilde \Gamma(\tau) \big) \big|
\le 
C \big(|\tau|+ \mu (|\tau|^\alpha + \ep^\alpha)\big)|b|
\end{align*}
\[
|\tilde b_1 (\tau) - \tilde b_{01}(\tau)| \le C\mu\ep^\alpha \big(|\tau| |b| +  \min\{|\tau|^{1-\alpha},\, \ep^{1-\alpha}  (1+| \log \ep|)\} |b_1|  \big),  
\]
\[
|\tilde b_2 (\tau) - \tilde b_{02}(\tau)| \le C\mu\big( \ep^\alpha |\tau|^{\alpha_1} |b| + \big( \tfrac \ep{\ep+|\tau|} + \log \big( 1+ \tfrac {C\ep^2}{\tau^2}\big)\big) |b_1|  \big).
\]
\end{corollary}

\begin{proof}
The estimates on $\tilde b$ follows from straight forward calculation based on  \eqref{E:Gamma-temp-1} and the bound on $\tilde B -I$ given in Lemma \ref{L:Gamma-B} and we omit the details. 

Regrading $\tilde b(\tau) -\tilde b_0(\tau)$, let $\tilde B_{jl}$ denote the entries of $\tilde B$. Using Lemma \ref{L:Gamma-B} where the estimates are uniform in $\ep>0$, we have 
\begin{align*}
& |\tilde b_2 (\tau) - \tilde b_{02}(\tau)| \le  (1+ |\Gamma_0|) |\tilde B - \tilde B_{0}| |b|  + |\Gamma- \Gamma_0| (|\tilde B_{11}| |b_{1} |  +  |\tilde B_{12}| |b_{2}|) \\
\le & C\big( (1+ \mu \big|\log |\tau|\big|) |\tilde B - \tilde B_{0}| |b| +|\Gamma- \Gamma_0| (
|b_1| + (|\tau|+\mu |\tau|^{\alpha'}) |b_{2}|  ) \big)\\
\le & C\mu \big(\ep^\alpha |\tau|^{\alpha_1} |b|  + \big(\tfrac \ep{\ep+ |\tau|} + \log \big( 1+ \tfrac {C\ep^2}{\tau^2}\big)\big) (|b_1| +  |\tau|^{\alpha'} |b_{2}|  ) \big).
\end{align*}
Since 
\be \label{E:temp-3.2-1}
|\tau|^\beta \big(\tfrac \ep{\ep+|\tau|} + \log \big( 1+ \tfrac {C\ep^2}{\tau^2}\big)\big) \le C \ep^\beta, \quad \beta\in (0, 1], 
\ee
the upper on $\tilde b_2(\tau) -\tilde b_{02}(\tau)$ follows accordingly. 

To derive the estimate on $\tilde b_1(\tau)-\tilde b_{01}(\tau)$, we notice $\tilde b_1(0)= \tilde b_{01} (0)=b_1$ and the desired estimate follows from integrating $\p_\tau (\tilde b_1-\tilde b_{01}) = \tilde b_2 - \tilde b_{02}$ using \eqref{E:dGamma-p}.
\end{proof}

\begin{remark} \label{R:A-esti-w}
The above estimates 
imply, that for any solution $w(\tau)$ to \eqref{E:Ray-H2}  
\begin{align}
&\big|w_1 (\tau) - \big(b_1 - \tilde \Phi_1(\tau)\big)\big| \le C (|\tau|+ \mu^{2}|\tau|^\alpha ) \big|b - \tilde \Phi(\tau)\big|, \label{E:A-esti-w-1}\\
&\Big|w_2 (\tau) - \Big(b_2 - \tilde \Phi_2(\tau) + 
\tilde \Gamma(\tau) \big(b_1 - \tilde \Phi_1(\tau)\big) \Big)\Big|
\le 
C \big(|\tau|+ \mu (|\tau|^\alpha + \ep^\alpha)\big) \big|b - \tilde \Phi(\tau)\big|.\label{E:A-esti-w-2}  
\end{align}
\end{remark} 

The following lemma gives another estimate of $w(\tau)$ in terms of some initial value $w(\tau_0)$ which we shall use mainly for $\tau_0$ away from $0$. 

\begin{lemma} \label{L:apriori-1}
For any $M> 0$ satisfying \eqref{E:M} and $\alpha \in (0, 1)$, there exists $C >0$ depending only on $M$, $\alpha$, $|U'|_{C^2}$, and $|(U')^{-1}|_{C^0}$, such that, for any $0< \ep < M$, and $\tau_0, \tau \in  [-M,M]$, 
the following hold for any solution $w(\tau)$ to \eqref{E:Ray-H2}:
\begin{subequations} \label{E:Ray-H2-ap}
\be \label{E:Ray-H2-ap-1} \begin{split} 
\big|w_1(\tau) - w_1 &(\tau_0)\big| \le  C |\tau - \tau_0| \big(|w(\tau_0)| + \mu \big| \log \big(\tau_0^2 +\ep^2\big) \big| |w_1(\tau_0)|\big) \\
&+ C \mu |\tau-\tau_0|^\alpha \big(|w(\tau_0)| +\big|\tilde \Phi_1 (\cdot)- \tilde \Phi_1(\tau_0)\big|_{L^\infty [\tau_0, \tau]}\big) + C \big| \tilde \Phi( \cdot) - \tilde \Phi(\tau_0)\big|_{L^1 [\tau_0, \tau]},
\end{split}\ee
\be \label{E:Ray-H2-ap-2} \begin{split} 
& \big|w_2(\tau) - \big( w_2(\tau_0) + \tilde \Phi_2 (\tau_0) - \tilde \Phi_2(\tau) - \tilde \Gamma(\tau_0) w_1(\tau_0) + \tilde \Gamma (\tau) ( w_1(\tau_0)+ \tilde \Phi_1 (\tau_0) - \tilde \Phi_1(\tau)) \big)\big| \\
\le & C\Big(( |\tau|^{\alpha} +\mu \ep^{\alpha}) \big| \tilde \Phi (\tau_0)- \tilde \Phi(\tau)\big|+ \big(\mu {\ep^\alpha} + |\tau|^\alpha+ |\tau_0|^\alpha (1+ \mu |\log (\tau^2+\ep^2)|) \big) |w(\tau_0)|    \\
& + \mu |\tau|^\alpha |\log (\tau_0^2+\ep^2)|  |w_1(\tau_0)| \Big). 
\end{split}\ee
\end{subequations} 
\end{lemma}

\begin{proof} 
We shall first estimate $b -\tilde \Phi (\tau_0)$ based on $w(\tau_0)$ and then apply Corollary \ref{C:A-esti-w}. From \eqref{E:tw-1} and $\det \tilde B=1$ which allows us to write $\tilde B^{-1}$ explicitly, we have 
\be \label{E:b-0}
b - \tilde \Phi(\tau_0)=  \begin{pmatrix} \tilde B_{22} & - \tilde B_{12} \\  -\tilde B_{21} &  \tilde B_{11} \end{pmatrix}  \begin{pmatrix} 1 & 0 \\ -\Gamma &1 \end{pmatrix} w\Big|_{\tau_0} = \begin{pmatrix} \tilde B_{22}  + \Gamma \tilde B_{12} & - \tilde B_{12} \\  -\tilde B_{21} -\Gamma  \tilde B_{11} & \tilde B_{11} \end{pmatrix} w\Big|_{\tau_0}.   
\ee
Using Lemma \ref{L:Gamma-B}, one may estimate 
\be \label{E:b-1}
|b_1 - \tilde \Phi_1(\tau_0) - w_1(\tau_0)| \le C (|\tau_0| + \mu |\tau_0|^\alpha) |w(\tau_0)|, 
\ee
\be \label{E:b-2}
|b_2 - \tilde \Phi_2(\tau_0) + \tilde \Gamma(\tau_0) w_1(\tau_0) - w_2(\tau_0)| \le C \big(|\tau_0| + \mu (|\tau_0|^\alpha + \ep^\alpha)\big) |w(\tau_0)|,
\ee
where we also used \eqref{E:Gamma-temp-1}. Combining these inequalities and Corollary \ref{C:A-esti-w}, we obtain 
\begin{align*}
& \big|w_2(\tau) - \big( w_2(\tau_0) + \tilde \Phi_2 (\tau_0) - \tilde \Phi_2(\tau) - \tilde \Gamma(\tau_0) w_1(\tau_0) + \tilde \Gamma (\tau) ( w_1(\tau_0)+ \tilde \Phi_1 (\tau_0) - \tilde \Phi_1(\tau)) \big)\big| \\
\le & C \big(|\tau|+ \mu(|\tau|^{\alpha'} + \ep^{\alpha'})\big) \big|b - \tilde \Phi(\tau)\big| + C \big(|\tau_0| + \mu (|\tau_0|^\alpha + \ep^\alpha) + | \tilde \Gamma (\tau)|(|\tau_0| + \mu |\tau_0|^\alpha)\big) |w(\tau_0)|\\
\le & C( |\tau|^{\alpha'} +\mu \ep^{\alpha'}) \big( \big| \tilde \Phi (\tau_0)- \tilde \Phi(\tau)\big| + |w(\tau_0)| + |\tilde \Gamma(\tau_0)| |w_1(\tau_0)|\big)\\
&+ C \big(\mu\ep^\alpha + |\tau_0|^\alpha (1+ \mu |\log (\tau^2+\ep^2)|) \big) |w(\tau_0)| \\
\le & C\Big(( |\tau|^{\alpha} +\mu \ep^{\alpha}) \big| \tilde \Phi (\tau_0)- \tilde \Phi(\tau)\big|+ \big(\mu\ep^\alpha + |\tau|^\alpha+ |\tau_0|^\alpha (1+ \mu |\log (\tau^2+\ep^2)|) \big) |w(\tau_0)|    \\
& + \mu |\tau|^\alpha |\log (\tau_0^2+\ep^2)|  |w_1(\tau_0)| \Big),
\end{align*}
where $\alpha' \in (\alpha, 1)$. This yields inequality \eqref{E:Ray-H2-ap-2} of $w_2(\tau)$. The estimate of $w_1(\tau)$ is obtained through integrating that of $w_2(\tau)= w_{1\tau}(\tau)$. 
\end{proof} 

\noindent $\bullet$ {\bf Convergence estimates as $c_I \to 0+$.}  
As $\ep= \mu^{-1} c_I= \langle k \rangle c_I \to 0+$, from Lemma \ref{L:Gamma-B}, it is natural to expect that the limit of solutions to the non-homogenous Rayleigh equation \eqref{E:Ray-NH-1} is also given by formula \eqref{E:tw-1} with $\Gamma$, $\tilde B$, and $\tilde \Phi$ replaced by $\Gamma_0$, $\tilde B_0$, and $\tilde \Phi_0 =\lim_{\ep \to 0+} \tilde \Phi$. 

With the above preparations, we are ready to obtain the convergence and error estimates of solutions to the Rayleigh equation \eqref{E:Ray-H0-1}. While the limits of non-homogeneous Rayleigh equation under appropriate assumptions on $\phi(k, c, x_2)$ can be studied in the framework in this section, we shall just focus on the homogeneous case, i.e. with $\phi \equiv 0$, and leave the non-homogeneous one to Section \ref{S:Ray-BC}. In fact, \eqref{E:Gamma0} and Lemma \ref{L:Gamma-B} imply that, as $c_I\to 0$, $w_1(\tau)$ would converge to a H\"older continuous limit, while $w_2(\tau)$ develops a jump proportional to $w_1(0)$ and a logarithmic singularity at $\tau=0$. More precisely, the limit $W(\tau)$ of solutions should (see the proposition in the below) satisfy the Rayleigh equation \eqref{E:Ray-H0-1} with $c\in \R$ for $\tau\ne 0$ and satisfy at $\tau=0$, 
\be \label{E:Ray-H0-jump} \begin{cases} 
&W_1 \in C^0\big([-M, M]\big), \quad W_2 \in C^0\big([-M, M] \backslash \{0\}\big), \\
&  \lim_{\tau \to 0+} \big( W_2(\tau) - W_2 (-\tau) \big) = \frac {i\pi \mu U''(x_2^c)}{U'(x_2^c)} W_1(0).
\end{cases} \ee
It is worth pointing out that the existence of the limit of $W_2(\tau) -W_2(-\tau)$ does not imply a simple jump discontinuity of $W_2$, which actually has a symmetric logarithmic singularity. In the distribution sense, the limit homogeneous Rayleigh equation \eqref{E:Ray-H0-1} (with $\phi=0$) along with \eqref{E:Ray-H0-jump} 
can be written as 
\be \label{E:Ray-NH0-1} \begin{split}
W_\tau = & (P.V.)_\tau \begin{pmatrix} 0 & 1 \\  1- \mu^2 + \frac {\mu^2 U'' ( x_2^c + \mu \tau )}{U( x_2^c + \mu \tau ) -c} & 0 \end{pmatrix} W 
+ \begin{pmatrix} 0 \\ \frac {i\pi \mu U''(x_2^c) }{U'(x_2^c)} W_1(0) \end{pmatrix}\delta (\tau). 
\end{split} \ee
Here  $\delta(\tau)$ denotes the delta function of $\tau$ and ``$(P.V.)_\tau$'' indicate the principle value when the corresponding distributions are applied to test functions of $\tau$. They occur in $W_{2\tau}$ only. In terms of the original unknown $y(x_2)$, the limit of \eqref{E:Ray-NH-1} as $c_I \to 0+$ is 
\be \label{E:Ray-H0-3} 
- y'' +  k^2 y + (P.V.)_{x_2} \big(\tfrac{U''y}{U- c}\big)  = 
- \tfrac {i\pi U''(x_2^c)}{U'(x_2^c)}  y(x_2^c)  \delta_{x_2} (x_2-x_2^c), 
\ee
where the subscript $\cdot_{x_2}$ indicates the distributions as generalized functions of $x_2$. 
 For $c_I \to 0-$, the parallel results hold except with the complex conjugate. It also means that homogeneous Rayleigh equation takes different limit as $c_I \to \pm 0$. 

\begin{lemma} \label{L:W-general}
General solutions of homogeneous \eqref{E:Ray-H0-1} (with $\phi=0$) along with \eqref{E:Ray-H0-jump} are 
\be \label{E:W-general} 
W(\tau) = \begin{pmatrix} 1 & 0 \\ \Gamma_0 (\mu, c_R, \tau) &1 \end{pmatrix} \tilde B_0(\mu, c_R, \tau) b_0, \quad b_0=  \begin{pmatrix} b_{01} \\ b_{02} \end{pmatrix}  \in \C^2.
\ee
Moreover, $W(\tau) \in C^0$ if $W_1(0)=0$. 
\end{lemma}

\begin{proof} 
On $[-M, 0)$ and $(0, M]$, \eqref{E:Ray-H0-1} is regular and thus Lemma \ref{L:Gamma-B}, in particular the form \eqref{E:tw-1} of the general solutions implies the above \eqref{E:W-general} with parameters $b_0^\pm =(b_{01}^\pm, b_{02}^\pm)^T \in \C^2$. The continuity of $W_1(\tau)$ and the estimates of $\Gamma$ and $\tilde B$ in Lemma \ref{L:Gamma-B} immediately yields $b_{01}^+= b_{01}^-$. Finally $b_{02}^+ = b_{02}^-$ follows from the jump condition of $W_2(\tau)$ at $\tau=0$ after writing $b_{02}^\pm$ using \eqref{E:b-0} and again using the estimates of $\Gamma$ and $\tilde B$.  

Finally, the continuity of $W(\tau)$ under the assumptions $W_1(0)=0$ follows from \eqref{E:W-general}, the H\"older continuity of $\tilde B$, and the logarithmic upper bound of $\Gamma_0$. 
\end{proof}

The following proposition provides the convergence estimates. 

\begin{proposition} \label{P:converg}
For any $M> 0$ satisfying \eqref{E:M} and $\alpha, \alpha' \in (0, 1)$, there exists $C >0$ depending only on $M$, $\alpha$, $\alpha'$, $|U'|_{C^2}$, and $|(U')^{-1}|_{C^0}$, such that, for any $0< \ep < M$,  $\tau\in  [-M,M]$, 
and solutions $w(\tau)$ and $W(\tau)$ to \eqref{E:Ray-H2} and \eqref{E:Ray-H0-1} (with $\phi=0$) in the forms \eqref{E:tw-1} and \eqref{E:W-general} with parameters $b, b_0 \in \C^2$, respectively,  
the following hold:
\be \label{E:w-d-1b} \begin{split}
|w_1(\tau) - W_1(\tau) - (w_1 (0) - & W_1(0))| \le  C\big( |\tau| (|b_2- b_{02}|+ \mu \ep^\alpha |b_{02}| )\\ 
& + (|\tau|+\mu|\tau|^\alpha)|w_1(0)- W_1(0) | +  \ep^{\alpha'} \mu |\tau|^{1-\alpha'} |W_1(0)| \big),
\end{split}\ee
\be \label{E:w-d-2b} \begin{split}
|w_2(\tau) - W_2  (\tau )|  \le& C \Big( \mu\ep^\alpha |\tau|^{\frac {1-\alpha}2} (|W_1(0)|+ |b_{02}|) + (1 + \mu \big|\log (|\tau| +\ep)\big|)|w_1(0)- W_1(0) | \\
&+ |b_2 - b_{02}|   +\mu \big( \tfrac \ep{\ep+ |\tau|} + \log \big( 1+ \tfrac {C\ep^2}{\tau^2}\big)\big) |W_1(0) | \Big).
\end{split} \ee
Moreover, for any $\tau, \tau_0 \in [-M, M]$, let $\tau_*=\min\{|\tau|, |\tau_0|\}>0$, 
we have 
\be \label{E:w-d-1} \begin{split}
|w_1(\tau) - W_1  (\tau) - (w_1(\tau_0) - & W_1(\tau_0))| \le  C \mu \ep^\alpha |\tau-\tau_0|   |W(\tau_0)|  + C|\tau-\tau_0|^{\alpha}  |w(\tau_0) - W(\tau_0)| \\
&+ C\mu \Big( |\tau - \tau_0| \big(\tfrac \ep{\ep+|\tau_0|} + \log \big( 1+ \tfrac {C\ep^2}{\tau_0^2}\big) \big)  +\ep^{\alpha'} |\tau-\tau_0|^{1-\alpha'}\Big) \\
&\times \big(|W_1(\tau_0)| + |\tau_0|^\alpha |W(\tau_0)|  \big), 
\end{split}\ee
\be \label{E:w-d-2} \begin{split}
|w_2(\tau) -  W_2  (\tau)| \le &  C \Big( \big(1+ \mu |\log (\ep+|\tau|)\big|\big) |w(\tau_0) - W(\tau_0)|    \\
&+   \mu \ep^\alpha |W(\tau_0)|
+\mu \big(\tfrac \ep{\ep+ \tau_*} + \log \big( 1+ \tfrac {C\ep^2}{\tau_*^2}\big)\big)  \big(|W_1(\tau_0)| + |\tau_0|^\alpha |W(\tau_0)| \big)   \Big). 
\end{split}\ee
\end{proposition}

\begin{remark} \label{R:converg}
When the above convergence estimate is applied in the rest of the manuscript, it always holds that $|W_1 (\tau_0)| \le M |\tau_0|^{\alpha_0}$ for some $\alpha_0>0$ which makes the right sides of \eqref{E:w-d-1} and  \eqref{E:w-d-2} converging to $0$ as $\ep \to 0$ locally uniformly in $\tau \ne 0$.   
\end{remark}

\begin{proof} 
We first work on the error estimates in terms of $W_1 (0)$ and $b_2$. Let 
\[
\tilde w(\tau) = \begin{pmatrix} 1 & 0 \\ \Gamma(\mu, c_R, \ep, \tau) &1 \end{pmatrix} \tilde B(\mu, c_R, \ep, \tau) b_0. 
\]
Controlling $w_2 - \tilde w_2$ and $\tilde w_2 - W_2$ by Corollary \ref{C:A-esti-w} ($w_2 - \tilde w_2$ by \eqref{E:A-esti-w-1} and \eqref{E:A-esti-w-2} in particular), where we recall the estimates are uniform in $\ep>0$,  we have 
\be \label{E:temp-3.2-2} \begin{split}
|w_2 (\tau) - W_2&(\tau)| \le  |w_2 (\tau) - \tilde w_2(\tau)| + |\tilde w_2 (\tau) - W_2(\tau)| \\
\le & C \Big( |b  - b_0| + |\tilde \Gamma(\tau)| |b_1 - b_{01} | + \mu \ep^\alpha |\tau|^{\alpha_1} |b_0| +\mu\big( \tfrac \ep{|\tau|+\ep} + \log \big( 1+ \tfrac {C\ep^2}{\tau^2}\big)\big) |b_{01} |  \Big)\\
\le& C \Big( |b_2 - b_{02}|  + (1 + \mu \big|\log (|\tau| +\ep)\big|)|w_1(0)- W_1(0) | \\
&  +\mu\big( \tfrac \ep{\ep+|\tau|} + \log \big( 1+ \tfrac {C\ep^2}{\tau^2}\big)\big)  |W_1(0) | + \mu\ep^\alpha |\tau|^{\alpha_1}\big(|W_1(0) | + |b_{02}| \big) \Big).
\end{split} \ee
where $\alpha_1 \in [0, 1-\alpha)$ and we also used 
\[
W_1 (0) = b_{01}, \quad w_1 (0) = b_{1}.  
\]
This completes the proof of inequality \eqref{E:w-d-2b}. The estimate \eqref{E:w-d-1b} on $w_1-W_1$
is derived by integrating $\p_\tau (w_1-W_1) = w_2 - W_2$ and using \eqref{E:dGamma-p} and \eqref{E:Gamma-d}. 

In the following, based on \eqref{E:temp-3.2-2} we establish the error estimates in terms of initial values given at some $\tau_0 \ne 0$. From formula \eqref{E:b-0} we have 
\begin{align*}
& b  - b_{0}  - \big( \Gamma_0 (\tau_0) W_1(\tau_0) - \Gamma (\tau_0) w_1(\tau_0)\big) (0,\, 1)^T \\
=& \Big(\tilde B^{-1} (w-W) + (\tilde B^{-1} - \tilde B_0^{-1}) W + \Big(\Gamma_0 (W_1-w_1) (\tilde B_0^{-1} -I) + \Gamma_0 w_1 (\tilde B_0^{-1} - \tilde B^{-1}) \\
& + (\Gamma_0-\Gamma) w_1 (\tilde B^{-1} - I)\Big) (0, \, 1)^T\Big)\Big|_{\tau_0}.
\end{align*}
From \eqref{E:temp-3.2-1} and Lemma \ref{L:Gamma-B}, one may estimate 
\[
|\Gamma_0(\tau_0)| |\tilde B(\tau_0)^{-1} -I | \le C \big(1 +\mu  \big|\log |\tau_0| \big|\big) (|\tau_0| +\mu^2 |\tau_0|^{\alpha'}) \le C (|\tau_0|+ \mu |\tau_0|^\alpha), 
\]
\[
|\Gamma_0| |\tilde B_0^{-1} - \tilde B^{-1} | \big|_{\tau_0} \le C \big(1 + \mu \big|\log |\tau_0| \big| \big) \mu \ep^{\alpha} (|\tau_0|^{1-\alpha} + \mu|\tau_0|^{\alpha_1})  \le C \mu \ep^\alpha  |\tau_0|^{\alpha_1}, 
\]
\[
|\Gamma_0 - \Gamma| |\tilde B^{-1} -I | \big|_{\tau_0} \le C\mu(|\tau_0| +\mu^2 |\tau_0|^{\alpha'}) \big(\mu \ep | \log\ep| + \tfrac \ep{\ep+ |\tau_0|} + \log (1+ \tfrac {C\ep^2}{\tau_0^2}) \big) \le C\mu \ep^\alpha  |\tau_0|^{\alpha_1},
\]
where $\alpha_1 \in [0, 1-\alpha)$.  Therefore we obtain 
\begin{align*}
\big|b_2 -&b_{02}  - \big( \Gamma_0 (\tau_0) W_1(\tau_0) - \Gamma (\tau_0) w_1(\tau_0)\big) \big| + \big|b_1 - b_{01} \big| \le  C \big( |(w - W)(\tau_0)| + \mu \ep^\alpha |\tau_0|^{\alpha_1} |W(\tau_0)| \big). 
\end{align*}  
Applying \eqref{E:b-1} and \eqref{E:b-2} to control $b_0$ in \eqref{E:temp-3.2-2}, we can estimate
\begin{align*}
|w_2(\tau) -  W_2  (\tau)| \le&  C \Big(  \big(1+ \mu \big|\log (\ep+ |\tau|)\big|\big) \big( |w(\tau_0) - W(\tau_0)| + \mu \ep^\alpha|\tau_0|^{\alpha_1} |W(\tau_0)| \big)   \\
& + |(\Gamma_0 W_1 - \Gamma w_1)(\tau_0)| + \mu \ep^\alpha \big( |W(\tau_0)|  +  \mu \big|\log |\tau_0|\big| |W_1(\tau_0)| \big) \\
&+\mu \big(\tfrac \ep{\ep+|\tau|} + \log \big( 1+ \tfrac {C\ep^2}{\tau^2}\big)\big) \big(|W_1(\tau_0)| + |\tau_0|^\alpha |W(\tau_0)| \big)   \Big)
\end{align*}
Inequality \eqref{E:w-d-2} is obtained by simplifying the above. In particular, we used 
\[
\ep^\alpha |\log \tau_*| \le C\ep^{\alpha'} \, \text{ if } \, \tau_*\ge \min\{1, \ep^2\} \, \text{ and } \log \big( 1+ \tfrac {C\ep^2}{\tau_*^2}\big) \ge \log \big( \tfrac {\ep^2}{\tau_*} \tfrac 1{\tau_*}\big) \ge |\log \tau_*| \, \text{ if } \, \tau_*\le \min\{1, \ep^2\}, 
\]
to absorb the term 
$\mu^2 \ep^\alpha \big|\log |\tau_0|\big| |W_1(\tau_0)|$. 

Again we integrate \eqref{E:w-d-2} to derive \eqref{E:w-d-1}. The only non-trivial terms are those involving $\tau_*$ 
\begin{align*}
\left| \int_{\tau_0}^\tau \tfrac \ep{\ep+ \min\{|\tau'|, |\tau_0|\}} + \log \big( 1+ \tfrac {C\ep^2}{\min\{|\tau'|, |\tau_0|\}^2}\big)  d\tau'\right| \le& C \ep^\alpha |\tau-\tau_0|^{1-\alpha} + \big| |\tau| - |\tau_0| \big|\big(\tfrac \ep{\ep+|\tau_0|} + \log \big( 1+ \tfrac {C\ep^2}{\tau_0^2}\big) \big) \\
\le & C \ep^\alpha |\tau-\tau_0|^{1-\alpha} + |\tau - \tau_0| \big(\tfrac \ep{\ep+|\tau_0|} + \log \big( 1+ \tfrac {C\ep^2}{\tau_0^2}\big) \big), 
\end{align*}
which are obtained by considering whether $|\tau'| \ge |\tau_0|$ and using \eqref{E:dGamma-p}.  
\end{proof}

\subsection{A priori bounds on the two fundamental solutions $y_\pm(k, c, x_2)$ to the homogeneous Rayleigh equation with $c_I \ge 0$} \label{SS:Ray-H-Fund}

In this subsection, we analyze and derive the basic estimates of of two fixed solutions $y_\pm (k, c, x_2)$  to the homogeneous equation \eqref{E:Ray-H1-1} with initial values 
\be \label{E:y-pm}
y_- (-h) =0, \quad y_-'(-h) =1, \quad \text{ and } \quad y_+ (0)=\frac {(U(0)-c)^2}{g+\sigma k^2}, \quad y_+'(0)=1 + \frac {U'(0)(U(0)-c)}{g+\sigma k^2}, 
\ee
which also depend on parameters $k$ and $c \in \C$. The initial condition of $y_+$ at $x_2=0$ is motivated by the linearized capillary gravity water wave problem \eqref{E:Ray}. (If it had been the linearized Euler equation at a shear flow in the channel, then naturally the boundary condition would be $y_+ (0)=0$ and $y_+'(0)=1$.)
As throughout this section, we often skip the arguments rather than $x_2$.  Particularly when working near $x_2^c = U^{-1} (c_R)$, we shall continue using the notations introduced in Subsection \ref{SS:a-esti}, like $c_R, \mu, \ep$, {\it etc.} The following lemma is standard. Due to conjugacy, we only consider $c_I\ge 0$.



\begin{lemma} \label{L:Ray-H-reg}
For $c \notin U([-h, 0])$ and $x_2 \in [-h, 0]$, the solutions $y_\pm (k, c, x_2)$ are even in $k$, analytic in $k^2$ and $c$, and is $C^{l_0+2}$ in $x_2$. Moreover $y_\pm (k, \bar c, x_2) = \overline {y_\pm (k, c, x_2)}$. 
\end{lemma}

In the next step we give a priori estimates of $y_\pm (k, c, x_2)$. In particular, we consider up to three subintervals,
\be \label{E:kc-away-2}
\CI_2 := (x_{2l}, x_{2r}) = \left\{ x_2 \in [-h, 0]\, : \,   \frac 
1{|U(x_2)-c|} >  
\rho_0 \mu^{-\frac 32} \right\}, 
\quad \rho_0 = \frac 
4{h_0 \inf_{[-h_0-h, h_0]} U'},    
\ee
\be \label{E:CIs}
\CI_1 = [-h, x_{2l}), 
\quad \CI_3=(x_{2r}, 0]. 
\ee
Here $\mu= \langle k \rangle^{-1}$ as in \eqref{E:tau}. Clearly $[-h, 0]= \CI_1 \cup \CI_2 \cup \CI_3$ and any of these subintervals may be empty. 
If $\CI_2 =\emptyset$, then $[-h, 0]$ is considered as $\CI_1$ for $y_-$ and as $\CI_3$ for $y_+$  
in the statement of the following lemma. The choice of the above constant $\rho_0$ and the fact $0\le \mu \le 1$ ensure 
\be \label{E:rho_0}
c_R \in U([-\tfrac 14 h_0-h, \tfrac 14 h_0]) \; \text{ if } \; \CI_2 \ne \emptyset. 
\ee

\begin{lemma} \label{L:y-pm}  
For any $\alpha  \in (0, \frac 12)$, there exists $C >0$ depending only on $\alpha$, $|U'|_{C^2}$, and $|(U')^{-1}|_{C^0}$ (also on $g$ and $\sigma$ for the estimates of $y_+$), such that, for any $c \in \C \backslash U([-h, 0])$, the following hold:
\be \label{E:y-_1}
|\mu^{-1} y_- (x_2) - \sinh (\mu^{-1}(x_2+h))|   \le C \mu^{\alpha}\sinh (\mu^{-1}(x_2+h)),
\ee
\be \label{E:y+_1}
|\mu^{-1} y_+ (x_2) - \sinh (\mu^{-1}x_2)|   \le C\big(\mu^{\alpha} +\mu |c|^2 \big)\cosh (\mu^{-1}x_2), 
\ee
for all $x_2 \in [-h, 0]$. Moreover, if $\CI_2 =\emptyset$, 
then for all $x_2\in [-h, 0]$, 
\be \label{E:y-_2}
|y_-'(x_2) - \cosh (\mu^{-1}(x_2+h))|  \le C \mu^{\alpha}\sinh (\mu^{-1}(x_2+h)),  
\ee
\be \label{E:y+_2}
|y_+'(x_2) - \cosh (\mu^{-1}x_2)|  \le  C \big(\mu^{\alpha} + \mu |c|^2 \big)\cosh (\mu^{-1}x_2).
\ee
If otherwise $\CI_2 \ne \emptyset$, 
then 
\be \label{E:y-_3}
|y_-'(x_2) - \cosh (\mu^{-1}(x_2+h))|  \le
\begin{cases}  
C \mu^{\alpha}\sinh (\mu^{-1}(x_2+h)),  \qquad & x_2 \in \CI_1\\
C\mu^{\alpha}\cosh (\mu^{-1}(x_2+h)),  & x_2 \in \CI_3
\end{cases} \ee
\be \label{E:y+_3}
|y_+'(x_2) - \cosh (\mu^{-1}x_2)|  \le 
C\mu^{\alpha}\cosh (\mu^{-1}x_2), \qquad  x_2 \in \CI_1 \cup \CI_3, 
\ee
and for $x_2 \in \CI_2$, 
\be \label{E:y-_4} \begin{split}
\big|y_-'(x_2) - \cosh (\mu^{-1}(x_2+h)) &- \frac {U''(x_2^c)}{U'(x_2^c)} y_-(x_{2l})\log |U(x_2) - c|  \big| \\
& \le C \mu^{\alpha} \big( 1 + \mu \big| \log |U(x_2) - c|\big| \big)\cosh (\mu^{-1} (x_2+h)),   
\end{split} \ee
\be \label{E:y+_4} \begin{split}
\big|y_+'(x_2) - \cosh (\mu^{-1} x_2) &- \frac {U''(x_2^c)}{U'(x_2^c)} y_+ (x_{2r})\log |U(x_2) - c|  \big| \\
& \le C \mu^{\alpha} \big( 1  + \mu  \big| \log |U(x_2) - c|\big| \big)\cosh (\mu^{-1} x_2).   
\end{split} \ee
\end{lemma}

\begin{remark} \label{R:y-pm}
Even though the lemma assumes $c \in \C \backslash U([-h, 0])$, the estimates are uniform in $c$ and thus they also hold for the limits of solutions as $c_I \to 0+$, while the limits as $c_I \to 0-$ are the conjugates of those as $ c_I \to 0+$. 
Moreover, the constant $C$ does not depend on $\sigma>0$, and in particular, $C$ for $y_-$ does not depend on $g$ either.  
\end{remark}

It is possible that $x_2^c \notin [-h, 0]$ as the domain of $U$ has been extended. However, the constant $C$ in \eqref{E:y-_2}, \eqref{E:y+_2}, \eqref{E:y-_3}, and \eqref{E:y+_3} are independent of the extensions of $U$ satisfying \eqref{E:U-ext}. 


\begin{proof}
The estimates of $y_\pm$ can be derived in exactly the same procedure by reversing the direction of the variable $x_2$. 
We shall focus on $y_-(k,c, x_2)$ and give a brief description on the argument for $y_+$ afterwards. The cases of $x_2$ close to  and away from $x_2^c$ will be considered differently based on Lemma \ref{L:Ray-regular-1} and Proposition \ref{P:converg}, respectively. 

{\it Step 1.} Assume $\CI_1 \ne \emptyset$. We consider $k$ in two cases. The first on is for those larger $|k|$ such that 
\be \label{E:rho-1}
\rho:= \rho_0 \mu^{-\frac 32} k^{-2} (1+|U''|_{C^0 ([-h_0-h, h_0])})  
\le \min\{ 1, C \mu^{\frac 12}\},  
\ee
where \eqref{E:kc-away-1} is satisfied and Lemma \ref{L:Ray-regular-1} is applicable. Observe
\be\label{E:mu-k}
\mu^{-1} -k  = \sqrt{1+k^2} -k = \tfrac 1{\mu^{-1} + k} \in (0, \mu), 
\ee
and 
\be \label{E:mu-k-s} \begin{split}
|\sinh (\mu^{-1}(x_2+h)) - \sinh (k(x_2+h))| =& 2\sinh \tfrac {x_2+h}{\mu^{-1} + k} \cosh (\tfrac 12(\mu^{-1} + k) (x_2+h)) \\
\le & C \mu \sinh (\mu^{-1}(x_2+h)),
\end{split} \ee
where the last inequality could be derived by considering whether $\mu^{-1}(x_2+h) \ge1$. The same upper bound also holds for $\cosh$. Therefore applying Lemma \ref{L:Ray-regular-1} on $\CI_1$ with $s=0$ and $C_0=0$, 
we immediately obtain the desired estimates \eqref{E:y-_1}, \eqref{E:y-_2}, \eqref{E:y-_3} on $y_-$ and $y_-'$ on $\CI_1$, respectively. Otherwise in the case of smaller $|k|$, the desired estimates follows from Lemma \ref{L:regular-small-k} with $\phi=0$. 

{\it Step 2.} Assume $\CI_2 \ne \emptyset$ and $x_{2r} > x_{2l}$ otherwise step 1 has completed the proof. In this case, $x_2^c \in [-\tfrac {h_0}4-h, \tfrac {h_0}4]$ due to \eqref{E:rho_0}. 
Let 
\be \label{E:M-1} 
M = \rho_0^{-1} |(U')^{-1}|_{C^0} = \tfrac 14 h_0, 
\ee
which implies   
\be \label{E:CI-2}  
\CI_2 
\subset [x_2^c -\mu M, x_2^c + \mu M] \subset [x_2^c - 2\mu M, x_2^c +2\mu M] \subset [-h_0-h, h_0].   
\ee
Therefore results in Subsection \ref{SS:a-esti} in the corresponding rescaled variables $w_{1,2}(\tau)$ and $x_2 = x_2^c +\mu \tau$ given in \eqref{E:tau} are applicable. Moreover the definition of $\CI_2$ further yields 
\[
|\tau| = \mu^{-1} |x_2 -x_2^c|  \le C \mu^{\frac 12}, \quad \forall \, x_2 \in\CI_2. 
\]
Let 
\[
\tau_0 = \mu^{-1} (x_{2l} - x_2^c).
\]
Lemma \ref{L:apriori-1} (with $\phi=0$) implies that, for any $x_2 \in \CI_2$
\begin{align*}
\Big| y_-'(x_2) +  \tfrac {U''(x_2^c)}{U'(x_2^c)} y_-(x_{2l})  \log  \tfrac {|U(x_{2l}) - c|}{|U(x_{2}) - c|}&   - y_-'(x_{2l})\Big| \le C  \big(1 + \mu^{\frac {\alpha'}2} \big| \log |U(x_{2l}) - c|\big|\big) |y_-(x_{2l})| \\
& + C \big( \mu^{\frac {\alpha'}2} +\mu |\tau_0|^{\alpha'} \big|\log |U(x_2) - c|\big|\big) \big(\mu^{-1}|y_-(x_{2l})| + |y_-'(x_{2l})|\big), 
\end{align*}
for any $\alpha' \in (0, 1)$. Moving the $\log |U(x_{2l})-c|$ term to the right side, we obtain 
\be \label{E:y-temp-0} \begin{split}
\Big| y_-'(x_2) -  \tfrac {U''(x_2^c)}{U'(x_2^c)} y_-(x_{2l}) & \log |U(x_2) - c|   - y_-'(x_{2l})\Big| \le C  \big(1 + \big| \log |U(x_{2l}) - c|\big|\big) |y_-(x_{2l})| \\
& \quad + C \big( \mu^{\frac {\alpha'}2} +\mu |\tau_0|^{\alpha'} \big|\log |U(x_2) - c|\big|\big) \big(\mu^{-1}|y_-(x_{2l})| + |y_-'(x_{2l})|\big).
\end{split} \ee
Notice that, no matter whether $\CI_1 =\emptyset$ or not, \eqref{E:y-_1} and \eqref{E:y-_2} are satisfied at $x_{2l}$ due to either the initial condition of $y_-(x_2)$ or the above step 1. On the one hand, regarding the above first term on the right side, it holds that either $y_-(x_{2l})=0$ if $x_{2l}= -h$ or $\mu^{\frac 32}\le C |x_2^c -x_{2l}|$ if $x_{2l}>-h$, hence this term would only contribute an error term of at most $O(\mu^{-\alpha''} |y_-(x_{2l})|)$, for any $\alpha'' >0$, 
in the upper bounds. On the other hand,  $0\le x_2 -x_{2l} \le C \mu^{\frac 32}$ implies that replacing the above $\mu^{-1} y_-(x_{2l})$, $y_-'(x_{2l})$ and $\cosh (\mu^{-1} (x_{2l}+h))$ by $\cosh (\mu^{-1}(x_2+h))$ would also only produce an error terms of at most $O\big(\mu^{\frac {\alpha'}2} \cosh (\mu^{-1}(x_2+h)) \big)$ in the upper bounds. Therefore we have 
\be \label{E:y-temp-1} \begin{split}
\big|y_-'(x_2) - \cosh (\mu^{-1}(x_2+h)) &- \tfrac {U''(x_2^c)}{U'(x_2^c)} y_-(x_{2l})\log |U(x_2) - c|  \big| \\
& \le C\big( \mu^{\frac {\alpha'}2} + \mu |\tau_0|^{\alpha'} \big| \log |U(x_2) - c|\big| \big)\cosh (\mu^{-1}(x_2+h)),   
\end{split} \ee
and thus \eqref{E:y-_4} follows by letting $\alpha'=2\alpha$. 

Integrating \eqref{E:y-temp-1} over $[x_{2l}, x_2] \subset \CI_2$, we have, for $\alpha' \in (2\alpha, 1)$,  
\begin{align*}
|\mu^{-1} y_-(x_2) - \sinh  (\mu^{-1} & (x_2+h))| \le C \mu^{\alpha} \sinh (\mu^{-1}(x_{2l} +h)) + \frac C\mu \int_{x_{2l}}^{x_2}| y_-(x_{2l})| \big(1+ \big| \log |x_2' - x_2^c|\big|\big) \\
& \qquad \qquad \quad + \big( \mu^{\frac {\alpha'}2} + \mu |\tau_0|^{\alpha'} \big| \log |x_2' - x_2^c|\big| \big) \cosh (\mu^{-1}(x_{2l}+h)) dx_2'\\
\le & C \mu^{\alpha} \sinh (\mu^{-1} (x_{2} +h)) + C |\tau_0|^{\alpha'} \cosh (\mu^{-1}(x_{2l}+h)) \int_{x_{2l}}^{x_2}  \big| \log |x_2' - x_2^c|\big|  dx_2'.  
\end{align*}
where we used \eqref{E:y-_1}, 
$|x_2 - x_{2l}|\le C \mu^{\frac 32}$, and 
and the first term of the right side of \eqref{E:y-temp-0} was incorporated into others as remarked just below \eqref{E:y-temp-0}. 
For $|x_2- x_{2l}| \le \frac 12 |x_{2l} - x_2^c|$, we have 
\[
|\tau_0|^{\alpha'} \int_{x_{2l}}^{x_2}  \big| \log |x_2' - x_2^c|\big|  dx_2' \le \mu^{-\alpha'} |x_{2l} - x_2^c|^{\alpha'} |x_2 - x_{2l}| \big(1+\big| \log |x_{2l} - x_2^c|\big| \big) \le C \mu^{\alpha}|x_2 - x_{2l}|, 
\] 
while for $|x_2- x_{2l}| \ge \frac 12 |x_{2l} - x_2^c|$, 
\[
|\tau_0|^{\alpha'} \int_{x_{2l}}^{x_2}  \big| \log |x_2' - x_2^c|\big|  dx_2' \le C \mu^{-\alpha'} |x_{2l} - x_2^c|^{\alpha'} |x_2 - x_{2l}|^{1-\frac 13(\alpha' - 2\alpha)}  \le C \mu^{\alpha}|x_2 - x_{2l}|. 
\] 
Therefore we obtain  
\begin{align*}
|\mu^{-1} y_-(x_2) - \sinh  (\mu^{-1}(x_2+h))| \le & C \mu^{\alpha} \big(\sinh (\mu^{-1}(x_{2} +h)) + |x_2 - x_{2l}|\cosh (\mu^{-1}(x_{2l}+h))\big)\\
\le & C \mu^{\alpha} \sinh (\mu^{-1} (x_{2} +h))
\end{align*}
which proves \eqref{E:y-_1} on $\CI_2$. 

{\it Step 3.} Assume $\CI_3 = [x_{2r}, 0] \ne \emptyset$, which implies $x_{2r} >-h$. In this case, surely $\CI_2 \ne \emptyset$ either and $\mu^{\frac 32} \le C|U(x_{2r}) -c|$. With \eqref{E:y-_1} for $y_-$ and \eqref{E:y-_4}  for $y_-'$ established at $x_2 = x_{2r}$, $y_-(x_2)$ satisfies assumption \eqref{E:y-ini-1} for the interval $\CI_3$ with $\Theta_1 =\sinh$, $\Theta_2 =\cosh$, and $C_0= C \mu^\alpha$.  

As in the step 1, for larger $|k|$ so that \eqref{E:rho-1} holds, the desired estimates \eqref{E:y-_1} and \eqref{E:y-_4} in $\CI_3$ follow directly from \eqref{E:mu-k}, \eqref{E:mu-k-s}, and Lemma \ref{L:Ray-regular-1}.

For smaller $k$, say, $|k| \le k_1$, we express $y_-(x_2)$ and $y_-'(x_2)$ in terms of $w(\tau)$, $\tau \in [\mu^{-1}(-h - x_2^c), -\mu^{-1} x_2^c]$, as in \eqref{E:tau}. Let 
\[
M= (1+k_1^2)^{\frac 12} (2h_0+h),  
\quad \tau_0 = \mu^{-1} (-h - x_2^c).
\]
Since $\CI_2 \ne \emptyset$, otherwise $[-h, 0] = \CI_1$ for $y_-(x_2)$, it along with \eqref{E:rho_0} and $|k|\le k_1$ implies 
\[
x_2^c \in [-h_0-h, h_0] \Longrightarrow |h+x_2^c|, \, |x_2^c| \le 2h_0+h \Longrightarrow [\mu^{-1}(-h - x_2^c), -\mu^{-1} x_2^c] \subset [-M, M].
\] 
Namely, the domain of $w(\tau)$ is contained in $[-M, M]$. Applying \eqref{E:Ray-H2-ap-2} (with $\phi=0$), 
using $w_1(\tau_0)=0$, $w_2(\tau_0)=1$, and 
\[
\CI_3 \ne \emptyset \Longrightarrow \rho_0^{-1} \mu^{\frac 32} 
=|U(x_{2r}) -c| \le |U(x_2) -c|, \quad \forall x_2 \in \CI_3,
\]
we obtain $|y_-' (x_2) | \le C$ on $\CI_3$. It in turn implies 
\begin{align*}
|\mu^{-1} y_-(x_2) - \sinh (\mu^{-1}(x_2 +h))| \le & \mu^{-1} |y_-(x_2) - y_-(x_{2r})| + |\mu^{-1} y_- (x_{2r}) - \sinh (\mu^{-1}(x_{2r} +h))|  \\
&+ |\sinh (\mu^{-1}(x_{2r} + h)) - \sinh (\mu^{-1}(x_2+h))|\\ 
\le & C \big( |x_2 - x_{2r}| + \sinh (\mu^{-1}(x_{2r}+h))\big) \le C \sinh (\mu^{-1}(x_2+h)).
\end{align*} 
Therefore \eqref{E:y-_1} and \eqref{E:y-_3} hold on $\CI_3$ due to $|k| \le k_1$. 

{\it Estimating $y_+$} Finally, we give a brief sketch of the argument for $y_+$, for which we proceed from $\CI_3$ to $\CI_1$.  

Suppose $\CI_3 \ne \emptyset$. The initial values of $y_+$ at $x_2=0$ satisfy \eqref{E:y-ini-2} with $\Theta_1=\Theta_2=\cosh$ and $C_0= C (1+|c|^2)\mu$.  For larger $|k|$ so that \eqref{E:rho-1} holds, the desired estimates \eqref{E:y+_1} and \eqref{E:y+_2} in $\CI_3$ follow directly from \eqref{E:mu-k}, \eqref{E:mu-k-s}, and Lemma \ref{L:Ray-regular-1}. The estimates for smaller $k$ is again a consequence of Lemma \ref{L:regular-small-k}. 

Suppose $\CI_2 \ne \emptyset$ which implies $|c|\le C$. Inequality \eqref{E:y-temp-0} with $x_{2l}$ replaced by $x_{2r}$ still follows from exactly the same argument, namely, for $x_2 \in \CI_2$ and any $\alpha' \in [0, 1)$,  
\begin{align*}
\Big| y_+'(x_2) -  \tfrac {U''(x_2^c)}{U'(x_2^c)} y_+(x_{2r}) & \log |U(x_2) - c|   - y_+'(x_{2r})\Big| \le C  \big(1 + \big| \log |U(x_{2r}) - c|\big|\big) |y_+(x_{2r})| \\
& \quad + C \big( \mu^{\frac {\alpha'}2} +\mu |\tau_0|^{\alpha'} \big|\log |U(x_2) - c|\big|\big) \big(\mu^{-1}|y_+(x_{2r})| + |y_+'(x_{2r})|\big).
\end{align*}
If $x_{2r}=0$, then 
\[
\big| \log |U(x_{2r}) - c|\big| |y_+(x_{2r})| = \big| \log |U(0) - c|\big| |y_+(0)| \le C \mu^2 \cosh \mu^{-1} x_2. 
\]
Otherwise, $x_{2r}<0$ and thus, for any $\alpha' \in (0, 1)$, 
\[
|U(x_{2r})-c| = \rho_0^{-1} \mu^{\frac 32} \, \Longrightarrow \big| \log |U(x_{2r}) - c|\big| |y_+(x_{2r})| 
\le \mu^{\alpha'}   \cosh \mu^{-1} x_2, 
\]
where \eqref{E:y+_1} at $x_2 = x_{2r}$ was also used. These estimates, along with \eqref{E:y+_1} and \eqref{E:y+_2} at $x_2=x_{2r}$ yield \eqref{E:y+_4} on $\CI_2$. Inequality \eqref{E:y+_1} follows from direct integrating the estimate on $y_+'$, actually without going through the technical argument at the end of step 2 for $y_-$ since the $\cosh$, instead of $\sinh$, is in the upper bound in \eqref{E:y+_1}. 

Suppose $\CI_1 \ne \emptyset$ where it must hold $\CI_2 \ne \emptyset$ and $|c|\le C$. From step 2, $y_+(x_2)$ satisfies assumption \eqref{E:y-ini-2} for the interval $\CI_1$ with $\Theta_1 = \Theta_2 =\cosh$, and $C_0= C\mu^\alpha $. For larger $|k|$, the desired estimates \eqref{E:y+_1} and \eqref{E:y+_3} follow from Lemma \ref{L:Ray-regular-1} and for smaller $|k|$ from Lemma \ref{L:regular-small-k}.
\end{proof}

\subsection{Limits of solutions to the homogeneous Rayleigh equation with $c_I=0+$
} \label{SS:Ray-H-sing}

Now that the convergence of solutions of the Rayleigh equation as $c_I \to 0+$ has been established in Proposition \ref{P:converg}, in this subsection, we shall focus on the analysis of the limit equation \eqref{E:Ray-H0-1} along with the jump condition \eqref{E:Ray-H0-jump} at the singularity $\tau=0$. In this subsection we consider $c = U(x_2^c) \in U\big([-\tfrac 12 h_0-h, \tfrac 12 h_0]\big)$ unless otherwise specified. 
As transformation \eqref{E:tw-1} was rather helpful in the proof of Proposition \ref{P:converg}, its limit would also turn out to be an effective tool in the study of  \eqref{E:Ray-H0-1}. However $\tilde B(\tau)$ as well as $\tilde B_0(\tau)$ appears only H\"older in $\tau$, or equivalently in $x_2$. 
In the notations given in \eqref{E:tau} in Subsection \ref{SS:a-esti}, we first prove the following lemma to isolate the singularity in $\tilde B_0$. Recall $U\in C^{l_0}$, $x_2^c$ and $c_R$ correspond to each other via \eqref{E:x2c}, $\tilde U, U_1 \in C^{l_0-1}$, and $U_2 \in C^{l_0-2}$ are defined in \eqref{E:tU}, and $\Gamma_0 (\mu, c, \tau) = \Gamma (\mu, c, \ep=0, \tau)$ in \eqref{E:Gamma0}. 

\begin{lemma} \label{L:B} 
Assume $l_0\ge 3$. There exists a unique continuous-in-$\tau$ real $2\times 2$ matrix valued $B(\mu, c, \tau)$  satisfying 
\be \label{E:B-1} 
B_\tau = \begin{pmatrix} 0 & 1 \\ 1-\mu^2 +  \frac {\mu U_2}{\tilde U} &0 \end{pmatrix} B - B \begin{pmatrix} 0 & 0 \\ 1 +\frac {\mu U_2(0)}{\tilde U(0) \tau} &0 \end{pmatrix}, \;\; \tau\ne 0; \; \quad B(\mu, c, 0)=I_{2\times 2}.
\ee
Moreover the following hold. 
\begin{enumerate} 
\item The matrix $B(\mu, c, \tau)$ is $C^{l_0-2}$ in $c\in U\big([-\tfrac 12 h_0-h, \tfrac 12 h_0]\big)$, $\tau$, and $\mu$ and 
\[
\det B=1, \;\; B_\tau(\mu, c, 0) = \begin{pmatrix} - \frac {\mu U''(x_2^c)}{U'(x_2^c)} & 1 \\ \mu^2\big(  -1 + \frac {U'''(x_2^c)}{U'(x_2^c)} - \frac {5U''(x_2^c)^2}{2U'(x_2^c)^2} \big)& \frac {\mu U''(x_2^c)}{U'(x_2^c)} \end{pmatrix}, 
\]
\[
B(0, c, \tau) =  \begin{pmatrix} \cosh \tau -\tau \sinh \tau & \sinh \tau \\ \sinh \tau - \tau \cosh \tau  & \cosh \tau \end{pmatrix} =  \begin{pmatrix} \cosh \tau & \sinh \tau \\ \sinh \tau  & \cosh \tau \end{pmatrix}\begin{pmatrix} 1 & 0 \\ -\tau & 1\end{pmatrix}.  
\]  
Moreover for any $M>0$ satisfying \eqref{E:M}, there exists $C>0$ depending only on $|U'|_{C^{l_0-1}}$  and $|(U')^{-1}|_{C^0}$, such that $|B|_{C^{l_0-2}} \le C$. 
\item $B$ and $\tilde B_0$ are conjugate, namely, 
\be \label{E:conj-tB-B}
B (\mu, c, \tau) = \begin{pmatrix} 1 & 0 \\ \Gamma_0(\mu, c, \tau) & 1\end{pmatrix}  \tilde B_0 (\mu, c, \tau) \begin{pmatrix} 1 & 0 \\ -  \Gamma_0^\# (\mu, c, \tau)    & 1\end{pmatrix}, 
\ee
where 
\[
\Gamma_0^\# (\mu, c, \tau) = \tau + \frac {\mu U''(x_2^c)}{U'(x_2^c)} \big( \log |\tau| + \frac {i\pi}2 (sgn(\tau)+1) \big). 
\]
\item General solutions to \eqref{E:Ray-H0-1} satisfying \eqref{E:Ray-H0-jump} 
are  
\be \label{E:Ray-H0-2} \begin{split} 
W(\tau) = &\begin{pmatrix} W_1(\tau)\\ W_2(\tau) \end{pmatrix} =B (\mu, c, \tau) \begin{pmatrix} 1 & 0 \\ \Gamma_0^\# (\mu, c, \tau) & 1\end{pmatrix} \big(b - \tilde \Phi_0(\mu, c, \tau)\big), \quad b= \begin{pmatrix} b_1 \\ b_2 \end{pmatrix} \in \C^2,\\
= & \begin{pmatrix} (B_{11} +\Gamma_0^\# B_{12}) (b_1 - \tilde \Phi_{01}) + B_{12} (b_2 - \tilde \Phi_{02}) \\  (B_{21} +\Gamma_0^\# B_{22}) (b_1 - \tilde \Phi_{01}) + B_{22} (b_2 - \tilde \Phi_{02}) \end{pmatrix}, 
\end{split}\ee
where $B_{j_1j_2}$ are the entries of $B$ and $\tilde \Phi_0 = (\tilde \Phi_{01}, \tilde \Phi_{02})^T =\lim_{\ep \to 0+} \tilde \Phi$ with $\tilde \Phi$ given in \eqref{E:tPhi-0}.
\item If $\phi\equiv 0$, the general solution $W(\tau)$ to \eqref{E:Ray-H0-1}--\eqref{E:Ray-H0-jump} with $b\in \C^2$ as in \eqref{E:Ray-H0-2} satisfies $W_1 \in C_{loc}^{\alpha'}$ for any $\alpha' \in (0,1)$, $W_1 (0)=  b_1$, and
\[
\lim_{\tau \to 0}  \Big (
W_2(\tau) - b_2 - W_1 (0) \tfrac {\mu U''(x_2^c)}{U'(x_2^c)} \big(\log (U'(x_2^c) |\tau|) + \frac {i\pi}2 (sgn(\tau)+1) \Big) =0.
\]
\item Finally, $W(\tau)$ are $C^{l_0-2}$ in $\mu$, $c_R$, and $\tau$ if $\phi\equiv 0$ and $b_1=W_1(0)=0$.  
\end{enumerate}
\end{lemma}

\begin{remark}
If needed, higher order Taylor expansions of $B$ can be obtained based on \eqref{E:B-2} through rather standard calculations in the analysis of local invariant manifolds. 
\end{remark}

One is reminded that both $\Gamma_0(\tau)$ has a logarithmic singularity and a jump at $\tau=0$ which leads to such singularities of $W_2(\tau)$ there even in the homogeneous case. Since $\Gamma_0 \notin \R$ for $\tau>0$, $\tilde B_0$ should not be real for $\tau>0$. Hence it is a non-obvious statement that this conjugate matrix $B$ is real. The above lemma isolates the singularity of $\tilde B_0$ into the explicit $\Gamma_0$ along with the smooth $B$. Conceptually, the smoothness of $B$ in $c_R$ is related to the smoothness of the spectral resolution of the identity with respect to the spectral parameter, and thus would play crucial role in proving the partial inviscid damping to the linearized Euler equation at the shear flow $U(x_2)$. 

\begin{proof}
The construction of $B(\mu, c, \tau)$ is adapted  from the one in \cite{BSWZ16}, where the main issue is to handle the singularity caused by $\tilde U(\mu, c, 0)=0$. We first make \eqref{E:B-1} autonomous by changing the independent variable an auxiliary one $s$ such that $\tau_s = \tau$ and thus we have 
\be \label{E:B-2} \begin{cases} 
B_s = \begin{pmatrix} 0 & \tau \\ (1-\mu^2) \tau +  \frac {\mu \tau U_2}{\tilde U} &0 \end{pmatrix} B - B \begin{pmatrix} 0 & 0 \\  \tau +\frac {\mu U_2(0)}{U_1(0)} &0 \end{pmatrix},\\
\tau_s = \tau.
\end{cases} \ee   
Obviously solutions to \eqref{E:B-1} correspond (up to a translation in $s$) to those to the $C^{l_0-2}$ ODE system \eqref{E:B-2} of 5-dim which converge to $(I_{2\times 2}, 0)$ as $s \to -\infty$, namely those on the unstable manifold of the steady state $(I_{2\times 2}, 0)$. The linearized system of \eqref{E:B-2} at $(I_{2\times 2}, 0)$ is given by 
\[\begin{cases} 
B_s =  \frac {\mu U''(x_2^c)}{U'(x_2^c)} \mathcal{A} B+\tau \begin{pmatrix} 0 & 1 \\ \mu^2\big( -1 + \frac {U_3(0)}{U_1(0)} - \frac {U_2(0)^2}{2U_1(0)^2}\big) &0 \end{pmatrix}, \quad \text{ where } \mathcal{A} B =\begin{pmatrix} 0 & 0 \\   1 &0 \end{pmatrix} B - B \begin{pmatrix} 0 & 0 \\ 1 &0 \end{pmatrix},  \\
\tau_s = \tau.
\end{cases}\]  
It is easy to compute that, on the one hand, an eigenvector associated to the eigenvalue $1$ is  
\[
(B_1, 1), \quad B_1= \begin{pmatrix} - \frac {\mu U_2(0)}{U_1(0)} & 1 \\  \mu^2\big(  -1 + \frac {U_3(0)}{U_1(0)} - \frac {5U_2(0)^2}{2U_1(0)^2} \big) & \frac {\mu U_2(0)}{U_1(0)} \end{pmatrix}.
\]
On the other hand, one may verify  
\[
e^{s\mathcal{A}} B = \begin{pmatrix} 1 & 0 \\   s &1 \end{pmatrix} B \begin{pmatrix} 1 & 0 \\   -s &1 \end{pmatrix} 
\]
which implies that in the 4-dim center subspace $\{ \tau=0\}$ there is not any decay backward in $s$. 
Therefore there exists a unique $C^{l_0-2}$ unstable manifold of 1-dim which corresponds a unique solution $B(\mu, c, \tau)$ satisfying $B(\mu, c, 0)=I$ and $B_\tau (\mu, c, 0)= B_1$ and $C^{l_0-2}$ in all its variables. In fact, the 4-dim center subspace $\{ \tau=0\}$ is also invariant under the nonlinear system \eqref{E:B-2}, where the flow is given by the above non-decaying linear flow of conjugation. Therefore this $B(\mu, c_R, \tau)$ is the only solution to \eqref{E:B-1} decaying to $I$ as $s \to -\infty$, or equivalently $\tau \to 0+$. Even though this construction is local in $\tau$, the domain of $B$ can be extended due to the linearity of equation \eqref{E:B-1}. 

With the existence of the $C^{l_0-2}$ solution $B(\mu, c, \tau)$ to \eqref{E:B-1} established through \eqref{E:B-2}, letting $\mu =0$ in \eqref{E:B-2} and then transforming back to \eqref{E:B-1}, we have 
\[
B_\tau (0, c, \tau) = \begin{pmatrix} 0 & 1 \\ 1  &0 \end{pmatrix} B(0, c, \tau)  - B(0, c, \tau) \begin{pmatrix} 0 & 0 \\ 1 &0 \end{pmatrix}, \;\; \tau\ne 0;  \quad B(0, c, 0)=I_{2\times 2}.
\]
This equation can be solved explicitly to yield 
\[
B (0, c, \tau) = \begin{pmatrix} \cosh \tau & \sinh \tau \\ \sinh \tau  & \cosh \tau \end{pmatrix}\begin{pmatrix} 1 & 0 \\ - \tau  & 1 \end{pmatrix}= \begin{pmatrix} \cosh \tau -\tau \sinh \tau & \sinh \tau \\ \sinh \tau - \tau \cosh \tau  & \cosh \tau \end{pmatrix}. 
\]

The conjugation relation is the consequence of the facts that both $B$ and the right side of \eqref{E:conj-tB-B} a.) are equal to $I$ at $\tau=0$, b.) satisfy the same ODE system \eqref{E:B-1} for $\tau\ne 0$, c.) are continuous in $\tau$ due to the construction of $B$ and \eqref{E:tB-2} in Lemma \ref{L:Gamma-B}, and d.) the uniqueness of solutions to \eqref{E:B-1} satisfying a.)--c.), which is obtained in the above construction based on the local invariant manifold theory. The property $\det B=1$ follows directly from \eqref{E:conj-tB-B} and \eqref{E:tB-2}.

Formula \eqref{E:Ray-H0-2} of the general solutions follows from the conjugacy relation \eqref{E:conj-tB-B} and Lemma \ref{L:W-general}. 
Under the assumption $\phi \equiv 0$, since $W_2(\tau)$ has at most logarithmic singularity at $\tau=0$ and $W_{1\tau} = W_2$, the H\"older continuity of $W_1$ in $\tau$ follows. From formula \eqref{E:Ray-H0-2} and $|B(\mu, c, 0)-I|= O(|\tau|)$, we obtain $W_1(0)= b_1$. The limit property of $W_2(\tau) - b_2$ also follows from similar calculation. Finally, the $C^{l_0-2}$ smoothness of $W(\tau)$ under the assumptions $\phi\equiv0$ and $b_1 = W_1(0)=0$ is again obvious from the representation of the solution \eqref{E:Ray-H0-2}. 
The proof of the lemma is complete. 
\end{proof}

For $c \in U([-\tfrac {h_0}2 -h, \tfrac {h_0}2])$, with the help of $B(\mu, c, \tau)$ and Lemma \ref{L:B} we shall analyze the $2\times 2$ fundamental matrices in two different forms of the homogeneous problem \eqref{E:Ray-H0-1} with the condition \eqref{E:Ray-H0-jump} at $\tau=0$
\be \label{E:Ray-H0-FM-0} \begin{split}
& S^0 (\mu, c, \tau) =B (\mu, c, \tau) \begin{pmatrix} 1 & 0 \\ \Gamma_0^\#(\mu, c, \tau)  & 1\end{pmatrix}, \\ 
& S(\mu, c, \tau, \tau_0) 
= B (\mu, c, \tau) \begin{pmatrix} 1 & 0 \\ \Gamma_0^\#(\mu, c, \tau) -\Gamma_0^\# (\mu, c, \tau_0) & 1\end{pmatrix} B(\mu, c, \tau_0)^{-1}, 
\end{split} \ee
where $\tau_0$ in $S$ is the initial value of the independent variable and hence $S(\mu, c, \tau_0, \tau_0)=I$. 
To analyze $S^0$ and $S$, let 
\be \label{E:S-sing-1} \begin{split} 
\tilde S^0 (\mu, c, \tau) =&B (\mu, c, \tau) \begin{pmatrix} 0 & 0 \\ 1  & 0\end{pmatrix} = \begin{pmatrix} B_{12} (\mu, c, \tau) & 0 \\ B_{22} (\mu, c, \tau)  & 0\end{pmatrix},  \\
\tilde S (\mu, c, \tau, \tau_0) =& B (\mu, c, \tau) \begin{pmatrix} 0 & 0 \\ 1 & 0 \end{pmatrix} B(\mu, c, \tau_0)^{-1}\\
=  & \begin{pmatrix} B_{12} (\mu, c,  \tau)B_{22} (\mu, c,  \tau_0) & - B_{12} (\mu, c,  \tau)B_{12} (\mu, c,  \tau_0) \\ B_{22} (\mu, c,  \tau) B_{22} (\mu, c,  \tau_0) & - B_{22} (\mu, c,  \tau) B_{12} (\mu, c,  \tau_0)    \end{pmatrix},
\end{split}\ee
where $\det B=1$ was used to compute the more explicit form of $\tilde S$ in the above, and 
\be \label{E:S-err-1} \begin{split}
& S_{err}^0  = S^0   -  \begin{pmatrix} \cosh \tau & \sinh \tau \\ \sinh \tau  & \cosh \tau \end{pmatrix} - \tfrac {\mu U''(x_2^c)}{U'(x_2^c)} \big( \log |\tau|  + \tfrac {i \pi} 2 \big(sgn(\tau) +1 \big)\big) \tilde S^0, \\
& S_{err} 
= S
-  \begin{pmatrix} \cosh (\tau-\tau_0) & \sinh (\tau-\tau_0)  \\ \sinh (\tau-\tau_0)   & \cosh (\tau-\tau_0)  \end{pmatrix} - \tfrac {\mu U''(x_2^c)}{U'(x_2^c)} \big( \log | \tfrac \tau{\tau_0} |  + \tfrac {i \pi} 2 \big(sgn(\tau)- sgn(\tau_0) \big)\big) \tilde S.
\end{split}\ee
The following lemma provides some very basic estimates on $S$. More detailed ones on $S_{jl}$ will be derived when needed.  


\begin{lemma} \label{L:Ray-H0-FM}
Assume $U\in C^{l_0}$, $l_0\ge 3$. The fundamental matrices $S^0(\mu, c, \tau)$ and $S(\mu, c, \tau, \tau_0)$ and their entries $S_{jl}^0$ and $S_{j_1j_2}$ satisfy the following for any $\alpha \in (0, 1)$.  
\begin{enumerate} 
\item $S^0$  is $C^{l_0-2}$ in its variables if $\tau \ne0$ and $S$ is $C^{l_0-2}$ in its variables if $\tau \ne0$ and $\tau_0\ne 0$.  
\item $S_{11}^0$, $S_{12}^0$, and $S_{22}^0$ are $C^\alpha$ in $\tau$ and $C^{l_0-2}$ in $\mu$ and $c$. If $\tau_0 \ne 0$, then $S_{11}$ and $S_{12}$ are $C^\alpha$ in $\tau$ and $C^{l_0-2}$ in $\mu$, $c$, and $\tau_0$.  
\item If $\tau \ne 0$, then $S_{12}$ and $S_{22}$ are $C^\alpha$ in $\tau_0$ and $C^{l_0-2}$ in $\mu$, $c$, and $\tau$. 
\item $S_{12}$ and $\tau_0 S_{11}$ are $C^\alpha$ in $\tau_0$  and $\tau$ and $C^{l_0-2}$ in $\mu$ and $c$.
\item $S_{err}$ and $S_{err}^0$ are $C^{l_0-2}$ in $\mu\in [0, 1]$, $c\in U([-\frac 12 h_0-h, \frac 12 h_0])$, and $\tau, \tau_0 \in [-M, M]$. 
\item For any $M$ satisfying \eqref{E:M}, there exists $C>0$ depending only on $M$, $|U'|_{C^{l_0-1}}$,  and $|(U')^{-1}|_{C^0}$ such that for any $\tau, \tau_0 \in [-M, M]$, 
\[
|\p_\mu^{j_1} D_1\ldots D_{j_2}  S_{err}| \le C|\tau-\tau_0|, \quad |D_1\ldots D_{j_1}  S_{err}| \le C\mu, \quad  |\p_\mu^{j_1} \p_c^{j_2} S_{err}^0 | \le C|\tau|,
\]
for $0\le j_1\le l_0-3$, $0\le j_2 \le l_0-3-j_1$, and $D_1, \ldots D_j \in \{\p_c, \tfrac 1{U'(x_2^c)} (\p_\tau + \p_{\tau_0})\}$, and for $l_0\ge 4$, 
\[
|D_1\ldots D_{j_2}  S_{err}| \le C\mu |\tau-\tau_0|, \quad |\p_c^{j_2} S_{err}^0 | \le C\mu |\tau|,
\]
for $0\le j_2 \le l_0-4$.
\end{enumerate} 
\end{lemma}

The reason we consider $\p_\tau + \p_{\tau_0}$ of $S$ instead of individual $\p_\tau$ or $\p_{\tau_0}$ is not only that it yields better estimate. Recall the change of variables $\tau = \mu^{-1} (x_2 - x_2^c)$. The above fundamental matrix is in the form of $S\big(\mu, c, \mu^{-1} (x_2 - x_2^c), \mu^{-1} (x_{20} - x_2^c)\big)$. Therefore $\p_c - \tfrac {\p_\tau + \p_{\tau_0}}{\mu U'(x_2^c)}$ corresponds to the partial differentiation with respect to $c$ in the $(c, x_2)$ coordinates. Here we also used 
\be \label{E:pc-x2c}
\p_c x_2^c = \tfrac 1{U'(x_2^c)}. 
\ee

\begin{proof}
The argument for $S^0$ and $S$ are very similar and we shall mainly focus on $S$. Let  
\be \label{E:S-reg}
S^\# (\mu, c, \tau, \tau_0) = B (\mu, c, \tau) \begin{pmatrix} 1 & 0 \\ \tau -\tau_0 
& 1\end{pmatrix} B(\mu, c, \tau_0)^{-1}.
\ee
Clearly we have 
\be \label{E:S-decom}
S = S^\# + \tfrac {\mu U''(x_2^c)}{U'(x_2^c)} \Big( \log \big| \tfrac \tau{\tau_0}\big|  + \tfrac {i \pi} 2 \big(sgn(\tau)- sgn(\tau_0) \big)\Big) \tilde S.  
\ee
All the $C^{l_0-2}$ smoothness follows from that of $B$. The $C^\alpha$ H\"older regularity in $\tau$ and $\tau_0$ in statements (2)--(4) is due to $B_{12}(\mu, c, 0)=0$. 

Straight forward computation based on Lemma \ref{L:B} yields
\[
S^\#(0, c, \tau, \tau_0)= \begin{pmatrix} \cosh (\tau-\tau_0) & \sinh (\tau-\tau_0)  \\ \sinh (\tau-\tau_0)   & \cosh (\tau-\tau_0)  \end{pmatrix}, \quad S^\#(\mu, c, \tau_0, \tau_0)= I.
\]
Therefore 
\[
S_{err} (\mu, c, \tau, \tau_0)= S^\# (\mu, c, \tau, \tau_0)- S^\#(0, c, \tau, \tau_0).
\]
It follows immediately that $S_{err}$ and its derivatives in $c$, $\tau$, and $\tau_0$ are of the order $O(|\mu|)$.  
By mimicking $f(\mu, s) = f(0, s) + \mu \int_0^1 f_\mu (\theta_1 \mu, 0) + s\int_0^1 f_{\mu s} (\theta_1 \mu, \theta_2 s) d\theta_2 d\theta_1$, we have  
\begin{align*}
|S^\#(\mu, c, \tau, \tau_0) - S^\#(0, c, \tau, \tau_0)| =& \mu |\tau -\tau_0| \left| \int_0^1 \int_0^1 \p_\mu \p_\tau S^\#(\theta_1 \mu, c, \tau_0 + \theta_2 (\tau-\tau_0), \tau_0) d\theta_2 d\theta_1\right| \\
\le & C\mu |\tau-\tau_0|.  
\end{align*}
Moreover, for $1\le j_2  \le l_0-4$ and $D_1, \ldots D_{j_2} \in \{\p_c, \tfrac 1{U'(x_2^c)} (\p_\tau + \p_{\tau_0})\}$, we have    
\[
D_1\ldots D_{j_2} S^\# =0,  \text{ if } \mu=0, \qquad \p_\mu D_1\ldots D_{j_2} S^\# =0,  \text{ if } \tau=\tau_0.   
\]
A similar procedure yields 
\begin{align*}
|  D_1\ldots D_{j_2} S^\# (\mu, c, \tau, \tau_0)| =& \mu |\tau-\tau_0| \left|\int_0^1 \int_0^1 \p_\mu \p_\tau D_1\ldots D_{j_2} S^\# (\theta_1\mu, c, \tau_0+ \theta_2 (\tau-\tau_0), \tau_0) d\theta_2d\theta_1\right| \\
\le& C\mu |\tau-\tau_0|.  
\end{align*}
Finally, since $S^\#$ is $C^{l_0-2}$ in all variables, for $l_0\ge 4$, $1\le j_1 \le l_0-3$, and $0\le j_2 \le l_0- j_1- 3$, the estimate on $\p_\mu^{j_1} D_1\ldots D_{j_2} S_{err}$ follows from its $C^1$ smoothness and vanishing at $\tau=\tau_0$.  

To analyze $S^0$, parallelly we consider 
\[
S_0^\# (\mu, c, \tau) = B (\mu, c, \tau) \begin{pmatrix} 1 & 0 \\ 
\tau
& 1\end{pmatrix}.
\]
Subsequently we have 
\[
S^0 = S_0^\# + \tfrac {\mu U''(x_2^c)}{U'(x_2^c)} \big( \log |\tau|  + \tfrac {i \pi} 2 (sgn(\tau) +1)\big) \tilde S^0, \;\; S_0^\#|_{\mu =0}= \begin{pmatrix} \cosh \tau & \sinh \tau \\ \sinh \tau  & \cosh \tau \end{pmatrix}, \;\; S_0^\#|_{\tau=0} = I. 
\]
The rest of the proof follows exactly as in the case of $S$.  
\end{proof} 


Recall the expressions \eqref{E:Ray-H0-2} of a solution $W(\tau)$ to the non-homogeneous Rayleigh equation \eqref{E:Ray-H0-1} along with \eqref{E:Ray-H0-jump} at $\tau=0$. 
We can use this formula to solve for the parameter $b$ from $W(\tau_0)$ for some $\tau_0\in [-M, M]$ and then rewrite $W(\tau)$ using the fundamental matrix $S(\mu, c, \tau, \tau_0)$ as 
\be \label{E:Ray-H0-4} \begin{split} 
W (\tau) = & S(\mu, c, \tau, \tau_0) W(\tau_0) - B (\mu, c, \tau) \begin{pmatrix} 1 & 0 \\ \Gamma_0^\#(\mu, c, \tau) & 1\end{pmatrix} \big(\tilde \Phi_0(\mu, c, \tau) - \tilde \Phi_0(\mu, c, \tau_0) \big). 
\end{split} \ee

\subsection{Dependence in $c$ and $k$ of the fundamental solutions to the Homogeneous Rayleigh equation \eqref{E:Ray-H1-1} with $c_I=0+$} \label{SS:pcy0}

In this subsection we revisit the two fundamental solutions 
\be \label{E:y0}
y_{0\pm} (k, c, x_2) = \lim_{c_I \to 0+} y_\pm (k, c+ic_I, x_2), \quad x_2 \in [-h, 0],   
\ee
of the homogeneous Rayleigh equation \eqref{E:Ray-H1-1} for $c\in U([-\tfrac {h_0}2-h, \tfrac {h_0}2])$ satisfying initial conditions \eqref{E:y-pm}. We often skip the dependence on $c$ and $k$ (or equivalently, on $\mu=(1+k^2)^{-\frac 12}$) when there is no confusion. The following lemma is a summary of results from Proposition \ref{P:converg}, Lemmas \ref{L:B}, and Remark \ref{R:A-esti-w}, where $x_2^c$ is defined in \eqref{E:x2c}.  

 \begin{lemma} \label{L:y0}
Assume $U \in C^{l_0}$, $l_0\ge 3$. For $c \in U([-\tfrac {h_0}2-h, \tfrac {h_0}2])$ and $x_2 \in [-h, 0]$, the following hold.
\begin{enumerate} 
\item As $c_I \to 0+$, $y_\pm (k, c + ic_I, x_2) \to y_{0\pm} (k, c, x_2)$ uniformly in $x_2$ and $c$. 
\item As $c_I \to 0+$, $y_\pm' \to y_{0\pm}'$ locally uniformly in $\{U(x_2) \ne c\}$ 
and also in $L_c^\infty L_{x_2}^r$ and $L_{x_2}^\infty L_c^r$ for any $r \in [1, \infty)$. 
\item For each $c$, $y_{0-} (x_2) \in \R$ if $U(x_2) \le c$, $y_{0+}(x_2) \in \R$ if $U(x_2) \ge c$, $y_{0\pm} \in C^\alpha ([-\tfrac {h_0}2-h, \tfrac {h_0}2 ])$ for any $\alpha \in [0, 1)$ and $C^{l_0}$ in $x_2\ne x_2^c$. 
\item Moreover,  
\[
\begin{pmatrix} \tfrac 1\mu y_{0\pm} (x_2)  \\ y_{0\pm}' (x_2) \end{pmatrix}= B \big(\mu, c, \tfrac 1\mu(x_2-x_2^c)\big) \begin{pmatrix} 1 & 0 \\ \Gamma_0^\#\big(\mu, c, \tfrac 1\mu(x_2-x_2^c)\big) & 1\end{pmatrix}\begin{pmatrix} \tfrac 1\mu y_{0\pm} (x_2^c) \\ b_{2\pm} \end{pmatrix}, 
\]
where 
\[
b_{2\pm} = \lim_{x_2 \to x_2^c} \Big( y_{0\pm}' (x_2) - \tfrac {U''(x_2^c)}{U'(x_2^c)} y_{0\pm} (x_2^c) \big( \log \big(\tfrac{U'(x_2^c)}\mu  |x_2 - x_2^c|\big) + \tfrac {i\pi}2 (sgn(x_2-x_2^c)+1)  \big) \Big)  \text{ exists}.
\]
\end{enumerate}
\end{lemma}
 
\begin{remark} \label{R:Ray-lim}
When $c$ takes the  end point values $U(-h)$, 
according to the above representation formula and the smoothness of $B$, actually $y_{0-} \in C^{l_0} ([-h_0-h, h_0])$. 
\end{remark}

\begin{remark} \label{R:jump-at-sing}
Suppose $c \in U\big( (-h, 0)\big)$ and $y(k, c, x_2) = \lim_{\ep\to 0+} y(k, c+i\ep, x_2)$ where $y(k, c+i\ep, x_2)$ is a solution to the homogeneous Rayleigh equation \eqref{E:Ray-H1-1} with $y(-h), y'(-h)\in \R$. The above analysis in Subsection \ref{SS:a-esti} implies that a.) $y(k, c, x_2) \in \R$ for $x_2 \in [-h, x_2^c]$; and b.) if $U''(x_2^c)\ne 0$, 
an imaginary part $\IP\, y(k, c, x_2)$ occurs for $x_2 > x_2^c$ which satisfies the homogeneous Rayleigh equation \eqref{E:Ray-H1-1} for $x_2 \in [x_2^c, 0]$ with initial condition 
\[
\IP\, y(x_2^c) =0,  \quad \IP\, y'(x_2^c) = \tfrac {\pi U''(x_2^c)}{U'(x_2^c)} y (x_2^c). 
\]
\end{remark}


The main goal of this subsection is to analyze the differentiation of $y_{0-}$ in $c$. 
Even though most of the results also hold for $y_{0+}$, the proof is slightly more technical. We shall skip those analysis of $y_+$ as they are not necessary for the rest of the paper.  See Remark \ref{R:y_+}. 
 
The proof of the following lemma would be embedded in those of the four subsequent lemmas, actually mainly Lemma \ref{L:pcy2}.  

\begin{lemma} \label{L:pcy0}
Assume $U\in C^{l_0}$, $l_0\ge 3$. For $k, c\in \R$, it holds that  \\
a.) $y_{0-}$ is locally $C^\alpha$ in both $k$ and $c$ for any $\alpha \in [0, 1)$; \\ 
b.) $(y_{0-}, y_{0-}')$ are locally $C^\alpha$ in both $k$ and $c$ for any $\alpha \in [0, 1)$ at any $(k, c, x_2)$ satisfying $U(x_2) \ne c$;  \\
c.) $(y_{0-}, y_{0-}')$ 
are $C^{l_0-2}$ in both $k$ and $c$ at any $(k, c, x_2)$ satisfying  $U(x_2) \ne c$ and $c \ne U(-h)$; \\
d.) $y_{0-} (k, c, x_2^c)$ is $C^{l_0-2}$ in $c$ and $k$ if $c \in U([-h, 0])$;\\
e.) $(y_{0-}, y_{0-}')$ are $C^{l_0-2}$ in $k$, at any $(k, c, x_2)$ except for $y_{0-}'$ at $c = U(x_2)$;\\
f.) assume $l_0\ge 4$, then,  
for any $l=0,1$, $j_1, j_2 \ge 0$, $j_1+j_2 \le l_0-4$, $r \in [1, \infty)$, and $x_2 \in [-h, 0]$, 
\[
(U(-h) -c)^{j_2} \p_k^{j_1} \p_c^{j_2} \p_{x_2}^l y_{0-} (k, c, x_2), 
\]
is locally $L_k^\infty W_c^{1,r}$ for $c$ near $U(-h)$. 
\end{lemma}

To obtain the estimates, for fixed $c\in \R$ near $U([-h, 0])$, as in Lemma \ref{L:y-pm}, we divide $[-h, 0]$ into subintervals 
\be \label{E:kc-away-4}
\CI_2 := (x_{2l}, x_{2r}) = \left\{ x_2 \in [-h, 0]\, :  \tfrac {1}{|U(x_2)-c|} > \tfrac {\rho_0}\mu \right\},
\quad \CI_1 = [-h, x_{2l}], 
\quad \CI_3=[x_{2r}, 0], 
\ee
where $\rho_0$ is defined as in \eqref{E:kc-away-2}. $\CI_2$ is an interval due to the monotonic assumption of $U$. 
Clearly $[-h, 0]= \CI_1 \cup \CI_2 \cup \CI_3$ and any of these subintervals may be empty. 
If $\CI_2=\emptyset$, 
then $[-h, 0]$ is considered as $\CI_1$ for $y_{0-}$ and as $\CI_3$ for $y_{0+}$. If $\CI_2 \ne \emptyset$, then \eqref{E:rho_0} holds and $x_2^c \in [-\tfrac 12 h_0-h, \tfrac 12 h_0]$ is well defined.
In the next three lemmas, we obtain the estimates on $y_{0-}$ on subintervals in the order of $\CI_1$, $\CI_2$, and $\CI_3$. The proof of Lemma \ref{L:pcy0} is mainly contained in that of Lemma \ref{L:pcy2} as the smooth dependence of solutions to the Rayleigh equation on $k$ and $c$ and the initial values is trivial on $\CI_1$ and $\CI_3$. While we mainly focus on $y_{0-}$ in the following lemmas, we shall also just outline the estimates on $\p_c y_{0+}$, which would be enough for the rest of the paper.  

\begin{lemma} \label{L:pcy1}
Assume $l_0\ge 2$ and $\CI_1 \ne \emptyset$. For any $k\in \R$ and any $c \in \R$, the following estimates hold for $x_2 \in \CI_1$ and $j_1, j_2\ge 0$ with $j_1+j_2>0$, 
\be \label{E:pcy1} \begin{split} 
\mu^{-1} |\p_k^{j_1} \p_c^{j_2} y_{0-} (x_2)| + |\p_k^{j_1} \p_c^{j_2} y_{0-}'(x_2)| \le& C_{j_1, j_2} \mu  \big(|U(x_2) -c|^{-j_2} + |U(-h) -c|^{-j_2}\big) \\
&\times \big( 1+ \mu^{-j_1} (x_2+h)^{j_1} \big) \sinh (\mu^{-1} (x_2+ h)) \\
\le & C_{j_1, j_2} \mu^{1-j_1-j_2} \sinh (\mu^{-1} (x_2+ h)),
\end{split}\ee
where $C_{j_1, j_2}>0$ depends only on $j_1$, $j_2$,  $|U'|_{C^2}$, and $|(U')^{-1}|_{C^0}$. Moreover, it also holds, for any $x_2 \in \CI_3$ 
\[
\mu^{-1} |\p_c y_{0+} (x_2)| + |\p_c y_{0+}' (x_2)| \le C (\sinh \mu^{-1}|x_2| +  \mu (1+ |c|) \cosh \mu^{-1}x_2).  
\]
\end{lemma} 

The above estimate holds in a neighborhood of $\CI_1$ actually. 

\begin{proof} 
It is obvious that, for $x_2 \in \CI_1$, $y_{0-}$ is analytic in $c$ and $k$. Let $K=k^2 = \mu^{-2}-1$. One may compute that $\p_K^{j_1} \p_c^{j_2} y_{0-} (x_2)$ satisfies the non-homogeneous Rayleigh equation \eqref{E:Ray-NH-1} in the form of 
\be \label{E:pcy1-temp-1}
-\p_K^{j_1} \p_c^{j_2} y_{0-}'' + \big(K + \frac {U''}{U-c}\big) \p_K^{j_1} \p_c^{j_2} y_{0-}= - j_1 \p_K^{j_1-1} \p_c^{j_2} y_{0-}- \sum_{j'=0}^{j_2-1}  \frac {m_{j_2, j'} U''}{ (U-c)^{j_2+1-j'}} \p_K^{j_1} \p_c^{j'} y_{0-},  
\ee 
with  some constants $m_{j_2, j'}$. Note that the definition of $\CI_2$ implies that \eqref{E:kc-away-1} is satisfied on $\CI_1$ with 
\be \label{E:rho-2}
\rho= \rho_0 \mu^{-1}k^{-2} (1+|U''|_{C^0([-h_0-h, h_0])}) = \rho_0 k^{-1} \sqrt{1+k^{-2}} (1+|U''|_{C^0([-h_0-h, h_0])}).  
\ee
We shall estimate the derivatives of $y_{0-}$ with respect to $c$ and $k$ for large $k$ and small $k$ separately. 

For any $k_*\ge 1$ sufficiently large so that $\rho\le 1$, we shall apply \eqref{E:y-nh-1} with $x_{02} =-h$ to prove 
\be \label{E:pcy1-temp-2} \begin{split}
\mu^{-1} |\p_K^{j_1} \p_c^{j_2} y_{0-}(x_2) | +& |\p_K^{j_1} \p_c^{j_2} y_{0-}'(x_2) | \\
\le & C_{j_1, j_2} \mu (x_2+h)^{j_1}\big(|U(x_2)-c|^{-j_2} + |U(-h) - c|^{-j_2}\big) \sinh (\mu^{-1} (x_2+h)), 
\end{split} \ee
for any $|k|\ge k_*$, $j_1, j_2 \ge 0$ with $j_1 +j_2 >0$. The proof is a simple mathematical induction in $j_1+j_2$. 

Since \eqref{E:pcy1-temp-2} does not include the case $j_1=j_2=0$, there are two base cases $(j_1, j_2)=(0, 1)$ and $(j_1, j_2)=(1, 0)$, which we have to consider separately. For $\p_c y_{0-}$, from \eqref{E:pcy1-temp-1}, \eqref{E:y-nh-1}, Lemma \ref{L:y-pm}, and the definition of $\CI_2$, we have, for any $x_2 \in \CI_1$,
\begin{align*}
k | \p_c y_{0-}(x_2) |+ |\p_c y_{0-}'(x_2) | \le & C \int_{-h}^{x_2} \cosh (\mu^{-1} (x_2-x_2')) \frac {\mu \sinh (\mu^{-1} (x_2'+h))} {(U(x_2')-c)^{2}} dx_2' \\
\le & C\mu \sinh (\mu^{-1} (x_2+h)) \int_{-h}^{x_2} \frac 1 {(U(x_2')-c)^{2}} dx_2' \\
\le & C \mu \big(|U(x_2)-c|^{-1} + |U(-h) - c|^{-1}\big) \sinh (\mu^{-1} (x_2+h)),
\end{align*}
where \eqref{E:mu-k} and \eqref{E:mu-k-s} are also used for $k\ge k_*$ to convert the estimates in terms of $k$ into those in terms of $\mu$. Similarly, $\p_K y_{0-}$ satisfies 
\begin{align*}
k | \p_K y_{0-}(x_2) |+ |\p_K y_{0-}'(x_2) | \le & C\mu \int_{-h}^{x_2} \cosh (\mu^{-1} (x_2-x_2')) \sinh (\mu^{-1} (x_2'+h)) dx_2' \\
\le & C \mu (x_2+h) \sinh (\mu^{-1} (x_2+h)).
\end{align*}
With the estimates in the base cases established, for $j_1 + j_2>1$, using the induction assumption (and Lemma \ref{L:y-pm} for $j_1=j' =0$ in \eqref{E:pcy1-temp-1}) and proceeding much as in the above, we obtain 
\begin{align*}
& k |\p_K^{j_1} \p_c^{j_2} y_{0-}(x_2) |+ |\p_K^{j_1} \p_c^{j_2} y_{0-}'(x_2) | \\
\le & C\mu \sinh (\mu^{-1} (x_2+h)) \int_{-h}^{x_2}  j_1(x_2'+h)^{j_1-1}\big(|U(x_2')-c|^{-j_2} + |U(-h) - c|^{-j_2}\big) \\
& + (x_2'+h)^{j_1} |U(x_2')-c|^{-2}  \big(|U(x_2')-c|^{1-j_2} + |U(-h) - c|^{1-j_2}\big) dx_2',
\end{align*}
and  \eqref{E:pcy1-temp-2} follows consequently. 

For $|k| \le k_*$, as $\mu \sim 1$, we apply Lemma \ref{L:regular-small-k} to \eqref{E:pcy1-temp-1} on $[-h, x_2]$ with 
\[
C_0 = \max\{ (U(-h) -c)^{-1}, \ (U(x_2) -c)^{-1} \} \le \rho_0\mu^{-1} \le C. 
\]
Following a similar induction procedure and using Lemma \ref{L:regular-small-k}, we obtain, for $x_2 \in \CI_1$, $l=0, 1$, and $j_1, j_2\ge 0$ with $j_1+j_2>0$,  
\begin{align*}
|\p_K^{j_1} \p_c^{j_2} \p_{x_2}^l y_{0-}(x_2)| \le & C_{j_1, j_2} (x_2+h)^{j_1}  \big(|U(x_2)-c|^{-j_2} + |U(-h) - c|^{-j_2}\big).  
\end{align*}  
Therefore \eqref{E:pcy1-temp-2} holds for all $k \in \R$. 

Since 
\[
\p_k = 2k \p_K \Longrightarrow \p_k^j = \sum_{0\le l\le \frac j2} \tilde m_{j, l} k^{j-2l} \p_K^{j-l} 
\]
for some constants $\tilde m_{j, l}$, \eqref{E:pcy1-temp-2} implies \eqref{E:pcy1} on $\CI_1$ (actually in a neighborhood of $\CI_1$). 

{\bf Estimating $\p_c y_{0+}$ on $\CI_3$.} Let $y_1 (x_2)$ be solutions to the homogeneous Rayleigh equation \eqref{E:Ray-H1-1} with initial values 
\[
y_1(0)= -2 (U(0) -c)/(g +\sigma k^2), \quad  y_1' (0)= -U'(0)/(g+\sigma k^2), 
\] 
and $y_2(x_2)$ be the solution to the initial value problem of the non-homogeneous Rayleigh equation 
\[
-y_2'' + \big(k^2 + \tfrac {U''}{U-c}\big) y_{0+} = - \tfrac {U''}{(U-c)^2} y_{0+}, \quad y_2(0)= y_2'(0)=0. 
\]
On $\CI_3$, $y_2$ can be estimated much as $y_{0-}$ on $\CI_1$, while $y_1$ much as in the proof of Lemma \ref{L:y-pm}. When Lemma \ref{L:Ray-regular-1} is used to estimate $y_1$ for large $|k|$, we set $s=0$, $\Theta_1=\Theta_2=\cosh$, and $C_0= C \mu (1+|c|)$. The desired inequality on $\p_c y_{0+}$ follows from  $\p_c y_{0+} = y_1 + y_2$. 
\end{proof}

\begin{lemma} \label{L:pcy2}
Assume $U\in C^{l_0}$, $l_0\ge 3$, and $\CI_2 \ne \emptyset$, then Lemma \ref{L:pcy0}a.)--e.) hold for $x_2\in \CI_2$. Moreover, there exists $C >0$ depending only on $|U'|_{C^{l_0-1}}$, and $|(U')^{-1}|_{C^0}$, such that, for any $k\in \R$ and any $c \in \R$, the following estimates hold.
\begin{enumerate} 
\item For  $1\le j\le l_0-2$,  
\be \label{E:pcy4} \begin{split} 
&| \p_c^j  \big(y_{0-} (k, c, x_2^c)\big) |  \le  C \mu^{1-j} \cosh (\mu^{-1}(x_2^c +h)), \; \text{ if } x_2^c \in [-h, 0]; \\
& \p_c \big(y_{0-} (k, c, x_2^c)\big)|_{c=U(-h)} = U'(-h)^{-1}.   
\end{split} \ee
\item If $l_0\ge 5$, then, for any $x_2 \in \CI_2$, we have  
\be \label{E:pcy3} \begin{split}
\mu^{-1} |\p_c y_{0-} (x_2) | \le & 
C\left(1+ \left|\log \frac {|U(x_2) - c|}{|U(x_{2l})-c|}\right|\right) \sinh (\mu^{-1}(x_{2}+h)), 
\end{split} \ee
\be \label{E:pcy3'} \begin{split}
\Big|\p_c y_{0-}' (x_2) + \frac {U''(x_2^c)}{U'(x_2^c)} \big((P.V.)_c (\frac 1{U(x_2) -c} ) &+ i \pi \delta_c (U(x_2)-c) \big) y_{0-}(x_2^c) \Big|\\
\le & C \Big( 1 + \Big|\log \frac {|U(x_2) - c|}{|U(x_{2l})-c|}\Big|\Big) \cosh (\mu^{-1}(x_{2}+h)),
\end{split} \ee
\[
\mu^{-1} |\p_c y_{0+} (x_2) | \le C\left(1+ \left|\log \frac {|U(x_2) - c|}{|U(x_{2r})-c|}\right|\right) \cosh (\mu^{-1}x_{2}), 
\]
\begin{align*}
\Big|\p_c y_{0+}' (x_2) + \frac {U''(x_2^c)}{U'(x_2^c)} \big((P.V.)_c (\frac 1{U(x_2) -c} ) &+ i \pi \delta_c (U(x_2)-c) \big) y_{0+}(x_2^c) \Big|\\
\le & C \Big( 1 + \Big|\log \frac {|U(x_2) - c|}{|U(x_{2r})-c|}\Big|\Big) \cosh (\mu^{-1}x_{2}),
\end{align*}
and for  $2\le j\le l_0-4$ and $c\ne U(x_2)$ and $c \ne U(-h)$, 
\be \label{E:pcy3-1}\begin{split}
 \mu^{-1} |\p_c^j  y_{0-} (x_2)| 
\le& C  (|U(x_2) -c|^{1-j} + |U(-h) -c|^{1-j}) \sin (\mu^{-1} (x_2+ h)), 
\end{split}\ee 
\be \label{E:pcy3-2}\begin{split}
 |\p_c^j  y_{0-}' (x_2)| 
\le& C \mu |U(x_2) -c|^{-1}(|U(x_2) -c|^{1-j} + |U(-h) -c|^{1-j}) \sinh (\mu^{-1} (x_2+ h)). 
\end{split}\ee 
\end{enumerate}
\end{lemma}

In the above lemma $\delta(\cdot)$ denotes the delta mass supported at $0$ and $(P.V.)_c$ and $\delta_c$ emphasize them as distributions of the variable $c$. 
Near $U(x_2) =c$ or $U(-h)=c$, singular distributions of $\p_c^j y_{0-}$ and $\p_c^j y_{0-}'$ at the level comparable to those negative exponents in \eqref{E:pcy3-1} and \eqref{E:pcy3-2} would occur. The quantities with $log$ upper bounds are $L^p$ functions for any $p\in [1, \infty)$. 


\begin{remark} \label{R:pcy2}
Statement (2) also holds for $l_0\ge 3$ with slightly weaker upper bounds. From the proof, it is easy to see that if $l_0\ge 4$, then \eqref{E:pcy3}, \eqref{E:pcy3'}, and \eqref{E:pcy3-1} and \eqref{E:pcy3-2} for $j\le l_0-3$ hold with an additional $\mu^{-1}$ or all $\sinh$ replaced by $\cosh$ on the right sides. If $l_0\ge 3$, then these inequality hold for $j \le l_0-2$ with all $\sinh$ on the right sides replaced by $\cosh$ besides the additional $\mu^{-1}$.  
\end{remark}

\begin{proof}
Since $\CI_2 \ne \emptyset$, it is easy to prove that \eqref{E:rho_0} holds and $x_2^c \in [-\tfrac 14 h_0-h, \tfrac 14 h_0]$ is well defined. Let $M$ be defined as in \eqref{E:M-1} and \eqref{E:CI-2} still holds. This allows us to work in the $\tau = \mu^{-1} (x_2 -x_2^c)$ coordinate and apply Lemma  \ref{L:B}, \ref{L:Ray-H0-FM}, and \ref{L:y0}. 
It is natural to express $y_{0-}$ using the fundamental matrix $S(\mu, c, \tau, \tau_0)$ defined in \eqref{E:Ray-H0-FM-0}. One is reminded that $x_{2l}$ depends on $c$. To study the regularity of $y_{0-}$ and $y_{0-}'$ with respect to $c$ at some $c_* \in U([-\tfrac 12 h_0-h, \tfrac 12 h_0])$, we fix some   $x_{20}\in [-h, x_{2l}(c_*)]$ (so independent of $c$) in a $O(\mu)$ neighborhood of $x_{2l} (c_*)$. 
For $c$ near $c_*$, $x_2 \in \CI_2$, we can write 
\be \label{E:pcy-temp-1}
\begin{pmatrix} \mu^{-1} y_{0-} (k, c, x_2)  \\ y_{0-}' (k, c, x_2)  \end{pmatrix} = S\big( \mu, c, \tau, \tau_0 \big) \begin{pmatrix} \mu^{-1} y_{0-} (k, c, x_{20})  \\ y_{0-}' (k, c, x_{20})  \end{pmatrix}, \quad \tau = \frac {x_2-x_2^c}\mu, \; \; \tau_0=\frac {x_{20}-x_2^c}\mu.  
\ee
Note that $\tau= \mu^{-1} (x_2-x_2^c) =0$ iff $U(x_2) =c$ and $\tau_0=\mu^{-1} (x_{20}-x_2^c) =0$ iff $U(x_{20}) =c$, the latter of which happens iff $U(-h)=c_*$. 
Clearly $y_-(x_{20})$ and $y_-'(x_{20})$ are smooth in $c$ and $k$ either due to the initial conditions or due to the smoothness of the Rayleigh equation on $\CI_1$. 
Hence the regularity statement (c) of Lemma \ref{L:pcy0} follows from statement (1) in Lemma \ref{L:Ray-H0-FM}. If $c\ne U(x_2)$ is close to $U(-h)$, then we could fix $x_{20}=-h$. In this case, $y_{0-}$ and $y_{0-}'$ involve only $S_{12}$ and $S_{22}$ due to $y_{0-} (-h)=0$, and thus statement (b) follows from statement (3) in Lemma \ref{L:Ray-H0-FM}. When $c$ is close to $U(x_2)$, the $C^\alpha$ regularity of $y_{0-}$ in $k$ and $c$ is a consequence of statement (2) in Lemma \ref{L:Ray-H0-FM}, unless $c=U(x_2)= U(x_{20})= U(-h)$. Near the last exceptional case, the $C^\alpha$ regularity of $y_-$ in $k$ and $c$ is due to (4) of Lemma \ref{L:Ray-H0-FM}. Statement (e) of the $C^{l_0-2}$ smoothness in $k$ of $(y_{0-}, y_{0-}')$ also following from the properties of $S$ given in Lemma \ref{L:Ray-H0-FM}.  

We shall derive the estimates of the differentiation by $\p_c$ at $c_*$ in two cases. 

* {\it Case 1: $x_{2l}(c_*) \ge \mu -h$.} In this case, fix $x_{20} = x_{2l} (c_*)$  which implies $-C \tau_0 \ge 1$. 
Hence $sgn(\tau_0) =-1$ and $\log |\tau_0|$ as well as its derivatives are of order $O(1)$  when $c$ varies slightly. Therefore the $\tau_0$ related terms  can be estimated easily. From the estimate at $x_{20}$ derived in Lemma \ref{L:pcy1} (or from the initial condition at $x_2=-h$), \eqref{E:pc-x2c}, and Lemma \ref{L:Ray-H0-FM}, for $1\le j \le l_0-3$, it holds on $\CI_2$, 
\begin{align*}
\p_c^j \begin{pmatrix} \mu^{-1} y_{0-} (x_2)  \\ y_{0-}' (x_2)  \end{pmatrix} = & \sum_{j'=0}^j \big(\p_c - \tfrac { \p_\tau + \p_{\tau_0}}{\mu U'(x_2^c)}\big)^{j'} \Big(\tfrac {\mu U''(x_2^c)}{U'(x_2^c)} \big(\log |\tfrac \tau{\tau_0}| + \tfrac {i \pi}2(sgn(\tau)- sgn(\tau_0)) \big) \tilde S (\tau, \tau_0)\Big) \\
& \times \p_c^{j-j'} \begin{pmatrix} \mu^{-1} y_{0-} (x_{20})  \\ y_{0-}' (x_{20})  \end{pmatrix} + O\big(\mu^{1-j} \sinh (\mu^{-1} (x_2+h))\big)\\
= & \sum_{j'=0}^j \big(\p_c - \tfrac { \p_\tau + \p_{\tau_0}}{\mu U'(x_2^c)}\big)^{j'} \Big(\tfrac {\mu U''(x_2^c)}{U'(x_2^c)} \big(\log | \tau| + \tfrac {i \pi}2 sgn(\tau) \big) \tilde S (\tau, \tau_0)\Big) \p_c^{j-j'} \begin{pmatrix} \mu^{-1} y_{0-} (x_{20})  \\ y_{0-}' (x_{20})  \end{pmatrix} \\
&+ O\big(\mu^{1-j}  \sinh (\mu^{-1} (x_2+h))\big), \\
=&  \sum_{j'=0}^j \p_c^{j'} \Big(\tfrac {\mu U''(x_2^c)}{U'(x_2^c)} \big(  \log \tfrac {|U(x_2) - c|}\mu + \tfrac {i \pi}2 sgn (U(x_2) -c) \big) \tilde S \big(\tfrac {U(x_2) -c}\mu, \tfrac {U(x_{20}) -c}\mu \big)\Big) \\
&\times \p_c^{j-j'} \begin{pmatrix} \mu^{-1} y_{0-} (x_{20})  \\ y_{0-}' (x_{20})  \end{pmatrix} + O\big(\mu^{1-j}  \sinh (\mu^{-1} (x_2+h))\big),
\end{align*}
where $\tilde S$ is given in \eqref{E:S-sing-1} and the constant $C$ in the $O(\cdot)$ terms depends only on $|U'|_{C^{l_0-1}}$ and $|(U')^{-1}|_{C^0}$. We also used that $\sinh$ and $\cosh$ are comparable at $\tfrac {x_2+h}\mu$ for $x_2 \in \CI_2$ in this case.  

For $j=1$, keeping the most singular terms arising from the derivatives of $\log$ and $sgn$ in the distribution sense, we have 
\begin{align*}
\p_c \begin{pmatrix} \mu^{-1} y_{0-} (x_2)  \\ y_{0-}' (x_2)  \end{pmatrix} = & - \tfrac {\mu U''(x_2^c)}{U'(x_2^c)} \big((P.V.)_c (\tfrac 1{U(x_2) -c} )+  i \pi  \delta_c (U(x_2) -c) \big) \tilde S (\tau, \tau_0) \begin{pmatrix} \mu^{-1} y_{0-} (x_{20})  \\ y_{0-}' (x_{20})  \end{pmatrix} \\
&+ O\big( \big( 1 + \big| \log \tfrac {|U(x_2) -c|}\mu \big| \big) \sinh (\mu^{-1} (x_2+h))\big).
\end{align*}
Using \eqref{E:S-sing-1}, \eqref{E:Ray-H0-2}, the smoothness of $B$ and $B(\mu, c, 0)=I$, one may compute 
\be \label{E:pcy-temp-5.3}
\big( B_{22} (\tau_0),  -B_{12} (\tau_0)\big) \big( \mu^{-1} y_{0-} (x_{20}),  y_{0-}' (x_{20})  \big)^T = \mu^{-1} y_{0-} (x_{2}^c),
\ee 
\be \label{E:pcy-temp-3}
\tilde S (\tau, \tau_0) \begin{pmatrix} \mu^{-1} y_{0-} (x_{20})  \\ y_{0-}' (x_{20})  \end{pmatrix} = \mu^{-1} y_{0-} (x_2^c) \begin{pmatrix} B_{12} (\mu, c, \tau)  \\ B_{22} (\mu, c, \tau)  \end{pmatrix} =  \mu^{-1} y_{0-} (x_2^c) \left( \begin{pmatrix} 0 \\ 1  \end{pmatrix} + O(|\tau|) \right), 
\ee
and it yields the desired estimates for $j=1$ in this case.

Similarly, at $x_2 \ne x_2^c$ for $2\le j \le l_0-3$, keeping the worst term and using \eqref{E:pcy-temp-3}, we have 
\begin{align*}
\p_c^j \begin{pmatrix} \mu^{-1} y_{0-} (x_2)  \\ y_{0-}' (x_2)  \end{pmatrix} = &  \big(\big( - \tfrac { \p_\tau }{\mu U'(x_2^c)}\big)^{j} \log | \tau|  \big) \tfrac {\mu U''(x_2^c)}{U'(x_2^c)} \tilde S (\tau, \tau_0)  \begin{pmatrix} \mu^{-1} y_{0-} (x_{20})  \\ y_{0-}' (x_{20})  \end{pmatrix} \\
&+ O\big(\mu^{1-j} | \tau|^{1-j} \sinh (\mu^{-1} (x_2+h))\big)\\
= & \mu^{1-j} \sinh (\mu^{-1} (x_2+h)) \begin{pmatrix} O(| \tau|^{1-j})    \\  O( | \tau|^{-j})  \end{pmatrix}.
\end{align*}
The desired inequality \eqref{E:pcy3-1} in case 1  follows. 

To finish the analysis in this case, we consider $y_{0-} (k, c, x_2^c)$. From \eqref{E:pcy-temp-5.3}
we obtain $C^{l_0-2}$ smoothness in $k$ and $c$. Differentiating  \eqref{E:pcy-temp-5.3} in $c$ and using Lemma \ref{L:B} and \ref{L:pcy1}, one may estimate, for $1\le j \le l_0-2$,  
\begin{align*}
\big|\p_c^j \big(y(c, x_2^c)\big)\big| \le & \left| \sum_{j'=0}^j \Big( \big((\p_c - \tfrac {\p_\tau}{\mu U'(x_2^c)})^{j-j'} B_{22} \big) (c, \tau_0)  \p_c^{j'} y_{0-} (c, x_{20}) \right.\\
& \quad \left. - \mu \big( (\p_c - \tfrac {\p_\tau}{\mu U'(x_2^c)})^{j-j'} B_{12} \big) (c, \tau_0)  \p_c^{j'} y_{0-}' (c, x_{20}) \Big) \right| \le C \mu^{1-j}  \cosh (\mu^{-1} (x_2^c+h)),
\end{align*} 
which proves \eqref{E:pcy4} in case 1. 

* {\it Case 2: $-h \le x_{2l} (c_*) \le \mu-h$.} In this case, let $x_{20} =-h$. While we have to deal with possibly very small $\tau_0$ in \eqref{E:pcy-temp-1}, the initial values $(y_{0-} (x_{20}), y_{0-}'(x_{20})) =(0,1)$. Hence from Lemma \ref{L:Ray-H0-FM} we obtain,  for $0\le j \le l_0-4$, $x_2 \in \CI_2$, 
\be \label{E:pcy-temp-5} \begin{split}
&\p_c^j \begin{pmatrix} \mu^{-1} y_{0-} (x_2)  \\ y_{0-}' (x_2)  \end{pmatrix} =  \big(\p_c - \tfrac { \p_\tau + \p_{\tau_0}}{\mu U'(x_2^c)}\big)^{j} \begin{pmatrix} S_{12} (\tau, \tau_0) \\  S_{22} (\tau, \tau_0) \end{pmatrix}  \\
= & \big(\p_c - \tfrac { \p_\tau + \p_{\tau_0}}{\mu U'(x_2^c)}\big)^{j} \Big(\tfrac {\mu U''(x_2^c)}{U'(x_2^c)} \big(\log |\tfrac \tau{\tau_0}| + \tfrac {i \pi}2(sgn(\tau)- sgn(\tau_0)) \big) \begin{pmatrix} \tilde S_{12} (\tau, \tau_0) \\ \tilde S_{22} (\tau, \tau_0) \end{pmatrix}\Big) \\
&   + O(\mu^{1-j} |\tau-\tau_0|). 
\end{split} \ee

From \eqref{E:S-sing-1} and Lemma \ref{L:B}, $\tfrac {\tilde S_{12}}{\tau \tau_0}$ and $\tfrac {\tilde S_{22}}{\tau_0}$ are $C^{l_0-3}$ functions, which could be used to reduce some singularity. As $\mu |\tau-\tau_0|= |x_2 +h|$, one can compute for $j=1$,  
\begin{align*}
\mu^{-1} \p_c y_{0-} (x_2) = & - \tfrac {U''(x_2^c) }{U'(x_2^c)^2} \tfrac {\tilde S_{12}}{\tau \tau_0} (\p_\tau + \p_{\tau_0}) \Big( \tau \tau_0 \big(\log |\tfrac \tau{\tau_0}| + \tfrac {i \pi}2(sgn(\tau)- sgn(\tau_0)) \big) \Big) \\
& +O\Big( \big| \tau \tau_0 \big(\log |\tfrac \tau{\tau_0}| + \tfrac {i \pi}2(sgn(\tau)- sgn(\tau_0)) \big)\big|\Big) +  O (\mu^{-1} |x_2 +h|).
\end{align*}
We use the following elementary inequalities to handle the above $\log$ terms:
\[
\big| \log |\tfrac \tau{\tau_0}| \big| =  \left|\int_{|\tau|}^{|\tau_0|} \tfrac 1{\tau'} d\tau'\right| \le \frac {\big||\tau| -|\tau_0|\big|}{\min\{ |\tau|, |\tau_0|\}}, \quad |\tau| + |\tau_0| \le |\tau-\tau_0| + 2\min\{|\tau|, |\tau_0|\},    
\]
which also imply
\[
\big|\tau \tau_0 \log |\tfrac \tau{\tau_0}| \big| \le C |\tau -\tau_0|, \quad ( |\tau| + |\tau_0|)\big| \log |\tfrac \tau{\tau_0}| \big| \le |\tau -\tau_0| \big( 2+ \big| \log |\tfrac \tau{\tau_0}| \big| \big). 
\]
The delta functions produced by differentiating $sgn$ are cancelled by $\tau\tau_0$. Finally $sgn(\tau)- sgn(\tau_0) \ne 0$ only when $-h \le x_2^c \le x_2$ which implies $\mu (|\tau_0| + |\tau|) = x_2 +h$. Summarizing these estimates we obtain 
\[
|\mu^{-1} \p_c y_{0-} (x_2) | \le C \mu^{-1} |x_2 +h| \big( 1 + \big| \log |\tfrac \tau{\tau_0}| \big| \big). 
\] 
If $\mu-h \ge x_{2l}(c_*) >-h$, then $\tfrac 1C \mu \le x_2^c(c_*) - x_{2l}(c_*) \le - \mu \tau_0 \le C \mu$, while $\mu |\tau_0| = |x_{2l}(c_*) - x_2^c (c_*)|$ if $x_{2l}(c_*) = -h$. Hence $\big| \log |\mu\tau_0| - \log |x_{2l} (c_*) - x_2^c (c_*)|\big| \le C$, which along with the estimate in case 1 yields \eqref{E:pcy3}.

Much as in the above, we estimate $\p_c y_{0-}' (x_2)$ in case 2 using \eqref{E:S-sing-1} and Lemma \ref{L:B}
\begin{align*}
\p_c y_{0-}' (x_2) = & \tfrac {\mu U''(x_2^c) }{U'(x_2^c)} \tfrac {\tilde S_{22}}{\tau_0} \tfrac {U(x_{20})-c}\mu \p_c  \big(\log \tfrac {|U(x_2)-c|}{|U(x_{20}-c|} + \tfrac {i \pi}2(sgn(U(x_2)-c)- sgn(U(x_{20})-c)) \big) + O\big(1+ \big| \log |\tfrac \tau{\tau_0}|\big| \big) \\
=& - \tfrac {\mu U''(x_2^c) }{U'(x_2)} \tilde S_{22}  \big((P.V.)_c (\tfrac 1{U(x_{2})-c}) + i \pi\delta_c(U(x_{2})-c) \big) + O\big(1+ \big| \log |\tfrac \tau{\tau_0}|\big| \big)\\
=&  \tfrac {\mu U''(x_2^c) }{U'(x_2^c)} B_{12}(\tau_0)  \big((P.V.)_c (\tfrac 1{U(x_{2})-c}) + i \pi\delta_c(U(x_{2})-c) \big) + O\big(1+ \big| \log |\tfrac \tau{\tau_0}|\big| \big).
\end{align*}
It  along with  \eqref{E:pcy-temp-5.3}
implies \eqref{E:pcy3'}. 

Similarly, for $\tau \ne 0$, $\tau_0\ne 0$, and $2\le j \le l_0-4$, where $sgn$ are constants, one may compute
\begin{align*}
\p_c^j \begin{pmatrix} \mu^{-1} y_{0-} (x_2)  \\ y_{0-}' (x_2)  \end{pmatrix} \le & C\mu^{1-j} \Big( \sum_{j'=0}^j  \left| ( \p_\tau + \p_{\tau_0})^{j'}  \Big( 
\begin{pmatrix} \tau \tau_0 \\ \tau_0 \end{pmatrix}  \log |\tfrac \tau{\tau_0}|\Big)\right| +  |\tau-\tau_0| \Big). 
\end{align*}
For $j\ge3$, we have 
\begin{align*}
& | ( \p_\tau + \p_{\tau_0})^{j}  \big(\tau\tau_0 \log |\tfrac \tau{\tau_0}| \big) \le C(|\tau^{2-j} - \tau_0^{2-j}| + |\tau +\tau_0| |\tau^{1-j} - \tau_0^{1-j}| + |\tau\tau_0||\tau^{-j} - \tau_0^{-j}|)  \\
\le & C |\tau\tau_0| |\tau^{-1}-\tau_0^{-1}| (|\tau|^{1-j} + |\tau|^{2-j} |\tau_0|^{-1} + \ldots + |\tau_0|^{1-j}) \le  C |\tau-\tau_0| (|\tau|^{1-j}  + |\tau_0|^{1-j}).  
\end{align*} 
If $j=2$, the first term on the right side of the first inequality would be $\log  |\tfrac \tau{\tau_0}|$ which as shown previously also satisfies the above final estimate. Similarly, one can also calculate,  for $\tau\ne 0$, $\tau_0\ne 0$, and $j\ge 2$, 
\begin{align*}
&| ( \p_\tau + \p_{\tau_0})^{j}  \big(\tau_0 \log |\tfrac \tau{\tau_0}| \big) \le  C( |\tau^{1-j} - \tau_0^{1-j}| + |\tau_0||\tau^{-j} - \tau_0^{-j}|)  \\
\le & C |\tau_0| |\tau^{-1}-\tau_0^{-1}| (|\tau|^{1-j} + |\tau|^{2-j} |\tau_0|^{-1} + \ldots + |\tau_0|^{1-j}) \le  C|\tau|^{-1} |\tau-\tau_0| (|\tau|^{1-j}  + |\tau_0|^{1-j}).   
\end{align*}
The cases of $j=0,1$ have been considered earlier and would only make minor contributions. Therefore \eqref{E:pcy3-1} and \eqref{E:pcy3-2} are satisfied in case 2  as well. 

Regarding $y_{0-} (k, c, x_2^c)$, much as in case 1, but with much simpler initial value at $\tau_0 = \tfrac {-h-x_2^c}\mu$, we have 
\[
y_{0-}(k, c, x_2^c) = - \mu B_{12} (\mu, c, \tau_0)  
\]
which also yields its $C^{l_0-2}$ smoothness. Differentiating in $c$ and using Lemma \ref{L:B} and \ref{L:pcy1}, one may estimate, for $1\le j \le l_0-2$,  
\begin{align*}
\big|\p_c^j \big(y_{0-}(k, c, x_2^c)\big)\big| = & \mu \big|\big((\p_c - \tfrac {\p_\tau}{\mu U'(x_2^c)})^j B_{12} \big) (\mu, c, \tau_0) \big| \le C \mu^{1-j}, 
\end{align*} 
which proves the inequality in \eqref{E:pcy4} in case 2. Finally, from Lemma \ref{L:B}, 
\[
\p_c \big(y_{0-} (k, c, x_2^c)\big)|_{c=U(-h)} = -\mu \big((\p_c - \tfrac {\p_\tau}{\mu U'(-h)}) B_{12} \big) (\mu, U(-h), 0) = U'(-h)^{-1}. 
\]

{\bf Estimating $\p_c y_{0+}$ on $\CI_2$.} In this case $\CI_2\ne \emptyset$ implies $|c|\le C$. Much as in the above argument for $y_{0-}$, we consider the estimates related to $\p_c y_{0+}$ at some $c_* \in [-\frac 12 h_0-h, \frac 12 h_0]$. Observe that, as an expression of solution to the homogeneous Rayleigh equation, \eqref{E:pcy-temp-1} also applies to $y_{0+}$ on $\CI_2$ with $x_{20}$ chosen near $x_{2r}(c_*)$. In the case of $x_{2r}(c_*) \le -\mu$,  the same arguments yields the desired estimates of $\p_c y_+$. 

In the case of $x_{2r}(c_*) \in [-\mu, 0]$, we take $x_{20}=0$ and proceed roughly as in the above case 2. Due to the initial condition \eqref{E:y-pm}, equation  \eqref{E:pcy-temp-5} is replaced by 
\begin{align*}
&\p_c \begin{pmatrix} \mu^{-1} y_{0+} (x_2)  \\ y_{0+}' (x_2)  \end{pmatrix} = \big(\p_c - \tfrac { \p_\tau + \p_{\tau_0}}{\mu U'(x_2^c)}\big) \left( S(\tau, \tau_0) \begin{pmatrix} \mu^{-1} y_{0+} (0)  \\ y_{0+}' (0)  \end{pmatrix}\right)\\
=&  \big(\p_c - \tfrac { \p_\tau + \p_{\tau_0}}{\mu U'(x_2^c)}\big) \left(\tfrac {\mu U''(x_2^c)}{U'(x_2^c)} \big(\log |\tfrac \tau{\tau_0}| + \tfrac {i \pi}2(sgn(\tau)- sgn(\tau_0)) \big) \tilde S (\tau, \tau_0) \begin{pmatrix} \mu^{-1} y_{0+} (0)  \\ y_{0+}' (0)  \end{pmatrix}\right)   + O( |\tau-\tau_0|), 
\end{align*}
where \eqref{E:S-err-1} and Lemma \ref{L:Ray-H0-FM} are used. Let 
\[
W(\tau, \tau_0) = (W_1, W_2)^T = \tilde S(\tau, \tau_0) \big(  \mu^{-1} y_{0+} (0),  y_{0+}' (0) \big)^T. 
\]
Recall from  initial condition \eqref{E:y-pm}
\[
|y_{0+}(0)| \le C \mu^4 \tau_0^2, \quad |y_{0+}'(0)-1| \le C \mu^3 |\tau_0|. 
\]
On the one hand, from \eqref{E:S-sing-1}, and Lemma \ref{L:B}, we have that 
\[
\Big( \frac {W_1}{\tau\tau_0}, \frac {W_2}{\tau_0}\Big) = \Big(\frac {B_{12} (\tau)}{\tau}, B_{22} (\mu) \Big) \Big( B_{22}(\tau_0)  \frac {y_{0+}(0)}{\mu \tau_0} - \frac {B_{12} (\tau_0)}{\tau_0} y_{0-}'(0) \Big)
\]
are $C^{l_0-3}$ function with bounds uniform in $c$ and $\mu$. Hence the estimate on $\p_c y_{0+}(x_2)$ is obtained much as that of $\p_c y_{0-}(x_2)$. On the other hand, as \eqref{E:pcy-temp-3} and \eqref{E:pcy-temp-5.3} also apply to $y_{0+}$, it holds 
\[
W_2 =B_{22} (\tau) \mu^{-1} y_{0+} (x_2^c)  = (1+ O(|\tau|)) \mu^{-1} y_{0+} (x_2^c). 
\]
With these estimates, the desired estimate on $\p_c y_{0+}' (x_2)$ follows much as that of $\p_c y_{0-}'(x_2)$. This completes the proof. 
\end{proof} 


\begin{lemma} \label{L:pcy3}
Assume $l_0\ge 3$, $\CI_2 \ne \emptyset$, and $\CI_3 \ne \emptyset$, then Lemma \ref{L:pcy0}a.)--e.) hold for $x_2 \in \CI_3$. Moreover, if $l_0\ge 5$, then there exists $C >0$ depending only on $|U'|_{C^{l_0-1}}$ and $|(U')^{-1}|_{C^0}$, such that, for any $k\in \R$ and any $c \in \R$, the following estimates hold for $x_2 \in \CI_3$ 
\be \label{E:pcy2} 
\mu^{-1} |\p_c y_{0-} (x_{2})| + |\p_c y_{0-}' (x_{2})| \le C \Big(1 +  \log \frac \mu{\min\{\mu,  |U(-h)-c|\}} \Big) \cosh (\mu^{-1}(x_{2}+h)), 
\ee
and for $2\le j \le l_0-4$, 
\[
\mu^{-1} |\p_c^j y_{0-} (x_{2})| + |\p_c^j y_{0-}' (x_{2})| \le C\big( \mu^{1-j} + |U(-h)-c|^{1-j} 
\big) \cosh (\mu^{-1}(x_{2}+h)).
\]
Moreover, if $\CI_2\ne \emptyset$ and $\CI_1\ne \emptyset$, then it also holds for $x_2 \in \CI_1$, 
\[
\mu^{-1} |\p_c y_{0+} (x_{2})| + |\p_c y_{0+}' (x_{2})| \le C \Big(1 +  \log \frac \mu{\min\{\mu,  |U(0)-c|\}} \Big) \cosh (\mu^{-1}x_{2}). 
\]
\end{lemma}

\begin{remark} 
Using Remark \ref{R:pcy2}, the above estimates with an an additional $\mu^{-1}$ on the right sides also hold for $l_0\ge 4$ and $j \le l_0-3$.
\end{remark} 

\begin{proof}
The assumption $\CI_2\ne \emptyset$ and $\CI_3=[x_{2r}, 0]\ne \emptyset$ imply 
$x_{2r}>-h$ and 
$\mu^{-1} \big(U(x_{2r}) -c \big)$ is uniformly bounded from above and below away from 0. The regularity of $y_{0-}$ and $y_{0-}'$ in $c$ and $k$ for $x_2 \in \CI_3$
follow directly from such smoothness at $x_{2r}$ obtained in Lemma \ref{L:pcy2}. Their estimates at $x_{2r}$ can be summarized into 
\[
\mu^{-1} |\p_c y_{0-} (x_{2r})| + |\p_c y_{0-}' (x_{2r})| \le C \Big(1 +  \log \frac \mu{\min\{\mu,  |U(-h)-c|\}} \Big) \cosh (\mu^{-1}(x_{2r}+h)), 
\]
and for $2\le j \le l_0-4$ 
\[
\mu^{-1} |\p_c^j y_{0-} (x_{2r})| + |\p_c^j y_{0-}' (x_{2r})| \le C\big( \mu^{1-j} + |U(-h)-c|^{1-j} \big) \cosh (\mu^{-1}(x_{2r}+h))
\]
where we also used $1 \le \mu^{-1}(x_{2r}+h)$ as $x_{2l}>-h$. 
Much as the proof of Lemma \ref{L:pcy1}, we shall obtain the estimates  inductively in $j$ by considering the cases of small and large $k$ separately. 

As $\rho_0>0$, we take $k_*\ge 1$ such that $ \rho <1$ (defined in \eqref{E:rho-2}) for $|k|\ge k_*$ and thus 
\eqref{E:kc-away-1} is satisfied on $\CI_3$ with $\rho< \min\{1, C \mu\}$. We shall obtain the estimates for this case of $|k|\ge k_*$ by splitting $\p_c y_{0-}$ into homogeneous and non-homogeneous parts. 
For $j \ge 1$, let $y_1(x_2)$ be the solution to the homogeneous Rayleigh equation \eqref{E:Ray-H1-1} with initial condition 
\[
y_1(x_{2r}) = \p_c^j y_{0-}(x_{2r}), \; \; y_1'(x_{2r})=\p_c^j y_{0-}'(x_{2r}), 
\]
and $y_2(x_2)$ be the solution to the non-homogeneous Rayleigh equation \eqref{E:Ray-NH-1} with the zero initial conditions at $x_2 = x_{2r}$ and the non-homogeneous term given by the right side of \eqref{E:pcy1-temp-1} (with $j_1=0$ and $j_2=j$).
Clearly it holds 
\be \label{E:pcy-temp-6}
\p_c^j y_{0-} = y_1+y_2, \;\text{ on } \; \CI_3.
\ee 
Using the the above estimates on $\p_c y_{0-}$ at $x_{2r}$,
we apply Lemma \ref{L:Ray-regular-1} to $y_1$  
with 
\[
\Theta_1 =\Theta_2 =\cosh, \quad s=0, \quad C_0= \mu^{-1} |\p_c^j y_{0-} (x_{2r})| +|\p_c^j y_{0-}' (x_{2r})| +1 
\]
to obtain, for $x_2 \in \CI_3$,  
\begin{align*}
& \mu^{-1} |y_{1} (x_2)| + |y_{1}' (x_2)| \le C \big(\mu^{-1} |\p_c^j y_{0-} (x_{2r})| +|\p_c^j y_{0-}' (x_{2r})| +1 \big) \cosh \mu^{-1} (x_2 - x_{2r}).  
\end{align*}
Concerning $y_2(x_2)$, Lemmas \ref{L:Ray-regular-1} and the same computation as in the proof of Lemma \ref{L:pcy1} implies, for any $x_2 \in \CI_3$,  
\[
|\mu^{-1} y_2(x_2)| + |y_2' (x_2)| \le C \sum_{j'=0}^{j-1} \int_{x_{2r}}^{x_2}  \frac {\cosh(\mu^{-1} (x_2- x_2'))}{|U(x_{2}') -c|^{j+1-j'} } |\p_c^{j'} y_{0-} (x_2')| dx_2'.  
\] 
The desired estimate for $j=1$ follows from \eqref{E:pcy-temp-6}, Lemma \ref{L:y-pm}, and direct integration. For $j\ge 2$, one may compute inductively using the above estimates and \eqref{E:pcy-temp-1},
\begin{align*}
& |\mu^{-1} y_2(x_2)| + |y_2' (x_2)| \le C \sum_{j'=0}^{j-1} (U(x_{2r}) -c)^{j'-j} \frac {|\p_c^{j'} y_{0-} (x_{2r})|}{\cosh\mu^{-1} (x_{2r}+h)} \cosh\mu^{-1} (x_2 +h) \\
\le & C\Big( \mu^{1-j} + \mu^{-1} |U(-h)-c|^{2-j} +\mu^{1-j} \log \frac \mu{\min\{\mu,  |U(-h)-c|\}}\Big)\cosh\mu^{-1} (x_2 +h). 
\end{align*}
If $|U(-h)-c|\ge \mu$, the desired estimate follows immediately, otherwise it follows from the fact $\log x \le x$ for any $x\ge 1$. 

In the case $k\le k_*$, $\mu\sim 1$ and Lemma \ref{L:regular-small-k} yields the estimates through a similar induction. 

The estimates on $\p_c y_{0+}$ is also obtained much as $\p_c y_{0-}$ using Lemmas \ref{L:Ray-regular-1} and \ref{L:regular-small-k} based on the estimates of $\p_c y_{0+}$ at $x_{2l}$ obtained in Lemma \ref{L:pcy2}. In particular, the fact that $\CI_2 \ne \emptyset$ also implies $|c|<C$ is also used.  
We skip the details. 
\end{proof}


The following lemma proves Lemma \ref{L:pcy0}(f) and (the case of $x_2=0$) will be used in analyzing the eigenvalues. 

\begin{lemma} \label{L:pcy-4}
Assume $U \in C^{l_0}$, $l_0\ge 3$. For any $k \in \R$ and $x_2 \in [-h, 0]$, there exist $R, \tilde C>0$ such that 
\be \label{E:pcy5}
|U(-h) -c|^{j} | \p_k^{j_1} \p_c^{j_2} \p_{x_2}^l y_{0-} (k, c, x_2) | \le \begin{cases}  \tilde C \big(1+ \big| \log |U(-h)-c|\big|\big), & j +1= j_2  \in \{0, 1\}, \\
\tilde C, & j +1= j_2\ge 2, 
\end{cases} \ee  
for any $|c-U(-h)|\le R$, $l=0,1$, $j_1, j_2 \ge 0$, $ j_1+j_2\le l_0-3$. Here $\tilde C$ can be taken independent of $k$ for $k$ in any bounded set.  
\end{lemma} 

Unlike in most other lemmas, the constants $R$ and $\tilde C$ may depend on $k$ and $x_2$. 

\begin{proof}
The lemma is trivial if $x_2 =-h$, so we assume $x_2 > -h$. Since the lemma is concerned with $c$ close to $U(-h)$ where $R$ and $\tilde C$ may depend on $x_2$ and $k$, we consider 
\[
c = U(x_2^c), \;\; x_2^c \in [-h_0-h, (-h+x_2)/2] \, \Longrightarrow \, \tau =\mu^{-1} (x_2 - x_2^c) \ge \mu^{-1} (x_2+ h)/2 >0. 
\]
If $x_2 - x_2^c >\mu$, then \eqref{E:pcy5} clearly holds as $x_2$ is away from the singularity of $y_{0-}$. Otherwise,  let $\tau_0 = - \mu^{-1} (x_2^c +h)$. From \eqref{E:S-err-1} and the $C^{l_0-2}$ smoothness of $S_{err}$ due to Lemma \ref{L:Ray-H0-FM}, 
we have 
\begin{align*}
& \tau_0^{j} \p_k^{j_1} \p_c^{j_2} \begin{pmatrix} \mu^{-1} y_{0-} (x_2)  \\ y_{0-}' (x_2)  \end{pmatrix} \\
= &\tau_0^{j}\big(\p_c - \tfrac { \p_\tau + \p_{\tau_0}}{\mu U'(x_2^c)}\big)^{j_2} \left[ \big( \log |\tfrac \tau{\tau_0}| + \tfrac {i \pi}2  (sgn(\tau)- sgn(\tau_0)) \big) \p_k^{j_1}  \Big(\tfrac {\mu U''(x_2^c)}{U'(x_2^c)} \begin{pmatrix} \tilde S_{12} (\mu, c, \tau, \tau_0)  \\ \tilde S_{22} (\mu, c, \tau, \tau_0) \end{pmatrix} \Big)\right] + O(1) 
\end{align*}
where we also used that $\log |\tfrac \tau{\tau_0}| + \tfrac {i \pi}2  (sgn(\tau)- sgn(\tau_0))$ is independent of $k$. The desired inequality \eqref{E:pcy5} follows from straight forward calculations using the $C^{l_0-3}$ smoothness of $\tfrac {\tilde S_{12}}{\tau_0}$, $\tfrac {\tilde S_{22}}{\tau_0}$, and $\tfrac {U(-h) - c}{\tau_0}$. 
\end{proof}

\begin{remark} \label{R:y_+}
Most of the above regularity results and estimates also hold for $y_{0+}(k, c, x_2)$. Since $y_+$ plays a less substantial role as $y_-$ in the rest of the paper, we only gave the basic estimates on $y_{+}$. 
\end{remark} 

In the above $\p_c y_\pm$ was considered only for $c\in U([-\frac {h_0}2 -h, \frac {h_0}2])$. To end this subsection, we extend some estimate for $c \in \C$ using the analyticity of $y_\pm$ in $c$ in the following lemma. 

\begin{lemma} \label{L:pcy-complex}
Assume $U\in C^5$. The following hold. 
\begin{enumerate} 
\item For any $c \in \C$ with $c_I>0$,
it holds 
\[
\p_c y_- (k, c, x_2)  = \frac 1{2\pi i} \int_{\R} \frac {\p_c y_{0-} (k, c', x_2)}{c' -c} dc', 
\]
\[
(U(x_2)-c) \p_c y_-' (k, c, x_2)  = \frac 1{2\pi i} \int_{\R} \frac {(U(x_2)-c') \p_c y_{0-}' (k, c', x_2)}{c' -c} dc'. 
\]
\[
\frac {\p_c y_+(k, c, x_2)}{(U(x_2)-c-i)^2} = \frac 1{2\pi i} \int_{\R} \frac {\p_c y_{0+} (k, c', x_2)}{(U(x_2)-c'-i)^2(c' -c)} dc',
\]
\[
\frac {(U(x_2)-c) \p_c y_+'(k, c, x_2)}{(U(x_2)-c-i)^3} = \frac 1{2\pi i} \int_{\R} \frac {(U(x_2)-c')\p_c y_{0+} (k, c', x_2)}{(U(x_2)-c'-i)^3(c' -c)} dc'.
\]
\item  For any $r\in (1, \infty)$, 
\begin{enumerate} 
\item there exists $C >0$ depending only on $r$, $|U'|_{C^{4}}$, and $|(U')^{-1}|_{C^0}$, such that for any $k \in \R$, $x_2\in [-h, 0]$, $c_I > 0$,
\[
\mu^{-1} |\p_c y_- (k, c, x_2)|_{L_{c_R}^r(\R)} + |(U(x_2)- c) \p_c y_{-}' (k, c, x_2)|_{L_{c_R}^r(\R)} \le C \cosh \mu^{-1} (x_2+h); 
\]
\item 
as $c_I \to 0+$, $\p_c y_-$ and $(U-c) \p_c y_{-}'$ converge to $\p_c y_{0-}$ and $(U-c_R) \p_c y_{0-}'$ in $L_{c_R}^r(\R)$, respectively, for any $x_2\in [-h, 0]$. Moreover, the convergence also holds in $L_{c_R, x_2}^r (\R \times [-h, 0])$. 
\end{enumerate}
\item For any $r\in (1, \infty)$ and compact interval $\CI \subset \R$, 
\begin{enumerate} 
\item there exists $C >0$ depending only on $r$, $\CI$, $|U'|_{C^{4}}$, and $|(U')^{-1}|_{C^0}$, such that for any $k \in \R$, $x_2\in [-h, 0]$, $c_I > 0$, 
\[
\mu^{-1} |\p_c y_+ (k, c, x_2)|_{L_{c_R}^r(\CI)} + |(U(x_2)- c) \p_c y_{+}' (k, c, x_2)|_{L_{c_R}^r(\CI)} \le C \cosh \mu^{-1} (x_2+h); 
\]
\item 
as $c_I \to 0+$, $\p_c y_+$ and $(U-c) \p_c y_{+}'$ converge to $\p_c y_{0+}$ and $(U-c_R) \p_c y_{0+}'$ in $L_{c_R}^r(\CI)$, respectively, for any $x_2\in [-h, 0]$. Moreover, the convergence also holds in $L_{c_R, x_2}^r (\CI \times [-h, 0])$. 
\end{enumerate}\end{enumerate}
\end{lemma}

The multiplier $U(x_2)-c$ in front of $\p_c y_\pm'$ is added to regularize their singularities near $U(x_2)=c$ and the denominators $(U(x_2)-c-i)^n$, $n=2,3$, in the expressions related to $\p_c y_+$ are to make it decay as $|c| \to \infty$ (recall the initial conditions \eqref{E:y-pm} of $y_+$ involving $c$).  

\begin{proof}
Let $B_{h_0} \subset \C$ be the open disk with diameter segment $U([-\frac {h_0}2 -h, \frac {h_0}2])$. For any $c \notin B_{h_0}$, let 
\[
\rho = k^{-2} (1+ |U''|_{C^0}) \max_{[-h, 0]} |U-c|^{-1} \le Ck^{-2} (1+|c|)^{-1}.  
\]
There exists $k_*>0$ such that $\rho<1$ for any $|k| \ge k_*$. Lemma \ref{L:Ray-regular-1} (with $x_{2l} =-h$, $\CI=[-h, 0]$, $C_0=0$, and $\Theta_1=\Theta_2 =\sinh$) implies, for $|k|\ge k_*$ and $c \notin B_{h_0}$, 
\be \label{E:pcy-c-temp-0.5} \begin{split} 
\big|y_- (k, c, x_2) - |k|^{-1} \sinh |k| (x_2+h)\big| + \mu \big|y_-' (k, c, & x_2) - \cosh |k| (x_2+h)\big| \\
\le & C\mu k^{-1} (1+ |c|)^{-1} \sinh |k| (x_2+h). 
\end{split} \ee
For $|k|< k_*$, Lemma \ref{L:regular-small-k} implies that the above inequality still holds for  $c \notin B_{h_0}$. 

From equation \eqref{E:pcy1-temp-1} ($j_1=0$ and $j_2=1$) of $\p_c y_-$, applying \eqref{E:y-nh-1}  with $\phi = -\frac {U''}{(u-c)^2} y_-$ and using Lemma \ref{L:y-pm}, we have for $|k|\ge k_*$  and $c \notin B_{h_0}$, 
\begin{align}
\mu^{-1} |\p_c y_- (k, c, x_2)| + |\p_c y_-' (k, c, x_2)|  \le & C\mu (1+ |c|)^{-2} \sinh \mu^{-1} (x_2+h).    \label{E:pcy-c-temp-1}
\end{align}
For $|k|\le k_*$, Lemma \ref{L:regular-small-k} implies that the above inequality still holds for $c \notin B_{h_0}$.   

For any $c \in \C$ with $c_I>0$, the analyticity of $\p_c y_-$ and its $O(|c|^{-2})$ decay as $|c|\to \infty$ imply, for any $\beta \in (0, c_I)$, 
\[
\p_c y_- (k, c, x_2) = \frac 1{2\pi i} \int_{\R + i\beta} \frac {\p_c y_- (k, c', x_2)}{c' -c} dc' = \frac 1{2\pi i} \int_{\R + i\beta} \frac {y_- (k, c', x_2)}{(c' -c)^2} dc',
\]
where the boundary terms at infinity in the above integration by parts vanish due to the uniform-in-$c$ bound on $|y_-|$ given in \eqref{E:pcy-c-temp-0.5}. Letting $\beta\to 0+$, the same bound 
and Lemma \ref{L:y0} yield 
\[
\p_c y_- (k, c, x_2) = \frac 1{2\pi i} \int_{\R} \frac {y_{0-} (k, c', x_2)}{(c' -c)^2} dc' = \frac 1{2\pi i} \int_{\R} \frac {\p_c y_{0-} (k, c', x_2)}{c' -c} dc'= - \frac 1{2\pi i} \int_{\R} \frac {\p_c y_{0-} (k, c', x_2)}{(c_R-c')+ ic_I} dc',
\]
where we integrated by parts again. The desired estimate on $|\p_c y_-|_{L_{c_R}^r}$ follows from the boundedness of the convolution kernel $\frac 1{c'+ic_I}$ on $L^r (\R)$, \eqref{E:pcy-c-temp-1} for $|c|\gg1$, and Lemmas \ref{L:pcy1}--\ref{L:pcy3}. 

The results for $\p_c y_-'$ are derived in the same manner. In fact 
\begin{align*}
& (U(x_2)-c) \p_c y_-' (k, c, x_2) = \frac 1{2\pi i} \int_{\R + i\beta} \frac {(U(x_2)-c') \p_c y_-' (k, c', x_2)}{c' -c} dc' \\
=& \frac 1{2\pi i} \int_{\R + i\beta} \frac {(U(x_2)-c) \p_c y_-' (k, c', x_2)}{c' -c} - \p_c y_-' (k, c', x_2) dc' \\
=& \frac 1{2\pi i} \int_{\R + i\beta} \frac {(U(x_2)-c)y_-' (k, c', x_2)}{(c' -c)^2} dc'= \frac 1{2\pi i} \int_{\R} \frac {(U(x_2)-c)y_{0-}' (k, c', x_2)}{(c' -c)^2} dc'\\
=&  \frac 1{2\pi i} \int_{\R } \frac {(U(x_2)-c') \p_c y_{0-}' (k, c', x_2)}{c' -c} dc',
\end{align*}
where we used \eqref{E:pcy-c-temp-0.5} to cancel the two boundary terms at infinity  in the above both  integrations by parts and also used the integrability of $(U(x_2)-c) \p_c y_-' (k, c, x_2)$ near $U(x_2)=c$ given in Lemma \ref{L:pcy2}. The latter also yields the estimate on $(U-c)\p_c y_-'$.   

In statement (2b), the pointwise-in-$x_2$ convergence in $L_{c_R}^r$ is standard due to the convergence of the convolution kernel $\frac 1{c'+ic_I}$ on $L^r (\R)$ as $c_I\to 0+$, as well as the analyticity of $y_-$ for $c_I>0$. The convergence in $L_{c_R, x_2}^r$ follows from the pointwise-in-$x_2$ convergence in $L_{c_R}^r$, the $L_{x_2}^\infty L_{c_R}^r$ bounds in statement (2a), and the dominant convergence theorem. 

Finally, $\p_c y_+$ can be analyzed similar. However, the initial values \eqref{E:y-pm} induce an $O(|c|^2)$ growth in $y_+$ and $y_+'$ and an $O(|c|)$ growth of $\p_c y_+$ and $\p_c y_+'$ for $|c|\gg 1$ (Lemma \ref{L:regular-small-k}). Instead we consider, for $c_I>0$, 
\[
\frac {\p_c y_+(k, c, x_2)}{(U(x_2)-c-i)^2} = \frac 1{2\pi i} \int_{\R + i\beta} \frac {\p_c y_+ (k, c', x_2)}{(U(x_2)-c'-i)^2(c' -c)} dc',
\]
which holds for any $\beta \in (0, c_I)$. From this Cauchy integral formula we proceed much as in the above and obtain the integral representation in term of $\p_c y_{0+}$. The derivation of the corresponding formula of $\p_c y_+'$ is also similar. The desired convergence and estimates of $\p_c y_+$ and $\p_c y_+'$ in $L_{c_R}^r(\CI)$ on a compact interval $\CI$ again follow from the properties of the convolution by the kernel $\frac 1{c'+ic_I}$. 
\end{proof}

\subsection{An important quantity $Y= y_{-}'(0)/y_{-} (0)$} \label{SS:Y}

To end this section, we analyze a quantity related to the Reynolds stress, which is crucial for the linearized water wave problem: 
\be \label{E:Y} \begin{split}
&Y(k, c) = Y_R (k, c) + iY_I (k, c) := \frac {y_{-}' (k, c, 0)}{y_{-} (k, c, 0)},  \quad  c = c_R + i c_I \in \C \setminus U([-h, 0]), \\
&Y(k, c) = \lim_{\ep \to 0+} Y(k, c+i\ep)= \frac {y_{0-}' (k, c, 0)}{y_{0-} (k, c, 0)}, \quad c\in U\big([-h, 0)\big),
\end{split} \ee
where $y_{-} (k, c,x_2)$ is the solution to homogeneous Rayleigh equation \eqref{E:Ray-H1-1} satisfying $y_{-}(-h)=0$ and $y_{-}'(-h)=1$ defined in Subsection \ref{SS:Ray-H-Fund} and $y_{0-}(k, c, x_2) = \lim_{\ep \to 0+} y_- (k, c+i \ep, x_2)$ for $c \in \R$. 
Due to Remark \ref{R:y-pm}, $y_-(k, c_R+ic_I, x_2)$ satisfies estimates uniform in $0< \ep \ll1$. With slight abuse of notations, we would not distinguish $y_{0-}$ from $y_{-}$ in the rest of this section. Apparently the domain of $Y(k, c)$ is given by 
\[
D(Y) = \{ (k, c) \in \R \times \C \mid c\ne U(0), \ y_- (k, c, 0)\ne 0\},
\] 
and those excluded points (except $c= U(0)$) exactly are the eigenvalues of of the linearized Euler equation in the fixed channel $x_2 \in [-h, 0]$ at the shear flow $(U(x_2), 0)$. $Y$ is not defined at $c = U(0)$ since $y_-'(x_2)$ has singularity at $x_2 =0$. We first summarize some basic or standard properties of $y_-(k, c, 0)$ in the following lemma.  


\begin{lemma} \label{L:y-lower-b}
Assume $U \in C^3$. The following hold.
\begin{enumerate} 
\item For any $k \in \R$, $y_- (k, c, x_2)>0$ for any $x_2\in (-h, 0]$ and $c\in \R \setminus U\big((-h, 0)\big)$.
\item  There exists $C>0$ depending on $U$ such that, for any $k, c  \in \R$, it holds 
\be \label{E:y-lower-b-1} 
(Ck)^{-1} \sinh k(x_2+h) \le y_-(k, c, x_2) \le C k^{-1} \sinh k(x_2+h) \; \text{ if } \; (U(x_2)-c)(U(-h) -c) \ge 0.
\ee
\item There exists $C>0$ depending only on $U$ such that, for any $c = U(x_2^c)$, $x_2^c \in [-h, 0)$, it holds,  for any $k \in \R$, 
\[
C^{-1}\mu^2 |U''(x_2^c)| \sinh \mu^{-1}(x_2^c+h) \sinh \mu^{-1} |x_2^c| \le | \IP\, y_- (k, c, 0)| \le C\mu^2 |U''(x_2^c)| \sinh \mu^{-1}(x_2^c+h) \sinh \mu^{-1} |x_2^c|.
\]
\item There exists $k_*>0$ and $C>0$ depending only on $M>0$, $|U'|_{C^2}$ and $|(U')^{-1}|_{C^0}$ such that, if $|k| \ge k_*$ or $|c - (U(-h)+U(0))/2|\ge M + (U(0)-U(-h))/2$ then 
\be \label{E:y-lower-b} 
|y_-(k, c, 0)| \ge (Ck)^{-1} \sinh kh.  
\ee
\item Suppose a closed subset $S \subset \C$ satisfies $y_-(k,c, 0) \ne 0$ for all $c\in S$ and $k \in \BK$ where $\BK =\R$ or $\frac {2\pi}L \mathbb{Z}$, then there exists $C>0$ depending only on $S$ and $U$ such that \eqref{E:y-lower-b} holds for all $k \in \BK$ and $c \in S$. 
\end{enumerate}
\end{lemma}

\begin{remark} \label{R:y-lower-b}
According to Lemmas \ref{L:regular-small-k} and \ref{L:y-pm}, the assumption $y_-(k, c, 0)\ne 0$ on $S$ in Statement (5) is automatically satisfied except possibly a compact set of $(k, c) \subset \BK \times S$. In particular, due to statement (3), it is satisfied for $S=\C$ if $U'' \ne 0$ on $[-h, 0]$.  We also recall $y_-(k, c, 0)=0$ is equivalent to that $-ikc$ is an eigenvalue of the linearized Euler equation at the shear flow $U$ on the fixed channel $x_2 \in (-h, 0)$ associated with an eigenfunction $v_2 (x_2) = e^{ikx_1}y_-(k, c, x_2)$. 
\end{remark}

\begin{proof}
We first claim the following standard result.

{\it Claim.} Let $y(x_2)$ is a solution to the homogeneous Rayleigh equation \eqref{E:Ray-H1-1} on an interval $\CI=(x_{2l}, x_{2r}) \subset [-h, 0]$ with $c \in \R \setminus U(\CI)$ such that $(y(x_{20}), y'(x_{20})) \in \{0\} \times (\R\setminus \{0\})$ at some $x_{20} \in \overline{\CI}$, then $y(x_2) \in \R \setminus \{0\}$ at any $x_2 \in \overline \CI \setminus \{x_{20}\}$.  

If $U(x_{20}) \ne c$, then the claim $y(x_2) \in \R$ and $y$ is in $C^\alpha$ on $\overline \CI$ are obvious since the coefficients of \eqref{E:Ray-H1-1} are real. If $U(x_{20})=c\notin U(\CI)$, then it must hold $x_{20} \in \{x_{2l}, x_{2r}\}$ and Lemma \ref{L:B} implies that $y\in C^1 (\overline \CI)$ and $W= (\mu^{-1} y, y') (\cdot + x_{20})$ satisfies \eqref{E:Ray-H0-2} with $W_1(0)=0$ and $\tilde \Phi_0\equiv0$. This formula yields $y\in \R$. Finally, suppose $y(x_{21})=0$ at some $x_{21} \in \overline \CI \setminus \{x_{20}\}$. Let $y = (U - c) \xi$. Again $\xi \in C^1(\CI \cup\{x_{20}, x_{21}\})$ due to Lemma \ref{L:B} and it is standard to verify 
\be \label{E:Y-temp-1}
- \big( (U-c)^2 \xi' \big)' + k^2 (U-c)^2 \xi =0, \quad x_2 \in \CI.
\ee
Multiplying it by $\xi$ and integrating it between $x_{20}$ and $x_{21}$ leads to a contradiction. Hence the claim is proved. 

For $c \in \R$, applying the above claim to $y_-$ on the interval $[-h, 0]$ if $c \notin U\big((-h, 0)\big)$ and on $[-h, x_2^c]$ if $c \in U\big((-h, 0)\big)$, respectively, implies that $y_-(x_2)\in \R$ does not change signs on these intervals. Hence we obtain statement (1) and $y_-(x_2) >0$ for $x_2 \in (-h, x_2^c]$ if $c \in U\big((-h, 0)\big)$. Along with \eqref{E:y-_1}, the continuity of $\frac {y_-(k, c, x_2)}{\mu \sinh \mu^{-1} (x_2+h)}$, and Lemma \ref{L:regular-small-k}, it also yields Statement (2). 

In the view of Lemma \ref{L:y0}, Remark \ref{R:jump-at-sing}, and statement (1), $y(x_2)=\IP\, y_-(x_2)$ is also a solution on $[x_2^c, 0]$ satisfying $y(x_2^c)=0$ and $y'(x_2^c)= \frac {\pi U''(x_2^c)}{U'(x_2^c)} y_-(x_2^c)$. Statement (3) follows from statement (2) applied to $y_-$ on $[-h, x_2^c]$ and to $\IP \, y_-$ on $[x_2^c, 0]$. 

From \eqref{E:y-_1} and Remark \ref{R:y-pm}, there exists $k_*>0$ such that \eqref{E:y-lower-b} holds for all $|k| \ge k_*$ and $c \in \C$. For $|k|\le k_*$, the restriction on $c$ involving $M>0$ ensures $y_-(k, c, 0)\ne 0$ due to the semicircle theorem (of the channel flow) and thus  
we obtain \eqref{E:y-lower-b}  from Lemma \ref{L:regular-small-k}, which completes the proof of statement (4). 

Finally assume $y_-(k, c, 0)\ne 0$ for all $k\in \BK$ and $c \in S$. Recalling  the convergence estimates \eqref{E:w-d-1} and the locally H\"older continuity of $y_-$ in $c\in \R$ (Lemma \ref{L:pcy0}), we obtain the continuity of $y_-$ in $c\in \C$ for $c_I\ge 0$. Lemmas  \ref{L:regular-small-k} and \ref{L:y-pm} along with the continuity of $y_-(k, c, 0)$ and the non-vanishing assumption imply that \eqref{E:y-lower-b} holds for all $k \in \BK$ and $c \in S$ with $c_I \ge 0$. As $y_-(k, \bar c, x_2) = \overline {y_-(k, c, x_2)}$, statement (5) follows  and it completes the proof of the lemma. 
\end{proof}


In the following we give some basic properties of $Y(k, c)$. 

\begin{lemma} \label{L:Y-def}
Assume $U\in C^{l_0}$, $l_0\ge 3$. It holds that $Y(k, \bar c) = \overline {Y(k, c)}$ and  $Y$ is a.) analytic in both $(k, c) \in D(Y)\setminus (\R \times U([-h, 0]))$, and, when restricted to $c_I\ge 0$, b.) $C^{l_0-2}$ in $(k, c) \in D(Y) \setminus (\R \times \{U(-h)\})$, and c.) $C^{l_0-2}$ in $k$ and locally $C^\alpha$ in $(k, c) \in D(Y)$  for any $\alpha \in [0, 1)$.  Moreover, 
\begin{enumerate}
\item $Y(k, U(-h)) \in \R$ and $Y(0, U(-h)) =\frac {U'(0)}{U(0)-U(-h)}$.
\item There exists $C, \rho>0$ depending only on $U$ such that 
\[
|Y(k, c)| \le C \big(\mu^{-1} + \big|\log  \min \big\{1, |U(0) -c|\big\}\big| \big), \; \;  \forall k\in \R, \; |c - U(0)| \le \rho. 
\]
\item For any $\alpha \in (0, \frac 12)$, there exist $k_0>0$ and $C>0$ depending only on  $\alpha$, $|U'|_{C^2}$, and $|(U')^{-1}|_{C^0}$ such that,  
\[
|Y(k, c) - k \coth kh| \le C (\mu^{\alpha-1} + |\log \min\{1, \, |U(0)-c|\}|), \quad \forall |k| \ge k_0, \; c \ne U(0).  
\]
\item For any $M>0$ and $k_*>0$, there exists $C>0$ depending only on $k_*$ and $M$ such that 
\[
|Y(k, c) - k \coth kh| \le \frac C{ dist(c, U([-h, 0]))}, \quad \forall |k| \le k_*, \; \Big|c - \frac {U(-h) + U(0)}2\Big|\ge M + \frac {U(0)-U(-h)}2. 
\]
\end{enumerate} 
\end{lemma}

\begin{proof} 
The analyticity and the conjugacy property of $Y$ are obvious from its definition. The property $Y(k, U(-h))\in \R$ is a direct corollary of Lemma \ref{L:y-lower-b}(1).
The $C^{l_0-2}$ smoothness of $Y$ away from $U(-h)$ and $U(0)$ 
follows from Lemma \ref{L:pcy0} and the analyticity of $Y$ in $c$ with $c_I>0$. 
The  H\"older continuity of $Y$ 
is again a corollary of Lemma \ref{L:pcy0} for $c$ varying along $\R$ and Proposition \ref{P:converg} for $c$ varying along $i\R$. 
The explicit form of $Y\big(0, U(-h)\big)$ is a direct consequence of the observation 
\be\label{E:y-0}
y_-\big(0, U(-h), x_2\big) = \big(U(x_2)-U(-h)\big)/U'(-h). 
\ee

To end the proof of the lemma, we obtain the quantitive estimate on $Y(k, c)$. From Lemma \ref{L:y-lower-b}, $y_-(k, U(0), 0)\ne 0$ for any $k\in \R$. Along with Lemma \ref{L:y-pm}, it implies that \eqref{E:y-lower-b} holds for $|c-U(0)| \le \rho$ for some $\rho>0$ depending only on $U$. Statement (2) follows from 
the upper bound of $|y_{-}'(k, c, 0)|$ given in Lemma \ref{L:y-pm}. 
Statement (3) is also a direct consequence of Lemma \ref{L:y-pm} where $k_0$ is involved to ensure $y_-(k, c, 0)\ne 0$. In statement (4), the restriction on $c$ guarantees $y_-(k, c, 0) \ne 0$ due to the semicircle theorem and the desired inequality follows Lemma \ref{L:regular-small-k}. 
\end{proof}


The analyticity of $Y$ in $c$ allows us to use the Cauchy integral to analyze $Y(k, c)$. For $r>0$, let 
\be \label{E:CD_r}
\CD_r = B\big( U([-h, 0]), r\big) \subset \C
\ee 
be the $r$-neighborhood of $U([-h, 0]) \subset \C$. 

\begin{lemma} \label{L:Y-Cauchy-0} 
Assume $U \in C^3$, $k\in \R$, and $r>0$ satisfy $y_-(k, \cdot, 0)\ne 0$ on $\C \setminus \CD_r$, then for any $c \in \C \setminus \overline{\CD_r}$ 
and $n \ge 1$ we have  
\be \label{E:Y-Cauchy-0} 
Y(k, c) = k \coth kh - \frac 1{2\pi i}  \oint_{\p \CD_r} \frac {Y(k, c')}{c'-c} dc', \quad \p_c^n Y(k, c) = - \frac {n!}{2\pi i}  \oint_{\p \CD_r} \frac {Y(k, c')}{(c'-c)^{n+1}} dc', 
\ee
where $\oint$ denote the integral along the contours counterclockwisely. 
\end{lemma}

Here $\p_c Y = \frac 12 (\p_{c_R} - i \p_{c_I})Y $ denotes the derivative of $Y$ as a function of the complex variable $c$ and thus $\p_c Y= \p_{c_R} Y= -i \p_{c_I} Y$ due to its analyticity. 

\begin{proof}
The assumption implies $Y(k, \cdot)$ is analytic in $\C \setminus \overline{\CD_r}$ and continuous in $\C \setminus \CD_r$.  
For any $r' \gg1$, the analyticity of $Y$ 
yields   
\be \label{E:Y-Cauchy-0.5} 
Y(k, c) = \frac 1{2\pi i} \left( \oint_{\p \CD_{r'}} - \oint_{\p \CD_{r}} \right) \frac {Y(k, c')}{c'-c} dc'.  
\ee
Applying Lemma \ref{L:Y-def}(4) with $M=1$ and $k_*=1+|k|$, we have 
\[
|Y(k, c') - k \coth kh| \le  C/dist\big(c', U([-h, 0])\big), \quad \forall c' \in \p \CD_{r'}. 
\]
Therefore 
\[
\lim_{r' \to +\infty} \left( \frac 1{2\pi i} \oint_{\p \CD_{r'}} \frac {Y(k, c')}{c'-c} dc' - k\coth kh \right) = 0 
\]
and thus the desired integral formula of $Y(k, c)$ follows. 
The representation of $\p_c^n Y$ 
simply follows from direct differentiation. 
\end{proof}

\begin{remark} 
Though not needed in the rest of the paper, this lemma could be modified for general $k$ and $c \notin U([-h,0])$. In this case, $0< r<dist\big(c, U([-h,0])\big) $ should be chosen so that $y_-(k, \cdot, 0)\ne 0$  along $\p \CD_r$. The integral representation formula would involve the residue at those roots of $y_-(k , \cdot, 0)$ outside $\CD_r$. The estimates should also be modified accordingly. 
\end{remark}

To analyze the remaining integral in \eqref{E:Y-Cauchy-0}, 
we start with the imaginary part $Y_I$ of $Y$.  


\begin{lemma} \label{L:Y-I}
$Y_I (k, c)=0$ for $c\in \R\backslash U\big((-h, 0]\big)$. 
Assume $U \in C^3$, $c = U(x_2^c) \in U\big((-h, 0)\big)$, and $y_-(k, c, 0) \ne 0$, then  
\[
Y_I (k, c) =\frac {\pi U''(x_2^c) y_{-} (k, c, x_2^c)^2}{U'(x_2^c) |y_{-} (k, c, 0)|^2}.
\]  
\end{lemma}

\begin{proof}
The vanishing of $Y_I(k, c)$ for $c\in \R\backslash U((-h, 0])$ is obvious from its definition and Lemma \ref{L:y-lower-b}(1). To derive the expression of $Y_I(k, c)$ for $c = U(x_2^c)$ with $x_2^c \in (-h, 0)$ and 
$y_-(k, c, 0) \ne 0$, 
we may consider $y(\ep, x_2) = \frac {y_{-} (k, c+i\ep, x_2)}{y_{-}(k, c+i\ep, 0)}$, $\ep>0$, which is also a solution to the homogeneous Rayleigh equation with $y(\ep, -h) =0$ and $y(\ep, 0)=1$. It is straight forward to calculate 
\[
\IP \, y'(\ep, 0) = \frac 1{2i} \int_{-h}^0 \p_{x_2} (y' \bar y - y \bar y') dx_2 =\int_{-h}^0 \frac {\ep U'' |y|^2}{|U-c|^2 +\ep^2} dx_2.  
\]   
Applying the convergence estimates \eqref{E:w-d-1}  
and the H\"older continuity of $y_{0-} (k, c, x_2)\in \R$ in $x_2$, we obtain the desired 
\[
Y_I (k, c) = \lim_{\ep \to 0+} \IP \, y'(\ep, 0) = \frac {\pi U''(x_2^c) y_{-} (k, c, x_2^c)^2}{U'(x_2^c) |y_{-} (k, c, 0)|^2}.   
\]
This completes the proof of the lemma. 
\end{proof} 

The above formula yields some refined estimates of $Y_I$ for $c \in U([-h, 0])$. 

\begin{lemma} \label{L:Y-I-1}
Assume $U\in C^{l_0}$, $l_0\ge 3$,
then the following hold for $Y_I (k, c)$.
\begin{enumerate}  
\item $Y_I(k, c)$ is $C^{l_0-2}$ in $(k, c) \in D(Y) \cap \big(\R \times U\big((-h, 0)\big)\big)$ and it satisfies  
\[
\lim_{c \to U(0)-} Y_I (k, c) = \frac {\pi U''(0)}{U'(0)}, \quad \lim_{c\to U(-h)+} \frac{ Y_I(k, c)}{(c-U(-h))^2} = \frac {\pi U''(-h) }{U'(-h)^3 |y_-(k, U(-h), 0)|^2}.
\]
Moreover, if $l_0\ge 4$, then,  for any $q\in [1, \infty)$, $j_1, j_2\ge 0$, $j_2 \le 2$, and $j_1+j_2\le l_0-4$, $\p_k^{j_1} \p_{c_R}^{j_2} Y_I$ is $L_k^\infty W_{c_R}^{1, q}$ locally in $(k, c) \in D(Y) \cap \big(\R \times U\big([-h, 0)\big)\big)$.
\item Assume $U\in C^5$, then there exists $C>0$ depending only on $|U'|_{C^{4}}$ and $|(U')^{-1}|_{C^0}$ such that, for any 
$c \in U\big([-h, 0)\big) \cap D\big(Y(k, \cdot)\big)$, we have   
\[
\frac{\mu^2 \sinh^2 (\mu^{-1} (x_2^c +h))}{C|y_{-} (k, c, 0)|^2} \le Y_I \big(k, c\big) \le \frac{C\mu^2  \sinh^2 (\mu^{-1} (x_2^c +h))}{|y_{-} (k, c, 0)|^2},  
\]
\begin{align*}
|\p_{c_R} Y_I(k, c)| \le & C \frac { \mu \sinh (2\mu^{-1}(x_2^c+h)) }{|y_{-} (k, c, 0)|^2}\\
&+ C \frac{\mu^3 \sinh (\mu^{-1}h) \sinh (2\mu^{-1} (x_2^c+h))}{|y_{-} (k, c, 0)|^3} 
\Big( 1+ \big|\log \min\{1, |\mu^{-1}(U(0)-c)|\}\big|\Big), 
\end{align*}
where $\mu = \langle k\rangle^{-1}= (1+k^2)^{-\frac 12}$. 
\end{enumerate}
\end{lemma}

\begin{proof}
Lemma \ref{L:pcy0} implies the $C^{l_0-2}$ smoothness of $y_-(k, c, 0)$ in $k$ and $c \in U\big( (-h, 0)\big)$ and that of $y_-(k, c, x_2^c)$ in $c \in U([-h, 0])$, hence $Y_I$ is $C^{l_0-2}$   in $(k, c) \in D(Y) \cap \R \times U\big((-h, 0)\big)$. Moreover,  $y_-(k, c, 0) >0$ for $c \in \{(U(-h), U(0)\}$ due to Lemma \ref{L:y-lower-b}(1), which along with Lemma \ref{L:pcy0}a.) implies that $y_-(k, c, 0) \ne 0$  for $c \in U([-h, 0])$ near $U(-h)$ or $U(0)$ and thus $Y_I(k, c)$ is H\"older continuous  for $c \in U([-h, 0])$ near $U(-h)$ and  $U(0)$. 
The local regularity of $\p_k^{j_1} \p_{c_R}^{j_2} Y_I$ follows from Lemma \ref{L:pcy0}(f). 

The upper bound estimate of $Y_I$ and its limits as $c$ approaches $U(0)-$ and $U(-h)+$ are direct corollaries of Lemmas \ref{L:y-pm} and \ref{L:y-lower-b}(5) and Remark \ref{R:y-pm}, as well as  \eqref{E:pcy4}, \eqref{E:y-lower-b-1} and \eqref{E:y-lower-b}. 
In  
\begin{align*}
\p_{c_R} Y_I (k, c) = & \p_{c} \left( \frac {\pi U''(x_2^c)}{U'(x_2^c)}\right) \frac{y_{-} (k, c, x_2^c)^2} {|y_{-} (k, c, 0)|^2} + \frac {2\pi U''(x_2^c) y_{-} (k, c, x_2^c) \p_{c} \big(y_{-} (k, c, x_2^c)\big)}{U'(x_2^c) |y_{-} (k, c, 0)|^2} \\
& - \frac {2\pi U''(x_2^c) y_{-} (k, c, x_2^c)^2 \big(y_-(k, c, 0) \cdot \p_c y_-(k, c, 0)\big)}{U'(x_2^c) |y_{-} (k, c, 0)|^4},  
\end{align*}
$ \p_{c} \big(y_{-} (k, c, x_2^c)\big)$ is estimated by \eqref{E:pcy4}. The other key term $\p_c y_-(k, c, 0)$ will be considered
in three possible cases of $c\in U([-h, 0])$ according to the division of $[-h, 0]= \CI_1 \cup \CI_2 \cup \CI_3$ defined in \eqref{E:kc-away-4} in Subsection \ref{SS:pcy0}. Observing $c\in U([-h, 0])$ implies $\CI_2\ne \emptyset$ and $x_2=0\in \CI_2 \cup \CI_3$. 

* {\it Case 1}: $x_2=0\in \CI_3$ and $x_{2l} =-h$. The former happens  if and only if $c\le U(0)- \rho_0^{-1} \mu$, while $x_{2l}=-h$ if and only if $c \le U(-h)+ \rho_0^{-1} \mu$. 
Lemma \ref{L:pcy3} implies 
\[
|\p_c y_{-} (k, c, 0)| \le C \mu  \big( 1+ \big|\log \big(\mu^{-1}(c-U(-h))\big)\big|\big)\cosh \mu^{-1}h.  
\]

* {\it Case 2}: $x_2=0\in \CI_3$ and $x_{2l} > -h$ which occurs if and only if $U(-h)+ \rho_0^{-1} \mu \le c \le U(0)- \rho_0^{-1}\mu$. Also from Lemma \ref{L:pcy3}, we have
\[
|\p_c y_{-} (k, c, 0)| \le C \mu \cosh \mu^{-1}h.
\]

* {\it Case 3}: $x_2=0 \in \CI_2$ which happens iff $U(0)-c\le \rho_0^{-1}\mu$ and thus $x_{2r}=0$. From the definitions \eqref{E:kc-away-4} of $\CI_2$, \eqref{E:kc-away-2} of $\rho_0$, and \eqref{E:h0} of $h_0$, it holds   
\[
0\le U(x_{2r}) - U(x_{2l}) \le 2 \rho_0^{-1} \mu \le \tfrac 12 h_0 \inf U' \Longrightarrow -x_{2l} =x_{2r} - x_{2l} \le \tfrac 12 h_0 \Longrightarrow x_{2l} >-h.  
\]
This in turn implies $c- U(x_{2l}) = \rho_0^{-1} \mu$ and thus Lemma \ref{L:pcy2} yields 
\[
|\p_c y_{-} (k, c, 0)| \le C \mu  \big( 1+ \big|\log \big(\mu^{-1}(U(0) -c)\big)\big|\big)\cosh \mu^{-1} h. 
\]

The desired estimates on $\p_{c_R} Y_I$ follow from \eqref{E:pcy4}, Lemmas \ref{L:y-lower-b} and \ref{L:y-pm}, and the above estimates. In particular, in the above case 1, we also used 
\[
\mu \sinh |\mu^{-1} (x_2^c+h)| \big|\log |\mu^{-1} \big(c-U(-h)\big)| \big| \le C\mu \cosh (\mu^{-1}(x_2^c+h)),
\]
which can be shown by considering whether $|\mu^{-1}(c-U(-h))| \le 1$ separately. 
\end{proof}


In the following we analyze $Y(k, c)$ by writing it as a Cauchy integral of $Y_I$. 

\begin{lemma} \label{L:Y-Cauchy}
Assume $U \in C^{l_0}$, $l_0\ge 5$, and $k \in \R$ satisfy that $y_-(k, c, 0)\ne 0$ for all $c \in U([-h, 0])$, then $Y(k, c)$ and $\p_k^{j_1} \p_{c_R}^{j_2} Y(k, c)$ are $L_k^\infty L_{c_R}^q$ locally in $k \in \R$ and $c_R$ in the domain $D(Y) \cap\{c_I\ge 0\}$ 
for any $q\in (1, \infty)$, $0 \le j_2 \le 2$, and $0\le j_1\le l_0 -3-j_2$. Assume, in addition, $y_-(k, c, 0)\ne 0$ for all $c \in \C$, then,  for any $c\notin U([-h, 0])$,  
\be \label{E:Y-Cauchy-1}
Y(k, c)=\frac{1}{\pi}\int_{U(-h)}^{U(0)} \frac{Y_I (k, c')}{c'-c}dc'+k\coth kh,
\ee
and for $c\in U\big([-h, 0)\big)$, 
\be \label{E:Y-Cauchy-2}
Y(k, c)= - \CH \big(Y_I(k, \cdot)\big) (c) + i Y_I(k,c) +k\coth kh. 
\ee 
\end{lemma}

Here $\CH$ denotes the Hilbert transform in $c\in \R$, namely, 
\[
\CH \big(Y_I(k, \cdot)\big) (c) =\frac{1}{\pi} \text{P.V.} \int_\R \frac{Y_I (k, c')}{c-c'}dc' = \frac{1}{\pi} \text{P.V.} \int_{U(-h)}^{U(0)} \frac{Y_I (k, c')}{c-c'}dc', 
\] 
where P.V.$\int$ represent the principle value of the singular integral. We also recall $Y(k, c) = Y(k, c+i0)$ and $Y(k, c-i0) = \overline {Y(k, c+i0)}$ for $c \in \R$.   

\begin{proof}
Let us first assume 
$y_-(k, c, 0)\ne 0$ for all $c \in \C$, then $Y(k, c)$ is well-defined for all $c \ne U(0)$. 
We shall apply Lemma \ref{L:Y-Cauchy-0}. 
The contour $\p \CD_{r}$ is the union of two segments $[U(-h), U(0)] \pm ir$, the left half circle centered at $U(-h)$ with radius $r$, and the right half circle centered at $U(0)$ with radius $r$. As $r \to 0+$, due to the continuity of $Y$ (when restricted to $\{c_I\ge 0\}$) at $c\ne U(0)$ and its logarithmic upper bound near $U(0)$ given in Lemma \ref{L:Y-def}, the Cauchy integrals along the two half circles converge to zero as $r \to 0+$. Hence the integral form \eqref{E:Y-Cauchy-1} of $Y(k, c)$ follows from taking the limit of \eqref{E:Y-Cauchy-0} as $r \to 0+$ and the conjugacy $Y(k, \overline {c'}) = \overline {Y(k, c')}$. 

For $c_I \ne 0$, the integral form \eqref{E:Y-Cauchy-1} can be rewritten as 
\[
Y(k, c) = \frac 1{\pi } \int_{U(-h)}^{U(0)} \frac {(c'-c_R + i c_I)}{(c'-c_R)^2 + c_I^2} Y_I (k, c') dc' +k\coth kh.
\]
A standard treatment of the above singular integral as $c_I \to 0+$, along with the regularity of $Y_I(k, c')$ in $c' \in U\big([-h, 0)\big)$ given in Lemma \ref{L:Y-I} and \ref{L:Y-I-1}, yields \eqref{E:Y-Cauchy-2}. 

The regularity of $Y$ follows from that of $Y_I$ and the boundedness in $L^q$ of the convolution by $\frac 1{c' + i c_I}$ with the parameter $c_I \ge 0$. Here the singularity of $Y_I$ near $c=U(0)$ does not affect the regularity of $Y$ away from $U(0)$ due to the localization property of this convolution operator. 

Finally, if we only assume $y_-(k, c, 0)\ne 0$ for $c \in U([-h, 0])$, due to its analyticity in $\C \setminus U([-h, 0])$, its continuity when restricted to $\{c_I\ge 0\}$, and $\lim_{|c| \to \infty} y_-(k, c, 0)= k^{-1} \sinh kh$ (Lemma \ref{L:regular-small-k}), $Y(k, \cdot)$ has at most  finitely many singular points $\C \setminus U([-h, 0])$. Hence there would be at most  finitely many  additional contour integrals of $Y$ in \eqref{E:Y-Cauchy-1} along contours in $\C\setminus U([-h, 0])$ enclosing the roots of $y_-(k, \cdot, 0)$. Those integrals in the analytic region of $Y$ would not affect the regularity of $Y$. The proof of the lemma is complete.  
\end{proof}

With the representation of $Y$ in terms of Cauchy integrals, we may also calculate its derivatives in more details. 

\begin{corollary} \label{C:Y-Cauchy} 
It holds, for $c \notin U([-h, 0])$,   
\be \label{E:Y-Cauchy-3}
\p_c Y(k, c)=\frac{1}{\pi}\int_{U(-h)}^{U(0)} \frac{Y_I (k, c')}{(c'-c)^2}dc'=\frac{1}{\pi}\int_{U(-h)}^{U(0)} \frac{\p_{c_R} Y_I (k, c')}{c'-c}dc' -\frac {U''(0)}{U'(0) \big(U(0)-c\big)},
\ee
and for $c\in U\big([-h, 0)\big)$, 
\be \label{E:Y-Cauchy-4} \begin{split}
\p_c Y(k, c)= - \CH \big(\p_{c_R} Y_I(k, \cdot)\big) (c) &+ i \p_{c_R} Y_I(k,c)-\frac {U''(0)}{U'(0) \big(U(0)-c\big)}.
\end{split} \ee 
\end{corollary}

Using the regularity of and estimates on $Y_I$ and $\p_{c_R} Y_I$ given in Lemma \ref{L:Y-I-1} , \eqref{E:Y-Cauchy-3} follows from direct differentiation and integration by parts, along with the explicit form of $Y_I \big(k, U(0)-\big)$. Equality \eqref{E:Y-Cauchy-4} is obtained by taking the limit of \eqref{E:Y-Cauchy-3} as $c_I \to 0+$. We omit the details of these straight forward calculations.

\section{Eigenvalues of the linearization of the water wave at shear flows} \label{S:e-values}

In this section, we shall discuss the distribution of eigenvalues of the linearized gravity-capillary water wave system \eqref{E:LEuler} at the shear flow $\big( U(x_2), 0\big)^T$. As \eqref{E:LEuler}  preserves Fourier mode $e^{ikx_1}$ for any $k$, the wave number $k\in \R$ would be treated as a parameter in this section. According to Lemma \ref{L:e-value}, $-ikc \in \C$, $c \in \C \setminus U([-h, 0])$, is an eigenvalue of \eqref{E:LEuler} with parameter $k$ if and only if 
\be \label{E:BF} \begin{split} 
& \BF (k, c) = \BF_R + i \BF_I := (g + \sigma k^2)  y_+  (k, c, -h)= (g + \sigma k^2) (y_+ y_-' - y_+' y_-) (k, c, 0)   \\
= & \big(U(0)-c\big)^2 y_-'(k, c, 0) - \big(U'(0)\big(U(0)-c\big)+ g+\sigma k^2\big) y_-(k, c, 0) =0,  
\end{split} \ee
where the last equal sign in the first row is due to the conservation of the Wronskian of \eqref{E:Ray-H1-1}. Let 
\[
\BF(k, c)= \lim_{\ep \to 0+} \BF(k, c+i\ep) = \overline{\lim_{\ep \to 0+} \BF(k, c-i\ep)}, \quad c \in U\big([-h, 0]\big). 
\]
It is easy to see that, if $\BF(k, c)=0$, then $y_-(k, c, x_2)$ also generates the associated eigenfunction of \eqref{E:LEuler}. In the literatures, those zero point $c$ of $\BF$ with $c_I>0$ are often referred to as unstable modes, while those zero point $c\in \R$ as neutral modes. We recall that Yih proved that the semicircle theorem also holds for free boundary problem \cite{Yih72}, namely, \eqref{E:semi-circle} holds for all unstable modes.

From the analysis in Subsection \ref{SS:pcy0}, it is not clear whether $\BF$ is $C^1$ at $c = U(-h)$ which would be crucial for the bifurcation analysis of eigenvalues. We also consider an almost equivalent quantity 
\be \label{E:dispersion}
F(k, c) = y_-(k, c, 0)^{-1} \BF = F_R + i F_I = Y(k, c)\big(U(0)-c\big)^2-U'(0)\big(U(0)-c\big)-(g+\sigma k^2) =0,  
\ee
where $Y(k, c)$ is defined in \eqref{E:Y}, and 
\[
F(k, c)= \lim_{\ep \to 0+} F(k, c+i\ep) = \overline{\lim_{\ep \to 0+} F(k, c-i\ep)},  \quad \forall c \in U\big([-h, 0]\big). 
\]
Apparently $\BF$ and $F$ satisfy 
\be \label{E:BF-conj}
\BF(-k, c) = \BF(k, c) = \overline {\BF(k, \bar c)}, \; \forall c\notin U\big((-h, 0)\big); 
\ee
\be \label{E:f-conj}
F(-k, c) = F(k, c) = \overline {F(k, \bar c)}, \;  c\in D(Y)\setminus U\big((-h, 0)\big). 
\ee
From Lemma \ref{L:Y-Cauchy} $F$ is $C^{1, \alpha}$ near $c_0=U(-h)$ if $y_- (k, c, 0) \ne 0$ for all $c\in U([-h, 0])$, which is crucial for the bifurcation analysis.

\subsection{Basic properties of eigenvalues}  \label{SS:e-v-basic}

Apparently 
it holds that  
\be \label{E:BF-ana}
\BF \text{ is analytic in } k\in \R \ \& \  c\notin U([-h, 0])  \text{ and } \, F \text{ analytic in } k\in \R \ \& \  c\in  D(Y) \setminus U([-h, 0]),
\ee
\[
\BF(k, c) =0 \Longleftrightarrow \; c \text{ is a non-singular or singular mode of \eqref{E:Ray}}.   
\]
In the following we first give some basic properties of $\BF$ under minimal assumptions. 

\begin{lemma} \label{L:e-v-basic-1}
Assume $U \in C^{l_0}$, $l_0\ge 3$, then for any $k \in \R$, the following hold. 
\begin{enumerate}
\item $\BF$ is well defined for all $k \in \R$ and $c \in \C$. When restricted to $c_I\ge 0$, $\BF$ is $C^{l_0-2}$ in $k$ and $c \notin \{U(-h), U(0)\}$ and 
$\BF$ is also $C^\alpha$ in both $k$ and $c$ with $c_I\ge 0$. 
\item $F(k, c)$ is well-defined for $c$ close to $U(-h)$ and $U(0)$, $C^1$ near $c=U(0)$, and 
\[
F(k, U(-h)) \in \R, \quad F\big(0, U(-h)\big) = -g, \quad F\big(k, U(0)\big) = -g - \sigma k^2, \quad 
\p_c F\big(k, U(0)\big)= U'(0). 
\] 
\item Assume $U \in C^5$, then for any $r\in (1, \infty)$, there exists $C>0$ determined only by $r$, $|U'|_{C^4}$, and $|(U')^{-1}|_{C^0}$,  such that, for any $c_I \ge 0$ and $k \in \R$, 
\[
|\p_c \BF(k, \cdot + i c_I)|_{L_{c_R}^r} \le C \mu^{-1} e^{\mu^{-1}h},  \quad \lim_{c_I \to 0+} | \BF(k, \cdot + i c_I) -  \BF(k, \cdot )|_{W_{c_R}^{1,r}}=0,   
\]
where the norm is taken on $c_R \in [-\frac 12 h_0-h, \frac 12 h_0]$ and we recall $\mu = (1+k^2)^{-\frac 12}$. 
\item $\BF(k, c) \ne 0$ if $y_-(k, c, 0)=0$. Hence $\{c \mid \BF(k, c)=0\} = \{c\mid F(k, c) =0\}$ for any $k \in \R$.  
\item $\BF(k, c)=0$ iff there exists a $C^2$ solution $y(x_2)$ to \eqref{E:Ray-reg} satisfying the corresponding homogeneous boundary conditions of (\ref{E:Ray-2}-\ref{E:Ray-3}). 
\item For any $x_2\in (0, -h)$, $F_I(k, U(x_2))\ne 0$ if $U''(x_2)\ne 0$.
\end{enumerate}
\end{lemma}

\begin{proof}
For $c_R \in U\big([-h, 0)\big)$, the convergence of $\BF(k, c_R + i c_I)$ as $c_I \to 0+$ follows from the convergence estimates given in Proposition \ref{P:converg}. For $c$ near $U(0)$, the logarithmic singularity in $y_-'(k, c, 0)$ is cancelled by $(U(0)-c)^2$ and thus the convergence of $\BF(k, U(0) + i c_I)$ and the continuity of $\BF$ at $c=U(0)$ follow. The $C^\alpha$ and $C^{l_0-2}$ smoothness of $\BF$ is obtained from those of $y_-(k, c, 0)$ and $y_-'(k, c, 0)$ (Lemmas \ref{L:pcy0} and \ref{L:pcy2}) as well as using the factor $(U(0)-c)^2$ multiplied to $y_-'(k, c, 0)$. 

From Lemma \ref{L:y-lower-b}(1), $y_-(k, U(-h), 0), y_-(k, U(0), 0)>0$ and thus $F$ is well-defined near $c=U(-h), U(0)$. 
The property $F(k, U(-h)) \in \R$ and the value of $F(0, U(-h))$ are due to those of $Y$ given in Lemma \ref{L:Y-def}(1). The $C^1$ smoothness of $F$ for $c$ near $U(0)$ follows from Lemma \ref{L:Y-def}(2) and the definition of $F$. 
The values of $F$ and $\p_c F$ at $(k, U(0))$ is obtained by direct computation. 

Statement (3) is a corollary of Proposition \ref{P:converg}, Lemma \ref{L:pcy-complex}(3), and the definition of $\BF$.  

Suppose $y_-(k, c, 0)=0$. Lemma \ref{L:y-lower-b}(1) implies $c \ne U(0)$. As a non-trivial solution to the homogeneous Rayleigh equation \eqref{E:Ray-H1-1}, it must hold $y_-'(k, c, 0) \ne 0$. Therefore $\BF(k, c) \ne 0$. 

To prove statement (5), we first observe that $\BF(k, c)=0$ iff $y_-$ satisfies the corresponding homogeneous boundary conditions of \eqref{E:Ray-3}, which happens only if $y_-(k, c, 0)\ne 0$ and thus $Y(k, c)$ and $F(k, c)$ are well-defined. Moreover the statement is obvious for $c \notin U\big([-h, 0]\big)$ and also for $c= U(-h)$ due to the smoothness of $y_-$ (Lemma \ref{L:B}), while $F(k, U(0)) \ne 0$ due to statement (2). Hence we focus on $c \in U\big((-h, 0)\big)$ only. 
``$\Longrightarrow$'': As $c \in U\big((-h, 0)\big)$, $F(k, c)=0$ implies $Y_I(k, c) =0$ and consequently $U''(x_2^c)=0$ according to Lemma \ref{L:Y-I}. Consequently Lemma \ref{L:B}, particularly formula \eqref{E:Ray-H0-2}, and the definition of $\Gamma_0$ yield the smoothness of $y_-$ which apparently satisfies \eqref{E:Ray-reg}. ``$\Longleftarrow$'': 
This solution $y(x_2)$ has to be proportional to $y_-$ on $[-h, x_2^c]$ which yields $y(x_2^c)\ne 0$ due to \ref{L:y-lower-b}(2). Hence the smoothness of $y(x_2)$ and equation \eqref{E:Ray-reg} imply 
$U''(x_2^c)=0$. Consequently both \eqref{E:Ray-reg} and the homogeneous Rayleigh equation  \eqref{E:Ray-H1-1} are regular on $[-h, 0]$ and are equivalent to each other. Therefore $y_-(x_2)$ and $y(x_2)$ are proportional on $[-h, 0]$ and thus $y_-$ satisfies the boundary condition at $x_2=0$. 

To prove the last statement, let $c = U(x_2)$, $x_2 \in (-h, 0)$. According to Lemma \ref{L:y-lower-b}, $\IP\, y_-(k, c, 0) \ne 0$ if $U''(x_{2}) \ne0$ and thus $Y(k, c)$ is well-defined. Lemma \ref{L:Y-I} yields 
\be \label{E:Im-F} 
F_I(k, c) = \big(U(0)-c\big)^2 Y_I(k, c) = \frac {\pi (U(0)-c)^2 U''(x_{2}) y_{-} (k, c, x_{2})^2}{U'(x_{2}) |y_{-} (k, c, 0)|^2} \ne 0, 
\ee 
which prove statement (5). This is the same argument as in \cite{Yih72} in the case of gravity waves. 
\end{proof} 

\begin{remark} \label{R:singularM}
The monotonicity assumption on $U$ is used in the above proof of statement (5). If $U$ is not monotonic, $U^{-1}(c)$ may contain several points in $[-h, 0]$ for a root of $F(k, \cdot)$ and the corresponding solution $y_-(k, c, x_2)$ may not be in $H_{x_2}^2$. Therefore the set of roots of $F(k, \cdot)$, which is what really matters, may be larger than those defined as singular modes in Definition \ref{D:modes}. 
\end{remark}

In the next step, we consider $\BF$ for $|k|\gg1$. Unlike the linearized Euler equation on a fixed channel where no eigenvalues exist for large $k$. Eigenvalues do exist for each large $k$ for the linearized water wave system. According to Lemma \ref{L:ev-large-k}(2), we often consider $F(k, c)$ as well. 

\begin{lemma} \label{L:ev-large-k} 
Assume $U \in C^3$, then the following hold for any $\alpha \in (0, \frac 12)$. 
\begin{enumerate} 
\item There exists $C>0$ depending only on $\alpha$, $|U'|_{C^2}$, and $|(U')^{-1}|_{C^0}$,  such that 
\begin{align*}
&|\BF + \sigma k^2 \mu \sinh \mu^{-1}h - (U(0)-c)^2 \cosh \mu^{-1}h| \le  C \big( \mu^{\alpha-1} + |c|^2 \mu^\alpha \big)\cosh \mu^{-1} h,
\end{align*}
where we recall $\mu = (1+k^2)^{-\frac 12}$.
\item For any $k_*, M>0$, there exists $C>0$ depending only on $M$, $k_*$, $|U'|_{C^2}$, and $|(U')^{-1}|_{C^0}$,  such that, for any $|k|\le k_*$ and  $c$ satisfying $dist(c, U([-h, 0]) \ge M$, 
\[
|\BF -  (U(0)-c)^2 \cosh kh| \le C \big(1+ |c| + |U(0)-c|^2 dist(c, U([-h, 0])\big)^{-1} \big).
\]
\item There exist $k_0>0$ and $C>0$ depending only on $|U'|_{C^2}$, and $|(U')^{-1}|_{C^0}$, such that for any $|k|\ge k_0$, \eqref{E:BF} has exactly two solutions $c^\pm (k)$ which are contained in $\C \setminus U([-h, 0])$ and depend on $k$ analytically. Moreover they satisfy 
\[
c^\pm (k) \in \R, \quad c^\pm (-k) = c^\pm (k), \quad \left|c^\pm (k)  \mp \sqrt{\sigma |k|} -U(0) \right| \le C,
\quad \big|\p_c F\big(k, c^\pm (k) \big) \mp 2\sqrt{\sigma} |k|^{\frac 32} \big|  \le C |k|.
\]  
\end{enumerate}
\end{lemma}


\begin{proof}
The first statement follows directly from Lemma \ref{L:y-pm}, where the factor $(U(0)-c)^2$ is used to cancel the logarithmic singularity in the estimate of $y_-'$, and the second from Lemma \ref{L:regular-small-k} with $C_0 =  dist(c, U([-h, 0])^{-1}$. We focus on the roots of $\BF$. From Lemma \ref{L:y-pm}, 
\[
\exists k_0>0, \text{ s. t. } | y_-(k, c, 0)| \ge (1/2) \mu \sinh \mu^{-1} h  >0, \quad \forall |k| \ge k_0, \, c\in \C,
\]
and thus we can work with $F(k, c)$ and $Y(k, c)$. Let 
\[
S_k = \{ c \in \C \mid |c|\ge  \sqrt{\sigma|k|}/2\}. 
\]
From statement (1) and  Lemma \ref{L:Y-def}(3), it holds that there exist $k_0>0$ such that,  for any $|k|\ge k_0$, $F(k, c)=0$ only if $c \in S_k$. We may take larger $k_0>0$ if necessary such that $dist\big(S_k, U([-h, 0])\big)\ge 1$. From Lemma \ref{L:Y-Cauchy} and \ref{L:Y-I-1} and Corollary \ref{C:Y-Cauchy}, there exists $C>0$ depending only on $U$ such that, for all 
$|k|\ge k_0$, $c \in S_k$,   
\[
|Y(k, c) - k\coth kh | \le  \frac {C}{(1+|c|) \sinh^2 \frac h\mu}   \int_{U(-h)}^{U(0)} \sinh^2 \tfrac 1\mu ( U^{-1} (c') + h) dc', 
\]
\[
|\p_c Y(k, c) | \le \frac {C}{1+ |c|} + \frac {C}{(1+|c|)\mu \sinh^2 \frac h\mu} \int_{U(-h)}^{U(0)}  \sinh \tfrac 2\mu \big( U^{-1} (c') + h\big) dc'.  
\]
By a substitution $\tau = \tfrac 1\mu \big(U^{-1} (c') +h\big)$ 
we obtain 
\[
|Y(k, c) - k\coth kh | \le C(|k|+1)^{-1} (1+|c|)^{-1}, \quad |\p_c Y(k, c)| \le C(1+|c|)^{-1}, \quad \forall |k|\ge k_0.
\]
On the other hand, 
viewing $F(k, c)=0$ as a quadratic equation of $U(0)-c$, its roots also satisfy   
\[
c = f_\pm (k, c), \; \text{ where } \; f_\pm (k, c) = U(0) -  \frac {U'(0)}{2Y(k, c)} \pm \sqrt{\frac {U'(0)^2}{4Y(k, c)^2} + \frac {g + \sigma k^2}{Y(k, c)}}.
\] 
Using the above estimates on $Y$ and $\coth s =  1+ \tfrac 2{e^{2s}-1}$, it is straight forward to verify that for any $|k|\ge k_0$ and $c\in S_k$, 
\[
\left|f_\pm (k, c) \mp \sqrt{\sigma |k|} - U(0)\right| \le C, 
\]
and 
\begin{align*}
|\p_c f_\pm (k, c)| = & \left|\frac {U'(0)}{2Y^2} \mp \frac 12 \left(\frac {U'(0)^2}{4Y^2} + \frac {g + \sigma k^2}{Y}\right)^{-\frac 12} \left(\frac {U'(0)^2}{2Y^3} + \frac {g + \sigma k^2}{Y^2}\right)   \right| |\p_c Y| \le \frac C{|k|}.
\end{align*}
Therefore $f_\pm (k, \cdot)$ are contractions acting on $S_k$. Their fixed points $c^\pm (k)$, analytic in $k$, are the only solutions to \eqref{E:BF}, or equivalently \eqref{E:dispersion}. These $c^\pm (k) \in \R$ since $f_\pm (k, c) \in \R$ for $c \in \R$ which allows the iteration to be taken in $\R$. Finally, one may compute 
\be \label{E:pcF}
\p_c F =  (U(0)-c)^2\p_c Y + 2 (c-U(0)) Y + U'(0).  
\ee
Using the above estimates on $Y-|k|$, $\p_c Y$, and $c^\pm (k)$, one may compute
\[
\big|\p_c F\big(k, c^\pm (k) \big) \mp 2\sqrt{\sigma} |k|^{\frac 32} \big| = \big| 2Y \big(c - U(0)\big) \mp 2\sqrt{\sigma} |k|^{\frac 32} + \p_c Y \big(U(0) - c\big)^2 +U'(0)\big|_{c= c^\pm (k)} \le C |k|. 
\]
The evenness of $c^\pm (k)$ in $k$ is due to that of $\BF(k, c)$ and the uniqueness of the fixed points of the above contractions. This completes the proof of the lemma. 
 \end{proof} 

We shall track the two roots $c^\pm (k)$ of the analytic function $F(k, \cdot)$ as $|k|$ decreases, based  on a standard  analytic continuation argument.  

\begin{lemma} \label{L:continuation}
Assume $U \in C^3$. Suppose $k_0\in \R$ and $c_0 \in \C \setminus U([-h, 0])$ satisfy $\BF(k_0, c_0)=0$ and $\p_c \BF(k_0, c_0)\ne 0$, then the following hold. 
\begin{enumerate}
\item There exists an analytic function $c(k) \in \C \setminus U([-h, 0])$ defined on a maximal interval $ (k_-, k_+)\ni k_0$ such that $\BF\big(k, c(k)\big)=0$ and $\p_c \BF\big(k, c (k)\big)\ne 0$. 
\item $c(k) \in \R$ for all $k \in (k_-, k_+)$ if and only if $c_0 \in \R$.  
\item If $k_+ < \infty$ (or $k_- > -\infty$), then 
\begin{enumerate} \item 
$\lim_{k\to (k_+)-} dist(c(k), U([-h, 0])) = 0$ (or $\lim_{k\to (k_-)+} dist(c(k), U([-h, 0])) = 0$ if $k_->-\infty$), or 
\item $\liminf_{k\to (k_+)-} \min \{ |c(k) -c| \, : \, \forall c \text{ s. t. } \BF(k, c)=0, \, c\ne c(k) \} = 0$ (or $\liminf_{k\to (k_-)+} \min \{ |c(k) -c| \, : \,  \BF(k, c)=0, \, c\ne c(k) \} = 0$ if $k_- > -\infty$).
\end{enumerate}
\end{enumerate}
\end{lemma}

\begin{proof}
We start the proof with a simple and standard consideration of the index of complex analytic functions. 
Suppose $\BF(k, c) \ne 0$ at any $c \in \p \Omega$ where $\Omega \subset \C \setminus U([-h, 0])$ is a domain with  piecewise smooth boundary $\p \Omega$, then the index   
\be \label{E:ind}
\text{Ind}\big(\BF(k, \cdot), \Omega\big):=\frac 1{2\pi i} \oint_{\p \Omega} \frac {\p_c \BF(k, c)}{\BF(k, c)} dc \in \mathbb{N} \cup \{0\}
\ee 
is equal to the number of zeros of $\BF(k, \cdot)$ inside $\Omega$, counting their multiplicities. Therefore the analyticity of $F$ in $k$ and $c$ implies that Ind$\big(\BF(k, \cdot), \Omega\big)$ is a constant in $k$ as long as $\BF(k, c) =0$ does not occur on $\p \Omega$. 

As a consequence, starting with the simple root $c_0 \in \C \setminus U([-h, 0])$ of $\BF(k_0, \cdot)$, a unique continuation of $c(k) \subset \C \setminus U([-h, 0])$ of {\it simple} roots of $\BF(k, \cdot)$ exists and is analytic in $k$. The simplicity of $c(k)$ is due to the fact Ind$\big(\BF(k, \cdot), \Omega\big)=1$ for any sufficiently small neighborhood $\Omega$ of $c(k)$ in the continuation procedure. For any $c \in \R \setminus U([-h, 0])$, we have $\BF(k, c) \in \R$ and $\p_{c_R} \BF_R(k, c) = \p_c \BF(k, c) \ne 0$. Therefore  if $c(k_1) \in \R \setminus U([-h, 0])$ for some $k_1$ along the continuation curve, then the unique extension $c(k)$ coincides with the (real) root of $\BF_R(k, c_R)$ obtained by applying the Implicit Function Theorem to the real function $\BF_R(k, c_R)$.
Hence $c(k)\in \R$ if and only if $c_0 \in \R$.  

Let $(k_-, k_+)$ be the {\it max} interval of the continuation $c(k) \subset \C \setminus U([-h, 0])$ as simple roots of $\BF(k, \cdot)$ and we shall prove statement (3). Suppose 
$k_->-\infty$, while the other case $k_+<+\infty$ can be analyzed similarly. As $k \to (k_-) +$, the solution curve $c(k)$ is bounded due to Lemma \ref{L:ev-large-k}(2). Therefore there exists a sequence $(k_j)_{j=1}^\infty \subset (k_-, k_+)$ such that $\lim_{j\to \infty} k_j = k_-$ and $c_- = \lim_{j\to \infty} c(k_j) \in \C$ exists. Statement (2) implies that $c(k)$ stays in the closure of either the upper or lower half of $\C$ 
and thus $\BF(k_-, c_-)=0$. 
Assume statement (3)(a) does not hold, then such a subsequence can be chosen such that $c_- \notin U([-h, 0])$. Therefore $c_-$ is a root in the domain of analyticity of $\BF(k_-, \cdot)$. Clearly $c_-$ is not a simple zero of $\BF(k_-, \cdot)$, otherwise $c(k)$ can be extended beyond $k_-$. Recall $c_-$ has to be an isolated root of $\BF(k_-, \cdot)$ since all roots of non-trivial analytic functions are isolated. Therefore, there exists a small neighborhood $\Omega$ of $c_-$ such that, for any $k\ge k_-$ sufficiently close to $k_-$, it hold Ind$\big(\BF(k, \cdot), \Omega\big) \ge 2$. Consequently, for each $k_j$ close to $k_-$, there exists at least another root $c$ of $\BF(k_j, \cdot)$ in $\Omega$ and thus (3)(b) holds. 
\end{proof}

The semicircle theorem of Yih \cite{Yih72} states that all imaginary roots $c$ of $\BF(k, \cdot)$ are contained in the circle with the diameter segment $U([-h, 0])$, so the only possibility for the branches $c^\pm(k)$ of simple roots of $\BF(k, \cdot)$ obtained in Lemma \ref{L:ev-large-k} can not be extended for all $k \in \R$ is when they reaches $U(0)$ or $U(-h)$, respectively. As a corollary of $\BF\big(k, U(0)\big)\ne 0$ and we have  

\begin{corollary} \label{C:branches-1} 
(1) The branch $c^+(k)$ 
can be extended for all $k \in \R$. Moreover $c^+(k) \in \R$ is even in $k$, $\p_c F\big(k, c^+(k)\big)> 0$, and $c^+(k) > U(0) + \rho_0$ for all $k\in \R$, for some $\rho_0>0$ independent of $k$. \\
(2) If $\BF(k, U(-h)) \ne 0$ for all $k \in \R$, then $c^-(k)$ of simple roots of $\BF(k, \cdot)$ obtained in can also be extended for all $k \in \R$. Moreover $c^-(k) \in \R$ is even in $k$, $\p_c F\big(k, c^-(k)\big)< 0$, and $c^-(k) < U(-h) - \rho_0$ for all $k\in \R$, for some $\rho_0>0$ independent of $k$.
\end{corollary}

\begin{proof} 
Let $k_0$ be given in Lemma \ref{L:ev-large-k}(3) and we only need to focus on $|k| \le k_0$. We may assume $k_0$ is sufficiently large such that $c^+(k_0) > U(0)$ and $c^-(k_0) < U(-h)$. From Lemma \ref{L:ev-large-k}(2), there exists $R>0$ such that $\BF(k, c) \ne 0$ for all $k \in [-k_0, k_0]$ and $|c|\ge R$. Hence $c^+ (k_0) \in (U(0), R)$ and $c^-(k_0)\in (-R, U(-h))$ are the only roots of $\BF(\pm k_0, \cdot)$, which are also simple with $\pm \p_c F(k_0, c^\pm (k_0)) >0$. 

We first consider $c^+(k)$. Let 
\[
\Omega = \{ c\in \C \mid c_R \in ( U(0), R), \, c_I \in (-1, 1)\}. 
\]
According to Lemma \ref{L:e-v-basic-1}(2), $\BF(k, U(0))\ne 0$ for any $k$. Hence the semicircle theorem and the choice of $R$ imply that a.) $c^+(k_0) \in \Omega$ and b.) $\BF(k, c)  \ne 0$  for all $|k|\le k_0$ and $c \in \p \Omega$, and thus 
\[
\text{Ind}(\BF (k, \cdot), \Omega) =\text{Ind}(\BF (k_0, \cdot), \Omega) = 1, \forall |k| \le k_0. 
\]
Therefore none of the possibilities in Lemma \ref{L:continuation}(3ab) can happen to the extension $c^+(k) \in \Omega$ starting from $k=k_0$, so this branch of simple root of $\BF(k, c\dot)$ can be uniquely extended for all $k \in [-k_0, k_0]$ with $c^+(k) \in (U(0), R)$ as the only root of $\BF(k, \cdot)$ in $\Omega$. The value of this extension at $k=-k_0$ has to coincide with $c^+ (-k_0)= c^+(k_0)$ as $c^\pm (-k_0)$ are the only roots of $\BF (-k_0, \cdot)$ while $c^-(-k_0) < U(-h)$. Therefore the extensions starting from $c^+(\pm k_0)$ have to coincide. The evenness of $c^+(k)$ in $k \in [-k_0, k_0]$ follows from that of $\BF$ and the uniqueness of its root in $\Omega$. The sign of $\p_c F\big(k, c^+(k)\big)$ remains positive from $k=k_0$ as $c^+(k)$ is always simple. The existence of $\rho_0>0$ is simple due to the continuity of $\BF$. The same argument applies to $c^-(k)$ under the assumption $\BF(k, U(-h))\ne 0$ all for $k$. The proof is complete.  
\end{proof} 

Based on the above analysis, we shall conclude that $-ik c^\pm (k)$ are the only eigenvalues of the linearized capillary gravity wave under the additonal assumption of the absence of singular modes
\be \label{E:no-S-M}
\BF(k, U(x_2)) \ne 0, \quad \forall k \in \BK, \; x_2 \in [-h, 0],
\ee
where $\BK= \R$ or $\frac {2\pi}L \mathbb{N}$ and $L$ is the period of the water wave in the $x_1$ direction. 

\begin{proposition} \label{P:e-v-0}
Assume $U \in C^3$ and \eqref{E:no-S-M} for $\BK= \R$ or $\frac {2\pi}L \mathbb{N}$, then there exists $\rho>0$ such that 
\begin{enumerate} 
\item $F_0 \triangleq \inf \{ (1+k^2)^{-\frac 12} e^{-\frac h\mu} |\BF(k, c)| \mid k \in \BK, \, c_R \in [U(-h)-\rho, U(0)+\rho], \, c_I \in [-\rho, \rho]\}  > 0$.
\item Assume $\BK=\R$, then 
$\{ c\mid \BF(k, c) =0 \}= \{c^\pm (k)\}$.
\end{enumerate}
\end{proposition}

\begin{proof}
The first statement is a direct corollary of the continuity of $\BF$, its analyticity outside $U([-h, 0])$, assumption \eqref{E:no-S-M}, and Lemma \ref{L:ev-large-k}. 

Let us consider statement (2). Corollary \ref{C:branches-1} and \eqref{E:no-S-M} imply that both $c^+ (k) \in (U(0), +\infty)$ and $c^-(k) \in (-\infty, U(-h))$ can be extended as even analytic functions of $k\in \R$.  Let $k_0, R>0$ be taken as in the proof of Corollary \ref{C:branches-1}  and we only need to focus on $|k| \le k_0$. Assumption \eqref{E:no-S-M} also yields $\rho>0$ such that 
\[
\BF(k, c) \ne 0, \;\; \forall dist\big(c, U([-h, 0])\big) = \rho, \; |k|\le k_0. 
\]
Let 
\[
\Omega = \{ c \in \C \mid |c| < R, \, dist\big(c, U([-h, 0])\big) > \rho \}, 
\]
then we have $\BF(k, c)  \ne 0$  for all $|k|\le k_0$ and $c \in \p \Omega$. Therefore 
\[
\text{Ind}(\BF (k, \cdot), \Omega) =\text{Ind}(\BF (k_0, \cdot), \Omega) = 2, \forall |k| \le k_0, 
\]
and $\BF(k, \cdot)$ does not have any other roots. 
\end{proof}

In order to obtain a more complete picture of the eigenvalue distribution we shall derive some sign properties in the following lemma, where $F$ and $Y$ are viewed as function of $c$ and $K=k^2\ge 0$. According to Lemma \ref{L:y-lower-b}(1), $F$ is well-defined for $c$ in a neighborhood of $\R \setminus U\big((-h, 0)\big)$. 

\begin{lemma} \label{L:F-signs}
Assume $U \in C^3$, then we have 
\[
\p_K^2 \big(F\big(\sqrt{K}, c \big)\big) <0, \;\ \forall k\in \R, \; c\in \R \setminus U\big((-h, 0]\big),
\]
\[
\p_K F(0, c) < - \sigma + \int_{-h}^0 \big(U(x_2) - c \big)^2 dx_2, \quad \forall c \in \R \setminus U([-h, 0]), 
\]
\[ 
\p_K F\big(0, U(-h)\big) =- \sigma + \int_{-h}^0 \big(U(x_2) - U(-h)\big)^2 dx_2. 
\]
\end{lemma}

\begin{proof}
For $K\ge 0$ and $c \in \C$ with $y_-(k, c, 0)\ne 0$ and $c_I\ge 0$, let 
\be \label{E:Ray-op}
\CR=\CR (K, c) = - \p_{x_2}^2 + K + \frac {U''(x_2)}{U(x_2) -c}, \;\; \tilde y(K, c, x_2) = \frac {y_{-} (\sqrt{K}, c, x_2)}{y_{-} (\sqrt{K}, c, 0)}, \quad \; x_2 \in [-h, 0],
\ee
be the differential operator in the Rayleigh equation \eqref{E:Ray-H1-1} and the normalization of the fundamental solution $y_{-}$ defined in \eqref{E:y-pm} and \eqref{E:y0}. Clearly 
\[
\tilde y(-h)=0, \quad \tilde y'(-h)=y_{-} (\sqrt{K}, c, 0)^{-1}, \quad  \tilde y(x_2) > 0, \;  x_2 \in (-h, 0), \quad \tilde y(0)=1, \quad Y(\sqrt{K}, c)= \tilde y' (0), 
\]
where the sign properties follows from Lemma \ref{L:y-lower-b}(1). 
It is straight forward to compute, for $c \in \R \setminus U\big((-h, 0)\big)$ and $x_2 \in (-h, 0)$, 
\[
\CR \p_K \tilde y = -\tilde y <0, \quad  \CR \p_{K}^2 \tilde y = - 2 \p_K \tilde y,  
\] 
where the smoothness of $\tilde y$ in $K$ is ensured by Lemma \ref{L:B}. 
The following claim is used to analyze these and some other functions.

{\it Claim.} Suppose $y\in C^2((-h, 0)) \cap C^0([-h, 0])$ is a solution to $(\CR y)(x_2) = f(x_2)$ and $y(-h)=y(0)=0$ with $c \in \R \setminus U\big((-h, 0)\big)$, where $f$ is $C^0$ on $[-h, 0]$, then we have the following through direct computations  
\be \label{E:F-signs-0}
(\tilde y'  y - \tilde y  y')' = \tilde y f  \Rightarrow  y'(0) =- \int_{-h}^0 \tilde yf dx_2, \;\; 
y (x_2) =  \tilde y(x_2) \int_{x_2}^0 \frac 1{\tilde y (x_2')^{2}} \int_{-h}^{x_2'} \tilde y(x_2'') f(x_2'') dx_2'' dx_2'.
\ee
 
Applying this claim to $\p_K \tilde y$ and $\p_{KK} \tilde y$ implies, for $c \in \R \setminus U\big((-h, 0]\big)$, 
\be \label{E:F-signs-0.1} \begin{split}
&\p_K Y = \p_K \tilde y' (0) = \int_{-h}^0 \tilde y^2 dx_2 >0, \\
& \p_{K}^2 Y = -2 \int_{-h}^0 \tilde y(x_2)^2 \int_{x_2}^0 \tilde y (x_2')^{-2} \int_{-h}^{x_2'} \tilde y(x_2'')^2 dx_2'' dx_2' dx_2 <0.
\end{split} \ee
The definition of $F$ implies $\p_K^2 F<0$. 

For $k=0$, through direct calculation, one may verify, for $c \notin U([-h, 0])$, 
\be \label{E:y_--0}
y_- (0, c, x_2) = (U(x_2)-c) \int_{-h}^{x_2} \frac {U(-h)-c}{(U(x_2')-c)^2} dx_2'.
\ee
For $c \in \R \setminus U([-h, 0])$, from \eqref{E:F-signs-0.1}, we have 
\begin{align*}
\p_K Y(0, c) =& \int_{-h}^0 \tilde y^2 dx_2 = \int_{-h}^0 \frac {(U-c)^2}{(U(0)-c)^2} \Big(\int_{-h}^{x_2}  \frac {dx_2' }{(U(x_2')-c)^2} \Big)^2  dx_2 \Big( \int_{-h}^{0} \frac {dx_2'}{(U(x_2')-c)^2} \Big)^{-2}, 
\end{align*}
and thus 
\be  \label{E:F-esti-temp-0.5} \begin{split}
\p_K F(0, c) = & (U(0)-c)^2 \p_K Y(0, c) -\sigma 
\\
=& \int_{-h}^0 (U-c)^2 \Big(\int_{-h}^{x_2}  \frac {dx_2' }{(U(x_2')-c)^2} \Big)^2  dx_2 \Big( \int_{-h}^{0} \frac {dx_2'}{(U(x_2')-c)^2} \Big)^{-2} - \sigma \\
< & \int_{-h}^0 (U-c)^2 dx_2 -\sigma. 
\end{split} \ee

For $k=0$ and $c = U(-h)$, we can use \eqref{E:y-0} to compute 
\be \label{E:F-esti-temp-1}
\tilde y \big(0, U(-h), x_2\big) =\big(U(x_2) - U(-h)\big)/ \big(U(0) - U(-h)\big). 
\ee
Consequently, one obtains 
explicitly
\[
\p_K Y\big(0, U(-h)\big) = \int_{-h}^0 \frac {\big(U(x_2) - U(-h)\big)^2}{\big(U(0) - U(-h)\big)^2} dx_2,   
\]
which in turn yields the desired formula of $\p_K F\big(0, U(-h)\big)$. 
\end{proof}

The information on the derivatives of $F$ leads to the following  properties of the roots of $F$.  

\begin{lemma} \label{L:g-thresh}
Assume $U\in C^3$, the following hold. 
\begin{enumerate}
\item If 
\be \label{E:sigma-1} 
\sigma \ge \int_{-h}^0 \big(U(x_2) - U(-h)\big)^2 dx_2 \Longleftrightarrow \p_K F\big(0, U(-h)\big) \le 0,
\ee 
then $F\big(k, U(-h) \big) \le -g = F(0, U(-h))$ for all $k \in \R$.
\item 
Let 
\begin{align*}
g_\# = & \max \big\{ Y\big(k, U(-h)\big) \big(U(0)-U(-h)\big)^2 - U'(0) \big(U(0)-U(-h)\big) -\sigma k^2 \mid k \in \R\big\}\\
= & \max \big\{ F\big(k, U(-h)\big) +g  \mid k \in \R\big\},    
\end{align*}
then we have
\begin{enumerate}
\item $g_\# \ge F\big(0, U(-h)\big) +g = 0$ and ``='' in the ``$\le$" holds if and only if \eqref{E:sigma-1} holds.
\item If $g > g_\#$, then $F\big(k, U(-h)\big) <0$ for all $k \in \R$.
\item If $0< g =g_\#$ , then there exists a unique $k_\# >0$ such that $F\big( \pm k_\#, U(-h)\big) =0$ and $F\big(k, U(-h)\big) <0$ for all $|k| \ne k_\#$.
\item If $0< g< g_\#$, then there exist $k_\#^+ > k_\#^- >0$ such that 
\[
F\big(k, U(-h)\big) <0, \;\;  |k| \notin (k_\#^-, k_\#^+); \quad F\big(k, U(-h)\big) >0, \;\; |k| \in (k_\#^-, k_\#^+); \quad \mp \p_k F(k_\#^\pm, U(-h)) >0. 
\] 
\end{enumerate} \end{enumerate}
\end{lemma} 

\begin{proof} 
Statement (1) is a direct consequence of the concavity of $F\big(k, U(-h)\big)$ in $K=k^2$ and $F\big(0, U(-h)\big) = -g<0$. 
Statement (2) is also an immediate implication of this concavity and Lemma \ref{L:ev-large-k}(1). 
\end{proof}

Along with statement (2b ) and Corollary \ref{C:branches-1}, \eqref{E:sigma-1} provides an explicit sufficient condition ensuring that the branch $c^-(k)$ does not reach $U([-h, 0])$ and thus staying in $(-\infty, U(-h))$ for all $k \in \R$. 

To end this subsection we prove the following monotonicity of the even functions $c^\pm (k)$ which will be used in obtaining the conjugacy between the irrotational linearized capillary gravity water waves and the component of the solutions   linearized at the shear $U(x_2)$. From the definition of $F$ and \eqref{E:y_--0}, we first compute, for $c \notin  U([-h, 0])$,   
\[
Y(0, c) = \frac{U'(0) \int_{-h}^0 (U-c)^{-2} dx_2 + (U(0)-c)^{-1} }{ (U(0)-c) \int_{-h}^0 (U-c)^{-2} dx_2} 
\]
and thus 
\[
F(0, c) = (U(0)-c)^2 Y(0, c) - U'(0) (U(0)-c) -g = \frac 1{\int_{-h}^0 (U-c)^{-2} dx_2} -g, 
\]
which is uniformly increasing on $(-\infty, U(-h))$ and uniformly decreasing on $(U(0), +\infty)$. Therefore $F(0, \cdot)$ has two real roots 
\be \label{E:c0pm}
c_0^+ \in (U(0), +\infty), \;\; c_0^- \in (-\infty, U(-h)), \; \text{ s. t. }\;  F(0, c_0^\pm) =  \frac 1{\int_{-h}^0 (U-c_0^\pm)^{-2} dx_2} -g =0,
\ee
which are unique in the above intervals. 

\begin{lemma} \label{L:c(k)-mono}
Assume $U\in C^3$,  
then the following hold. 
\begin{enumerate}
\item For $\dagger \in \{+, -\}$, suppose $c^\dagger(k) \in \R \setminus U([-h, 0])$ can be extended as simple roots of $F(k , \cdot)$ for all $k \ge k_*\ge 0$, then $(c^\dagger)'(k)=0$ has most one solution on $(k_*, +\infty)$, where $(c^\dagger)''(k)\ne 0$ is also satisfied.
\item For $\dagger \in \{+, -\}$, suppose $c^\dagger(k) \in \R \setminus U([-h, 0])$ can be extended  as simple roots of $F(k , \cdot)$ for all $k \in \R$, then $(c^\dagger)' (k) \ne 0$ for all $k \ne 0$ if and only if 
\be \label{E:c(k)-mono-1}
\sigma \ge g^{2} \int_{-h}^0 (U-c_0^\dagger)^2 \Big(\int_{-h}^{x_2}  \frac {dx_2' }{(U(x_2')-c_0^\dagger)^2} \Big)^2  dx_2, 
\ee
with $c_0^\dagger$ defined in \eqref{E:c0pm}. 
\end{enumerate}  
\end{lemma}


\begin{proof}

We shall work with $c^-(k)$, while the same proof works for $c^+(k)$. Suppose there exists $k_0>k_*\ge 0$ such that $(c^-)'(k_0)=0$, then 
\[
2k_0 (\p_K F)(k_0, c^- (k_0)) = \p_k F(k_0, c^- (k_0)) = - \p_c F(k_0, c^- (k_0)) (c^-)'(k_0)=0. 
\]
Computing the second order derivative at $k_0$, we have 
\[
(c^-)'' (k_0)= -\frac {\p_k^2 F
}{\p_c F}\Big|_{(k_0, c^- (k_0))} = -\frac { 4k_0^2 (\p_K^2 F)
+ 2(\p_K F)
}{\p_c F}\Big|_{(k_0, c^- (k_0))} = -\frac { 4k_0^2 (\p_K^2 F)
}{\p_c F}\Big|_{(k_0, c^- (k_0))},
\]
which along with Lemma \ref{L:F-signs} and $\p_c F(k, c^- (k_0))<0$ (Corollary \ref{C:branches-1}) implies 
$(c^-)''(k_0) <0$. 
Hence $k_0>k_*$ has to be the only positive critical point of $c^- (k)$. 

To prove Statement (2) where $k_*=0$, on the one hand, we first observe that since $c_0^-$ is the unique root of $F(0, \cdot)$ in $(-\infty, U(-h))$ and $c^-(0)$ is also such a root, so $c^-(0)=c_0^-$. Moreover, \eqref{E:F-esti-temp-0.5} implies that \eqref{E:c(k)-mono-1} is equivalent to $\p_K F(0, c^-(0)) \le 0$. On the other hand, observe that the evenness of $c^- (k)$ yields $(c^-)' (0)=0$. One may compute 
\[
(\p_K F)(0, c^-(0)) = \p_k^2 F(0, c^-(0))/2 =  -\p_c F(0, c^- (0)) \big((c^-)'' (0)\big)/2. 
\]

From Lemma \ref{L:ev-large-k}(3), $(c^-)'(k)<0$ for some $k \gg1$. Hence, on the one hand, if \eqref{E:c(k)-mono-1} does not hold, then $\p_c F(0, c^- (0))<0$ (Lemma \ref{L:ev-large-k}(3) and \ref{L:continuation}) and the above identity imply $(c^-)''(0) >0$. Along with $(c^-)'(0)=0$, 
it yields that $c^-$ has a critical point $k_0>0$. On the other hand, through the same argument, \eqref{E:c(k)-mono-1} yields $(c^-)''(0) \le 0$ while $(c^-)'(0)=0$. Therefore, if $\p_K F(0, c^-(0))<0$ which implies $(c-)''(0)<0$, it is impossible that there exists a unique critical point of $c^-$ where $(c^-)''<0$. In the borderline case of $\p_K F(0, c^-(0))=0$ which implies $(c-)''(0)<0$, further Taylor expansions of the even-in-$k$ functions $F(k, c)$ and $c^-(k)$ yields 
\[
\p_k^4 c^-(0) = -12 \p_K^2 F(0, c^-(0)) / \p_c F(0, c^-(0)) <0. 
\]
From the same reasoning, we obtain that $(c^-)' \ne 0$ for $k>0$.  The proof of the lemma is complete. 
\end{proof}

\subsection{Eigenvalue distribution of convex/concave shear flow $U$} \label{SS:ev-details}

To analyze eigenvalues under less implicit assumptions than \eqref{E:no-S-M}, particularly the generation of unstable modes from $c=U(-h)$, we further assume $U''\ne 0$ on $[-h, 0]$. Due to Lemma \ref{L:e-v-basic-1}(6), this rules out the possibility of roots of $\BF$ on $U\big((-h, 0])$ and provides better smoothness of $F$ for the bifurcation analysis.  

\begin{lemma} \label{L:e-v-large-k} 
Assume $U\in C^{l_0}$, $l_0\ge 5$, and  $U''\ne 0$ on $[-h, 0]$, then 
$F(k, c)$ is well defined for all $k\in \R$ and $c\in \C$ and \\
a.) $F$ is analytic in both $k\in \R$ and $c \notin U([-h, 0])$ and, when restricted to $c_I\ge 0$, is $C^{l_0-2}$ in both $k\in \R$ and $c \notin \{U(-h), U(0)\}$,  \\
b.) $\p_k^{j_1} \p_c^{j_2} F$ is locally $C^{\alpha}$ in both $k$ and $c\ne U(0)$ with $c_I\ge 0$ for any $\alpha \in [0, 1)$, $j_2=0,1$, and $0\le j_1\le l_0-4 -j_2$, \\
c.) $F$ is $C^1$ in $k$ and $c$ with $c_I\ge 0$. 
\end{lemma} 

\begin{remark}
Note that, in the above statement, for fixed $c \in U([-h, 0))$, $F$ is $C^{l_0-2}$ in $k$. This stronger regularity in $k$ follows from that of $(y_{0-}, y_{0-}')$ and $Y$ (see Lemmas \ref{L:y0} and \ref{L:Y-Cauchy}). Moreover, one could prove that $F$ and $\p_k F$ are also $C^{1, \alpha}$ near $c=U(0)$ with $c_I\ge 0$ by estimating $\p_{c_R}^2 Y_I (k, c) = O\big(|c-U(0)|^{-1}\big)$ using Lemmas \ref{L:pcy1}--\ref{L:pcy3} and \ref{L:Y-Cauchy} as well as Corollary \ref{C:Y-Cauchy}. 
\end{remark}


\begin{proof} 
The assumption $U''\ne 0$ implies that $y_-(k, c, 0)\ne 0$ for all $k$ and $c$ (Lemma \ref{L:y-lower-b}(5)) and thus $F$ is well defined. The analyticity and the $C^{l_0-2}$ and $C^{1, \alpha}$ (restricted to $c_I\ge 0$ for the latter two) regularity of $F$ follow directly from those of $Y$ given in Lemma \ref{L:Y-Cauchy} except at $c=U(0)$. Near $c\in U(0)$, the regularity and estimates on $Y$ (Lemma \ref{L:Y-def}, \ref{L:Y-I-1}, \ref{L:Y-Cauchy}) and $\p_c Y$ (Lemma \ref{L:Y-I-1} and Corollary \ref{C:Y-Cauchy}) yield the regularity of $F$. 
\end{proof}


As a corollary of the Lemmas \ref{L:continuation}, \ref{L:g-thresh} and \ref{L:e-v-large-k} and the semicircle theorem, we obtain a sufficient condition for \eqref{E:no-S-M} to hold for $\BK = \R$. 

\begin{corollary} \label{C:no-S-M}
Suppose $U'' \ne 0$ on $[-h, 0]$ and \eqref{E:sigma-1} hold, then \eqref{E:no-S-M} is true for all $k \in \R$.
\end{corollary}

Assuming $U''\ne 0$, in general $c=U(-h)$ is the only point outside the domain of analyticity of $F(k, \cdot)$ which might happen to be a root and also might be the end point of branches of roots of $F(k, \cdot)$, it is a crucial step to analyze zeros of $F$ around $U(-h)$. 

\begin{lemma} \label{L:F-esti}
Assume $U \in C^5$, then (a) $\p_c F(k, U(-h)) <0$ for all $k \in \R$ if $U''>0$ on $[-h, 0]$; and (b) if $U''<0$ on $[-h, 0]$, then $\p_c F(k, U(-h)) <0$ if $F(k, U(-h)) =0$, where $F$ is understood as restricted to $c_I \ge 0$. 
\end{lemma} 

\begin{proof}
We shall use the notations $\CR$ and $\tilde y$ defined in the proof of Lemma \ref{L:F-signs} and $F$ and $Y$ are also viewed as function of $c$ and $K=k^2\ge 0$.
It is straight forward to compute, for $c < U(-h)$ and $x_2 \in (-h, 0)$, 
\[
\CR \p_c \tilde y = -\tfrac {U''}{(U-c)^2}\tilde y, \quad \CR \p_{Kc} \tilde y = - \p_c \tilde y - \tfrac {U''}{(U-c)^2}\p_K \tilde y.  
\] 
Applying \eqref{E:F-signs-0} we obtain 
\be \label{E:F-signs} \begin{split}
& U''(0) \p_c Y
= U''(0) \p_c \tilde y' (0) =  U''(0)\int_{-h}^0 \frac {U'' \tilde y^2}{\big(U- c\big)^2} dx_2>0, \;\; U''(0) \p_{Kc} Y
<0, \quad \forall c < U(-h).
\end{split} \ee
These integral representation of $\p_c Y$ still holds as $c \to U(-h)-$, and thus also its sign. For $k=0$ and $c = U(-h)$, we can use \eqref{E:F-esti-temp-1} to compute 
\[
\p_c Y \big(0, U(-h)\big) 
= \frac {U'(0) - U'(-h)}{\big(U(0)-U(-h)\big)^2} \implies 
\p_c F\big(0, U(-h)\big) = - U'(-h) <0. 
\]

Finally we obtain the sign of $\p_c F(k, U(-h))$ in two cases separately, based on the sign of $U''$. Suppose $U''>0$. The above \eqref{E:F-signs-0.1} and \eqref{E:F-signs} implies that, for $c\le U(-h)$, $Y(\sqrt{K}, c)$ is strictly increasing in $K$ and $\p_c Y(\sqrt{K}, c)$ is strictly deceasing in $K$, and thus 
\[
\p_c F = (U(0)-c)^2 \p_c Y - 2 (U(0)-c) Y + U'(0)  
\]
is also strictly decreasing in $K$. Letting $c \to U(-h)-$, this monotonicity yields 
\[
\p_c F(k, U(-h)) \le \p_c F(0, U(-h)) = -U'(-h) <0. 
\] 

In the other case of $U''<0$, suppose $F(k, U(-h))=0$ for some $k \in \R$, which implies
\[
Y(k, U(-h)) = \frac {g +\sigma k^2}{(U(0)- U(-h))^2} + \frac {U'(0)}{U(0)- U(-h)}. 
\]
Therefore 
\begin{align*}
\p_c F(k, U(-h)) =& (U(0)- U(-h))^2\p_c Y(k, U(-h)) - 2 (U(0)- U(-h)) Y(k, U(-h)) + U'(0)\\
= & (U(0)- U(-h))^2\p_c Y(k, U(-h))  - U'(0)- 2 (g +\sigma k^2) /(U(0)- U(-h)).  
\end{align*}
We also have $\p_c Y(k, U(-h)) <0$ from taking the limit of \eqref{E:F-signs}.
Hence we obtain $\p_c F(k, U(-h)) <0$ and the proof of the lemma is complete.  
\end{proof}  

In the next step we shall study the roots of $F(k, \cdot)$ near $c=U(-h)$.   

\begin{lemma} \label{L:bifurcation}
Assume $U\in C^5$, and $U''\ne 0$ on $[-h, 0]$. Suppose $F\big(k_0, U(-h)\big)=0$, then there exist $\ep>0$, $\rho \in \big(0, U(0)-U(-h)\big)$, and $\CC \in C^{1, \alpha} \big([k_0-\ep, k_0+ \ep], \C\big)$ for any $\alpha\in [0, 1)$ such that 
\[
\CC(k_0)=U(-h), \;\; \CC(k) \notin U([-h, 0]), \;\; 0< |k -k_0|\le \ep, \quad U''\CC_I(k) \ge 0, \;\; |k-k_0|\le \ep,
\]
\[
\p_c F(k, \CC(k)) \ne 0, \; \text{ if } \; \CC_I (k)\ge 0,
\]
and for $c_I \in [0, \rho]$ and $|c_R-U(-h)|\le \rho$, 
\[
F(k, c) =0 \text{ with } k \in [k_0-\ep, k_0+\ep], \; 
\text{ iff } c = \CC(k) = \CC_R(k) + i \CC_I(k). 
\]
Moreover, without loss of generality assume $k_0>0$ (Lemma \ref{L:e-v-large-k} implies $k_0\ne 0$) and this branch of roots of $F$ satisfies 
\begin{enumerate}
\item If $\p_k F\big(k_0, U(-h)\big) =0$, then $\CC'(k_0)=0$, $\CC_I \equiv 0$ and $\CC (k)< U(-h)$ for all $0<|k-k_0|\le \ep$.
\item If $\pm \p_k F\big(k_0, U(-h)\big) >0$, then $\pm \CC_R'(k_0)>0$ and 
\[
\CC_R(k) < U(-h), \; \; \CC_I(k)=0, \quad \forall \, 0< \pm (k_0-k) \le \ep, 
\]
and for some $\tilde C>0$ determined by $k_0$ and $U$, 
\[
\CC_R(k) > U(-h), \;\; \left|\frac {\CC_I(k)}{ Y_I \big(k, \CC_R(k)\big)} + \frac {\big(\big(U(0)- U(-h)\big)^2} {\p_{c} F \big(k_0, U(-h)\big)}\right| \le \tilde C|k- k_0|^\alpha,     \quad \forall \, 0< \pm (k-k_0) \le \ep, 
\]
which implies 
\[
0 < |\CC_I(k)| \le 
\tilde C (k-k_0)^2, \;\; U''(0)\CC_I(k)>0,  \quad \forall \, 0< \pm (k-k_0) \le \ep.  
\] 
\end{enumerate}
\end{lemma}

In the generic case $\p_k F\big(k_0, U(-h)\big) \ne 0$, locally the roots of $F(k, c)$ consists of the intersection of the graph of $\CC(k)$ and the closure of the upper half complex plane,  along with its complex conjugate. 
In this case, however, one observes that $d \CC_I(k)/d \CC_R(k)=0$ at $k=k_0$ means $\CC_I$ is very weak when it is nonzero. The following proof is based on both the Implicit Function Theorem and the Intermediate Value Theorem.    

\begin{proof}
According to Lemma \ref{L:e-v-large-k}, $F$ is $C^{1, \alpha}$ in $k$ and $c$ in the region $c_I\ge 0$. As $F_I$ is not continuous at $c \in U\big((-h, 0]\big) \subset \C$ in general, let $\tilde F(k, c) = \tilde F_R + i \tilde F_I \in \C$ be a $C^{1, \alpha}$ extension of $F$ into a neighborhood of $\big(k_0, U(-h)\big) \in \R \times \C$ which coincides with $F$ for $c_I\ge 0$. From Lemma \ref{L:F-esti}, the $2\times 2$ Jacobian matrix of $D_c \tilde F$ satisfies 
\[
D_c \tilde F \big(k_0, U(-h)\big) = \begin{pmatrix} \p_{c_R} \tilde F_R & \p_{c_I} \tilde F_R \\ \p_{c_R} \tilde F_I & \p_{c_I} \tilde F_I \end{pmatrix}\Big|_{\big(k_0, U(-h)\big)} =  
\p_{c} F \big(k_0, U(-h)\big) I_{2\times 2},  \quad \p_c F \big(k_0, U(-h)\big)<0,
\]
where we used the Cauchy-Riemann equation and the fact $F(k, c) \in \R$ for all $c < U(-h)$. Therefore the Implicit Function Theorem implies that all roots of $\tilde F(k, c)$ near $\big(k_0, U(-h)\big)$ form the graph of a $C^{1, \alpha}$ complex-valued function $\CC(k)$ which contains $\big(k_0, U(-h)\big)$. To complete the proof of the lemma, we only need to prove that $\CC(k)$ satisfies properties (1) and (2). 
 
Firstly we prove $\CC(k) \in \R$ if $\CC_R (k) \le U(-h)$ and thus $F\big(k, \CC(k)\big)= \tilde F\big(k, \CC(k)\big)=0$ as well. When restricted to $\R$, $F_R\in C^1$ and $\p_{c_R} F_R \big(k_0, U(-h)\big) = \p_c F(k_0, U(-h))<0$. The Implicit Function Theorem yields a $C^1$ real-valued function $\tilde \CC(k)$ for $k$ near $k_0$ such that 
\be \label{E:tCC}
\tilde \CC(k_0)=U(-h), \quad F_R\big(k, \tilde \CC(k))=0. 
\ee
Since $F_I(k, c)=0$ if $c \le U(-h)$, the uniqueness of solutions ensured by the Implicit Function Theorem implies that $\CC(k) = \tilde \CC(k) \in \R$ if $\tilde \CC(k) \le U(-h)$.  

Next we consider the case $\p_k F\big(k_0, U(-h)\big)=0$. Along with 
\[
\p_{c_R} F_R \big(k_0, U(-h)\big), \, \p_{KK} F_R \big(k_0, U(-h)\big)<0, \, \text{ where } K=k^2, 
\]
it implies 
\[
F_R(k_0, c) >0, \; \forall 0 < U(-h)-c\ll 1, \quad F_R \big(k, U(-h)\big) <0, \; \forall k \in \R^+\setminus \{k_0\}. 
\]
From the Intermediate Value Theorem, for $k$ near $k_0$, there exist real roots of $F_R(k, \cdot)$ slightly smaller than $U(-h)$, which must belong to $\tilde \CC(k)$ due to the uniqueness of solutions ensured by the Implicit Function Theorem. Therefore along with the last step, we conclude $\CC(k) = \tilde \CC(k) <U(-h)$ for $k\ne k_0$ close to $k_0$. 

Finally, we consider the case of $\p_k F\big(k_0, U(-h)\big)> 0$, while the opposite case can be handled similarly. 
The fact $\p_c F \big(k_0, U(-h)\big)<0$ yields 
\[
\p_k \CC(k_0)= \p_k \tilde \CC(k_0)= - \frac{\p_k F\big(k_0, U(-h)\big)}{\p_c F\big(k_0, U(-h)\big)} >0,  
\] 
where in the calculation of $\tilde \CC(k_0)$ we also used $F_I(k, c)=0$ for $c\le U(-h)$ and the smoothness of $F$. Hence we obtain $\CC(k)=\tilde \CC(k) < U(-h)$ for $k$ slightly smaller than $k_0$. In the following we shall focus on $k>k_0$ where $\CC_R(k) > U(-h)$. In this case, apparently $\tilde F_I(k, \CC_R (k)) =F_I(k, \CC_R (k)) \ne 0$ and thus $\CC_I(k) \ne 0$. From the Mean Value Theorem, there exists $\theta$ between $0$ and $\CC_I(k)$ such that   
\[
0= \tilde F_I \big(k, \CC(k)\big) = F_I \big(k, \CC_R(k)\big) + \CC_I(k) \p_{c_I} \tilde F_I \big(k, \CC_R(k) + i \theta\big). 
\]
The $C^{1, \alpha}$ regularity of $F$ and $\CC (k)$ 
implies  
\begin{align*}
\CC_I (k)= &-\frac {F_I \big(k, \CC_R(k)\big)}{\p_{c_I} \tilde F_I \big(k, \CC_R(k) + i \theta\big)}=-\frac {Y_I \big(k, \CC_R(k)\big) \big(U(0)- \CC_R(k)\big)^2}{\p_{c_I} \tilde F_I \big(k, \CC_R(k)\big) + O\big(|\CC_I(k)|^\alpha\big)} \\
=&-\frac {Y_I \big(k, \CC_R(k)\big) \big(U(0)- \CC_R(k)\big)^2}{\p_{c_I} F_I \big(k, \CC_R(k)\big) + O\big(|\CC_I(k)|^\alpha\big)} 
= -\frac {Y_I \big(k, \CC_R(k)\big) \big(U(0)- U(-h) +O( |k-k_0|)\big)^2}{\p_{c} F \big(k_0, U(-h)\big) + O\big(|k- k_0|^\alpha\big)}.  
\end{align*}
The proof of the lemma is complete. 
\end{proof}

While the branch $c^+(k) \in (U(0), +\infty)$ of neutral modes is global in $k\in \R$ and contained in $\big(U(0), \infty\big)$ as addressed in Corollary \ref{C:branches-1}, in the following we completes the picture of the other branch $c^-(k)$ by combining Lemmas \ref{L:continuation} -- \ref{L:bifurcation} and finish the proof of Theorem \ref{T:e-values}. \\

\noindent {\it Proof of Theorem \ref{T:e-values}(3).} 
Let 
$g_\#\ge 0$, $k_\#$, and/or $k_\#^\pm$ be the thresholds given in Lemma \ref{L:g-thresh}. 

{\it Case 1. $g> g_\#$.} The desired result follows from Lemma \ref{L:g-thresh} and Corollary \ref{C:branches-1} immediately.  

We start the rest of the proof much as in that of Corollary \ref{C:branches-1} and Proposition \ref{P:e-v-0}. Namely, let $k_0$ be given by Lemma \ref{L:ev-large-k}(3) and we only need to focus on $c^-(k)$ for $|k| \le k_0$. 
From Lemma \ref{L:ev-large-k}(2), there exists $R>0$ such that $F(k, c) \ne 0$ for all $k \in [-k_0-1, k_0+1]$ and $|c|\ge R$, which also implies $c^+(k) \in \big(U(0), R)$ for all $|k| \le k_0+1$ and $c^-(k) \in \big(-R, U(-h)\big)$ for all $|k| \in [k_0, k_0+1]$. 

{\it Case 2. $g= g_\#$.} One the one hand, for any $k_1 \in (k_\#, k_0]$, Lemmas \ref{L:ev-large-k}, \ref{L:e-v-large-k}, and \ref{L:g-thresh} imply that there exists $r_0>0$ such that 
\[
F(k, c)\ne 0, \; \forall k\in [k_1, k_0], \, c \in \p \Omega_1 \cup \CD_{r_0}, \; \text{ where } \Omega_1 = \{c \in \C \mid |c| < R, \, c\notin \overline{\CD_{r_0}}\},
\]  
where $\CD_r$ is the $r$-neighborhood of $U([-h, 0])$ (see also \eqref{E:CD_r}). 
Hence for all $k \in [k_1, k_0]$, we have Ind$\big(F(k, \cdot), \Omega_1\big) = \big(F(k_0, \cdot), \Omega_1\big) =2$, which is equal to the number of roots of $F(k, \cdot)$ in $\Omega_1$. According to Corollary \ref{C:branches-1}, $c^+(k) \in \big(U(0), R\big)$, $\forall |k|\le k_0+1$, is one of them. Therefore neither cases in Lemma \ref{L:continuation}(3) can happen to the branch $c^-(k)$ and the simple root $c^-(k) \in \big(-\infty, U(-h)\big)$ can be extended analytically for all 
$k \in [k_1, k_0]$. Therefore $c^-(k)$ can be extended to at least $(k_\#, \infty)$ which along with $c^+ (k)$ are the only roots of $F(k, \cdot)$ for $k \in (k_\#, \infty)$. On the other hand, according to Lemma \ref{L:bifurcation}, there exists a $C^{1, \alpha}$ branch $\CC(k)$ of the only roots of $F(k, c)$ for $|k-k_\#|, |c-U(-h)|\ll1$. Moreover $\CC(k) < U(-h)$ for $0< |k-k_\#|\ll1$. Therefore $c^-(k)= \CC(k)$ for $0< k- k_\# \ll1$ as $c^\pm (k)$ are the only roots of $F(k, \cdot)$ for $k> k_\#$. In particular, $c^-(k)$ is thus extended to $|k_\# -k| \ll1$ as a $C^{1, \alpha}$ function with $c^-(k) \in (-R, U(-h))$ for $0< |k_\# -k| \ll1$. 


Moreover, on the one hand, $c^-(k)$ is the only root of $F(k, \cdot)$ near $U([-h, 0])$ for $k$ near $k_\#$ and it satisfies $c^-(k_\#)=U(-h)$. On the other hand, the continuity of $c^-(k)$ implies that there exists $\ep_1, r_1>0$ such that $F(k, c)\ne 0$ for any $|k-k_\#|\le \ep_1$ and $dist\big(c, U([-h, 0]) \big) =r_1$. It implies 
\[
\text{Ind}\big(F(k, \cdot), \Omega_2\big) = \text{Ind}\big(F(k_\#+\ep_1, \cdot), \Omega_2\big) =1, \ \forall |k-k_\#|\le \ep_1,
\]
due to the root $c^+(k)$, where
\[ 
\Omega_2 =  \{c \in \C \mid |c| < R, \, dist\big(c, U([-h, 0])\big)>r_1\}.
\]
Therefore, $c^\pm(k)$ are the only root of $F(k, \cdot)$ for $k$ near $k_\#$, which are also simple. As $c^- (k) \in (-\infty, U(-h))$ is away from $c^+(k)$, 
Lemma \ref{L:continuation} implies that the branch $c^-(k)$ of simple roots  can be extended at least to $(-k_\#, +\infty)$ and remains in $\big(-\infty, U(-h)\big)$. As $F$ is even in $k$, we have $c^\pm(k)$ are also the only roots of $F(-k, \cdot)$ for $k \in (-\infty, k_\#)$. Therefore the extension $c^-(k)$ must be even on $(-k_\#, k_\#)$ 
and we obtain the whole branch $c^-(k)$ for $k \in \R$.   

{\it Case 3a. $g< g_\#$ and $U''>0$.} Following the same arguments as in case 2, we obtain that $c^-(k) = c_R^-(k) + i c_I^-(k)$ can be extended to a $C^{1, \alpha}$ function on $(k_1, +\infty)$ for some $k_1 < k_\#^+$, such that  $c^\pm(k)$ and $\overline{c^-(k)}$ are the only roots of $F(k, \cdot)$ for all $k \in (k_1, +\infty)$ and $c_I^-(k) >0$ for $k \in (k_1, k_\#^+)$. Let $(k_1, k_\#^+)$ also denote the maximal interval of the analytic extension of $c^-(k)$ as a simple root of $F(k, \cdot)$ inside $\C \setminus U([-h, 0])$. The same above index based argument (in case 1) applied to $[k, k_\#^+-\ep]$ for  any $k \in(\max \{k_1, k_\#^-\}, k_\#^+)$ and $0\ll \ep < k_\#^+-k$ also implies that $c^\pm(k)$ and $\overline{c^-(k)}$ are the only roots of $F(k, \cdot)$ for all $k \in(\max \{k_1, k_\#^-\}, k_\#^+)$. According to Lemma \ref{L:ev-large-k} we have $k_1 \ge -k_0 >-\infty$. For $k\in (k_1, k_\#^+)$, the semicircle theorem implies that $c^-(k)$ lies  
in the closed upper semi-disk with the boundary diameter $U( [-h, 0])$  
and thus $|c^-(k) - c^+(k)| > \rho_0$ where $\rho_0>0$ is given in Corollary \ref{C:branches-1}. Moreover, since $F (k, c) \ne 0$ for any $c \in U\big((-h, 0]\big)$ (Lemmas \ref{L:e-v-basic-1} and \ref{L:e-v-large-k}), we obtain from Lemma \ref{L:continuation}
\[
\lim_{k \to k_1+} c^-(k) = U(-h) \Longrightarrow F\big(k_2, U(-h)\big) =0 \Longrightarrow k_2 \in\{ k_\#^-, \, -k_\#^-, \, -k_\#^+\}. 
\]
It must hold $k_2 = k_\#^-$, otherwise $c^-(k_\#^-) \ne U(-h)$, $F\big(k_\#^-, U(-h)\big)=0$, $\p_k F\big(k_\#^-, U(-h)\big) >0$, and Lemma \ref{L:bifurcation} imply that there exists the fourth root near $U(-h)$ for $0< k - k_\#^- \ll 1$. 
This contradicts that $F(k, \cdot)$ has exactly three roots for all $k \in(\max \{k_1, k_\#^-\}, k_\#^+)$ and thus $k_2 = k_\#^-$ and $c^-(k_\#^-) =U(-h)$. 
For $0< k_\#^-- k\ll1$, Lemma \ref{L:bifurcation} yields the further extension of $c^-(k)$ back into $\big(-\infty, U(-h)\big)$. From a similar argument, we can extend this branch to $k=-k_\#^-$ with $c^-(-k_\#^-) = U(-h)$.  Finally, the whole branch $c^-(k)$ for $k \in \R$ is obtained by the evenness $c^-(-k)= c^-(-k)$. 

{\it Case 3b. $g< g_\#$ and $U''< 0$.} Following the same arguments as in case 2, we obtain that $c^-(k) = c_R^-(k) + i c_I^-(k)$ can be extended to a $C^{1, \alpha}$ function on $[k_\#^+, +\infty)$ and $c^-(k_\#^+)=U(-h)$. However, for $0< k_\#^+-k \ll1$,  Lemma \ref{L:bifurcation} implies that there does not exist any roots of $F(k, \cdot)$ near $U(-h)$ (as $\CC_I <0$ due to $U''<0$). The same index argument further yields that $c^+(k)$ is the only root for $k \in (k_\#^-, k_\#^+)$. From Lemma \ref{L:bifurcation}, we obtain another branch of roots in $(-\infty, U(-h))$ of $F(k, \cdot)$ for $k \in (-k_\#^-, k_\#^-)$ which along with the $c^+(k)$ are the only roots. The final conclusion again follows from the even symmetry as in the above cases.    
\hfill $\square$

\begin{remark}
As in \cite{Yih72} for the gravity wave, the spectral stability in the case $U''<0$ can also be obtained by directly modifying the usual  proof of the Rayleigh theorem in the  fixed boundary case. Namely, multiplying \eqref{E:Ray-H1-1} by $\bar y$, integrating on $[-h, 0]$, using the homogeneous boundary condition as in \eqref{E:Ray-2} and \eqref{E:Ray-3}, and the semicircle theorem, a contradiction occurs if an unstable mode $c$ exists.  
Our above proof provides a complete picture of the eigenvalue distribution, however.  
\end{remark}

\subsection{Singular neutral modes at inflection values} \label{SS:inflectionV}

To end this section, we discuss the spectrum near inflection values of $U$, which are the only possible singular neutral modes other than $U(-h)$ according to Lemma \ref{L:e-v-basic-1}(6).

\begin{proposition} \label{P:inflectionV} 
Assume $U\in C^5$, $x_{20} \in [-h, 0)$, and $U''(x_{20})=0$, then the following hold for $c_0 = U(x_{20})$. 
\begin{enumerate}
\item For any $\alpha \in (0, \frac 12)$, there exist $C>0$ depending only on $U$, $g$, and $\alpha$, such that, with   
\[
k_* =C  \max\{1, (U(0)-c_0)^{-2}\},  \quad \sigma_0= (U(0)-c_0)^2/(2k_*), 
\]
for any $\sigma \in (0, \sigma_0 )$, there exists a unique  $k_0 \ge k_*$ such that $F(k_0, c_0)=0$. Moreover it  satisfies 
\[
|k_0 - (U(0)- c_0)^{2} / \sigma| \le C(U(0)- c_0)^{-2}, \quad |\p_k F(k_0, c_0) + (U(0)- c_0)^{2}| \le C (U(0)- c_0)^{-2\alpha} \sigma^\alpha.  
\]
\item In addition, suppose $x_{20} \ne -h$ and  
\[
\BF(k_0, c_0)=0, \; k_0>0, \; \p_k F(k_0, c_0)\ne 0,  \; U'''(x_{20}) \ne 0, 
\]
then there exist $\tilde C>0$, $\delta >0$, and a $C^1$ function $c(k)$ defined for 
\[
0\le |k-k_0| \le \delta, \quad (k-k_0) U'''(x_{20}) \p_k F(k_0, c_0) >0,
\] 
such that $c(k_0)=c_0$, $c_I(k)>0$ for the above $k\ne k_0$, and 
\[
F(k, c)=0, \; |k-k_0| \le \delta\text{ and } |c-c_0|\le \tilde C\delta \; \text{ iff } \; c \in \{c(k), \overline{c(k)}\}. 
\]
\end{enumerate}
\end{proposition}

In the above statement (2), note that $\BF(k_0, c_0)=0$ and Lemma \ref{L:e-v-basic-1}(4) imply $y_- (k_0, c_0, 0)\ne 0$ and thus  $Y(k, c_0)$ is well-defined which is actually real due to $U''(x_{20}) =0$ and Lemma \ref{L:Y-I}. Therefore it makes sense to talk about the sign of $\p_k F(k_0, c_0)$. Statement (1) also implies that assumptions of statement (2) may be satisfied at inflection values of $U$  with $|k|\gg1$ if $\sigma$ is small.  

\begin{proof} 
From Lemma \ref{L:y-pm} and Remark \ref{R:y-pm}, there exists $C_0>0$ such that 
\[
k y_-(k, c, x_2)  \ge (1/2) \sinh \mu^{-1} (x_2 +h) \implies y_-(k, c, 0) \ne 0, \quad \forall |k|\ge C_0, \, c \in \C,
\] 
and thus $F(k, c)$ and $Y(k, c)$ are defined for all $|k|\ge C_0$. According to \eqref{E:Im-F}, $F_I(k, c_0) =0$ for all $k \in \R$ and thus $F (k, c_0) \in \R$. Lemmas \ref{L:Y-Cauchy} and \ref{L:Y-I} imply, for  $|k|\ge C_0$ and $c\in U([-h, 0))$, 
\[
|Y_I(k, c)| \le C |U''(x_2^c)| e^{2\mu^{-1} x_2^c} \implies \big|Y(k, c_0) - |k|\big| \le C \mu.
\]
Therefore, for $|k|\ge C_0$, it holds 
\[
\big| |k|^{-1} F(k, c_0) - (U(0)-c_0)^2 + \sigma |k| \big| \le C \mu. 
\]
Let  
\[
k_* = \max\{C_0, 3C (U(0)-c_0)^{-2}\} \implies C \langle k_* \rangle^{-1} \le (U(0)-c_0)^2/3.
\]
From the Intermediate Value Theorem, for every $0< \sigma\le \sigma_0$, there exists a root  $k_0\in [k_*, +\infty)$ of $F(\cdot, c_0)$ close to $(U(0)- c_0)^{2} / \sigma$. 

To estimate $\p_k F(k_0, c_0)$ and obtain the uniqueness of $k_0$, we analyze $\p_k Y(k_0, c_0)$ using the same standard method used in the proof of Lemma \ref{L:F-signs}. Let 
\[
y(k, x_2) = \frac {y_{0-} (k, c_0, x_2)}{y_{0-} (k, c_0, 0)} \implies - y'' +\Big(k^2 + \frac {U''}{U - c_0}\Big) y =0, \;\; y(-h)=0, \; y(0)=1, \; Y(k, c_0) = y'(0),  
\]
where $\frac {U''}{U - c_0} \in C^3([-h, 0])$. Differentiating the above equation with respect to $k$ yields 
\begin{align*} 
&- \p_k y'' +\Big(k^2 + \frac {U''}{U - c_0}\Big) \p_k y = -2k y, \;\; \p_k y(-h)= \p_k y(0)=0, \; \p_k Y(k, c_0) = \p_k y'(0), \\
\implies & \p_k Y(k, c_0)= \int_{-h}^0 (\p_k y' y- \p_k y y')' dx_2  = 2k \int_{-h}^0 y(x_2)^2 dx_2. 
\end{align*} 
From Lemma \ref{L:y-pm}, we can estimate, for any $\alpha \in (0, \frac 12)$ and $|k|> k_*$, 
\begin{align*}
&\Big| \p_k Y(k, c_0) -  2k \int_{-h}^0 \Big(\frac {\sinh \mu^{-1} (x_2+h)}{ \sinh \mu^{-1} h}\Big)^2 dx_2 \Big|\le C \mu^{\alpha-1} \int_{-h}^0 \Big(\frac {\sinh \mu^{-1} (x_2+h)}{ \sinh \mu^{-1} h}\Big)^2 dx_2 \\
\implies & |\p_k Y(k, c_0) - sgn(k) | \le  C \mu^\alpha. 
\end{align*} 
Therefore we obtain  
\[
\p_k F(k, c_0) = (U(0)-c_0)^2 \p_k Y(k, c_0) - 2\sigma k = (U(0)-c_0)^2 sgn(k) - 2\sigma k + O(|k|^{-\alpha}), 
\]
which implies 
\[
\p_k F(k_0, c_0)=- (U(0)-c_0)^2 + O(k_0^{-\alpha}) \; \text{ if } \; k_0 \in (k_*, \infty) \text{ and } F(k_0, c_0)=0.
\]  
The desired estimate on $\p_k F(k_0, c_0)$ follows immediately, whose always negative sign also implies the uniqueness of such $k_0 \in (k_*, \infty)$. 

Under the assumption in statement (2) of the proposition, Lemma \ref{L:e-v-basic-1}(4) implies $y_- (k_0, c_0, 0)\ne 0$ and thus $F(k, c)$ is $C^1$ in $(k, c)$ near $(k_0, c_0)$ with $c_I\ge 0$. Much as in the proof of Lemma \ref{L:bifurcation}, statement (2) can be proved by applying the Implicit Function Theorem to $\tilde F(k, c)$, an extension of $F(k, c)$ which is $C^1$ in $(k, c)$ in $\R\times \C$ near $(k_0, c_0)$. The Jacobi matrix of $\tilde F$ is 
\[
D_c \tilde F (k_0, c_0) = \begin{pmatrix} \p_{c_R} \tilde F_R & \p_{c_I} \tilde F_R \\ \p_{c_R} \tilde F_I & \p_{c_I} \tilde F_I \end{pmatrix}\Big|_{(k_0, c_0)} =  
\begin{pmatrix} \p_{c_R} F_R & - \p_{c_R} F_I \\ \p_{c_R} F_I &  \p_{c_R} F_R \end{pmatrix}\Big|_{(k_0, c_0)},
\]
where we also used the Cauchy-Riemann equation. According to Lemma \ref{L:Y-I}, $Y(k, c_0) \in \R$ and  
\[
\p_{c_R} F_I(k_0, c_0) = (U(0)-c_0)^2 \p_{c_R} Y_I(k_0, c_0)= (U(0)-c_0)^2\frac {\pi U'''(x_{20}) y_{0-} (k_0, c_0, x_{20})^2} {U'(x_{20})^2 y_{0-} (k_0, c_0, 0)^2} \ne 0,
\] 
and has the same sign as $U'''(x_{20})$. Therefore $D_c \tilde F (k_0, c_0)$ is invertible and thus there exist $\delta >0$ and a $C^1$ function $c(k) = c_R (k) + i c_I (k)$ defined for all $|k-k_0|\le \delta$ such that $\tilde F(k, c)=0$ for $(k, c) \in \R \times \C$ iff $c = c(k)$. Consequently $F(k, c)=0$ for $(k, c)$ near $(k_0, c_0)$ iff $c\in \{c(k), \overline{c(k)}\}$ and $c_I(k) \ge 0$. Identifying complex numbers with 2-d column vectors, since 
\begin{align*}
\p_k c (k_0) = - (D_c \tilde F (k_0, c_0))^{-1} \p_k \tilde F (k_0, c_0)= - \p_k F(k_0, c_0) / \p_c F(k_0, c_0)   
\end{align*}
implies $c_I(k) (k-k_0) \p_k F(k_0, c_0) U'''(x_{20}) >0$ for $k$ near $k_0$, statement (2) follows readily. 
\end{proof}

\begin{remark} \label{R:inflectionV} 
In part (1) of the proposition, one may also seek $k_0$ satisfying $\BF(k_0, c_0)=0$ using the Intermediate Value Theorem instead. It is easy to see $\BF(k, c_0)\in \R$ approaches $-\infty$ as $k \to \infty$. Therefore such $k_0$ exists if $\sup_{k \ge 0} \BF(k, c_0) >0$ and only if $\sup_{k \ge 0} \BF(k, c_0) \ge 0$, which may not the case if $g$ and $\sigma$ are sufficiently large. This is different from the gravity waves (i.e. $\sigma=0$), see \cite{Yih72, HL08, HL13}. It is also worth pointing out that the smoothness of $F$ for $c_I\ge 0$ based on Section \ref{S:Ray-Homo} made the analysis using the Implicit Function Theorem in part (2) easier, compared with, e.g. \cite{HL08}. 
\end{remark} 


\section{Boundary value problems of the non-homogeneous Rayleigh equation} \label{S:Ray-BC}

In this section, using the fundamental solutions $y_\pm (k, c, x_2)$ to the homogeneous Rayleigh equation \eqref{E:Ray-H1-1}, we study the boundary value problem of the non-homogeneous Rayleigh equation 
\begin{subequations} \label{E:Ray-BVP} 
\be \label{E:Ray-NH-2}
- y'' + \big(k^2 + \frac {U''}{U-c}\big) y = \frac {\psi(c, x_2)}{U-c}, \quad x_2 \in (-h, 0); 
\ee 
\be \label{E:Ray-BC-1}
y(-h)= \zeta_- (c), \quad \big(U(0)-c\big)^2 y'(0) - \big(U'(0)(U(0)-c) + g + \sigma k^2 \big) y(0) = \zeta_+(c),
\ee
\end{subequations}
where the boundary conditions are from the linearized water wave system  \eqref{E:Ray}.

Using the two fundamental solutions $y_\pm$ to the homogeneous equation with zero boundary values, for $c \in \C \setminus U([-h, 0])$
it is standard to compute the solution to \eqref{E:Ray-BVP} in the form
\be \label{E:y_B}
y_B (k, c, x_2) = \frac {\zeta_+(c)}{\BF(k, c)} y_- (k, c, x_2) + \frac {\zeta_-(c)}{y_+(k, c, -h)} y_+(k, c, x_2) + y_{nh} (k, c, x_2), 
\ee
where $y_{nh}$ is the solution to \eqref{E:Ray-NH-2} with zero boundary values in \eqref{E:Ray-BC-1} given 
by 
\be \label{E:y_nh-0} 
y_{nh} (k, c, x_2) = \frac {y_+ (k, c, x_2)}{y_+(k, c, -h)} \int_{-h}^{x_2} \frac {(y_-  \psi) (k, c, x_2')}{U(x_2') -c}  dx_2' +  \frac {y_- (k, c, x_2)}{y_+(k, c, -h)} \int_{x_2}^0 \frac {(y_+  \psi) (k, c, x_2')}{U(x_2') -c}   dx_2'.  
\ee  
Its derivative in $x_2$ is given by 
\be \label{E:y_nh-1} 
y_{nh}' (k, c, x_2) = \frac {y_+' (k, c, x_2)}{y_+(k, c, -h)} \int_{-h}^{x_2} \frac {(y_-  \psi) (k, c, x_2')}{U(x_2') -c}   dx_2' +  \frac {y_-' (k, c, x_2)}{y_+(k, c, -h)} \int_{x_2}^0 \frac {(y_+  \psi) (k, c, x_2')}{U(x_2') -c}   dx_2'.  
\ee  
Here the unique solvability condition of \eqref{E:Ray-BVP} is $\BF(k, c) \ne 0$, where $\BF$ is defined in \eqref{E:BF}, as the Wronskian of the fundamental solutions $y_\pm$,
which is a constant in $x_2$, is given by 
\be \label{E:Wronskian}
y_+(k, c, -h) = (g+\sigma k^2)^{-1} \BF(k, c) =(y_+ y_-' - y_+' y_-)(k, c, x_2). 
\ee 

Throughout this section, we consider 
\[
c = c_R + i c_I, \quad c_R \in \CI=U( [-h-\rho_0, \rho_0]), \quad |c_I| \le \rho_0,  
\]
where $\rho_0\in [0,  h_0]$. 
By choosing $\rho_0$ smaller, we also have that, for some $C>0$ depending only on $|U|_{C^1}$ and $|(U')^{-1}|_{C^0}$, 
\be \label{E:h-0}
\RP \,  \big(g + \sigma k^2 + U'(0) (U(0)-c) \big) \ge (1+k^2)/C, \quad \forall k \in \R, \; c \in \CI + i[-\rho_0, \rho_0].
\ee
This and boundary condition \eqref{E:Ray-BC-1} imply
\be \label{E:y(0)}
|y(0)| \le C\mu^2 (|U(0)-c|^2 |y'(0)| + |\zeta_+|), 
\ee  
which will be used repeatedly to control $y(0)$ in terms of $y'(0)$. 

Throughout this section, we assume that, there exists $\rho_0>0$ such that  
\be \label{E:F_0}
F_0 = \inf \{ (1+k^2)^{-\frac 12} e^{-\frac h\mu} |\BF(k, c)| \mid c_R \in \CI= U([U(-h)-\rho_0, U(0)+\rho_0]), \, |c_I| \in [-\rho_0, \rho_0]\}  > 0. 
\ee 
In this subsection, mostly we shall not vary $k \in \R$, but carefully track the dependence of the estimates on $k$, or equivalently $\mu = (1+k^2)^{-\frac 12}$. From Lemma \ref{L:y-pm}, it is easy to compute that, for any $r_1 \in [1, \infty]$, $r_2 \in [1, \infty)$, and $|c_I| \le \rho_0$, 
\[
\mu^{-(1+ \frac 1{r_1})} |y_\pm|_{L_{x_2}^{r_1} L_{c_R}^\infty } + \mu^{-\frac 1{r_1}}  |y_\pm'|_{L_{x_2}^{r_1} L_{c_R}^{r_2} } + \mu^{-\frac 1{r_2}}  |y_\pm'|_{L_{c_R}^\infty L_{x_2}^{r_2} } +  |y_+' (-h)|_{L_{c_R}^{r_2} } +  |y_-' (0)|_{L_{c_R}^{r_2} } \le C  e^{\mu^{-1} h}, 
\] 
where $x_2 \in [-h, 0]$ and $c_R \in \CI$. This inequality will be used repeatedly in the rest of the paper. 

Solutions to this system are rather smooth away from $c \in \{U(x_2), \, U(0), \, U(-h)\}$   
and their singular behaviors near this set could be analyzed rather detailedly following the approach in Section \ref{S:Ray-Homo}, based on \eqref{E:tw-1} and \eqref{E:Ray-H0-2} and the estimates on $\tilde B$ and $B$. However, for the purpose of this paper, it is sufficient just to obtain certain bounds of the solutions based on the properties of the homogeneous solutions $y_\pm$, which is carried out in this section. 

As a preparation, in Subsection \ref{SS:Ray-NH-HBC} we shall first consider \eqref{E:Ray-BVP} with zero boundary conditions $\zeta_\pm =0$ in \eqref{E:Ray-BC-1}. Subsequently in Subsection \ref{SS:Ray-NH-CGW}, we study the non-homogenous Rayleigh system \eqref{E:Ray-BVP} with $\zeta_\pm$ linear in $c$, particularly focusing on the derivatives of the solutions on $c\in \CI + i [-\rho_0, \rho_0]$.  We sometimes skip writing parameters $k$ and $c$ explicitly.

\subsection{Non-homogeneous Rayleigh system \eqref{E:Ray-BVP} with zero boundary conditions $\zeta_\pm =0$} \label{SS:Ray-NH-HBC}

The formulas \eqref{E:y_nh-0} and \eqref{E:y_nh-1} of $y_{nh}(k, c, x_2)$ and $y_{nh}'(k, c, x_2)$ 
are actually consistent with \eqref{E:tw-1} for $x_2$ near $x_2^c$. In fact, \eqref{E:tw-1} implies that $\begin{pmatrix} 1 & 0 \\ \Gamma & 1 \end{pmatrix} \tilde B$ is a fundamental matrix of \eqref{E:Ray-H1-1} and hence $\tilde B$ can be rewritten in terms of $y_\pm$ and $\Gamma$. A straight forward calculation using \eqref{E:tPhi-0} and \eqref{E:tw-1} also yields \eqref{E:y_nh-0}. This solution also satisfy 
\[
\overline{y_{nh} (k, \bar c, x_2)} = y_{nh} (k, c, x_2) = y_{nh} (-k, c, x_2), 
\]
so we mainly focus on $c_I \ge 0$. Assume $\psi(c_R + ic_I, x_2) \to \psi_0(c_R, x_2)$ as $c_I \to 0+$. 
Due to the singularity of the non-homogeneous term at $x_2 = x_2^c$ (as defined in \eqref{E:x2c-1} by $U(x_2^c) =c_R$) as $c_I \to 0+$, the limits of $y_{nh}$ and $y_{nh}'$ involve $P.V.$ of integrals and delta masses 
\be \label{E:y_nh0-2-0} \begin{split} 
y_{nh0} (x_2)  &=  P.V. \int_{-h}^{0} \psi_0(x_2')\frac {  y_{0+} (x_2) y_{0-}(x_2') \chi_{\{x_2' < x_2\}} + y_{0-} (x_2) y_{0+} (x_2') \chi_{\{x_2' >x_2\}} }{y_{0+}(-h)(U(x_2') -c_R) } dx_2'\\
&+ \frac {i\pi\psi_0(x_2^c)} {U'( x_2^c)}\Big( \frac {y_{0+} (x_2) y_{0-} (x_2^c)}{y_{0+}( -h)} \chi_{\{U(x_2) > c_R> U(-h)\}} +\frac {y_{0-} (x_2) y_{0+} (x_2^c) }{y_{0+}( -h)} \chi_{\{U(0)> c_R >U(x_2) \}}    \Big),
\end{split} \ee
\be \label{E:y_nh0-2-1} \begin{split} 
y_{nh0}' (x_2)  &= P.V. \int_{-h}^{0} \psi_0(x_2')\frac {  y_{0+}' (x_2) y_{0-}(x_2') \chi_{\{x_2' < x_2\}} + y_{0-}' (x_2) y_{0+} (x_2') \chi_{\{x_2' >x_2\}} }{y_{0+}(-h)(U(x_2') -c_R) } dx_2'\\
&+ \frac {i\pi\psi_0(x_2^c)} {U'( x_2^c)}\Big( \frac {y_{0+}' (x_2) y_{0-} (x_2^c)}{y_{0+}( -h)} \chi_{\{U(x_2) > c_R> U(-h)\}} + \frac {y_{0-}' (x_2) y_{0+} (x_2^c) }{y_{0+}( -h)} \chi_{\{U(0)> c_R >U(x_2) \}}    \Big),
\end{split} \ee
where $\chi$ is the characteristic function and we skipped the dependence on $c_R$ of $\psi_0$, $y_{0\pm}$, and $y_{nh0}$. Naturally, in the above the $P.V.$ is taken only when there are singularities in the integral. 


We consider a priori and convergence estimates of $y_{nh}$ as $c_I \to 0+$ in the following two cases of $\psi(c, x_2)$, motivated by the non-homogenous Rayleigh system \eqref{E:Ray}  
and its differentiation in $c$. 

{\bf $\bullet$ Case 1: $\psi' (c, \cdot) \in L_{x_2}^r, \; r \in (1, \infty)$.}
While this case occurs in the linearized capillary gravity wave \eqref{E:Ray} when some regularity is assumed on the initial vorticity, it is also a crucial part of the analysis when \eqref{E:Ray} is differentiated in $c$. 
 
\begin{lemma} \label{L:y_nh-esti-2}
Assume \eqref{E:F_0}. For any $\ep>0$\footnote{Like the generic upper bound $C>0$, the small constant $\ep>0$ in this and the next section may change from line to line.}, there exists $C>0$ depending only on $r$, $\ep$, $F_0$, $\rho_0$, $|U'|_{C^2}$  and $|(U')^{-1}|_{C^0}$, such that the following hold.
\begin{enumerate} 
\item For any $k \in \R$, $x_2 \in [-h, 0]$, $c_I \in (0, \rho_0]$, and $c_R \in \CI
$ it holds
\begin{align*}
|y_{nh}(k, c, x_2)| \le C & \mu^{1-\frac 1r -\ep} 
(\mu |\psi'|_{L_{x_2}^r} +|\psi|_{L_{x_2}^r} ),  
\end{align*}
\begin{align*}
|y_{nh}'(k, c, x_2)| \le C \mu^{-\frac 1r -\ep} &
\big( 1+  \big|\log |U(x_2)-c|\big|\big) (\mu |\psi'|_{L_{x_2}^r} +|\psi|_{L_{x_2}^r} ).  
\end{align*}
\item Assume $\psi(\cdot + ic_I, \cdot) \to \psi(\cdot, \cdot)$ in $L_{c_R}^{r_1} W_{x_2}^{1, r}$ as $c_I \to 0+$ with $r_1 \in (1, \infty)$ and $r \in (1, \infty)$, then 
\begin{enumerate} 
\item $y_{nh} \to y_{nh0}$ in $L_{c_R}^{q_1} L_{x_2}^\infty$ for any $q_1\in [1, r_1)$ and $y_{nh}' \to y_{nh0}'$ in $L_{c_R}^{q_1} L_{x_2}^{q_2}$ for any $q_1\in [1, r_1)$ and $q_2 \in [1, \infty)$;
\item at $\tilde x_2 =-h$ and $\tilde x_2=0$, $y_{nh}(\cdot + ic_I, \tilde x_2) \to y_{nh0} (\cdot, \tilde x_2)$  and $y_{nh}' (\cdot + ic_I, \tilde x_2) \to y_{nh0}' (\cdot, \tilde x_2)$ in $L_{c_R}^{q_1}$ for any $q_1 \in [1, r_1)$. Moreover, and for any $\ep>0$, for any $k\in \R$, $c_I \in [0,1]$,
\[
|y_{nh}' (c, \tilde x_2)| \le C
\big( \mu^{-\frac 1r -\ep} (\mu |\psi'|_{L_{x_2}^r} +|\psi|_{L_{x_2}^r} ) + \big( 1+ \big|\log |U(\tilde x_2)-c|\big|\big) |\psi(\tilde x_2)|\big).
\]
\end{enumerate}
\end{enumerate} 
\end{lemma} 

Even though the above formulas of $y_{nh0}$ involve some subtlety at $x_2 =x_2^c$, the regularity of $y_{nh0}'$ in $x_2$ implies that $y_{nh0}$ is H\"older continuous. In fact, the continuity of $y_{nh0}$ at $x_2 = x_2^c$ can also be seen directly by using the rather precise local form of $y_{0\pm}$ near $x_2^c$ given in Lemma \ref{L:B}. Moreover, while the convergence is given in the integral norms, one could attempt to obtain more detailed convergence estimates near $x_2^c$ using the tools given in Lemma \ref{L:Gamma-B} and Proposition \ref{P:converg}.  

\begin{proof} 
Since $c_I>0$, no singularity is involved in \eqref{E:y_nh-0} and \eqref{E:y_nh-1}, one can compute via integration by parts 
\[
\int_{-h}^{x_2} \frac {y_-  \psi}{U-c} dx_2' = \int_{-h}^{x_2} \frac {y_-  \psi}{U'} \big(\log (U- c) \big)' dx_2'  = \big(\frac {y_-  \psi}{U'} \log (U- c)\big) (x_2) - \int_{-h}^{x_2} \big( \frac {y_-  \psi}{U'} \big)' \log (U- c) dx_2'.
\]
The other integral can be handled similarly,  
\[
\int_{x_2}^0 \frac {y_+  \psi}{U-c} dx_2' = \big(\frac {y_+  \psi}{U'} \log (U- c)\big)\Big|_{x_2}^0 - \int_{x_2}^0 \big( \frac {y_+  \psi}{U'} \big)' \log (U- c) dx_2'.
\]
Observing that the boundary terms at $x_2$ are canceled and we have 
\be \label{E:y_nh-temp-1} \begin{split}
y_{nh} (x_2) = &- \frac {y_+(x_2)}{y_+ (-h)} \int_{-h}^{x_2} \big( \frac {y_-  \psi}{U'} \big)' \log (U- c) dx_2' - \frac {y_-(x_2)}{y_+ (-h)} \int_{x_2}^0 \big( \frac {y_+  \psi}{U'} \big)' \log (U- c) dx_2' \\
& +\frac {y_-(x_2)}{y_+ (-h)}\big(\frac {y_+  \psi}{U'} \log (U- c)\big)(0). 
\end{split}\ee
The above two integrals can be estimated similarly and we shall focus on the first one only. Lemma \ref{L:y-pm} implies 
\begin{align*}
\big|\big( \frac {y_-  \psi}{U'} \big)' \log (U- c) \big| = &|U'|^{-2} \big| ( y_-'\psi U' + y_- \psi' U' - y_-\psi U'')  \log (U-c) \big|  \\
\le & C \cosh( \mu^{-1} (x_2+ h)) \Big(\mu |\psi'| + \big( 1 + \mu \big|\log |U-c|\big|\big) |\psi| \Big)\big( 1+ \big|\log |U-c|\big|\big). 
\end{align*}
Using the H\"older inequality we obtain  
\begin{align*}
& \Big| \int_{-h}^{x_2} \big( \frac {y_-  \psi}{U'} \big)' \log (U- c) dx_2' \Big| \\
\le & C (\mu |\psi'|_{L_{x_2}^r} + |\psi|_{L_{x_2}^r} ) \big|\cosh (\mu^{-1} (x_2'+h)) \big( 1+ \big|\log |U(x_2')-c|\big|^2\big) \big|_{L_{x_2'}^{\frac r{r-1}} ([-h, x_2])} \\
\le & C \mu^{1-\frac 1r -\ep}  (|\psi|_{L_{x_2}^{r}} +  \mu |\psi'|_{L_{x_2}^{r}})\cosh \mu^{-1} (x_2+h).  
\end{align*}
From the initial condition \eqref{E:y-pm} (in particular $y_+(0) = O(\mu^2 |c-U(0)|^2)$) and \eqref{E:Wronskian}, the remaining boundary term can be estimated as 
\[
\Big| \frac {y_-(x_2)}{y_+ (-h)}\big(\frac {y_+  \psi}{U'} \log (U- c)\big)(0)\Big| \le \frac {C\mu}{|\BF(k, c)|}  |\psi(0)| \sinh \mu^{-1}(x_2+h).  
\]
The desired estimate on $y_{nh}$ follows from \eqref{E:F_0}, \eqref{E:Wronskian}, Lemma \ref{L:y-pm}, the above inequalities, and the standard Sobolev inequality
\be \label{E:y_nh-temp-1.5}
|\psi|_{L_{x_2}^\infty} \le C (\mu^{1-\frac 1r} |\psi'|_{L_{x_2}^r} + \mu^{-\frac 1r} |\psi|_{L_{x_2}^r}). 
\ee 
The estimate of $y_{nh}'$ can be obtained much as in the above. Integrating by parts and using \eqref{E:Wronskian} to handle the boundary terms at $x_2$, we have  
\be \label{E:y_nh-temp-2} \begin{split}
y_{nh}' (x_2) = &  - \frac {y_+'(x_2)}{y_+ (-h)} \int_{-h}^{x_2} \big( \frac {y_-  \psi}{U'} \big)' \log (U- c) dx_2' - \frac {y_-' (x_2)}{y_+ (-h)} \int_{x_2}^0 \big( \frac {y_+  \psi}{U'} \big)' \log (U- c) dx_2' \\
&- \big(\frac {  \psi}{U'} \log (U- c)\big) (x_2) + \frac {y_-'(x_2)}{y_+ (-h)}\big(\frac {y_+  \psi}{U'} \log (U- c)\big)(0).
\end{split} \ee
The desired estimate on $y_{nh}'$ follows from \eqref{E:y_nh-1}, \eqref{E:y_nh-temp-1.5}, the above estimate on the integrals, and Lemma \ref{L:y-pm}. 

To consider the convergence of $y_{nh}$, we first note that, for $c_I>0$, the imaginary part of $\log(U(x_2) -c)$ belongs to $(-\pi, 0)$ and as $c_I \to 0+$, 
\be \label{E:y_nh-temp-2.5}
\log(U(x_2) -c) \to \log |U(x_2) -c_R| + \tfrac {i\pi}2 \big(sgn(U(x_2) -c_R)-1\big) \; \text{ in } \; L_{c_R}^\infty L_{x_2}^q, \;\; \forall q\in [1, \infty).
\ee
Using expression \eqref{E:y_nh-temp-1}, the estimates thereafter, bounds on $y_\pm$ in Lemmas \ref{L:y-pm}, and the convergence of $y_\pm$ to $y_{0\pm}$ as $c_I \to 0$ in Lemma \ref{L:y0}, it is straight forward to obtain 
\begin{align*}
y_{nh} (x_2) \to & - \frac {y_{0+} (x_2)}{y_{0+} (-h)} \int_{-h}^{x_2} \big( \frac {y_{0-}  \psi_0}{U'} \big)' \log |U - c_R| dx_2' - \frac {y_{0-}(x_2)}{y_{0+} (-h)} \int_{x_2}^0 \big( \frac {y_{0+}  \psi_0}{U'} \big)' \log |U- c_R| dx_2'\\
& + \frac {i\pi\psi_0(x_2^c)} {U'( x_2^c)}\Big( \frac {y_{0+} (x_2) y_{0-} (x_2^c)}{y_{0+}( -h)} \chi_{\{U(x_2) > c_R> U(-h)\}} + \frac {y_{0-} (x_2) y_{0+} (x_2^c) }{y_{0+}( -h)} \chi_{\{U(0)> c_R >U(x_2) \}}    \Big)\\
&
+\frac {y_{0-}(x_2)}{y_{0+} (-h)}\big(\frac {y_{0+}  \psi_0}{U'} \log |U- c_R|\big)(0),
\end{align*}
in $L_{c_R}^{q_1} L_{x_2}^\infty$ for any $q_1 \in [1, r_1)$, where,  for $c_R > U(0)$, two other terms involving $sgn(U-c_R)$ (one from upper limit term from the second integral and the other from the boundary term in \eqref{E:y_nh-temp-1}) cancelled each other. Here the loss of the integrability in $c_R$ in the convergence is due to the last logarithmic term. Since $(\log |U-c_R|)'= P.V. \tfrac {U'}{U-c_R}$ in the distribution sense, the above limit is equal to $y_{nh0}$ after integration by parts. The convergence of $y_{nh}'$ is obtained using  \eqref{E:y_nh-temp-2} along with \eqref{E:Wronskian} in a similar fashion   
\begin{align*}
y_{nh}' (x_2&)  \to  - \frac {y_{0+}' (x_2)}{y_{0+} (-h)} \int_{-h}^{x_2} \big( \frac {y_{0-}  \psi_0}{U'} \big)' \log |U - c_R| dx_2' - \frac {y_{0-}'(x_2)}{y_{0+} (-h)} \int_{x_2}^0 \big( \frac {y_{0+}  \psi_0}{U'} \big)' \log |U- c_R| dx_2'\\
& + i\pi\Big( \frac {y_{0+}' (x_2) y_{0-} (x_2^c)\psi_0(x_2^c)}{y_{0+}( -h)U'( x_2^c)} \chi_{\{U(x_2) > c_R> U(-h)\}} -  \frac {  \psi_0 (x_2)}{U'(x_2)} \chi_{\{c_R >U(x_2) \}} \\
&+ \frac {y_{0-}' (x_2) y_{0+} (x_2^c)\psi_0(x_2^c) }{y_{0+}( -h) U'(x_2^c)}  \chi_{\{U(0)> c_R >U(x_2) \}}   \Big)-\big(\big( \frac {  \psi_0}{U'} \log |U- c_R|\big) (x_2) - i\pi \frac {  \psi_0 (x_2)}{U'(x_2)} \chi_{\{c_R > U(x_2) \}}\big)\\
&+ \frac {y_{0-}'(x_2)}{y_{0+} (-h)}\big(\frac {y_{0+}  \psi_0}{U'} \log |U- c_R|\big)(0) 
\end{align*}
where again two other terms involving $sgn(U-c_R)$ cancelled each other for $c_R > U(0)$. Here the convergence in the slightly weaker norm $L_{c_R}^{q_1} L_{x_2}^{q_2}$, for any $q_1\in [1, r_1)$ and $q_2 \in [1, \infty)$ is due to the logarithmic singularity both explicitly  outside the integrals and in $y_\pm'$ (see also Lemma \ref{L:y0}). The limit can be simplified to
\begin{align*}
& - \frac {y_{0+}' (x_2)}{y_{0+} (-h)} \int_{-h}^{x_2} \big( \frac {y_{0-}  \psi_0}{U'} \big)' \log |U - c_R| dx_2' - \frac {y_{0-}'(x_2)}{y_{0+} (-h)} \int_{x_2}^0 \big( \frac {y_{0+}  \psi_0}{U'} \big)' \log |U- c_R| dx_2'\\
& + \frac {i\pi\psi_0(x_2^c)} {U'( x_2^c)}\Big( \frac {y_{0+}' (x_2) y_{0-} (x_2^c)}{y_{0+}( -h)} \chi_{\{U(x_2) > c_R> U(-h)\}} + \frac {y_{0-}' (x_2) y_{0+} (x_2^c) }{y_{0+}( -h)} \chi_{\{U(0)> c_R >U(x_2) \}}  \Big)\\
&-\big( \frac {  \psi_0}{U'} \log |U- c_R|\big) (x_2) + \frac {y_{0-}'(x_2)}{y_{0+} (-h)}\big(\frac {y_{0+}  \psi_0}{U'} \log |U- c_R|\big)(0),
\end{align*}
which is equal to $y_{nh0}'$ after an integration by parts.  

At the end point $x_2=-h, 0$, $y_{nh}(\tilde x_2)$ and $y_{nh}'(\tilde x_2)$ have only one integrals 
and, unlike for general $x_2\in (-h, 0)$, the terms $y_+(0)$, $y_+'(0)$ and $y_-'(-h)$ outside the integrals are prescribed in \eqref{E:y-pm} without any singularity. Hence the same above argument yields slightly better estimates and convergence. One may make the following  computations using \eqref{E:y_nh-0} and \eqref{E:y_nh-1}, 
\[
y_{nh}' (0) = \frac {y_+'(0)}{y_+ (-h)} \int_{-h}^{0} \frac {y_-  \psi}{U-c}  dx_2'  = - \frac {y_+'(0)}{y_+ (-h)} \int_{-h}^{0} \big( \frac {y_-  \psi}{U'} \big)' \log (U- c) dx_2' + \frac {(y_+' y_-\psi)(0)}{U'(0)y_+ (-h)}  \log (U(0)- c),
\]
\begin{align*}
y_{nh}' (-h) = \frac {1}{y_+ (-h)} \int_{-h}^{0} \frac {y_+ \psi}{U-c}  dx_2'  = & - \frac 1{y_+ (-h)} \int_{-h}^{0} \big( \frac {y_+  \psi}{U'} \big)' \log (U- c) dx_2' \\
& + \frac 1{y_+ (-h)} \Big(\frac {y_+\psi}{U'}  \log (U- c)\Big)\Big|_{-h}^0.
\end{align*}
The desired inequalities follow from \eqref{E:y-pm} and the above estimates, 
which  completes the proof of the lemma. 
\end{proof}

Assuming $\psi \in L_{c_R}^2 H_{x_2}^1$, 
in the following we estimate $y_{nh}$ and $y_{nh}'$ as well as their derivatives in $x_2$ in $L_{c_R, x_2}^2$, in particular their dependence on $k$, by an energy estimate approach.  

\begin{lemma} \label{L:y_nh0-2-esti}
Assume \eqref{E:F_0}. For any $\ep \in (0, 1)$, there exists $C>0$ depending only on $\ep$, $F_0$, $\rho_0$, $|U'|_{C^2}$, and $|(U')^{-1}|_{C^0}$, such that for any $c_I\ge 0$ and $k \in \R$, it holds 
\be \label{E:y_nh0-esti-2}
|y_{nh}'|_{L_{c_R, x_2}^2}^2 + \mu^{-2} |y_{nh}|_{L_{c_R, x_2}^2}^2 \le C (|\psi|_{L_{c_R, x_2}^2}^2 + \mu^{2-\ep} |\psi'|_{L_{c_R, x_2}^2}^2 ),  
\ee 
where the norms are taken for $c_R \in \CI$ and $x_2 \in [-h, 0]$. 
\end{lemma}

\begin{proof}
We first assume $c_I>0$ and drop the subscript $\cdot_{nh}$ for notation simplification. Multiplying the Rayleigh equation \eqref{E:Ray-NH-2} by $\bar y$ and integrating in both $c_R$ and $x_2$, we have 
\begin{align*}
&\quad  \int_\CI \int_{-h}^0 |y'|^2 + k^2 |y|^2 dx_2 dc_R =  \int_\CI y' \bar y dc_R\Big|_{x_2=0} +  \int_\CI \int_{-h}^0 \frac {\psi \bar y - U'' |y|^2}{U-c} dx_2 dc_R \\
&=\int_\CI y' \bar y dc_R\Big|_{x_2=0} + \int_\CI \int_{-h}^0 \frac {U'}{U-c} \Big( \big(\frac {\psi \bar y - U'' |y|^2}{U'}\big) \big(c, x_2) - \big(\frac {\psi \bar y - U'' |y|^2}{U'}\big) \big(c, x_2^c) \Big)  dx_2 dc_R \\
&\;+ \left(\int_{U(-h -\frac 12 \rho_0)}^{U(-\frac 12h)} +\int_{U(-\frac 12h)}^{U(\frac 12 \rho_0)} \right) \big(\frac {\psi \bar y - U'' |y|^2}{U'}\big) \big(c, x_2^c) \big(\log (U(0) - c) - \log (U(-h)-c)  \big) dc_R \triangleq \sum_{j=1}^4 A_j.
\end{align*}
The first term $A_1$ of boundary contribution can be estimated by Lemma \ref{L:y_nh-esti-2}(2b) and 
\eqref{E:y(0)} with $\zeta_\pm=0$, as well as \eqref{E:h-0},  \eqref{E:F_0} and \eqref{E:y_nh-temp-1.5}, 
\[
|A_1|\le \Big| \int_\CI y' \bar y dc_R\Big|_{x_2=0} \Big| \le C\mu^2 \Big| \int_\CI |U(0)- c|^2 |y' (0)|^2 dc_R \Big| \le C \mu^{1 -\ep} (\mu |\psi'|_{L_{c_R, x_2}^2} +|\psi|_{L_{c_R, x_2}^2} )^2. 
\]
Concerning the last integral $A_4$, we first split it as 
\[
|A_4| \le  \int_{U(-\frac 12h)}^{U(\frac 12 \rho_0)} \Big( \big|(\psi \bar y - U'' |y|^2)(c, \cdot) \big|_{C_{x_2}^\alpha} |x_2^c|^{\alpha} + \big|(\psi \bar y - U'' |y|^2)(c, 0) \big|\Big)  \big(1+ \big| \log |U(0) - c|\big|\big) dc_R. 
\] 
The above terms at $x_2=0$ can estimated much as $A_1$ and we obtain 
\begin{align*}
& \int_{U(-\frac 12h)}^{U(\frac 12\rho_0)} \big|(\psi \bar y - U'' |y|^2)(c, 0) \big|  \big(1+ \big| \log |U(0) - c|\big|\big) dc_R \\ 
\le & C\int_{U(-\frac 12h)}^{U(\frac 12 \rho_0)} (\mu^2 |\psi|^2 + \mu^2 |y'|^2 )(c, 0)  |U(0) - c|  dc_R \le C\mu^{1 -\ep} (\mu |\psi'|_{L_{c_R, x_2}^2} +|\psi|_{L_{c_R, x_2}^2} )^2.    
\end{align*}
We shall estimate all the remaining terms  using the H\"older norms of  $\psi \bar y$ and $|y|^2$. For any $H^1$ function $f(x)$ on an interval, 
it holds  
\be \label{E:y_nh-temp-3} 
|f|_{C^\alpha} \le C |f|_{L^2}^{\frac 12 - \alpha} |f|_{H^1}^{\frac 12 +\alpha}, \quad \alpha \in [0, \tfrac 12], 
\ee
which applies to $\psi \bar y$ and $|y|^2$. In the $f|_{H^1}$ can be replaced by $|f'|_{L^2}$ if $f$ vanishes somewhere in the interval. 
Hence for each fixed $c$ with $c_I>0$ and $c_R \in \CI$,   
\[
\big| |y|^2\big|_{C_{x_2}^\alpha} \le C |y|_{C_{x_2}^\alpha}|\bar y|_{C_{x_2}^0}  \le C |y|_{L_{x_2}^2}^{1 -\alpha} |y'|_{L_{x_2}^2}^{1 + \alpha},
\]
\[ 
|\psi \bar y|_{C_{x_2}^\alpha} 
\le  C\big( |\psi|_{L_{x_2}^2}^{\frac 12 -\alpha} |\psi|_{H_{x_2}^1}^{\frac 12 + \alpha} |y|_{L_{x_2}^2}^{\frac 12} |y'|_{L_{x_2}^2}^{\frac 12} +|\psi|_{L_{x_2}^2}^{\frac 12} |\psi|_{H_{x_2}^1}^{\frac 12} |y|_{L_{x_2}^2}^{\frac 12 -\alpha} |y'|_{L_{x_2}^2}^{\frac 12 + \alpha}\big).   
\]
For any $\alpha \in (0, \tfrac 12]$ and $k>0$, using $y(c, -h)= 0$ and the above estimates, 
we obtain   
\begin{align*}
& |y'|_{L_{c_R, x_2}^2}^2 + k^2 |y|_{L_{c_R, x_2}^2}^2 \le C \int_\CI \int_{-h}^0  \big|(\psi \bar y - U'' |y|^2)(c, \cdot) \big|_{C_{x_2}^\alpha} |x_2 -x_2^c|^{\alpha -1} dx_2 dc_R \\
&\qquad \qquad  + C \int_{U(-h-\frac 12\rho_0)}^{U(-\frac 12h)}  \big|(\psi \bar y - U'' |y|^2)(c, \cdot) \big|_{C_{x_2}^\alpha} |x_2^c+h|^{\alpha} \big(1+ \big|\log |U(-h)-c|\big|\big) dc_R \\
& \qquad \qquad  + C \int_{U(-\frac 12h)}^{U(\frac 12 \rho_0)}  \big|(\psi \bar y - U'' |y|^2)(c, \cdot) \big|_{C_{x_2}^\alpha} |x_2^c|^{\alpha} \big(1+ \big|\log |U(0)-c|\big|\big) dc_R \\
& \qquad \qquad + C\mu^{1 -\ep} (\mu |\psi'|_{L_{c_R, x_2}^2} +|\psi|_{L_{c_R, x_2}^2} )^2\\  
\le & C  
\int_\CI  \big(|\psi|_{L_{x_2}^2}^{\frac 12 -\alpha} |\psi|_{H_{x_2}^1}^{\frac 12 + \alpha} |y|_{L_{x_2}^2}^{\frac 12} |y'|_{L_{x_2}^2}^{\frac 12} +|\psi|_{L_{x_2}^2}^{\frac 12} |\psi|_{H_{x_2}^1}^{\frac 12} |y|_{L_{x_2}^2}^{\frac 12 -\alpha} |y' |_{L_{x_2}^2}^{\frac 12 + \alpha}  + |y|_{L_{x_2}^2}^{1 -\alpha} |y' |_{L_{x_2}^2}^{1 + \alpha}\big)\big|_{c} dc_R \\
& + C\mu^{1 -\ep} (\mu |\psi'|_{L_{c_R, x_2}^2} +|\psi|_{L_{c_R, x_2}^2} )^2\\
\le & C\big(|\psi|_{L_{c_R, x_2}^2}^{\frac 12 -\alpha} |\psi|_{L_{c_R}^2 H_{x_2}^1}^{\frac 12 + \alpha} |y|_{L_{c_R, x_2}^2}^{\frac 12} |y'|_{L_{c_R, x_2}^2}^{\frac 12} +|\psi|_{L_{c_R, x_2}^2}^{\frac 12} |\psi|_{L_{c_R}^2 H_{x_2}^1}^{\frac 12} |y|_{L_{c_R, x_2}^2}^{\frac 12 -\alpha} |y' |_{L_{c_R, x_2}^2}^{\frac 12 + \alpha}  + |y|_{L_{c_R, x_2}^2}^{1 -\alpha} |y' |_{L_{c_R, x_2}^2}^{1 + \alpha}\big) \\
& + C\mu^{1 -\ep} (\mu |\psi'|_{L_{c_R, x_2}^2} +|\psi|_{L_{c_R, x_2}^2} )^2\\
\le & \tfrac 12 |y'|_{L_{c_R, x_2}^2}^2 + (C + \tfrac 12 k^2) |y|_{L_{c_R, x_2}^2}^2 + C\big(|\psi|_{L_{c_R, x_2}^2}^2 + k^{-2(1-2\alpha) } |\psi'|_{L_{c_R, x_2}^2}^2  \big). 
\end{align*} 
By choosing $\alpha =\ep/2$, we have that, there exists $k_0>0$ such that for any $|k|\ge k_0$ and $c_I>0$, $y(\cdot +ic_I, \cdot)$ satisfies \eqref{E:y_nh0-esti-2}. To obtain the estimates for $y_{nh0}$ and $y_{nh0}'$ in the limiting case $c_I=0+$, for  $c_I>0$, let $y (c, x_2)$ and $y'(c, x_2)$ be defined by \eqref{E:y_nh-0} and \eqref{E:y_nh-1},  
which satisfy the desired estimates uniform in $c_I>0$. For $|k|\le k_0$ and $c_I>0$, the desired estimates simply follows from the estimates and convergence obtained in Lemma \ref{L:y_nh-esti-2}. 

Finally we consider the case $c_I=0$. Given $\psi(c_R, x_2) \in L_{c_R}^2 H_{x_2}^1$, let $y_{nh} (k, c_R+ic_I, x_2)$ be given by \eqref{E:y_nh-0} with $c= c_R + ic_I$ with $1\gg c_I>0$, which solves \eqref{E:Ray-NH-2}. 
From Lemma \ref{L:y_nh-esti-2}, it holds that $y(\cdot+ i c_I, \cdot) \to y_{nh0}$ 
and $y' (\cdot+ i c_I, \cdot) \to y_{nh0}'$ in $L_{c_R}^{\frac 32} L_{x_2}^2$ 
as $c_I \to 0+$. Therefore $y_{nh0}$ and $y_{nh0}'$ are also the weak limit of $y$ and $y'$ in $L_{c_R, x_2}^2$ as $c_I\to 0+$ and thus also satisfy \eqref{E:y_nh0-esti-2}. 
\end{proof}


{\bf $\bullet$ Case 2:}  
\be \label{E:NH-term-3}
\psi (c, x_2) =f (c, x_2) \psi_0 (x_2),  \quad f(\cdot + ic_I, \cdot) \in  L_{c_R}^{r_1} C_{x_2}^{\alpha}, \;\psi_0 \in L^{r}, \; r>1, \; r_1 \in [\tfrac r{r-1}, \infty], \; \alpha >0. 
\ee
Again we start with rough estimates on $y_{nh}$ and $y_{nh}'$. 

\begin{lemma} \label{L:y_nh-esti-3}
Assume \eqref{E:F_0} and \eqref{E:NH-term-3}. For any $q\in [1, \tfrac {rr_1}{r+r_1})$, the following hold for $x_2 \in [-h, 0]$ and $c_R \in \CI$. 
\begin{enumerate} 
\item There exists $C>0$ depending only on $r$, $r_1$, $q$, $\alpha$, $F_0$, $\rho_0$, $|U'|_{C^2}$, and $|(U')^{-1}|_{C^0}$, such that for any $k \in \R$ and $c_I \in (0, \rho_0]$, it holds
\[
|y_{nh}' (k, \cdot +ic_I, \cdot)|_{L_{x_2}^\infty L_{c_R}^q}  + \mu^{-1} |y_{nh} (k, \cdot +ic_I, \cdot)|_{L_{x_2}^\infty L_{c_R}^{\frac {rr_1}{r+r_1}} }  \le C\mu^{-\alpha} |f(\cdot +ic_I, \cdot)|_{L_{c_R}^{r_1} C_{x_2}^{\alpha}} |\psi|_{L^{r}}.  
\]
\item Assume $f(\cdot + ic_I, \cdot) \to f_0 (\cdot, \cdot)$ in $L_{c_R}^{r_1} C_{x_2}^{\alpha}$ as $c_I \to 0+$, then 
\begin{enumerate} 
\item $y_{nh} \to y_{nh0}$ in $L_{x_2}^\infty L_{c_R}^{\frac {rr_1}{r+r_1} } $ and $y_{nh}' \to y_{nh0}'$ in $L_{x_2}^{\infty} L_{c_R}^q$,
where $y_{nh0}$ and $y_{nh0}'$ are given by \eqref{E:y_nh0-2-0} and \eqref{E:y_nh0-2-1} with $\psi_0$ replaced by $f_0 \psi_0$;
\item at $\tilde x_2=-h, 0$, $y_{nh}' (k, \cdot +ic_I, \tilde x_2) \to y_{nh0}' (k, \cdot, \tilde x_2)$ in $L_{c_R}^{\frac {rr_1}{r+r_1} } $. Moreover, and for any $\ep \in (0, \frac 1r)$ with $\ep\le \alpha$, for any $k\in \R$, $c_I \ge 0$, it holds 
\[
|y_{nh}' (\cdot + ic_I, \tilde  x_2)|_{L_{c_R}^{\frac {rr_1}{r+r_1}}} \le C  \mu^{-\ep} |f|_{L_{c_R}^{r_1} C_{x_2}^\alpha} |\psi|_{L^r},
\]
where $C$ also depends on $\ep>0$. 
\end{enumerate} 
\end{enumerate} 
\end{lemma} 

\begin{proof}
Since the desired estimates are stronger and with weaker assumptions if  $\alpha \in (0, 1)$ is smaller (with possibly greater $C>0$), without loss of generality, we may assume $\alpha < \tfrac 1r$. 
In the following we shall need the  modification $\tilde x_2^c$ determined by $c_R \in \CI$:
\be \label{E:tx2c}
\tilde x_{2-}^c (c, x_2) = \begin{cases} 
\min\{ x_2, x_2^c\},  & \text{if } c_R >U(-h),\\
-h,  & \text{if } c_R \le U(-h), 
\end{cases}, \;\, \tilde x_{2+}^c (c, x_2) = \begin{cases} 
\max\{ x_2, x_2^c\},  & \text{if } c_R <U(0),\\
0,  & \text{if } c_R \ge U(0). 
\end{cases}\ee
For $c_I>0$, we first split $y_{nh}$ into 
\begin{align*}
y_1 (x_2)=& \frac {y_+(x_2)}{y_+(-h)} \big(\frac {y_- f}{ U'}\big)(\tilde x_{2-}^c) 
\int_{-h}^{x_2} \frac {\psi_0 U'}{U-c} dx_2'  + \frac {y_-(x_2)} {y_+(-h)} \big(\frac {y_+f} {U'}\big) (\tilde x_{2+}^c) 
\int_{x_2}^0 \frac {\psi_0 U'}{U-c} dx_2',
\end{align*}
and 
\begin{align*}
y_2(x_2) = & \frac {y_+(x_2) }{y_+(-h)} \int_{-h}^{x_2} \Big( \big( \frac {y_-f}{U'} \big) (x_2') - \big( \frac {y_-f}{U'} \big) (\tilde x_{2-}^c)
\Big) \frac {\psi_0 U'}{U-c} dx_2' \\
& + \frac {y_- (x_2) }{y_+(-h)} \int_{x_2}^0 \Big( \big( \frac {y_+ f}{U'} \big) (x_2') - \big( \frac {y_+f}{U'} \big) (\tilde x_{2+}^c)
\Big) \frac {\psi_0 U'}{U-c} dx_2',
\end{align*}
where we skipped all the dependence on $c$ and $k$. Clearly $y_{nh} = y_1+ y_2$. 

To estimate $y_1$, we can rewrite its integral part as 
\[
\int_{-h}^{x_2} \frac {\psi_0 U'}{U-c} dx_2' = \int_\R \frac {\chi_{U([-h, x_2])} (\psi_0 \circ U^{-1}) }{\tau-c_R - ic_I} d\tau = -\Big( \big(\frac 1{\tau + i c_I}\big) * \tilde \psi_- (x_2, \cdot) \Big)(c_R). 
\]
where 
\[
\tilde \psi_-(x_2, \tau) = \chi_{U([-h, x_2])} (\psi_0 \circ U^{-1})(\tau), \quad \tilde \psi_+(x_2, \tau) = \chi_{U([x_2, 0])} (\psi_0 \circ U^{-1})(\tau).
\]
The operator of convolution by $\tfrac 1{\tau + i c_I}$ is bounded on $L^r$ uniformly in $c_I>0$ and converges to $\pi (\CH + i I)$ strongly in $L^r$ as $c_I \to 0+$, where $\CH$ is the Hilbert transform and $I$ is the identity. The other integral can be treated similarly and we obtain from \eqref{E:F_0} and Lemma \ref{L:y-pm}
\begin{align*}
|y_1|_{L_{x_2}^\infty L_{c_R}^{\frac {rr_1}{r+r_1} }} \le &C\big(\big|\tfrac {y_+ (x_2) y_-(\tilde x_{2-}^c) }{y_+(-h)}\big|_{L_{c_R, x_2}^\infty} 
+ \big|\tfrac {y_- (x_2) y_+(\tilde x_{2+})}{y_+(-h)}\big|_{L_{c_R, x_2}^\infty} 
\big) |f|_{ L_{c_R}^{r_1} L_{x_2}^\infty} |\psi_0|_{L^r}\le C\mu  |f|_{L_{c_R}^{r_1} L_{x_2}^\infty} |\psi_0|_{L^r}.   
\end{align*} 
Moreover, since $x_2 \to \tilde \psi_\pm (x_2, \cdot)$ are two uniformly continuous mapping from $[-h, 0]$ to $L^r (\R)$ and the above convolution $\big(\frac 1{\tau + i c_I}\big) *$ is bounded on $L_{c_R}^r(\R)$ uniformly in $c_I>0$, we have that $\big(\frac 1{\tau + i c_I}\big) * \tilde \psi_\pm (x_2, \cdot)$ are two families (with parameter $c_I$) of equicontinuous functions (of $x_2$) from $[-h, 0]$ to $L_{c_R}^r$. As $c_I\to 0+$, they converge pointwisely (in $x_2$) to $\pi (\CH + i I) \tilde \psi_\pm (x_2, \cdot)\in L_{c_R}^r$ which are also uniformly continuous in $x_2$. The equicontinuity and the compactness of $[-h, 0]$ imply that the convergence is uniform in $x_2$. Therefore, along with the $L_{c_R, x_2}^\infty$ convergence of $y_\pm$ as $c_I\to 0+$ (Lemma \ref{L:y0}), we obtain that, as $c_I \to 0+$,  
\begin{align*}
y_1 (c_R+ ic_I, x_2) \to & \pi\frac {y_{0+}(c_R, x_2)}{ y_{0+} (c_R, -h)}  \big(\frac {y_{0-} f_0}{U'}\big)(c_R, \tilde x_{2-}^c) \big( (\CH + i I) \tilde \psi_- (x_2, \cdot) \big) (c_R) \\
&  + \pi\frac {y_{0-} (c_R, x_2)}{y_{0+} (c_R, -h)} \big(\frac {y_{0+} f_0}{U'}\big) (c_R, \tilde x_{2+}^c) \big( (\CH + i I) \tilde \psi_+ (x_2, \cdot) \big) (c_R) \quad\text{ in } L_{x_2}^\infty L_{c_R}^{\frac {rr_1}{r+r_1} }. 
\end{align*}

The other part $y_2$ can be estimated by the H\"older continuity of $f$ and $y_{\pm}$ in $x_2$ as 
\begin{align*}
|y_2(c, x_2)| \le& C\Big(\Big|\frac {y_+(x_2) }{y_+(-h)}\Big| |y_- f|_{C_{x_2'}^\alpha ([-h, x_2]) }  \int_{-h}^{x_2} \frac {|U- U(\tilde x_{2-}^c)|^\alpha }{|U-c|} |\psi_0|  U'dx_2'  \\
&\quad + \Big|\frac {y_-(x_2) }{y_+(-h)}\Big| |y_+ f|_{C_{x_2'}^\alpha ([x_2, 0])}  \int_{x_2}^0 \frac {|U- U(\tilde x_{2+}^c)|^\alpha }{|U-c|} |\psi_0|  U'dx_2'  \Big) \\
\le & C \mu^{1-\alpha} |f|_{C_{x_2}^\alpha}  \int_\R \frac {|\tau - c_R|^{\alpha}}{|\tau -c|} \big(\chi_{U([-h, 0])} (|\psi_0| \circ U^{-1}) \big) (\tau) d\tau,  
\end{align*}
where we also used 
\[
|y_- f|_{C_{x_2'}^\alpha [-h, x_2]} \le C |y_-|_{C_{x_2'}^\alpha([-h, x_2])} |f|_{C_{x_2}^\alpha} \le C \mu^{1-\alpha} e^{\mu^{-1} (x_2+h)} |f|_{C_{x_2}^\alpha}.
\]
and a similar estimate for $|y_+ f|_{C_{x_2'}^\alpha ([x_2, 0])}$ due to Lemma \ref{L:y-pm}. 
Since $\frac {|\tau|^{\alpha}}{|\tau + ic_I|}$ is a weak-$L^{\frac 1{1-\alpha}}$ function of $\tau$ with norm uniformly bounded in $c_I>0$, the weak Young's inequality 
yield 
\begin{align*}
|y_2|_{L_{x_2}^\infty L_{c_R}^{r_2} } 
\le C  \mu^{1-\alpha} |f|_{L_{c_R}^{r_1} C_{x_2}^{\alpha}} |\psi_0|_{L^r}, \; \text{ where } \; \tfrac 1{r_2} = \tfrac 1{r_1} + \tfrac 1r -\alpha < \tfrac 1{r_1} + \tfrac 1r. 
\end{align*}
To obtain the convergence of $y_2$ as $c_I \to 0$, using the $L_{c_R, x_2}^\infty$ convergence of $y_\pm$ and the $L_{c_R}^\infty L_{x_2}^{\tilde q}$ and $L_{x_2}^\infty L_{c_R}^{\tilde q}$, $\forall \tilde q\in (1, \infty)$, convergence of $y_\pm'$ (Lemma \ref{L:y0}), one may easily reduce the problem to the convergence 
of 
\begin{align*}
\tilde \Delta  = & \Big| \frac {y_{0+}(x_2) }{y_{0+}(-h)} \int_{-h}^{x_2} \Big( \big( \frac {y_{0-} f_0}{U'} \big) (x_2') - \big( \frac {y_{0-}f_0}{U'} \big) (\tilde x_{2-}^c)
\Big) \big( \frac 1{U-c} - \frac 1{U-c_R}\big)\psi_0 U' dx_2' \Big|_{ L_{x_2}^\infty L_{c_R}^{r_2}}  \\
\le & C \mu^{1-\alpha} \Big| |f_0|_{L_{C_{x_2}^\alpha}} \int_\R \big| \frac {|\tau - c_R|^{\alpha}}{|\tau -c|}  - |\tau -c_R|^{\alpha-1} \Big| \big(\chi_{U([-h, 0])} (|\psi_0| \circ U^{-1}) \big) (\tau) d\tau \Big|_{L_{c_R}^{r_2}}  
\end{align*}
and that of a similar term of the other integral. It is easy to see via a rescaling that, for $s\in [1, \frac 1{1-\alpha})$, 
\[
\Big|\frac {|\tau|^{\alpha}}{|\tau + ic_I|} - |\tau|^{\alpha-1} \Big|_{L^s} = |c_I|^{\alpha -1} \big|\gamma \big(\frac \tau{c_I}\big) \big|_{L^s} 
= |c_I|^{\frac 1s-1+\alpha} |\gamma|_{L^s},
\;\; \text{ where } \; \gamma(\tau) = \frac {|\tau|^{\alpha}}{|\tau + i|} - |\tau|^{\alpha-1}, 
\]
while with the weak-$L^{\frac 1{1-\alpha}}$ norm equal to $|\gamma|_{w-L^{\frac 1{1-\alpha}}}$. Hence
\[
\left| \Big|\frac {|\tau|^{\alpha}}{|\tau + ic_I|} - |\tau|^{\alpha-1}\Big| * \varphi \right|_{L^{\frac 1{\frac 1r -\alpha}}} \to 0, \quad \; \text{ as } c_I\to 0,
\]
for any $\varphi \in L^{\tilde r}$ with $\tilde r>r$. Through a standard density argument and using the above uniform bound on the weak-$L^{\frac 1{1-\alpha}}$ norm of the convolution kernel, this convergence also holds for any $\varphi \in L^r$. 
Therefore, we obtain $\tilde \Delta \to 0$ and thus 
\begin{align*}
y_2(c_R+ic_I, x_2) \to & \frac {y_{0+}(c_R, x_2) }{y_{0+}(c_R, -h)} \int_{-h}^{x_2} \Big( \big( \frac {y_{0-}f_0}{U'} \big) (c_R, x_2') - \big( \frac {y_{0-}f_0}{U'} \big) (c_R, \tilde x_{2-}^c)
\Big) \frac {\psi_0 U'}{U-c_R} dx_2' \\
& + \frac {y_{0-} (c_R, x_2) }{y_{0+} (c_R, -h)} \int_{x_2}^0 \Big( \big( \frac {y_{0+} f_0}{U'} \big) (c_R, x_2') - \big( \frac {y_{0+}f_)}{U'} \big) (c_R, \tilde x_{2+}^c)
\Big) \frac {\psi_0 U'}{U-c_R} dx_2'. 
\end{align*}

The above estimates of $y_1$ and $y_2$ together yield the desired estimates of $y_{nh}$ and its convergence as $c_I \to 0$. The analysis on $y_{nh}'$ also follows from the above estimates with minor modifications, mostly replacing some $|y_\pm|_{L_{c_R, x_2}^\infty}$ by $|y_\pm'|_{L_{x_2}^\infty L_{c_R}^s}$ or $|y_\pm'|_{L_{c_R}^\infty L_{x_2}^s}$ outside the integrals, needed to control its logarithmic singularity caused by $y_\pm'$. We omit the details.   

Finally, as in Lemma \ref{L:y_nh-esti-2}, stronger estimates and convergence can be obtained at $x_2 =-h, 0$ due to prescribed boundary values \eqref{E:y-pm}. 
In fact, 
\begin{align*}
y_{nh}' (0) = & \frac {y_+' (0)}{y_+( -h) } \int_{-h}^{0} \frac {y_- f \psi_0}{U-c}   dx_2'  \\
=& \frac {y_+'(0)} {y_+( -h)} \big( \frac {y_-  f}{U'}\big) (\tilde x_{2-}^c) \int_{-h}^{0} \frac {\psi_0 U'}{U-c}   dx_2' +  \frac {y_+'(0)}{y_+( -h) } \int_{-h}^{0} \Big(\big(\frac {y_- f}{U'} \big) (x_2') - \big(\frac {y_- f}{U'} \big) (\tilde x_{2-}^c)  \Big) \frac {\psi_0}{U-c}   dx_2' 
\end{align*} 
implies 
\begin{align*}
|y_{nh}' (0)|_{L_{c_R}^{\frac {rr_1}{r+r_1}}} \le & C \Big( |f|_{L_{c_R}^{r_1} L_{x_2}^\infty} \Big|  \int_{U(-h)}^{U(0)} \frac {\psi_0}{\tau-c}   d\tau \Big|_{L_{c_R}^r} \\
&+ \mu^{-1} e^{-\mu h} |y_-|_{L_{c_R}^{\frac 1\ep} C_{x_2}^{\ep}} |f|_{L_{c_R}^{r_1} C_{x_2}^{\ep}} \Big| \int_{U(-h)}^{U(0)} \frac {|\psi_0|}{|\tau-c|^{1-\ep}}  d\tau \Big|_{L_{c_R}^{\frac r{1-\ep r}}} \Big).  
\end{align*} 
From the same procedure as in estimating $y_1$ and $y_2$ in the above, we obtain the desired estimate. Its convergence follows much as that of $y_{nh}$. The same argument applies to $y_{nh}'(c, -h)$ and the proof of the lemma is complete. 
\end{proof}

The following is an estimate $y_{nh0}$ and $y_{nh0}'$ in $L_{c_R, x_2}^2$ and their dependence on $k$.

\begin{lemma} \label{L:y_nh0-3-esti}
In addition to \eqref{E:F_0} and \eqref{E:NH-term-3}, assume $\frac 12 \ge \frac 1r + \tfrac 1{r_1}$. For any $\ep \in (0, 1)$, there exists $C>0$ depending only on $\ep$, $r$, $r_1$, $F_0$, $\rho_0$, $|U'|_{C^2}$, and $|(U')^{-1}|_{C^0}$, such that for any $k \in \R$ and $c_I\ge 0$, it holds 
\[
|y_{nh}'|_{L_{c_R, x_2}^2}^2 + \mu^{-2} |y_{nh}|_{L_{c_R, x_2}^2}^2 \le C  \mu^{1-\ep}|f|_{L_{c_R}^{r_1} C_{x_2}^{\alpha}}^2  |\psi_0|_{L^r}^2. 
\] 
where the norms are taken for $c_R \in \CI$ and $x_2 \in [-h, 0]$. 
\end{lemma}


\begin{proof} 
As in the proof of Lemma \ref{L:y_nh0-2-esti}, we first consider for $c_I>0$ and drop the subscript $\cdot_{nh}$ for notation simplification.  
Multiplying the Rayleigh equation \eqref{E:Ray-NH-2} by $\bar y$ and integrating in both $c_R$ and $x_2$, we have 
\begin{align*}
& \int_\CI \int_{-h}^0 |y'|^2 + k^2 |y|^2 dx_2 dc_R = \int_\CI \int_{-h}^0 \frac {f \psi_0 \bar y - U'' |y|^2}{U-c} dx_2 dc_R + \int_\CI y' \bar y dc_R\Big|_{x_2=0} \\
=& \int_\CI \int_{-h}^0 \frac {U'\psi_0}{U-c} \Big( \big(\frac {f \bar y }{U'}\big) \big(c, x_2) - \big(\frac {f \bar y}{U'}\big) \big(c, x_2^c) \Big)  dx_2 dc_R + \int_\CI \big(\frac {f \bar y}{U'}\big) \big(c, x_2^c) \int_{-h}^{0} \frac {U' \psi_0}{U-c} dx_2'  dc_R \\
& -  \int_\CI \int_{-h}^0 \frac { U'' |y|^2}{U-c} dx_2 dc_R + \int_\CI y' \bar y dc_R\Big|_{x_2=0} \triangleq I_1 + I_2 +I_3 +I_4.
\end{align*}

The term $I_4$ can be estimated much as in the proof of Lemma \ref{L:y_nh0-2-esti} using Lemmas \ref{L:y-pm} and \ref{L:y_nh-esti-3}(2b) 
\[
|I_4| \le C\mu^2 \Big| \int_\CI |U(0)- c|^2 |y' (0)|^2 dc_R \Big| \le C \mu^{2-\ep} |f|_{L_{c_R}^{r_1} C_{x_2}^\alpha}^2 |\psi_0|_{L^r}^2.
\]
Choose $\alpha_1$ and $r_2$ such that  
\[
0< \alpha_1 \le \max\{\tfrac \ep2,  \alpha,  \tfrac 1r + \tfrac 1{r_1}\},  \; \; \tfrac 1{r_2} = 1+ \alpha_1 -\tfrac 1r - \tfrac 1{r_1} \in (\tfrac 12, 1],   
\]
which is possible due to our assumption on $\alpha$, $r$, and $r_1$. The integral $I_1$ can be controlled by the H\"older continuity of $f$ and $y$ in $x_2$, the weak Young's inequality, and the \eqref{E:y_nh-temp-3} type interpolation inequality as 
\begin{align*}
|I_1| \le & C \int_\R \int_\R \big( \chi_\CI |(f \bar y)(c_R, \cdot) |_{C_{x_2}^{\alpha_1}}\big)  |\tau - c_R|^{\alpha_1-1} |(\chi_{U([-h, 0])} \psi_0 \circ U^{-1} ) (\tau)| d\tau d c_R \\
\le & C |f\bar y|_{L_{c_R}^{\frac 1{1+\alpha_1 -\frac 1r}} C_{x_2}^{\alpha_1}} |\psi_0|_{L^r} 
\le C |f|_{L_{c_R}^{r_1} C_{x_2}^{\alpha_1}} |y|_{L_{c_R}^{r_2} C_{x_2}^{\alpha_1}} |\psi_0|_{L^r} \\
\le & C |f|_{L_{c_R}^{r_1}C_{x_2}^{\alpha_1}} \big| | y'|_{L_{x_2}^2}^{\frac 12 + \alpha_1}  |y|_{L_{x_2}^2}^{\frac 12 - \alpha_1} \big|_{L_{c_R}^{r_2}} |\psi_0|_{L^r} 
\le C |f|_{L_{c_R}^{r_1}C_{x_2}^{\alpha_1}} | y'|_{L_{c_R, x_2}^2}^{\frac 12 + \alpha_1}  | y|_{L_{c_R}^{r_3} L_{x_2}^2}^{\frac 12 - \alpha_1} |\psi_0|_{L^r}    
\end{align*}
where $r_3< 2$ is determined by $\frac {\frac 12 +\alpha_1}2 + \frac {\frac 12 -\alpha_1}{r_3} = \frac 1{r_2}$. Therefore we obtain 
\[
|I_1| \le \tfrac 14 \big( |y'|_{L_{c_R, x_2}^2}^2 + k^2 |y|_{L_{c_R, x_2}^2}^2\big) + C k^{- (1-2\alpha_1)} |f|_{L_{c_R}^{r_1}C_{x_2}^{\alpha_1}}^2 |\psi_0|_{L^r}^2. 
\]
The estimate of $I_2$ is much as in the proof of Lemma \ref{L:y_nh-esti-3} based on the boundedness of the convolution operator on $L^r$
\begin{align*}
|I_2| \le & C |\psi|_{L^r} |(f \bar y)(c_R, x_2^c)|_{L_{c_R}^{\frac r{r-1}}}  \le C |\psi|_{L^r} \big| |f|_{L_{x_2}^\infty} |y|_{L_{x_2}^2}^{\frac 12} |y'|_{L_{x_2}^2}^{\frac 12}  \big|_{L_{c_R}^{\frac r{r-1}}} \\
\le & C |\psi|_{L^r} |f|_{L_{c_R}^{r_4} L_{x_2}^\infty} |y|_{L_{c_R, x_2}^2}^{\frac 12} |y'|_{L_{c_R, x_2}^2}^{\frac 12},
\end{align*} 
where $r_4 = \frac {2r}{r-2} \le r_1$. Hence 
\[
|I_2| \le \tfrac 14 \big( |y'|_{L_{c_R, x_2}^2}^2 + k^2 |y|_{L_{c_R, x_2}^2}^2\big) + C k^{- 1} |f|_{L_{c_R}^{r_1}C_{x_2}^{\alpha_1}}^2 |\psi|_{L^r}^2. 
\]
Finally $I_3$ can be estimated exactly as in the proof of Lemma \ref{L:y_nh0-2-esti} (and also applying Lemma \ref{L:y_nh-esti-3}(2b)) and we have 
\[
|I_3| \le \tfrac 14 |y'|_{L_{c_R, x_2}^2}^2 + C \big(|y|_{L_{c_R, x_2}^2}^2 + \mu^{4-\ep} |f|_{L_{c_R}^{r_1} C_{x_2}^\alpha}^2 |\psi_0|_{L^r}^2\big).
\] 
Therefore, there exists $k_0>0$ such that $y$ and $y'$ satisfy the desired estimates for $|k| \ge k_0$ and $c_I>0$. For those $|k|\le k_0$, the $|y|_{L_{c_R, x_2}^2}^2$ term in the upper bound of $I_3$ can be controlled by Lemma \ref{L:y_nh-esti-3} directly and thus the desired estimates  are also satisfied by  $y$ and $y'$. The estimate in the limiting case of $c_I=0+$ can be obtained through the same weak convergence argument as in the proof of Lemma \ref{L:y_nh0-2-esti}.
\end{proof} 

\begin{remark} 
In some sense the $L_{c_R, x_2}^2$ assumption on $\psi$ and $\psi'$ in the Lemma \ref{L:y_nh0-2-esti} is the (unreachable) borderline case of  Lemma \ref{L:y_nh0-3-esti}. In fact, $\psi(c_R, x_2)$ can be written as $\psi \cdot 1$, where the former belongs to $L_{c_R}^{2} C_{x_2}^{\frac 12}$ with $r_1=2$. As $r<\infty$ and $\frac 1r + \frac 1{r_1} =\frac 12$ are assumed in \eqref{E:NH-term-3} and Lemma \ref{L:y_nh0-3-esti}, it does not apply in this case.   
\end{remark}

\subsection{Differentiation in $c$ of solutions to non-homogeneous Rayleigh system
} \label{SS:Ray-NH-CGW}

Based on the analysis of the non-homogeneous Rayleigh equation \eqref{E:Ray-BVP} with zero boundary conditions, in this subsection we shall mainly consider \eqref{E:Ray-3} type non-zero boundary conditions, in particular the estimates of the derivative of solutions $y_B(k, c, x_2)$ given in \eqref{E:y_B} with respect to $c$.   
%

Through straight forward calculations 
and applying Lemma \ref{L:y-pm}, 
we obtain 

\begin{lemma} \label{L:Ray-BVP-1}
Assume \eqref{E:F_0} and $c \in \CI + i [-\rho_0, \rho_0]$. 
For any $1< r_1 < r_2 < \infty$, there exists $C>0$ depending only on $r_1$, $r_2$, $F_0$, $\rho_0$, $|U'|_{C^2}$, and $|(U')^{-1}|_{C^0}$, such that for any $|c_I| \le \rho_0$, the unique solution $y_B(k, c, x_2)$ to \eqref{E:Ray-BVP} satisfies 
\[
|y_B|_{L_{c_R, x_2}^2}  \le C \big( |y_{nh}|_{L_{c_R, x_2}^2} + \mu^{\frac 52} |\zeta_+ |_{L_{c_R}^2} + \mu^{\frac 12} |\zeta_- |_{L_{c_R}^2}\big),    
\]
\[
|y_B'|_{L_{c_R, x_2}^2}\le C \big( |y_{nh}'|_{L_{c_R, x_2}^2} + \mu^{\frac 32} |\zeta_+ |_{L_{c_R}^2} + \mu^{-\frac 12} |\zeta_- |_{L_{c_R}^2}\big), 
\]
\[
|y_B'(-h)|_{L_{c_R}^{r_1}}  \le C \big( |y_{nh}' (-h)|_{L_{c_R}^{r_1}} + \mu^{-1} |\zeta_- |_{L_{c_R}^{r_1}} + |\zeta_- |_{L_{c_R}^{r_2}}
+ \mu e^{-\mu^{-1} h}   |\zeta_+ |_{L_{c_R}^{r_1}} \big),
\]
\[
|y_B' (0)|_{L_{c_R}^{r_1}}  \le C \big(|y_{nh}' (0)|_{L_{c_R}^{r_1}}  + \mu |\zeta_+ |_{L_{c_R}^{r_1}} + \mu^2  |\zeta_+ |_{L_{c_R}^{r_2}} + \mu^{-1} e^{-\mu^{-1} h} |\zeta_- |_{L_{c_R}^{r_1}} )\big),
\]
where the norm is taken on $c_R\in \CI$ and $x_2 \in [-h, 0]$. 
\end{lemma}

 
We shall also consider the limit 
\be \label{E:y_B0} 
y_{B0} = y_B|_{c_I=0+}=  \lim_{c_I\to 0+} y_B = b_{0-} y_{0-} + b_{0+} y_{0+} + y_{nh0}, 
\ee
which exists for appropriate $\psi (c, x_2)$ and satisfies the same estimates as $y_B$ (see Subsection \ref{SS:Ray-NH-HBC}). 
 


In the rest of the subsection, we shall focus on the special case motivated by \eqref{E:Ray}:
\be \label{E:Ray-BC-2} 
\psi=\psi_0(x_2), \quad \zeta_-(c)= \xi_-, \quad \zeta_+ (c) = \xi_1 + (U(0)-c) \xi_2, 
\ee
where $\psi_0$, $\xi_-$, $\xi_1$, and $\xi_2$ are all independent of $c$. Our goal is to obtain the estimates of the derivatives of the solution $y_B(k, c, x_2)$ to \eqref{E:Ray-BVP} in $c_R$.  

\begin{proposition} \label{P:pcy_B} 
Assume  $U \in C^{l_0}$, $l_0\ge 3$, \eqref{E:F_0}, and \eqref{E:Ray-BC-2}. For any $\ep \in (0, 1)$, $r \in (1, \infty)$, there exists $C>0$ depending on $\ep$, $r$, $F_0$, $\rho_0$, $|U'|_{C^{l_0-1}}$, and $|(U')^{-1}|_{C^0}$ such that the solution $y_B(k, c, x_2)$ to \eqref{E:Ray-BVP} satisfies that for any $|c_I| \le \rho_0$ and $k \in \R$, 
\[
|y_B|_{L_{c_R, x_2}^2} + \mu |y_B'|_{L_{c_R, x_2}^2} + \mu^{\frac 32}  |y_B' (0)|_{L_{c_R}^{2}} + \mu^{\frac 12} |y_B(0)|_{L_{c_R}^2}  \le C\mu^{\frac 52} \big( \mu^{-1-\ep} |\psi_0|_{L^2} +  |\xi_1| + |\xi_2| +\mu^{-2} |\xi_-| \big); 
\]
if $l_0\ge 4$, then
\begin{align*}
& |\p_{c_R} y_B|_{L_{c_R, x_2}^2} + \mu |\p_{c_R} y_B' + \tfrac 1{U'(x_2^c)} y_B''|_{L_{c_R, x_2}^2} + \mu^{\frac 32} |(\p_{c_R} y_B' + \tfrac 1{U'(x_2^c)} y_B'')(0)|_{L_{c_R}^2} + \mu^{\frac 12} |\p_{c_R} y_B(0)|_{L_{c_R}^2}\\
\le & C \mu^{\frac 32}\big(  \mu^{-1-\ep}|\psi_0|_{L^2} + \mu^{-\ep}|\psi_0'|_{L^2} + |\xi_1| +  |\xi_2| +\mu^{-2} |\xi_-|\big); 
\end{align*}
and, if $l_0\ge 5$, then 
\[
|\tilde y_B|_{L_{c_R, x_2}^2} \le  C\mu^{\frac 12}\big(  \mu^{-1-\ep}|\psi_0|_{L^2} + \mu^{-\ep}|\psi_0'|_{L^2} + \mu^{1 -\ep}|\psi_0''|_{L^2} +  |\xi_1| + |\xi_2| +\mu^{-2} |\xi_-|\big), 
\]
where 
\[
\tilde y_B =\p_{c_R}^2 y_B + \frac 1{U'(x_2^c)^2} \Big( - y_B'' + \frac { g+ \sigma k^2 
}{\BF(k, c)}  \big( y_B'' (-h) y_+ - y_B''(0) 
 y_-\big)
\Big), 
\]
and all the norms are taken on $(c_R, x_2) \in \CI \times [-h, 0]$. Moreover, as $c_I \to 0+$, the following hold. 
\begin{enumerate}
\item Assume $\psi_0\in L^2$ and $U \in C^3$, then for any $r\in [1, 2)$, $y_B \to y_{B0}$ in $L_{x_2}^\infty L_{c_R}^2$,  $y_B' \to y_{B0}'$ in $L_{x_2}^\infty L_{c_R}^r$, and $y_B'(0) \to y_{B0}'(0)$ in $L_{c_R}^2$. 
\item Assume $\psi_0\in H^1$ and $U \in C^4$, then for any $r\in [1, 2)$ and $q \in [1, \infty)$, $\p_{c_R} y_B \to \p_{c_R} y_{B0}$ in $L_{x_2}^\infty L_{c_R}^r$, $\p_{c} y_B' + \tfrac 1{U'(x_2^c)} y_B'' \to \p_{c} y_{B0}' + \tfrac 1{U'(x_2^c)} y_{B0}''$ in $L_{x_2}^q L_{c_R}^r$, and $(\p_{c} y_B' + \tfrac 1{U'(x_2^c)} y_B'')(0) \to (\p_{c} y_{B0}' + \tfrac 1{U'(x_2^c)} y_{B0}'')(0)$ in $L_{c_R}^r$. 
\item Assume $\psi_0\in H^2$  and $U \in C^5$, then for any $r\in [1, 2)$, $\tilde y_B$ also converges  in $L_{x_2}^\infty L_{c_R}^r$ to its limit $\tilde y_{B0}$.
\end{enumerate}    
\end{proposition}

Since $y_B$ is holomorphic in $c \notin U([-h, 0])$, $\p_c y_B = \p_{c_R} y_B$. From the Rayleigh equation, singularity at the level of delta mass appears in $y_B''$ along $U(x_2) = c_R$, $x_2 \in [-h, 0]$, as $c_I\to 0+$. Therefore $\p_{c}^2 y_B$ and $\p_c y_B'$ also display such singularities which are singled out in the above estimates. The $y_B''$ involved in the singular terms will be substituted by using the Rayleigh equation \eqref{E:Ray-NH-2} whenever necessary. 

\begin{proof}
The $L_{c_R, x_2}^2$ estimates on $y_B$ and $y_B'$, as well as the $L_{c_R}^r$ estimate of $y_B'(0)$ with $r \in (1, \infty)$, follow readily from \eqref{E:F_0}, \eqref{E:Ray-BC-2}, Lemmas \ref{L:y-pm}, \ref{L:y_nh0-3-esti}, \ref{L:y_nh-esti-3} (with $r=2$, $r_1=\infty$, and $f_0=1$), 
and \ref{L:Ray-BVP-1}. 
The estimate of $y_B(0)$ is simply obtained from those of $y_B$ and $y_B'$. 
Moreover, for the rest of the proof of the proposition we shall also need the following inequality for $r\in (1, \infty)$ which is also derived form Lemma \ref{L:y_nh-esti-3} and  Lemma \ref{L:Ray-BVP-1} and uniform in $c_I \in [0, \rho_0]$
\be \label{E:pcy_nh-temp-1}\begin{split}
&
|y_B' (0)|_{L_{c_R}^{r}}  \le C (\mu^{-\ep} |\psi_0|_{L^{r}} +\mu (|\xi_1| +|\xi_2|) + \mu^{-1} e^{-\mu^{-1} h} |\xi_-|), \\ 
&|y_B' (-h)|_{L_{c_R}^{r}} \le C (\mu^{-\ep} |\psi_0|_{L^{r}} + \mu e^{-\mu^{-1} h} (|\xi_1| +|\xi_2|) +\mu^{-1} |\xi_-| ). 
\end{split} \ee
The convergence of $y_B$, $y_B'$, and $y_B'(0)$ follow directly from the continuity of $\BF$ (Lemma \ref{L:e-v-basic-1}) and the convergence of $y_\pm$ and $y_\pm'$ (Lemma \ref{L:y0}) and $y_{nh}$ (Lemma \ref{L:y_nh-esti-3}). Moreover, we also have the convergence of $y_B'(-h)$ in $L_{c_R}^2$. 

In the following differentiations in $c_R$ are all carried out for $c_I>0$. The convergence analysis based on the convergence results of $y_\pm$ and those of $y_{nh}$ in Subsection \ref{SS:Ray-NH-HBC} ensure that the estimates hold also for $c_I=0+$. Directly differentiating the Rayleigh equation \eqref{E:Ray-NH-2} in $c_R$ directly would cause worse singularity in the equation. Instead we first consider 
 \be \label{E:Dc}
 D_c = U'(x_2^c) \p_{c_R} + \p_{x_2}, \qquad \p_{c_R} = U'(x_2^c)^{-1} (D_c -\p_{x_2}),  \qquad [D_c, \p_{x_2}]=0, 
 \ee
where $x_2^c$ is defined by $U(x_2^c) =c_R$ as in \eqref{E:x2c}. It satisfies 
\be \label{E:Dc-1}
D_c \big(\tfrac 1{U(x_2)-c}\big)  = - \tfrac {U'(x_2) - U'(x_2^c)}{(U(x_2)-c)^2}, \quad D_c^2 \big(\tfrac 1{U(x_2)-c}\big)  = \tfrac {2(U'(x_2) - U'(x_2^c))^2}{(U(x_2)-c)^3} - \tfrac {U''(x_2) - U''(x_2^c)}{(U(x_2)-c)^2}, 
\ee
where the singularity remains at the same level. 

$\bullet$ {\bf Estimating $\p_{c_R} y_B$.} 
Applying $D_c$ to \eqref{E:Ray-NH-2} and simplifying, we obtain 
\begin{subequations} \label{E:Dc1-Ray} 
\be \label{E:Dc1-Ray-NH} \begin{split}
- (D_c y_B)'' + &\big(k^2 + \frac {U''}{U-c}\big)  D_c y_B = \frac {\psi_0'(x_2)  + f_1 (c, x_2) \psi_0 (x_2)+ \psi_1(c, x_2)}{U-c};  
\end{split} \ee
where 
\[
\psi_1 = \big( \frac {U'' (U' - U'(x_2^c))}{U-c} - U'''\big) y_B, \quad f_1 = (U-c)D_c \big( \frac 1{U-c}\big)= - \frac {U' - U'(x_2^c)}{U-c},  
\]
and boundary conditions 
\be \label{E:Dc1-Ray-BC-1} 
D_c y_B(-h) = \zeta_{1-} \triangleq y_B'( -h);
\ee
\be \label{E:Dc1-Ray-BC-2} 
\big(U(0)-c\big)^2 (D_c y_B)' (0) - \big(U'(0) (U(0)-c) + g + \sigma k^2\big) D_c y_B(0) = \zeta_{1+} (c) 
\ee
where 
\begin{align*}
\zeta_{1+} = &  -\xi_2 U'(x_2^c) - (U(0)-c) \psi_0(0) + \big( (2U'(x_2^c) - U'(0)) (U(0)-c) -g -\sigma k^2\big) y_B'(0) \\
&+ \big( k^2 (U(0)-c)^2 + U''(0) (U(0)-c) - U'(x_2^c) U'(0) \big) y_B(0).   
\end{align*}
\end{subequations}

Let $\tilde y_1 (c, x_2)$ and $\tilde y_2 (c, x_2)$ be the solution to the non-homogeneous Rayleigh equation \eqref{E:Ray-NH-2}, but with zero boundary values in \eqref{E:Ray-BC-1}, with $\psi(c, x_2)$ replaced by $\psi_1$ and $\psi_0' + f_1 \psi_0$, respectively. Both are given by the formula \eqref{E:y_nh-0}. 
Using the estimates of $y_B$ derived in the above and apply Lemmas \ref{L:y_nh0-2-esti}, we have    
\[
|\tilde y_1|_{L_{c_R, x_2}^2} + \mu |\tilde y_1'|_{L_{c_R, x_2}^2} \le C\mu \big(|y_B|_{L_{c_R, x_2}^2} + \mu^{1-\frac \ep4} |y_B'|_{L_{c_R, x_2}^2} \big) \le C\mu^{\frac 52-\ep} \big(  \mu |\xi_1| + \mu |\xi_2| + |\psi_0|_{L^2}+\mu^{-1} |\xi_-|  \big). 
\]
Moreover, from  Lemma \ref{L:y_nh-esti-2}(2b) and \eqref{E:Ray-BC-1}, \eqref{E:y(0)}, and \eqref{E:Ray-BC-2}, one can compute   
\begin{align*}
|\tilde y_1'(c, 0)| \le &C\mu^{-\frac 12 (1+\ep)} (|y_B|_{L_{x_2}^2} + \mu |y_B'|_{L_{x_2}^2}) + C \mu^2 \big(1+ \big| \log |U(0)-c|\big|\big) (|\zeta_+|+ |U(0)-c|^2 |y_B' (c, 0)| ),
\end{align*}
where $y_B(0)$ was substituted by using  \eqref{E:Ray-BC-1}. It along with the above estimates on $y_B$ implies 
\begin{align*}
|\tilde y_1'(0)|_{L_{c_R}^2} \le &C\mu^{-\frac 12 (1+\ep)} (|y_B|_{L_{c_R, x_2}^2} + \mu |y_B'|_{L_{c_R, x_2}^2}) + C \mu^2 ( |\xi_1| +  |\xi_2|+ |y_B'(0)|_{L_{c_R}^2})\\
\le & C\mu^{1-\ep} \big(  \mu |\xi_1| + \mu |\xi_2| + |\psi_0|_{L^2} + \mu^{-1} |\xi_-| \big).
\end{align*}
The estimate at $x=-h$ based on Lemma \ref{L:y_nh-esti-2}(2b)  is similar
\[
|\tilde y_1'(c, -h)|
\le   C\mu^{-\frac 12 (1+\ep)} \big(|y_B|_{L_{x_2}^2} + \mu |y_B'|_{L_{x_2}^2} \big) + C\big(1+ \big| \log |U(-h)-c|\big|\big) |\xi_-|,
\]
which yields 
\begin{align*}
|\tilde y_1'(-h)|_{L_{c_R}^2} \le &C\mu^{-\frac 12 (1+\ep)} (|y_B|_{L_{c_R, x_2}^2} + \mu |y_B'|_{L_{c_R, x_2}^2}) + C|\xi_-| \\
\le & C\mu^{1-\ep} \big(  \mu |\xi_1| + \mu |\xi_2| + |\psi_0|_{L^2} +\mu^{-1} |\xi_-| \big). 
\end{align*}
From the convergence of $y_B$ and Lemma \ref{L:y_nh-esti-2}, as $c_I \to 0+$, we have the convergence of $\tilde y_1$ in $ L_{c_R}^r L_{x_2}^\infty$, $\tilde y_1'$ in $L_{c_R}^r L_{x_2}^q$, and $\tilde y_1'(0)$ in $L_{c_R}^r$, for any $r \in [1, 2)$ and $q\in [1, \infty)$.  

Due to the smoothness of $f_1$, we apply Lemmas \ref{L:y_nh0-3-esti}  
and \ref{L:y_nh-esti-3} 
instead to  estimate $\tilde y_2$
\[
|\tilde y_2|_{L_{c_R, x_2}^2} + \mu |\tilde y_2'|_{L_{c_R, x_2}^2} \le C\mu^{\frac 32-\ep}  |\psi_0|_{H^1}, \quad 
|\tilde y_2'(-h)|_{L_{c_R}^2}+|\tilde y_2'(0)|_{L_{c_R}^2} \le C\mu^{-\ep} |\psi_0|_{H^1}.
\]
Again from Lemma \ref{L:y_nh-esti-3}, as $c_I \to 0+$, we have the convergence of $\tilde y_2$ in $  L_{x_2}^\infty L_{c_R}^2$, $\tilde y_2'$ in $L_{x_2}^\infty L_{c_R}^r$, for any $r \in [1, 2)$, and $\tilde y_2'(0)$ in $L_{c_R}^2$.  

Finally, from \eqref{E:pcy_nh-temp-1} and \eqref{E:y(0)}, 
we have, for any $r \in (1, \infty)$,   
\[
|\zeta_{1-}|_{L_{c_R}^r} \le C\big(\mu^{-\ep} |\psi_0|_{L^r} + \mu e^{-\mu^{-1} h} (|\xi_1| +|\xi_2|) +\mu^{-1} |\xi_-| \big), 
\]
\[ 
|\zeta_{1+}|_{L_{c_R}^r} \le C \mu^{-1}\big(  |\xi_1|+  |\xi_2| + \mu|\psi_0(0)| + \mu^{-1-\ep} |\psi_0|_{L^r}+ \mu^{-2} e^{-\mu^{-1} h} |\xi_-| \big),
\]
where again we substituted $y_B(0)$ by \eqref{E:Ray-BC-1} and \eqref{E:y(0)}. Moreover, from the convergence of $y_B$ and $y_B'$, we have the convergence of $\zeta_{1\pm}$  
in $L_{c_R}^2$.  

As $\tilde y_1 + \tilde y_2$ plays the role of "$y_{nh}$" in the representation of $D_cy_B$ as given in Lemma \ref{L:Ray-BVP-1}, the above estimates 
imply 
\be \label{E:pcy_nh-temp-2}
|D_c y_B|_{L_{c_R, x_2}^2} + \mu |D_c y_B'|_{L_{c_R, x_2}^2} \le C \big(  \mu^{\frac 32} |\xi_1| + \mu^{\frac 32} |\xi_2| + \mu^{-\frac 12} |\xi_-| + \mu^{\frac 12-\ep}|\psi_0|_{L^2} + \mu^{\frac 32-\ep}|\psi_0'|_{L^2}\big), 
\ee
where the $\psi_0(0)$ term was bounded by the other norms of $\psi_0$ via interpolation. The desired $L_{c_R, x_2}^2$ estimates on $\p_{c_R} y_B$ and $\p_{c_R} y_B'$ follow from that of $y_B'$, \eqref{E:Dc}, and the above inequality. We also obtain the $L_{c_R}^2$ estimate of $D_c y_B(0)$ from \eqref{E:pcy_nh-temp-2} which in turn yields the $L_{c_R}^2$ bound on $\p_{c_R} y_B(0)$. The convergence of $\p_{c_R} y_B$ is a direct consequence of those of $\tilde y_1$, $\tilde y_2$, $\zeta_{1\pm}$, and the representation formula given in Lemma \ref{L:Ray-BVP-1}. Moreover, we also have the convergence of $D_c y_B'|_{x_2 =0, -h}$ in $L_{c_R}^r$ for any $r \in [1, 2)$. 

To complete the estimates on $\p_{c_R} y_B$ and also for the next step, we also need the following inequalities which are also derived from the above estimates and Lemma \ref{L:Ray-BVP-1} 
\begin{align*}
|D_c y_B'(-h)|_{L_{c_R}^2} \le &C \big(\mu^{-1} |\psi_0|_{L^2} +|\psi_0|_{L^r} + |\psi_0'|_{L^2} +\mu^{2-\ep} (|\xi_1| +|\xi_2|) \big) + C \mu^{-2}|\xi_-|\\
\le & C\mu^{-\ep} \big( \mu^2 (|\xi_1| +|\xi_2|) + \mu^{-1} |\psi_0|_{L^2}  + |\psi_0'|_{L^2} \big)+ C \mu^{-2}|\xi_-|,  
\end{align*}
\begin{align*}
|D_c y_B'(0)|_{L_{c_R}^2} \le &C  \big(\mu^{-1} |\psi_0|_{L^2} +|\psi_0|_{L^r} +\mu |\psi_0 (0)| + |\psi_0'|_{L^2} +|\xi_1| +|\xi_2|\big)\\
\le & C  \big(|\xi_1| +|\xi_2| + \mu^{-1-\ep} |\psi_0|_{L^2} +\mu^{-\ep}|\psi_0'|_{L^2} + \mu^{-\ep} |\xi_-| \big),
\end{align*}
where the terms involving $|\psi_0 (0)|$ and $|\psi_0|_{L^r}$, $r>2$, are bounded by other norms of $\psi_0$. 

$\bullet$ {\bf Estimating $\p_{c_R}^2 y_B$.} 
In order to analyze $\p_{c_R}^2 y_B$, we still first apply $D_c$ to \eqref{E:Dc1-Ray}. 
Due to the commutativity \eqref{E:Dc} between $D_c$ and $\p_{x_2}$, the Rayleigh equation \eqref{E:Ray-NH-2} and \eqref{E:Ray-BC-2} imply  
\begin{subequations} \label{E:Dc2-Ray}
\begin{align*} 
- (D_c^2 y_B)'' + \big(k^2 + \frac {U''}{U-c}\big)  D_c^2 y_B = &\frac {\psi_0'' - U^{(4)}y_B - 2 U''' D_c y_B}{U-c} + 2D_c \big( \frac 1{U-c}\big) \big( \psi_0' - D_c (U'' y_B)\big) \\
&+ D_c^2 \big( \frac 1{U-c}\big) (\psi_0 - U'' y_B). 
\end{align*}
We can write 
\be \label{E:Dc2-Ray-NH}
- (D_c^2 y_B)'' + \big(k^2 + \frac {U''}{U-c}\big)  D_c^2 y_B = \frac {\psi_0''(x_2)  + f_2 (c, x_2) \psi_0 (x_2) + 2 f_1 (c, x_2) \psi_0' (x_2)+ \psi_2(c, x_2)}{U-c},  
\ee
where $f_1$ was defined in \eqref{E:Dc1-Ray} and 
\[
f_2 = (U-c)D_c^2 \big( \frac 1{U-c}\big), \quad 
\psi_2 = - (2U'''+ U'' f_1 ) D_c y_B - (U^{(4)} + U''' f_1  +U'' f_2 ) y_B.   
\]
From \eqref{E:Dc-1} and the assumption $U \in C^4$, it holds $f_2$ and $f_3$ are $C^1$ in $x_2$ and $c_R$ with bounds uniform in $|c_I|\le \rho_0$. At $x_2=-h$, one can compute using \eqref{E:pc-x2c},  
\[
(D_c^2 y_B)(-h) = \big(U'(x_2^c)^2 \p_{c_R}^2 y_B + U''(x_2^c) \p_{c_R} y_B + U'(x_2^c) \p_{c_R} y_B' + (D_c y_B)' \big)\big|_{x_2=-h}. 
\]  
From \eqref{E:Ray-BC-2} and \eqref{E:Ray-NH-2}, we can write 
\be \label{E:Dc2-Ray-BC-1} 
(D_c^2 y_B)(c, -h) = \zeta_{2-} (c) \triangleq \big( 2(D_c y_B)' - y_B'' \big)(-h) = 2(D_c y_B)' (-h) + \frac {\psi_0(-h)}{U(-h) -c}.
\ee
At $x_2=0$, we write 
\be \label{E:Dc2-Ray-BC-2} 
\big(U(0)-c\big)^2 (D_c^2 y_B)' (0) - \big(U'(0) (U(0)-c) + g + \sigma k^2\big) D_c^2 y_B(0) = \zeta_{2+} (c).
\ee
\end{subequations} 
One may compute $\zeta_{2+}$ using \eqref{E:Dc} and \eqref{E:Dc1-Ray-BC-2} 
\begin{align*}
\zeta_{2+} =& U'(x_2^c)  \big(\p_{c_R} \zeta_{1+} + 2 (U(0)-c) (D_c y_B)' (0) - U'(0) D_c y_B(0) \big) \\
&+ \big(U(0)-c\big)^2 (D_c y_B)'' (0) - \big(U'(0) (U(0)-c) + g + \sigma k^2\big) (D_c y_B)' (0). 
\end{align*}
On the one hand, the $U'(x_2^c)\p_{c_R} \zeta_{1+}$ turns out to involve some of the most singular terms in $\zeta_{2+}$,  
\begin{align*}
U'(x_2^c) \p_{c_R} \zeta_{1+}=& -\xi_2 U''(x_2^c) + U'(x_2^c) \psi_0 (0) + \big( 2U''(x_2^c) (U(0)-c) - U'(x_2^c) (2U'(x_2^c) - U'(0) )\big) y_B' (0) \\
&+ U'(x_2^c) \big( (2U'(x_2^c) - U'(0)) (U(0)-c) -g -\sigma k^2\big) \p_{c_R} y_B'(0) \\
&+ U'(x_2^c)\big( k^2 (U(0)-c)^2  + U''(0) (U(0)-c) - U'(x_2^c) U'(0) \big) \p_{c_R} y_B(0) \\
&+ \big( U'(x_2^c)(2k^2 (c-U(0)) - U''(0)) - U''(x_2^c) U'(0)\big) y_B(0).
\end{align*}
We shall use \eqref{E:Dc} to replace $\p_{c_R} y_B$ and $\p_{c_R} y_B'$ by $D_c y_B$ and $(D_c y_B)'$, the latter of which would produce $y_B''(0)$. All those $y_B''(0)$ multiplied by $U(0)-c$ can be substituted by \eqref{E:Ray-NH-2}, but we keep other $y_B''(0)$ terms in the expression. On the other hand, we use \eqref{E:Dc1-Ray-NH} to substitute $(D_c y_B)''(0)$ in $\zeta_{2+}$, which turns out to be rather regular due to the multiplier $(U(0)-c)^2$. Finally, we can write 
\begin{align*}
\zeta_{2+} =& -\xi_2 U''(x_2^c) + f_3(c) \psi_0 (0) + f_4( c) \psi_0'(0) + f_5(k, c) y_B (0) +  f_6(k, c) y_B' (0) +  f_7(k, c) D_c y_B (0) \\
&+  f_8(k, c) (D_c y_B)' (0) +  ( g+ \sigma k^2) y_B''(0), 
\end{align*}
where the functions $f_j(k, c, x_2)$, $j=3, \ldots, 8$, are 
\begin{align*}
& f_3 =(c -U(0)) f_1 + 3 U'(x_2^c) - U'(0), \quad f_4 = c - U(0), \\
& f_5 = k^2 \big((4U'(x_2^c) - U'(0)) (c-U(0)) \big) +  U'''(0) (U(0)-c) -2 U''(0) U'(x_2^c) - U'(0) U''(x_2^c),  \\
& f_6 =-k^2 (U(0)-c)^2 + (2U''(x_2^c) -U''(0)) (U(0)-c) - 2U'(x_2^c) (U'(x_2^c) - U'(0) ), \\
& f_7 = 2\big(  k^2 (U(0)-c)^2  + U''(0) (U(0)-c) - U'(x_2^c) U'(0) \big), \\
& f_8= 2\big( (2U'(x_2^c) - U'(0)) (U(0)-c) -g -\sigma k^2\big),
\end{align*}
and are at least $C^1$ in $c_R$ and $x_2$.  

The terms $y_B''(-h)$ in $\zeta_{2-}$ and $y_B''(0)$ in $\zeta_{2+}$ generate the most singular part of $D_c^2 y_B$ which, based on Lemma \ref{L:Ray-BVP-1}, takes the form 
\[
y_S (x_2) = -\frac {y_B''(-h)}{y_+(-h)} y_+(x_2) + \frac { ( g+ \sigma k^2) y_B''(0) 
}{\BF(k, c)} y_-(x_2) = \frac { g+ \sigma k^2}{\BF(k, c)}\big( -y_B''(-h) y_+(x_2) +  y_B''(0)  y_-(x_2)\big). 
\]
Let 
\[
\tilde y = D_c^2 y_B - y_S.  
\]
Clearly, it satisfies the same non-homogeneous Rayleigh equation \eqref{E:Dc2-Ray-NH} and boundary conditions 
\be \label{E:Dc2-Ray-BC-3} 
\tilde y (c, -h) = \tilde \zeta_{2-} (c) \triangleq  2 (D_c y_B)' (c, -h)   
\ee
\be \label{E:Dc2-Ray-BC-4} 
\big(U(0)-c\big)^2 \tilde y' (0) - \big(U'(0) (U(0)-c) + g + \sigma k^2\big) \tilde y (0) = \tilde \zeta_{2+} (c) \triangleq \zeta_{2+} -  ( g+ \sigma k^2) y_B''(0).    
\ee
Let $\tilde y_3$ and $\tilde y_4$ be the solutions to \eqref{E:Dc2-Ray-NH}  with zero boundary values in \eqref{E:Ray-BC-1} and non-homogeneous terms 
\[
\frac {\psi_2}{U-c}, \quad \frac {\psi_0'' + f_2 \psi_0 + 2f_1 \psi_0'}{U-c},  
\]
respectively. Using the above estimates of $y_B$ and $D_c y_B$ and applying Lemma \ref{L:y_nh0-2-esti}, we obtain
\begin{align*}
|\tilde y_3|_{L_{c_R, x_2}^2} + \mu |\tilde y_3'|_{L_{c_R, x_2}^2} \le & C\mu \big(|y_B|_{L_{c_R, x_2}^2} + |D_c y_B|_{L_{c_R, x_2}^2} + \mu^{1-\ep} |y_B'|_{L_{c_R, x_2}^2} + \mu^{1- \ep} |(D_c y_B)'|_{L_{c_R, x_2}^2} \big) \\
\le & C\mu^{\frac 52-\ep} \big( |\xi_1| + |\xi_2| + \mu^{-1} |\psi_0|_{L^2} + |\psi_0'|_{L^2} + \mu^{-2} |\xi_-|\big). 
\end{align*}
As $c_I \to 0+$, the convergence of $y_B$ and $D_c y_B$ implies that of $\psi_2$ in $L_{c_R}^r W_{x_2}^{1,r}$ for any $r\in [1, 2)$. From Lemma \ref{L:y_nh-esti-2}(2a), we obtain the convergence of $\tilde y_3$ in $L_{c_R}^r L_{x_2}^\infty$.  

Again we apply Lemma \ref{L:y_nh0-3-esti} 
to  estimate $\tilde y_4$
\[
|\tilde y_4|_{L_{c_R, x_2}^2} + \mu |\tilde y_4'|_{L_{c_R, x_2}^2} \le C\mu^{\frac 32-\ep} \big(  |\psi_0|_{L^2} + |\psi_0'|_{L^2} + |\psi_0''|_{L^2} \big). 
\]
As $c_I \to 0+$, Lemma \ref{L:y_nh0-3-esti}(2a)  implies that $\tilde y_4$ converges in $L_{x_2}^\infty L_{c_R}^2$. 

The boundary values of $\tilde y$ satisfy 
\[
|\tilde \zeta_{2-}|_{L_{c_R}^2} \le C\mu^{-\ep} \big( \mu^2 (|\xi_1| +|\xi_2|) + \mu^{-1} |\psi_0|_{L^2}  + |\psi_0'|_{L^2} \big) + C \mu^{-2} |\xi_-|,  
\]
\begin{align*}
|\tilde \zeta_{2+}|_{L_{c_R}^2} \le &C \big( |\xi_2| + |\psi_0(0)| + |\psi_0'(0)| + \mu^{-2} \big( |y_B |_{L_{c_R}^2} + |D_c y_B |_{L_{c_R}^2} + |y_B' |_{L_{c_R}^2} + |(D_c y_B)' |_{L_{c_R}^2} \big)\big|_{x_2=0} \big)\\
\le & C \big( |\psi_0'(0)| + \mu^{-2} \big(
|\xi_1| +|\xi_2| + \mu^{-1-\ep} |\psi_0|_{L^2} +\mu^{-\ep}|\psi_0'|_{L^2} +\mu^{-\ep} |\xi_-|\big)\big),
\end{align*}
where we also used the boundary conditions of $y_B$ and $D_c y_B$ to express them in terms of $y_B'$ and $D_c y_B'$ at $x_2=0$. As $c_I \to 0+$, the convergence of $y_B$ and $D_c y_B$ at $x_2 =0, -h$ implies that of $\xi_\pm$ in $L_{c_R}^r$ for any $r\in [1, 2)$.

As $\tilde y_3 + \tilde y_4$ plays the role of "$y_{nh}$" in the representation of $D_c^2y_B$ as given in Lemma \ref{L:Ray-BVP-1}, the above estimates 
and Lemma \ref{L:Ray-BVP-1} imply 
\[
|\tilde y|_{L_{c_R, x_2}^2} + \mu |\tilde y'|_{L_{c_R, x_2}^2} \le C \big(  \mu^{\frac 12} |\xi_1| + \mu^{\frac 12} |\xi_2| + \mu^{-\frac 12-\ep}|\psi_0|_{L^2} + \mu^{\frac 12-\ep}|\psi_0'|_{L^2} + \mu^{\frac 32 -\ep}|\psi_0''|_{L^2} \big), 
\]
where the $\psi_0'(0)$ term was bounded by the other norms of $\psi_0$ via interpolation. Finally, using \eqref{E:Dc} one can compute 
\be \label{E:Dc-2}
\p_{c_R}^2 = U'(x_2^c)^{-2} (D_c^2 - 2\p_{x_2} D_c + \p_{x_2}^2)  - (U'(x_2^c))^{-3} U''( x_2^c) (D_c - \p_{x_2}). 
\ee
This relationship and the definition of $\tilde y_B$ and $\tilde y$ yield
\begin{align*}
\tilde y_B =& \p_{c_R}^2 y_B + \frac 1{U'(x_2^c)^2} \Big( - y_B'' + \frac { g+ \sigma k^2 
}{\BF(k, c)}  \big( y_B'' (-h) y_+ - y_B''(0) 
 y_-\big) \Big)\\
= & \frac 1{U'(x_2^c))^2} \Big(\tilde y - 2 (D_c y_B)' - \frac{U''(x_2^c)}{U'(x_2^c)} (D_c y_B - y_B')\Big).   
\end{align*}
Therefore the desired estimate on $\tilde y_B$ follows from those of $\tilde y$, $y_B$, and $D_c y_B$. The convergence of $\tilde y_B$ is also obtained much as that of $D_c y_B$. 
\end{proof}

\section{Solutions to the Euler equation linearized at shear flows} \label{S:Linear}

In this section, we finally return to the linearized flow of the capillary gravity water waves at the shear flow $U(x_2)$ in both the horizontally $L$-periodic (in $x_1$) case and the $x_1 \in \R$ case. 
Under the assumption \eqref{E:no-S-M} of the absence of singular modes for all $k$, we shall show that a.) inviscid damping occurs to a large component (remotely related to the rotational part) of the solutions and b.) what is left in the solutions are superpositions of non-singular modes (smooth eigenfunctions). The latter is a linear dispersive flow which is asymptotic to the linear irrotational flow for high spatial wave numbers $k$.

\subsection{Estimating each Fourier mode of the linear solutions} \label{SS:integral formulas}


Based on \eqref{E:Ray} and the formula of the inverse Laplace transform, we first derive some integral representation formulas of the linear solution $\big(\hat v (t, k, x_2), \hat \eta (t, k, x_2)\big)$ of \eqref{E:LE-F} for a fixed wave number $k \ne 0$ satisfying \eqref{E:F_0}. This procedure is essentially obtaining the linear solution group from contour integrals of the resolvents of the linear operator defined by the linearized water wave problem at the shear flow. Subsequently estimates of solutions are obtained using these formulas. Due to the conjugacy relation $\hat v(t, -k, x_2)= \overline {\hat v(t, k, x_2)}$ and $\hat \eta (t, -k)= \overline {\hat \eta(t, k)}$, we shall mostly work on estimates for $k>0$ in this subsection, unless otherwise specified. 

Recall $\BF$ defined in \eqref{E:BF}. Denote the set of non-singular modes 
\be \label{E:R(k)}
R(k) = \{ c \notin U([-h, 0]) \mid \BF(k, c)=0\}
\ee 
Throughout this subsection, we fix $k \ne 0$ and assume \eqref{E:F_0}. We shall also use \eqref{E:h-0}, possibly after choosing smaller $\rho_0$. The continuity of $\BF$ and \eqref{E:F_0} imply that $R(k)$ is a finite set, which consists of only simple roots $c^\pm (k)$ for large $k$ due to Lemma \ref{L:ev-large-k}(3).  
We shall work on the following type of neighborhoods of $U([-h, 0]) \subset \C$
\be \label{E:CD}
\CD_{r_1, r_2} =[-r_1 +U(-h), U(0)+ r_1] + i [-r_2, r_2] \subset R(k)^c, \quad r_1, r_2 \in (0, \rho_0), 
\ee     
where $\rho_0$ is given in \eqref{E:F_0}. 

Recall the Laplace transform $V_2 (k, c, x_2)$ of $\hat v_2(t, k, x_2)$, defined by \eqref{E:LT-1} and \eqref{E:s-c}, is the solution of the boundary value problem \eqref{E:Ray} of the Rayleigh equation, or equivalently, the solution to \eqref{E:Ray-BVP} and \eqref{E:Ray-BC-2} with 
\be \label{E:Ray-data-1}
\psi  =-\hat \omega_0(k, x_2) = -i k^{-1} (k^2 -\p_{x_2}^2) \hat v_{20}, \;\; \xi_-=0, \;\; \xi_1 = (g+ \sigma k^2) \hat \eta_0 (k), \;\; \xi_2 = -i k^{-1} \hat v_{20}' (k, 0),
\ee
and $\hat \omega_0(k, x_2)$, $\hat \eta_0(k)$ and $\hat v_{20} (k, x_2)$ are the Fourier transforms with respect to $x_1$ of the initial values $\omega_0(x)$, $\eta_0 (x_1)$ and $v_{20} (x)$. 
The solution $V_2(k, c, x_2)$ to \eqref{E:Ray} is still given by Lemma \ref{L:Ray-BVP-1} along with \eqref{E:Ray-BC-2} and \eqref{E:Ray-data-1}. More explicitly, if $\BF(k, c)\ne 0$, then 
\be \label{E:V2} \begin{split}
V_2(k, c, x_2) =  & \frac {(g+ \sigma k^2) \hat \eta_0 (k) -\frac ik (U(0)-c) \hat v_{20}' (k, 0) }{\BF(k, c)} y_-(k, c, x_2)  + y_{nh} (k, c, x_2),
\end{split} \ee 
where 
$y_\pm$ are solutions to the homogeneous Rayleigh equation \eqref{E:Ray-H0-1} satisfying initial conditions \eqref{E:y-pm}   
and $y_{nh}$ the solution to \eqref{E:Ray-BVP} 
given by \eqref{E:y_nh-0} with  $\zeta_\pm =0$ and $\psi=\hat \omega_0(k, x_2)$. 
The Laplace transform $\tilde \eta (k, c)$ of $\hat \eta (t, k)$ can be computed by using \eqref{E:t-eta} and the boundary condition \eqref{E:Ray-BC-1} along with \eqref{E:Ray-BC-2}, \eqref{E:Ray-data-1}, and \eqref{E:h-0}  
\be \label{E:LT-eta} \begin{split}
\tilde \eta (k, c) = & \frac {V_2(k, c, 0) +\hat \eta_0 (k)}{ik(U(0)-c)} 
= \frac { V_2'(k, c, 0) (U(0)-c) + U'(0)\hat \eta_0 (k) + \frac ik \hat v_{20}'(k, 0)   }{ik \big( U'(0)(U(0)-c) + g + \sigma k^2\big)}, \quad k \ne 0.
\end{split} \ee
We shall also need the following quantities   
\be \label{E:Bb} \begin{split}
&\Bb(t, k, c_*, x_2) = - (ik) Res \big( V_2 e^{-ik(c-c_*)t}, c_*\big), \\ 
& \Bb_S (t, k, c_*) = - (ik) Res \big( \tilde \eta e^{-ik(c-c_*)t}, c_*\big) = - Res \big( V_2(k, c, 0) e^{-ik(c-c_*)t}/(U(0)-c), c_*\big)
\end{split} \ee
where $Res(f(z), z_*)$ is the residue of a meromorphic function $f(z)$ at $z_*$. Apparently $\Bb= \Bb_S =0$ unless $\BF(k, c_*) =0$, or equivalently $c_* \in R(k)$. 
The following lemma is obtained from applying the inverse Laplace transform. 

\begin{lemma} \label{L:integralF}
Assume $U\in C^3$ and $k > 0$ satisfies \eqref{E:F_0}, then 
for any $r_1,r_2\in (0, \rho_0)$, we have 
\[
\hat v_2 (t, k, x_2) = \hat v_2^c + \hat v_2^p \triangleq - \frac k{2 \pi } \oint_{\p \CD_{r_1, r_2}}  e^{-ik ct} V_2(k, c, x_2)  d c + \sum_{c_* \in R(k)} 
e^{-ic_* kt} \Bb(t, k, c_*, x_2), 
\]
\[
\hat \eta (t, k) = \hat \eta^c + \hat \eta^p \triangleq - \frac k{2 \pi } \oint_{\p \CD_{r_1, r_2}}  e^{-ik ct} \tilde \eta (k, c)  d c + \sum_{c_* \in R(k)} 
e^{-ic_* kt} \Bb_S (k, c_*).
\]
\end{lemma}    

From Lemma \ref{L:ev-large-k}, $c_* \in R(k)$ implies $y_-(k, c, 0)\ne 0$ and thus $F(k, c)$ is well-defined for $c$ near $c_*$. In part (2), similar types of formula and estimates of $\Bb_S$ can be obtained from those of $\Bb$ and \eqref{E:Bb}. 
In the subsequent analysis, the limits of the above contour integrals as $\CD_{r_1, r_2}$ shrinks to $U([-h, 0])$ will be taken and estimated whenever needed.  
  
\begin{proof} 
From the definition \eqref{E:LT-1} and the inverse Laplace transform formula \eqref{E:ILT}, we have 
\[ 
\hat v_2 (t, k, x_2) = \frac k{2 \pi }\int_{ -\infty +i \gamma}^{ +\infty+ i\gamma}  e^{-ik ct} V_2(k, c, x_2)  d c, 
\]
where $\gamma >0$ is chosen such that the above integrand is analytic for $c_I > \gamma$. 
Apparently $V_2$ is analytic in $c \notin \big(U([-h, 0]) \cup \{ \BF=0\}\big)$. In order to analyze $V_2$ for $|c|\gg1$, we first consider  $y_+$ and then $y_{nh}$ for $|c|\gg1$. From Lemma \ref{L:regular-small-k} and initial conditions \eqref{E:y-pm}, it holds that 
\[
0< \liminf_{|c| \to \infty} |y_+(k, c, x_2)|/( 1+ |c|^2) \le \limsup_{|c| \to \infty} |y_+(k, c, x_2)|/( 1+ |c|^2) <\infty.
\]  
Along with \eqref{E:Ray-data-1} and Lemma \ref{L:y-pm} which yields the boundedness of $y_-$ for $|c|\gg1$, it implies 
\[
\limsup_{|c| \to \infty} |c| |y_{nh}(k, c, x_2)| < \infty. 
\]
From Lemma \ref{L:ev-large-k}(2) and again Lemma \ref{L:y-pm}, we obtain\footnote{Through a more careful analysis we may obtain  a Taylor expansion   of $V_2$ in terms of $\frac 1c$ as $|c|\to \infty$.
} 
\[
\limsup_{|c| \to \infty} |c| |V_2 (k, c, x_2)| < \infty. 
\]
As $|e^{-ikct}| = e^{kt \IP\, c}$, the Cauchy integral theorem yields     
\begin{align*} 
\int_{ -\infty -i \gamma}^{ +\infty- i\gamma}  e^{-ik ct} V_2(k, c, x_2)  d c =0 \; \text{ and }\;  \hat v_2 (t, k, x_2) = \frac k{2 \pi } \Big( \int_{ -\infty +i \gamma}^{ +\infty+ i\gamma}  -  \int_{ -\infty -i \gamma}^{ +\infty- i\gamma}  \Big) e^{-ik ct} V_2(k, c, x_2)  d c.
\end{align*}
The desired expression of $\hat v_2$ follows immediately from the residue calculation. 

Concerning $\hat \eta$, one first obtains
\[
\hat \eta(t, k) = \frac k{2 \pi } \Big( \int_{ -\infty +i \gamma}^{ +\infty+ i\gamma}  -  \int_{ -\infty -i \gamma}^{ +\infty- i\gamma}  \Big) e^{-ik ct}  
\tilde \eta (k, c) dc.  
\]
Using the first expression in \eqref{E:LT-eta}, the desired formula for $\hat \eta$ is derived via the same arguments as in the above. In particular, the $\hat \eta_0$ term does not contribute to the residue as $R(k)$ is away from $U([-h, 0])$ due to assumption \eqref{E:F_0}. 
\end{proof}

From the divergence free condition on the velocity, it holds that the Fourier transform (in $x_1$) of the velocity field satisfies $ik \hat v_1 =- \hat v_2'$. Therefore, we have  

\begin{corollary} \label{C:integralF}
Under the assumptions of Lemma \ref{L:integralF}, we have 
\[
\hat v_1 (t, k, x_2) = \hat v_1^c + \hat v_1^p \triangleq - \frac i{2 \pi } \oint_{\p \CD_{r_1, r_2}}  e^{-ik ct} V_2'(k, c, x_2)  d c +  \sum_{c_* \in R(k)
} \frac ik e^{-ic_* kt} \Bb' (t, k, c_*, x_2).  
\]
\end{corollary} 

In the following lemma, we give some basic properties of $\Bb(t, k, c, x_2)$ and $\Bb_S (t, k, c)$ at some $c_*\in R(k)$. Since $c_*$ is away from $U([-h, 0])$ and $\BF(k, \cdot)$ and $F(k, \cdot)$ are analytic in a neighborhood of $c_*$, the assumption \eqref{E:F_0} is not needed. 

\begin{lemma} \label{L:Bb} 
Assume $U\in C^{l_0}$, $l_0\ge 3$, and $k > 0$. 
Let $c_*\in R(k)$ be a root of $\BF(k, \cdot)$ (or equivalently, of $F(k, \cdot)$ defined in \eqref{E:dispersion}) of degree $n\ge 1$, then the following hold. 
\begin{enumerate} 
\item $e^{-ik c_* t} \Bb(t, k, c_*, x_2)$ is a solution to \eqref{E:LEuler-v2}.     
\item $\Bb (t, k, c_*, x_2)$ is a linear combination of $t^{l_1} \p_c^{l_2} y_-(k, c_*, x_2)$, $0\le l_1 + l_2 \le l=n-1$, and $\Bb_S(t, k, c_*)$ a linear combination of $t^l$, $0\le l \le n-1$, with coefficients depending on $k$ and $c_*$. The leading terms of $\Bb(t, k, c_*, x_2)$ with $l_1+l_2= n-1$ are given by 
\be \label{E:Bb-2} \begin{split}
\frac { (n!) (-ik)^{l_1+1}}{l_1! l_2! \p_c^n F(k, c_*)} \Big( &(g+ \sigma k^2) \hat \eta_0 (k) -\frac ik (U(0)-c_*) \hat v_{20}' (k, 0) \\
& +\frac  {(U(0)-c_*)^2 }{y_- (k, c_*, 0) } \int_{-h}^0 \frac {(y_-  \hat \omega_0) (k, c_*, x_2')}{U(x_2') -c_*}  dx_2' \Big) \frac {t^{l_1} \p_c^{l_2} y_-(k, c_*, x_2) }{y_-(k, c_*, 0) },   
\end{split} \ee
and the leading terms of $ik (U(0)-c_*)\Bb_S(t, k, c_*)$ is given by the above expression evaluated at $x_2=0$.  
\item If $c_*$ is a simple root of $\BF(k, \cdot)$, i.e., $n=1$, then $\Bb$ and $ik (U(0)-c_*)\Bb_S$ are given by the  above expression  and 
there exists $C>0$ determined only by $|U'|_{C^{l_0-1}}$ and $|(U')^{-1}|_{C^0}$ such that  
\begin{align*}
|\p_{x_2}^{n_2} \Bb (k, c_*, x_2)| \le C |\p_c F(k, c_*)|^{-1} \Big( & |k|\mu^{-2} |\hat \eta_0(k)| + (1+ |c_*|) |\hat v_{20}' (k, 0)| \\
& + \frac {|k| \mu^{\frac 32} e^{\mu^{-1} h} (1+|c_*|^2)|\hat \omega_0(k)|_{L_{x_2}^2}}{dist (c_*, U([-h, 0])) |y_-(k, c_*, 0)|} \Big)\Big|\frac {\mu^{1-n_2} e^{\mu^{-1} (x_2+h)}}{y_-(k, c_*, 0)}\Big|,
\end{align*}
for any $n_2 \in [0, l_0]$, where we recall $\mu= (1+k^2)^{-\frac 12}$. 
\end{enumerate} 
\end{lemma} 

\begin{proof}
According to Lemma \ref{L:e-v-basic-1}(4), $\BF (k, c_*)=0$ implies $y_-(k, c_*, 0)\ne 0$ and thus $F(k, c)$ is analytic in $c$ for $c$ near $c_*$ and the degree of $c_*$ as a root of both $F(k, \cdot)$ and $\BF(k, \cdot)$ is $n\ge1$. By the definition of $R(k)$ and the analyticity of $F(k, \cdot)$, $c_*\in R(k)$ is an isolated root of $F(k, \cdot)$. Let $1\gg R>0$ such that there are no other roots of $F(k, \cdot)$ in the disk $B(c_*, R)$ centered at $c_*$ with radius $R$. Using the fact that $V_2(k, c, x_2)$ solves \eqref{E:Ray}, one may compute 
\begin{align*}
& (\p_t + ik U) (k^2 - \p_{x_2}^2) (e^{-ik c_* t} \Bb(t, k, c_*, x_2)) = \frac {\p_t + ik U}{2\pi} \oint_{\p B(c_*, R)} k e^{-ik c t} (k^2 - \p_{x_2}^2) V_2(k, c, x_2)dc \\
=& \frac {i k^2}{2\pi } \oint_{\p B(c_*, R)} e^{-ik c t} (U-c)  (k^2 - \p_{x_2}^2) V_2(k, c, x_2)dc = -  \frac {i k^2 U''}{2\pi } \oint_{\p B(c_*, R)} e^{-ik c t} V_2(k, c, x_2)dc \\
=& - ik U'' e^{-ik c_* t} \Bb(t, k, c_*, x_2),
\end{align*}
and thus \eqref{E:LEuler-v2-1} is satisfied. Similar calculation also proves the boundary condition \eqref{E:LEuler-v2-2} at $x_2=0$. The zero boundary value at $x_2 =-h$ is obvious from that of $V_2$ at $x_2 =-h$. Therefore statement (1) is proved. 

To analyze $\Bb$ in more details, let 
\[
F_1 (c) = (c-c_*)^{-n} F(k, c) \implies F_1 (c_*) = \p_c^n F(k, c_*)/(n!) \ne 0, 
\]
and 
\[
\tilde y(c, x_2) = y_+ (k, c, x_2) - \frac {(U(0)-c)^2}{(g+ \sigma k^2) y_-(k, c, 0)} y_-(k, c, x_2). 
\]
From the initial conditions \eqref{E:y-pm} of $y_\pm$, it is straight forward to verify 
\[
\tilde y(c, 0) =0, \quad \tilde y'(c, 0)= - \frac {F(k,c)}{g+ \sigma k^2} = O(|c-c_*|^n) \implies y_+ (x_2) = \frac {(U(0)-c)^2y_-(x_2)}{(g+ \sigma k^2) y_-(0)}  + O(|c-c_*|^n). 
\] 

Using the above expression to substitute $y_+(k, c, x_2)$ in the residue (in the definition of $\Bb$) and observing that the $O(|c-c_*|^n)$ term cancels the singularity of $y_+ (k, c, -h)$ for $|c-c_*|\ll1$ which results in an analytic function contributing nothing to the residue, we have 
\[
Res(y_{nh} (k, c, x_2)e^{-ik(c-c_*)t}, c_*) 
= Res \Big( \frac {(U(0)-c)^2 y_- (k, c, x_2) e^{-ik(c-c_*)t}}{(g+\sigma k^2) y_- (k, c, 0) y_+(k, c, -h) } \int_{-h}^{0} \frac {(y_-  \hat \omega_0) (k, c, x_2')}{U(x_2') -c}  dx_2', c_*\Big).
\]
From definitions \eqref{E:y_nh-0}, \eqref{E:BF}, and \eqref{E:dispersion}, of $y_{nh}$, $\BF$, and $F$, \eqref{E:Ray-data-1},  
we have   
\begin{align*}
&Res(y_{nh} (k, c, x_2)e^{-ik(c-c_*)t}, c_*) 
= Res \Big( \frac {(U(0)-c)^2 y_- (k, c, x_2)e^{-ik(c-c_*)t}}{(c-c_*)^n F_1(c)y_- (k, c, 0)^2 } \int_{-h}^0 \frac {(y_-  \hat \omega_0) (k, c, x_2')}{U(x_2') -c}  dx_2', c_* \Big)\\
=& \frac 1{(n-1)!} \p_c^{n-1} \Big( \frac {(U(0)-c)^2 y_- (k, c, x_2)e^{-ik(c-c_*)t}}{F_1(c)y_- (k, c, 0)^2 } \int_{-h}^0 \frac {(y_-  \hat \omega_0) (k, c, x_2')}{U(x_2') -c}  dx_2' \Big)\Big|_{c=c_*}\\
=& \sum_{l=0}^{n-1} \frac 1{l! (n-l-1)!} \p_c^{n-l-1} \Big( \frac {(U(0)-c)^2 e^{-ik(c-c_*)t}}{F_1(c)y_- (k, c, 0)^2 } \int_{-h}^0 \frac {(y_-  \hat \omega_0) (k, c, x_2')}{U(x_2') -c}  dx_2' \Big)\Big|_{c=c_*} \p_c^l y_- (k, c_*, x_2). 
\end{align*}
Therefore this residue is a linear combination of $t^{l_1} \p_c^{l_2} y_-(k, c_*, x-2)$, $0\le l_1 + l_2 \le n-1$, with coefficients depending on $k$ and $c_*$. The coefficients for  $l_1+l_2= n-1$ are given by 
\begin{align*}
&\frac { (\p_c^{l_1}e^{-ik(c-c_*)t})|_{c=c_*}}{l_1! l_2! t^{l_1}} \frac {(U(0)-c_*)^2 }{F_1(c_*)y_- (k, c_*, 0)^2 } \int_{-h}^0 \frac {(y_-  \hat \omega_0) (k, c_*, x_2')}{U(x_2') -c_*}  dx_2'  \\
= & \frac { (n!) (-ik)^{l_1} (U(0)-c_*)^2 }{l_1! l_2! \p_c^n F(k, c_*)y_- (k, c_*, 0)^2 } \int_{-h}^0 \frac {(y_-  \hat \omega_0) (k, c_*, x_2')}{U(x_2') -c_*}  dx_2'.  
\end{align*}
The contributions of the terms involving $\eta_0(k)$ and $\hat v_{20}'(k, 0)$ can be analyzed similarly (actually simpler as $y_+$ is not involved) and we obtain the desired statement (2) on the form of $\Bb$ and $\Bb_S$. 

If $c_* \in R(k)$ is a simple root of $\BF(k, \cdot)$, i.e. $n=1$, then $\Bb$ and $\Bb_s$ have only one term with $l_1=l_2=0$ and are constants in $t$ as given in statement (2). 
It along with Lemma \ref{L:y-pm} readily leads to its estimate. 
\end{proof} 
 
\begin{corollary} \label{C:v^c} 
$\hat v_2^c$ is also a solution to \eqref{E:LEuler-v2}. Moreover if $c_*$ is a simple root of $F(k, \cdot)$, then the corresponding eigenvalue $-ik c_*$ is algebraically simple in the subspace of the $k$-th Fourier modes.  
\end{corollary}   
 
Based on the above lemmas, it is natural to define 
\be \label{E:BX-k}
\BP(k, c_*): (\hat v_{0}, \hat \eta_0) \to \big(i \Bb(0, k, c_*, \cdot)'/k, \Bb(0, k, c_*, \cdot), \Bb_S (0, k, c_*)\big), \; 
\BX (k, c_*) = range (\BP(k, c_*)).
\ee 
The following lemma gives that $\BP(k, c_*)$ defines the invariant spectral projection to the eigenspace $\BX(k, c_*)$ of $-ikc_*$ spanned by $\p_c^l y_- (k, c_*, \cdot)$, $0\le l \le n-1$. 
  
\begin{lemma} \label{L:S-projection-1} 
Assume the same conditions as in Lemma \ref{L:Bb}, then 
\[
\BX (k, c_*) = span\big\{ \big(i \p_c^l y_-' (k, c_*, \cdot)/k, \p_c^l y_-(k, c_*, \cdot), \p_c^l y_- (k, c_*, 0)/(ik (U(0)-c_*) \big) \mid l=0, \ldots, n-1\},
\]
is an invariant subspace of \eqref{E:LE-F} and 
\[
\BP(k, c_*): (\hat v_{0}, \hat \eta_0) \to \big(i \Bb(0, k, c_*, \cdot)'/k, \Bb(0, k, c_*, \cdot), \Bb_S (0, k, c_*)\big)
\]
is an invariant projection operator of \eqref{E:LE-F} to $\BX(k, c_*)$ with 
\[
\ker \big(\Sigma_{c_*\in R(k)} \BP(k, c_*)\big) = \big\{ \big(\hat v^c(0, k, \cdot), \hat \eta^c(0, k)\big)  \mid \text{ all initial values } \hat v_0(k, \cdot), \eta_0 (k) \big\}. 
\]
\end{lemma} 


\begin{proof} 
The statement of the lemma is rather standard in the operator calculus and Laplace transform, while constructing solutions to \eqref{E:LE-F} using Laplace transform is equivalent to using contour integrals of the resolvent operators in the complex spectral plane. We shall only outline the proof and skip some details. 

Due to the translation invariance in $t$ of solutions to \eqref{E:LE-F}, the $t=0$ in the definition of $\BX(k, c_*)$ can be replaced by any $t\in \R$. From Lemma \ref{L:Bb}, all solutions $(i\Bb'/k, \Bb, \Bb_S)$ are polynomials of $t$ of degree no more than $n-1$. It is standard to show inductively that $\BX(k, c_*)$ consists of all possible coefficients of $t^l$, which can be computed to be generated by $\p_c^l y_- (k, c_*, \cdot)$, $0\le l \le n-1$, using \eqref{E:Bb-2} and the relationship between $\Bb$ and $\Bb_S$. The invariance of $\BX(k, c_*)$ under \eqref{E:LE-F} is due to the fact that $(i\Bb/k, \Bb, \Bb_s)$ are solutions to  \eqref{E:LE-F}. To show $\BP(k, c_*)^2 = \BP(k, c_*)$, let $(\hat u_0, \hat \nu_0) = \BP(k, c_*) (\hat v_0, \hat \eta_0) \in \BX(k, c_*)$. With this initial value, the solution $(\hat u(t), \hat \nu(t))$ is simply the $(i\Bb/k, \Bb, \Bb_S)$ component of the solution with the initial value $(\hat v_0, \hat \eta_0)$. Hence $(\hat u(t), \hat \nu(t))$ takes the form given in Lemma \ref{L:Bb}(2). Its Laplace transform is analytic at all $c \ne c_*$ and thus the $(i\Bb/k, \Bb, \Bb_S)$ component of $(\hat u(t), \hat \nu(t))$ is equal to itself. Therefore we obtain $\BP(k, c_*) (\hat u_0, \hat \nu_0) = (\hat u_0, \hat \nu_0)$. Finally the description of the kernel of $\sum_{c_*\in R(k)} \BP(k, c_*)$ is obvious due to the fact that both $(i\Bb/k, \Bb, \Bb_S)$ and $(\hat v^c (t), \hat \eta^c(t))$ are solutions. 
\end{proof} 
 
\begin{remark} \label{R:e-func}
In particular, if 
\[
\hat v_{20} (k, x_2) = y_- (k, c_*, x_2), \quad \hat \eta_0 = y_-(k, c_*, 0)/\big(ik (U(0)-c_*\big),
\] 
then straight forward verification yields 
\[
V_2 (k, c, x_2) = \frac {y_-(k, c_*, x_2)}{ik (c_*-c)}, \quad \hat v^c =0, \quad  \Bb  =y_-(k, c_*, x_2) = \hat v_{20}, \quad \Bb_S  = \frac {y_-(k, c_*, 0)}{ik (U(0)-c_*)} = \hat \eta_0.  
\]
From Lemma \ref{L:e-value}, $-ic_*k$ is an eigenvalue (with the above eigenfunctions generated by $y_-(k, c_*, x_2)$) of the linearized capillary gravity water wave at the shear flow, which has to be geometrically simple when restricted to the $k$-th Fourier mode in $x_1$. Its algebraic multiplicity is equal to the degree of the root $c_*$ of $\BF(k, \cdot)$. The eigenfunctions of the linearized irrotational capillary gravity wave are generated by $\frac 1k \sinh k(x_2+h)$. From Lemmas \ref{L:Ray-regular-1}(1) (with $\rho= O(k^{-\frac 52})$, $s=0$,  $C_0=0$, and $\Theta_1=\Theta_2 =\sinh$) and \ref{L:ev-large-k}(3), it is straight forward to estimate that, after normalizing the $L^2$ norm of $v_2$ to be $1$, the $L^2$ and $H^1$ differences in the $v$ and $\eta$ components, respectively, between the eigenfunctions of \eqref{E:LE-F} and the irrotational capillary gravity waves linearized at zero is of order $O(k^{-\frac 32})$ as $|k| \to \infty$.   
\end{remark} 

In the rest of this subsection we consider $\hat v^c(t, k, x_2)$ and $\hat \eta^c(t, k)$. We shall always work on $c \in [U(-h)-\rho_0, U(0) +\rho_0] + i [-\rho_0, \rho_0]$. We first present some properties of $V_2$ and $\tilde \eta$. Let us keep in mind that for analytic functions, $\p_c$ and $\p_{c_R}$ are equivalent.  

\begin{lemma} \label{L:V2-esti-1}
It holds that $V_2$ and $\tilde \eta$ are analytic in $c \in \C  \setminus \big(U([-h, 0]) \cup R(k)\big)$ and satisfy 
\[
V_2(-k, \bar c, x_2) = \overline{V_2(k, c, x_2)}, \quad \tilde \eta(-k, \bar c, x_2) = \overline{\tilde \eta (k, c, x_2)}.
\] 
Assume $U \in C^{l_0}$, $l_0\ge 3$, and \eqref{E:F_0}, 
then the following hold for some . 
\begin{enumerate} 
\item For any $\ep>0$, there exists $C>0$ determined only by $\ep$, $F_0$, $\rho_0$, $|U'|_{C^{l_0-1}}$, and $|(U')^{-1}|_{C^0}$ (independent of $k \in \R$) such that  for any $c_I \in [0, \rho_0]$,  
\begin{align*}
&|V_2|_{L_{c_R, x_2}^2} + \mu |V_2'|_{L_{c_R, x_2}^2} + \mu^{\frac 32} |V_2' (0)|_{L_{c_R}^2} + \mu^{\frac 12} |V_2 (0)|_{L_{c_R}^2} \\
\le & C \big( \mu^{\frac 12} |\hat \eta_0(k)| + |k|^{-1} \mu^{\frac 52} |\hat v_{20}' (k, 0)| + \mu^{\frac 32-\ep} |\hat \omega_0(k)|_{L_{x_2}^2} \big);
\end{align*} 
\[
|\tilde \eta|_{L_{c_R}^2} \le C \big( |k|^{-1} \mu |\hat \eta_0(k)| + |k|^{-2} \mu^2 |\hat v_{20}' (k, 0)| +|k|^{-1} \mu^{2-\ep} |\hat \omega_0(k)|_{L_{x_2}^2} \big), 
\]
if $l_0\ge 4$, then 
\[
|\p_{c_R} \tilde \eta|_{L_{c_R}^2} \le  C \big( |k|^{-1} |\hat \eta_0 (k)| + |k|^{-2} \mu^2 |\hat v_{20}' (k, 0)| + |k|^{-1} \mu^{1-\ep} |\hat \omega_0 (k)|_{L_{x_2}^2} + |k|^{-1} \mu^{2-\ep} |\hat \omega_0' (k)|_{L_{x_2}^2} \big),
\]
\begin{align*}
&|\p_{c_R} V_2|_{L_{c_R, x_2}^2} + \mu |\p_{c_R} V_2' + \tfrac {V_2''}{U'} |_{L_{c_R, x_2}^2} + \mu^{\frac 32}\big|\big(\p_{c_R} V_2' + \tfrac {V_2''}{U'}\big) (0)\big|_{L_{c_R}^2} + \mu^{\frac 12} |\p_{c_R}V_2 (0)|_{L_{c_R}^2} \\ 
\le & C \big(  \mu^{-\frac 12} |\hat \eta_0(k) | + |k|^{-1} \mu^{\frac 32} |\hat v_{20}' (k, 0)| + \mu^{\frac 12-\ep}|\hat \omega_0(k)|_{L_{x_2}^2} + \mu^{\frac 32-\ep}| \hat \omega_0'(k)|_{L_{x_2}^2}\big);
\end{align*}
and if $U\in C^5$, then 
\[
|\tilde V_2|_{L_{c_R, x_2}^2} \le  C\big(  \mu^{-\frac 32} |\hat \eta_0(k) | + |k|^{-1} \mu^{\frac 12} |\hat v_{20}' (k, 0) | + \mu^{-\frac 12-\ep}|\hat \omega_0(k) |_{L_{x_2}^2} + \mu^{\frac 12-\ep} |\hat \omega_0'(k)|_{L_{x_2}^2}+ \mu^{\frac 32 -\ep} | \hat \omega_0''(k)|_{L_{x_2}^2} \big), 
\]
where 
\[
\tilde V_2 (x_2) =\p_{c_R}^2  V_2(x_2) - \frac {V_2''(x_2)}{U'(x_2)^2} +  \frac {g+\sigma k^2}{\BF(k, c)} \Big( \frac {V_2'' (-h)}{U'(-h)^2} y_+(x_2) - \frac {V_2''(0)}{U'(0)^2 } y_- (x_2)\Big), 
\]
and all the norms are taken on $(c_R, x_2) \in [U(-h)-\rho_0,  U(0)+\rho_0] \times [-h, 0]$. 
\item As $c_I\to 0+$, on $[-r_1 + U(-h), r_1]$
\[
V_{20} (k, c_R, x_2) \triangleq \lim_{c_I \to 0+} V_2(k, c_R+ic_I, x_2), \quad \tilde \eta_{0} (k, c_R, x_2) \triangleq \lim_{c_I \to 0+} \tilde \eta (k, c_R+ic_I, x_2) 
\]
exist and the following hold. 
\begin{enumerate}
\item Assume $\hat \omega_0(k)\in L^2$ and $U \in C^3$, then for any $r\in [1, 2)$, $V_2 \to V_{20}$ in $L_{x_2}^\infty L_{c_R}^2$,  $V_2' \to V_{20}'$ in $L_{x_2}^\infty L_{c_R}^r$, and $V_2'(0) \to V_{20}'(0)$ and $\tilde \eta \to \tilde \eta_0$ in $L_{c_R}^2$. 
\item Assume $\hat \omega_0(k)\in H^1$  and $U \in C^4$, then for any $r\in [1, 2)$ and $q \in [1, \infty)$, $\p_{c_R} V_2 \to \p_{c_R} V_{20}$ in $L_{x_2}^\infty L_{c_R}^r$, $\p_{c_R} V_2' + \tfrac {V_2''}{U'} \to \p_{c_R} V_{20}' + \tfrac {V_{20}'' }{U'(x_2)}$ in $L_{x_2}^q L_{c_R}^r$, and $\big(\p_{c_R} V_2' + \tfrac { V_2''}{U'} \big)(0) \to \big(\p_{c_R} V_{20}' + \tfrac {V_{20}'' }{U'} \big)(0)$ and $\p_{c_R} \tilde \eta \to \p_{c_R} \tilde \eta_0$ in $L_{c_R}^r$. 
\item Assume $\hat \omega_0(k)\in H^2$  and $U \in C^5$, then for any $r\in [1, 2)$, $\tilde V_2$ converges  to its limit $\tilde V_{20}$ in $L_{x_2}^\infty L_{c_R}^r$.
\end{enumerate}    
\end{enumerate}
\end{lemma} 

Compared to Proposition \ref{P:pcy_B}, the modifications in the definition of $\tilde V_2$ is to make it analytic in $c$ which will make it more convenient in applying Lemma \ref{L:auxL-1} in the below. 

\begin{proof}
The estimates of $V_2$, $V_2'$, $V_2'(0)$, $\p_{c_R} V_2$, and their convergences are all direct corollaries of \eqref{E:Ray-data-1} and Proposition \ref{P:pcy_B}. 
The estimate of $\tilde \eta$ and its convergence follows from the second expression of \eqref{E:LT-eta} and the above properties of $V_2$. 

We also notice that, compared to Propositions \ref{P:pcy_B}, in the definition of $\tilde V_2$ as well as in the estimate related to $\p_{c_R} V_2'$, the $U'(x_2^c)$ in front of $V_2''$, $V_2''(-h)$, and $V_2''(0)$ had been replaced by $U'(x_2)$, $U'(-h)$, and $U'(0)$, respectively. This modification brings at most minor changes to the upper bounds. In fact,  
\[
\big|\big(U'(x_2^c)^{-n} - U'(x_2)^{-n}\big) V_2'' \big|_{L_{c_R, x_2}^2} \le C \big| |U(x_2) - c| |V_2''| \big|_{L_{c_R, x_2}^2} \le C \big(\mu^{-2} |V_2|_{L_{c_R, x_2}^2} + |\hat \omega_0 (k) |_{L_{x_2}^2}\big),   
\]    
for $n=1, 2$, where the Rayleigh equation was also used. This error bound and the estimate on $V_2$ are  then used to obtain the desired inequality on $\p_{c_R} V_2'$. The term $\frac {V_2''}{U'(x_2)^2}$ in $\tilde V_2$ is handled by the same argument. Similarly, 
\[
\big|\big(U'(x_2^c)^{-1} - U'(0)^{-1}\big) V_2''(0) \big|_{L_{c_R}^2} \le C \big| |U(0) - c| |V_2''(0)| \big|_{L_{c_R}^2} \le C \big(\mu^{-2} |V_2(0)|_{L_{c_R}^2} + |\hat \omega_0 (k, 0) |\big),   
\]    
and this along with the estimate on $V_2(0)$ yields the estimate on $\big(\p_{c_R} V_2' + \tfrac { V_2''}{U'} \big)(0)$. It remain the consider the modifications to the correction terms in $\tilde V_2$ at $x_2=-h$ and $x_2=0$. Similarly, 
\begin{align*}
& \big|\big(U'(x_2^c)^{-2} - U'(0)^{-2}\big) V_2'' (0)  y_- (x_2)\big|_{L_{c_R, x_2}^2} \le  C \big(\mu^{-2} |V_2(0)|_{L_{c_R}^2} + |\hat \omega_0 (k, 0) |\big) |y_-|_{L_{c_R}^\infty L_{x_2}^2}  
\end{align*}
which is controlled using $ |y_-|_{L_{c_R}^\infty L_{x_2}^2} \le C \mu^{\frac 32} e^{\frac h\mu}$ due to lemma \ref{L:y-pm}. The last remaining modification from $U'(x_2^c)^{-2}$ to $U'(-h)^{-2}$ can be justified by the same argument (even easier as $V_2(-h)=0$.)
 
Finally we consider $\p_{c_R} \tilde \eta$ for in $c\in \CD_{\rho_0, \rho_0}$. 
From \eqref{E:LT-eta}, one may compute 
\begin{align*}
&\p_{c_R} \tilde \eta = \tfrac 1{ik} \big( U'(0)(U(0)-c) + g + \sigma k^2\big)^{-2} \Big(
 \p_{c_R} V_2'( 0) (U(0)-c) \big( U'(0)(U(0)-c) + g + \sigma k^2\big) \\
 &- (g +\sigma k^2) V_2'(0)  + U'(0)\big( U'(0)\hat \eta_0 (k) + \tfrac ik \hat v_{20}'(k, 0)\big)\Big)\\  
=& \tfrac 1{ik} \big( U'(U-c) + g + \sigma k^2\big)^{-2} \Big( \big( \p_{c_R} V_2' + \tfrac {V_2''}{U'} - \tfrac 1{U'} \big(k^2 V_2+ \tfrac {U''V_2 + \hat \omega_0}{U-c}\big) \big)  (U-c) \big( U'(U-c) + g + \sigma k^2\big) \\
 &- (g +\sigma k^2) V_2'  + U'\big( U'\hat \eta_0 (k) + \tfrac ik \hat v_{20}'(k, 0)\big)\Big)\Big|_{x_2=0}
\end{align*}
where we used the Rayleigh equation \eqref{E:Ray-NH-2} in the last step. 
Therefore from Proposition \ref{P:pcy_B} we have,  for any $k\in \R$ and $c_I \in [0, r_2]$, 
\begin{align*}
&|\p_{c_R} \tilde \eta|_{L_{c_R}^2} \le C \big( |k|^{-1} \mu^4 |\hat \eta_0 (k)| + |k|^{-2} \mu^4 |\hat v_{20}' (k, 0)| + |k|^{-1} \mu^2 |V_{2}'(0)|_{L_{c_R}^2} + |k|^{-1} \mu^2 \big|\big(\p_{c_R} V_{2}' + \tfrac {V_2''}{U'}\big)(0)\big|_{L_{c_R}^2} \\
&\qquad\qquad \quad+ |k|^{-1} \mu^2 |(k^2(U(0)-c) + U''(0))V_{2}(0) + \hat \omega(k, 0) |_{L_{c_R}^2} \big) \\
\le & C \big( |k|^{-1} |\hat \eta_0 (k)| + |k|^{-2} \mu^2 |\hat v_{20}' (k, 0)| + |k|^{-1} \mu^{1-\ep} |\hat \omega_0 (k)|_{L_{x_2}^2} + |k|^{-1} \mu^{2-\ep} |\hat \omega_0' (k)|_{L_{x_2}^2}+  |k|^{-1} \mu^2 |\hat \omega_0 (k, 0)|\big).
\end{align*}
The last terms can be controlled by the previous two terms, which completes the estimate on $\p_{c_R} \tilde \eta$. The convergence of $\p_{c_R} \eta$ also follows from those of $V_2(0)$, $V_2' (0)$ and $\big(\p_{c_R} V_2' + \tfrac 1{U'(x_2^c)} V_2''\big)(0)$. 
\end{proof} 

The following lemma will be used in the decay estimates. 

\begin{lemma} \label{L:auxL-1}
Suppose  $n \ge 0$ is an integer,  $q \in [2, \infty]$, $f(c)$ and $f_1(c)$ are analytic functions on    
\[
\CD \setminus \CI_0 \subset \C, \; \text{  where } \;\CD = \CI + i [-\rho, \rho], \;\;  \CI_0 \subset (b_1, b_2), \;\;
\CI =[b_1, b_2] \subset \R, \;\; \rho>0, 
\]
and there exists $M>0$ such that $ |(f^{(n)} - f_1)(\cdot +ic_I) |_{L^{\frac q{q-1}}(\CI)} \le M$ for all $0< |c_I| \le \rho $, then there exists $C>0$ depending only on $b_2-b_1$ such that, for any $k\ne 0$,  
\[
\Big| \oint_{\p \CD} e^{-i ckt} \big( t^n f(c)  - (ik)^{-n} f_1(c) \big) dc\Big|_{L_t^q (\R)} \le C |k|^{-n-\frac 1q} M.    
\]
\end{lemma}

\begin{proof}
Integrating by parts we have, for any $0< |r| \le \rho$,  
\[
\int_{\CI + ir} t^n e^{-i ckt} f(c) dc = (ik)^{-n} e^{r kt} \int_{\CI} e^{-i c_R kt} f^{(n)}(c_R + ir) dc_R - e^{-i c kt} \sum_{l=1}^{n} t^{n-l} (ik)^{-l} f^{(l-1)} (c) \big|_{c=b_1 + ir}^{c=b_2 + ir}. 
\]
For any $T>0$, the $L^{\frac q{q-1}} \to L^q$ boundedness (for $q\in [2, \infty]$) of the Fourier transform  
implies 
\[
\Big|  \int_{\CI } e^{-i c_R kt} (f^{(n)} - f_1)(c_R + ir) dc_R \Big|_{L_t^q ([-T, T])} \le C |k|^{-\frac 1q} |(f^{(n)} - f_1) (\cdot + ir)|_{L^{\frac q{q-1}}(\CI)} \le C |k|^{-\frac 1q}M.  
\]
From this inequality and the Cauchy integral theorem, we obtain, for any $r\in (0, \rho]$,  
\begin{align*}
&\Big| \oint_{\p \CD} e^{-i ckt}  \big( t^n f(c)  - (ik)^{-n} f_1(c) \big) dc\Big|_{L_t^q ([-T, T])} \\
= & \Big|  \oint_{\p (\CI + i[-r, r])} e^{-i ckt}  \big( t^n f(c)  - (ik)^{-n} f_1(c) \big)  dc\Big|_{L_t^q ([-T, T])} \\
\le & C |k|^{-n-\frac 1q}e^{r |k|T} M + \Big|  \Big( \int_{b_1 +ir}^{b_1-ir} + \int_{b_2-ir}^{b_2+ir} \Big)e^{-i ckt} \big( t^n f(c)  - (ik)^{-n} f_1(c) \big)  dc \\
&+ \sum_{l=1}^{n} t^{n-l} (ik)^{-l} \big( e^{-ickt}f^{(l-1)} (c) \big|_{c=b_1 + ir}^{c=b_2 + ir} - e^{-ickt} f^{(l-1)} (c) \big|_{c=b_1 - ir}^{c=b_2 - ir}\big) \Big|_{L_t^q ([-T, T])}.
\end{align*}
Letting $r\to 0$, the analyticity assumption of $f$ and $f_1$ implies all those terms on the vertical boundary of $\CD$ vanish and the above estimates on the integrals along the horizontal edges yield 
\[
\Big| \oint_{\p \CD} t^n e^{-i ckt} \big( t^n f(c)  - (ik)^{-n} f_1(c) \big) dc \Big|_{L_t^q ([-T, T])} \le C |k|^{-n-\frac 1q}M.
\]
The lemma follows by letting $T \to +\infty$. 
\end{proof} 

\begin{remark} \label{R:contour}
In the following applications of this lemma, we often use the $L^2$ norm to control the $L^{\frac q{q-1}}$ norm. This leads to fact that the regularity requirements in $x_1$ (i.e. the exponents of $k$) may not be close to optimal. 
\end{remark} 

Applying the above lemma, we first obtain the decay of $\hat v^c (t, k, x_2)$ and $\hat \eta^c (t, k)$.

\begin{lemma} \label{L:mode-decay-1}
Assume $U \in C^{l_0}$, $l_0\ge 3$, and \eqref{E:F_0},  
then for any $\ep\in (0, 1)$, $q\in [2, \infty]$, and integer $m \ge 0$, there exists $C>0$ determined only by $\ep$, $q$, $m$, $F_0$, $\rho_0$, $|U'|_{C^{l_0-1}}$, and $|(U')^{-1}|_{C^0}$ (independent of $k \ne 0$) such that   
\begin{align*}
|\p_t^m \hat v_2^c(k)|_{L_{x_2}^2 L_t^q (\R)} & + |k|\mu |\p_t^m \hat v_1^c(k)|_{L_{x_2}^2 L_t^q (\R) } + \mu |\p_t^m (\hat v_2^c)' (k)|_{L_{x_2}^2 L_t^q (\R)} \\
\le & C |k|^{m+1 -\frac 1q}  \big( \mu^{\frac 12} |\hat \eta_0(k)| + |k|^{-1} \mu^{\frac 52} |\hat v_{20}' (k, 0)| + \mu^{\frac 32-\ep} |\hat \omega_0(k)|_{L_{x_2}^2} \big),
\end{align*}
\[
|\p_t^m \hat \eta^c(k)|_{L_t^q (\R)}  \le C |k|^{m-1 -\frac 1q} \big( |k| \mu |\hat \eta_0(k)| + \mu^2 |\hat v_{20}' (k, 0)| +|k| \mu^{2-\ep} |\hat \omega_0(k)|_{L_{x_2}^2} \big);
\]
and if $l_0\ge 4$,
\[
|t \p_t^m \hat v_2^c(k)|_{L_{x_2}^2 L_t^q (\R) } \le C |k|^{m-\frac 1q} \big(  \mu^{-\frac 12} |\hat \eta_0(k) | + |k|^{-1} \mu^{\frac 32} |\hat v_{20}' (k, 0)| + \mu^{\frac 12-\ep}|\hat \omega_0(k)|_{L_{x_2}^2} + \mu^{\frac 32-\ep}| \hat \omega_0'(k)|_{L_{x_2}^2}\big),
\]
\[
|t \p_t^m \hat \eta^c(k)|_{L_t^q (\R)}  \le C |k|^{m-2-\frac 1q} \big( |k| |\hat \eta_0 (k)| +  \mu^2 |\hat v_{20}' (k, 0)| + |k| \mu^{1-\ep} |\hat \omega_0 (k)|_{L_{x_2}^2} + |k| \mu^{2-\ep} |\hat \omega_0' (k)|_{L_{x_2}^2} \big).
\]
\end{lemma}

\begin{proof}
The estimates of $\p_t^m \hat v_2^c$, $t \p_t^m \hat v_2^c$, $\p_t^m (\hat v_2^c)'$, $\p_t^m \hat v_1^c$, $\p_t^m \hat \eta^c$, and $t \p_t^m \hat \eta^c$ are based on the definitions of $\hat v^c (t, k, x_2)$ and $\hat \eta^c(t, k)$ from direct application of Lemma \ref{L:V2-esti-1} and Lemma \ref{L:auxL-1} on $\CD_{\rho_0, \rho_0}$ with $f_1=0$ and $f$ being $c^m V_2$ (with $n=0, 1$), $c^m V_2'$, $c^m \tilde \eta$ (with $n=0,1$), respectively.  We omit the details.
\end{proof} 

In the following we shall focus on $t \p_t^m \hat v_1^c(t, k, x_2)$, $t^2 \p_t^m \hat v_2^c(t, k, x_2)$, and $\p_t^m \hat \omega^c (t, k , x_2)$, where $\hat \omega^c$ is the Fourier transform (in $x_1$) of the vorticity $\omega^c = \p_{x_1} v_2^c - \p_{x_2} v_1^c$ of $v^c(t, x)$. In order to characterize their asymptotic behavior, define 
\be \label{E:Omega-1} \begin{split}
\hat \Omega^c(k, x_2) =& \hat \omega_0 (k, x_2) \\
&+ \tfrac 12 U''(x_2)  \big( (1+ sgn(kt))V_{20} (k, U(x_2), x_2)  + (1-sgn(kt)) \overline{V_{20} (-k, U(x_2), x_2)}  \big). 
\end{split} \ee  
In the above expression, exactly one of $1+ sgn(kt)$ and $1- sgn(kt)$ is equal to $2$ and the other equal to $0$. The dependence of $\hat \Omega^c$ on $t$ is only through its sign, so we skipped specifying the $t$ dependence. We also notice that $V_2$ may not be $C^0$ at $c\in U([-h, 0]) \subset \C$. The available conjugacy properties of $V_2$ are not sufficient to imply $\overline{V_{20} (-k, U(x_2), x_2)} = V_{20} (k, U(x_2), x_2)$. We shall see that $\hat \Omega^c$ provides the asymptotic profile of the vorticity $\hat \omega^c$. We first give the following some basic properties of $\hat \Omega^c$. 

\begin{lemma} \label{L:Omega} 
Assume $U \in C^4$ and \eqref{E:F_0}, then $\hat \Omega^c (-k, x_2) = \overline {\hat \Omega^c (k, x_2)}$ and, for any $\ep\in (0, 1)$, there exists $C>0$ determined only by $\ep$, $F_0$, $\rho_0$, $|U'|_{C^3}$, and $|(U')^{-1}|_{C^0}$ (independent of $k \ne 0$) such that  
\[
|\hat \Omega^c - \hat \omega_0|_{L_{x_2}^2} \le C\big( |\hat \eta_0(k) | + |k|^{-1} \mu^2 |\hat v_{20}' (k, 0)| + \mu^{1-\ep}|\hat \omega_0(k)|_{L_{x_2}^2} ), 
\] 
\[
|(\hat \Omega^c)' - \hat \omega_0'|_{L_{x_2}^2} \le C\mu^{-1} \big( |\hat \eta_0(k) | + |k|^{-1} \mu^2 |\hat v_{20}' (k, 0)| + \mu^{1-\ep}|\hat \omega_0(k)|_{L_{x_2}^2} + \mu^{2-\ep}| \hat \omega_0'(k)|_{L_{x_2}^2}\big). 
\]
\end{lemma} 

\begin{proof}
The conjugacy relation of $\hat \Omega^c$ is clear from its definition. According to Lemma \ref{L:V2-esti-1}, $V_{20}$ satisfies the same estimates as $V_2$ for $|c_I| \in (0, \rho_0]$. 
We have, for $x_2\in [-h, 0]$ and $c \in U([-h, 0])$, 
\be \label{E:V2-L2L^in} \begin{split}
|V_{20}(k, U(\cdot), \cdot)|_{L_{x_2}^2} \le & C |V_{20}|_{L_{c_R, x_2}^2}^{\frac 12} |V_{20}|_{L_{c_R}^2 H_{x_2}^1}^{\frac 12} \\
\le & C\big( |\hat \eta_0(k) | + |k|^{-1} \mu^2 |\hat v_{20}' (k, 0)| + \mu^{1-\ep}|\hat \omega_0(k)|_{L_{x_2}^2}),   
\end{split} \ee
which implies the estimate of $\hat \Omega^c$. Apparently the estimate of $(\hat \Omega^c)'$ depends on that of 
\[
\p_{x_2} \big( V_{20} (k, U(x_2), x_2)\big) = (D_c V_{20}) (k, U(x_2), x_2), 
\]
where $D_c$ was defined in \eqref{E:Dc}. From \eqref{E:pcy_nh-temp-2} and \eqref{E:Ray-data-1}, we have 
\begin{align*}
\big|\p_{x_2} \big( V_{20} (k, U(\cdot), \cdot)\big)\big|_{L_{x_2}^2} \le & C | D_c V_{20}|_{L_{c_R, x_2}^2}^{\frac 12} |D_c V_{20}|_{L_{c_R}^2 H_{x_2}^1}^{\frac 12} \\
\le & C\mu^{-1} \big( |\hat \eta_0(k) | + |k|^{-1} \mu^2 |\hat v_{20}' (k, 0)| + \mu^{1-\ep}|\hat \omega_0(k)|_{L_{x_2}^2} + \mu^{2-\ep}| \hat \omega_0'(k)|_{L_{x_2}^2}\big),
\end{align*}
which yields and completes the proof of the lemma. 
\end{proof} 

In the following lemma, we obtain the leading order terms of $t\hat v_1^c$, $(\hat v_2^c)''$, and $\hat \omega^c$.  

\begin{lemma} \label{L:mode-decay-2}
Assume $U \in C^4$ and \eqref{E:F_0},  
then, for any $\ep\in (0, 1)$, $q\in (2, \infty]$, and integer $m \ge 0$, there exists $C>0$ determined only by $\ep$, $q$, $m$, $F_0$, $\rho_0$, $|U'|_{C^3}$, and $|(U')^{-1}|_{C^0}$ (independent of $k \ne 0$) such that 
\begin{align*} 
& k^2 \big| \p_t^m \big( t \hat v_1^c (t, k, x_2) + i k^{-1}U'(x_2)^{-1} e^{-ik U(x_2) t} \hat \Omega^c (k , x_2) \big) \big|_{L_{x_2}^2 L_t^q (\R) } \\
&+ |k| \big| \p_t^m \big(\hat \omega^c (t, k, x_2) - e^{-ik U(x_2) t} \hat \Omega^c (k , x_2) \big) 
\big|_{L_{x_2}^2 L_t^q (\R) } \\
&+ \big|\p_t^m\big(  (\hat v_2^c)'' (t, k, x_2) -ik  e^{-ik U(x_2) t} \hat \Omega^c (k, x_2)\big) \big|_{L_{x_2}^2 L_t^q (\R) }\\
\le & C  |k|^{m+1-\frac 1q} \mu^{-\frac 32} \big( |\hat \eta_0(k) | + |k|^{-1} \mu^2 |\hat v_{20}' (k, 0)| + \mu^{1-\ep}|\hat \omega_0(k)|_{L_{x_2}^2} + \mu^{2-\ep}| \hat \omega_0'(k)|_{L_{x_2}^2}\big). 
\end{align*}
\end{lemma}

\begin{remark} 
This lemma also implies, for any integer $m \ge 1$, 
\begin{align*}
|t \p_t^m \big( e^{ikU(x_2)t}\hat v_1^c (t, k, x_2)\big)\Big|_{L_{x_2}^2 L_t^q (\R) } \le  C  |k|^{m-1-\frac 1q} \mu^{-\frac 32} \big( & |\hat \eta_0(k) | + |k|^{-1} \mu^2 |\hat v_{20}' (k, 0)| \\
&+ \mu^{1-\ep}|\hat \omega_0(k)|_{L_{x_2}^2} + \mu^{2-\ep}| \hat \omega_0'(k)|_{L_{x_2}^2}\big),
\end{align*} 
while for $m=0$, there is another term $i k^{-1}U'(x_2)^{-1} e^{-ik U(x_2) t}\hat \Omega^c (k , x_2)$ on the left side.  
The form in the lemma is more consistent with other estimates including that  of $t^2 \hat v_2^c$ to be given in the following, however. 
\end{remark}

\begin{proof}
The definition of $\hat v^c$ 
implies, for each $x_2 \in [-h, 0]$ and $r_1, r_2 \in (0, \rho_0]$,  
\begin{align*}
& t \p_t^m \hat v_1^c (t, k, x_2) = \frac {-i(-ik)^m}{2 \pi }  \oint_{\p \CD_{r_1, r_2}}   t e^{-ik ct} c^mV_2' (k, c, x_2)  d c, \\ 
&\p_t^m \big((ik \hat \omega^c + k^2 \hat v_2^c) (t, k, x_2)\big)= \p_t^m (\hat v_2^c)'' (t, k, x_2)  = \frac {-k(-ik)^m}{2 \pi }  \oint_{\p \CD_{r_1, r_2}}   e^{-ik ct} c^mV_2'' (k, c, x_2)  d c.
\end{align*}
Applying Lemma \ref{L:auxL-1} 
with $n=1$ and $f= c^m V_2'$ and $f_1 = - \frac {c^m}{U'(x_2)} V_2''$
and Lemma \ref{L:V2-esti-1}, we obtain 
\be \label{E:decay-temp-0.3} \begin{split}
& \Big| t \p_t^m \hat v_1^c (t, k, x_2) - \oint_{\p \CD_{r_1, r_2}} \frac {(-i)^m k^{m-1} c^m}{2\pi U'(x_2)} e^{-ik ct}V_2'' (k, c, x_2) dc\Big|_{L_{x_2}^2 L_t^q (\R) } \\
\le & C |k|^{m-1-\frac 1q} \sup_{|c_I| \in (0, r_2]} \big( \big| \p_{c_R} V_2' + U'(x_2)^{-1} V_2''\big|_{L_{c_R, x_2}^2} + | V_2' |_{L_{c_R, x_2}^2}\big) \\
\le & C  |k|^{m-1-\frac 1q} \mu^{-\frac 32} \big( |\hat \eta_0(k) | + |k|^{-1} \mu^2 |\hat v_{20}' (k, 0)| + \mu^{1-\ep}|\hat \omega_0(k)|_{L_{x_2}^2} + \mu^{2-\ep}| \hat \omega_0'(k)|_{L_{x_2}^2}\big).
\end{split}\ee

In the rest of the proof, we shall focus on the integral involving $V_2''$ which also yields the other desired estimates. 
Substituting the term $V_2''$ by the Rayleigh equation \eqref{E:Ray-1} and applying the Cauchy Integral Theorem yield 
\begin{align*}
& \oint_{\p \CD_{r_1, r_2}} \frac { c^m e^{-ik ct}} {U'(x_2)} V_2'' dc = \oint_{\p \CD_{r_1, r_2}} \frac { c^m e^{-ik ct}}{U'(x_2)} \big( k^2 V_2 + \frac {U''(x_2) V_2 + \hat \omega_0 (k, x_2)}{U(x_2)-c} \big) dc \\
=& \oint_{\p \CD_{r_1, r_2}} \frac { e^{-ik ct}}{U'(x_2)}  \Big( k^2 c^m + \big( \frac {U(x_2)^m}{U(x_2)-c} +  \frac {c^m -U(x_2)^m}{U(x_2)-c} \big) U''(x_2)  \Big)V_2  dc  - \frac {2\pi i U(x_2)^m \hat \omega_0 (k, x_2)}{U'(x_2)} e^{-ik U(x_2) t} 
\end{align*}
Since $k^2 c^m + \frac {c^m -U^m}{U-c} U''$ is bounded by $C\mu^{-2}$ on $ \CD_{r_1, r_2}$, 
we can control those terms using Lemma  \ref{L:auxL-1} and obtain 
\be \label{E:decay-temp-1} \begin{split}
& \Big| 
\oint_{\p \CD_{r_1, r_2}} \frac {(-i)^m k^{m-1} c^m}{2\pi U'(x_2)} e^{-ik ct}V_2'' (k, c, x_2) dc +  \frac {(-ik)^{m-1} U(x_2)^m }{U'(x_2)}e^{-ik U(x_2) t} \Big( \hat \omega_0 (k, x_2)
\\& \qquad \qquad \qquad \qquad \qquad \qquad 
- \frac {i U''(x_2) }{2\pi  } \oint_{\p \CD_{r_1, r_2}^{x_2}} \frac 1c e^{-ik ct}V_2  (k, c+U(x_2), x_2) dc \Big)\Big|_{L_{x_2}^2 L_t^q (\R) }\\
\le & C  |k|^{m-1-\frac 1q} \mu^{-\frac 32} \big( |\hat \eta_0(k) | + |k|^{-1} \mu^2 |\hat v_{20}' (k, 0)| + \mu^{1-\ep}|\hat \omega_0(k)|_{L_{x_2}^2} + \mu^{2-\ep}| \hat \omega_0'(k)|_{L_{x_2}^2}\big), 
\end{split} \ee
where we also changed the variable $c-U(x_2)  \to c$ in the last integral and 
\be \label{E:CDx_2}
\CD_{r_1, r_2}^{x_2} = \CD_{r_1, r_2} -U(x_2).
\ee
It remains to handle this integral term and we shall identify its leading terms. 

Fix $T>0$. We first let 
\[
w(k, c, x_2) = V_2 (k, c+U(x_2), x_2) - V_{20} (k, c_R+U(x_2), x_2) \ \Longrightarrow \  \lim_{c_I \to 0+}  |w(k, \cdot+ic_I, \cdot)|_{L_{x_2}^\infty W_{c_R}^{1, q_1}} = 0,     
\]
for any $q_1 \in [1, 2)$, where $\hat \omega_0 \in H_{x_2}^1$, Lemma \ref{L:V2-esti-1}
was used. In the rest of the proof of this lemma, we use $\p^\dagger \CD_{r_1, r_2}^{x_2}$, $\dagger =L, R, T, B$, to denote the left, right, top, bottom sides of the rectangle $\CD_{r_1, r_2}^{x_2}$ with the counterclockwise orientation. For any $r\in (0, r_2]$ and $k \ne 0$, and  $1\le \frac q{q-1} < q_1<2$, 
integrating by parts and using the $L^{\frac q{q-1}} \to L^q$ boundedness (for $q\in [2, \infty]$) of the Fourier transform,  we obtain     
\begin{align*}
& \Big| \oint_{\p^T \CD_{r_1, r}^{x_2}}  \frac 1c e^{-ik ct} w  dc \Big|_{L_{x_2}^2 L_t^q ([-T, T]) } = \Big| (e^{-ik ct} w \log c) \big|_{U(0) -U(x_2) +r_1+ ir}^{U(-h)-U(x_2)-r_1 +ir} \\
& \qquad \qquad \qquad \qquad \qquad\qquad - e^{krt} \oint_{\p^T \CD_{r_1, r}^{x_2}} e^{-ik c_Rt} (-ikt w+  \p_{c_R} w) \log c \, dc \Big|_{L_{x_2}^2 L_t^q ([-T,T]) } \\
\le & C e^{|k|rT} \big(T^{\frac 1q} ( 1 +| \log r_1|) |w(k, \cdot +ir, \cdot)|_{L_{x_2}^2 L_{c_R}^\infty} + |k|^{-\frac 1q} (1+|k|T) |w(k, \cdot +ir, \cdot)|_{L_{x_2}^2 W_{c_R}^{1, q_1}}\big),  
\end{align*}
where $\log$ is taken along $\p^T \CD_{r_1, r}^{x_2}$ which in the upper half plane. Next from Lemma \ref{L:V2-esti-1}   we have
\begin{align*}
&\Big| \oint_{\p^T \CD_{r_1, r}^{x_2}}  \frac 1c  e^{-ik ct} \big(V_{20} (k, c_R +U(x_2), x_2) - V_{20} (k, U(x_2), x_2) \big)  dc \Big|_{L_{x_2}^2 L_t^q ([-T, T]) } \\
\le & C |k|^{-\frac 1q} e^{|k|rT} |\p_{c_R} V_{20} |_{L_{x_2}^2 L_{c_R}^{q_1}} \\
\le & C |k|^{-\frac 1q} \mu^{-\frac 12}  e^{|k|rT} \big( |\hat \eta_0(k) | + |k|^{-1} \mu^2 |\hat v_{20}' (k, 0)| + \mu^{1-\ep}|\hat \omega_0(k)|_{L_{x_2}^2} + \mu^{2-\ep}| \hat \omega_0'(k)|_{L_{x_2}^2}\big).
\end{align*} 
The above error analysis implies that the main contribution of the integral along $\p^T \CD_{r_1, r}^{x_2}$ would come from the product 
\[
V_{20} (k, U(x_2), x_2) f(r, x_2, kt), \; \text{ where } \; f(r, x_2, \tau) = \oint_{\p^T \CD_{r_1, r}^{x_2}} \frac {e^{-i\tau c}}{c} dc.
\]
For any $r \in (0, r_2]$, on the one hand, 
\[
|f(r, x_2, \tau)| = e^{r\tau} \big| - (e^{-i \tau c_R} \log c)\big|_{U(-h)-U(x_2)-r_1}^{U(0)-U(x_2)+r_1} + i \tau \oint_{\p^T \CD_{r_1, r}^{x_2}} e^{-i\tau c_R} \log c \, dc \big| \le C (1+|\tau|) e^{r\tau},  
\]
which is useful for $|\tau|\le 1$. On the other hand,  
\[
f(r, x_2, \tau) =  - \Big(\int_{\R+ir} - \int_{(\R+ir) \setminus \p^T \CD_{r_1, r}^{x_2}}\Big) \frac 1c e^{-i\tau c } dc.
\]
The first integral can be evaluated as $ i \pi (sgn(\tau) +1)$ by using the Cauchy Integral Theorem. Integrating the second integral (in the way opposite to the above) we obtain 
\[
\Big|\int_{(\R+ir) \setminus \p^T \CD_{r_1, r}^{x_2}} \frac 1c e^{-i \tau c} dc\Big| = \frac {e^{r \tau}}{|\tau|} \Big| \frac {e^{-i\tau c_R}}c \big|_{U(0)-U(x_2)+r_1}^{U(-h)-U(x_2)-r_1} + \int_{(\R+ir) \setminus \p^T \CD_{r_1, r}^{x_2}} \frac {e^{-i \tau c_R}}{c^2} dc\Big| \le C \frac {e^{r \tau}}{|\tau|}. 
\]
Therefore 
\[
|f(r, x_2, \tau) - i \pi (sgn(\tau) +1)| \le C (1+ |\tau|)^{-1} e^{r |\tau|}, \quad \forall \tau \in \R. 
\]
Along with \eqref{E:V2-L2L^in}, we have 
\begin{align*}
& \Big| V_{20} (k, U(x_2), x_2) \Big( \oint_{\p^T \CD_{r_1, r}^{x_2}}  \frac 1c  e^{-ik ct}  dc -i \pi (sgn(kt) +1)\Big)  \Big|_{L_{x_2}^2 L_t^q ([-T, T]) } \\
\le & C |k|^{-\frac 1q} e^{r|k|T} \big( |\hat \eta_0(k) | + |k|^{-1} \mu^2 |\hat v_{20}' (k, 0)| + \mu^{1-\ep}|\hat \omega_0(k)|_{L_{x_2}^2} 
\big).
\end{align*}
The integrals along the vertical sides of $\p \CD_{r_1, r}^{x_2}$ converge to $0$ as $r \to 0+$ as all the integrands are smooth there. The integrals along $\p^B \CD_{r_1, r}$, $ r\in (0, r_2]$, can be treated much as in the above. Recall $V_2 (k, \bar c, x_2) = \overline{V_2 (-k, c, x_2)}$. Letting $r \to 0+$, the Cauchy Integral Theorem and the above error analysis imply 
\begin{align*}
& \Big| \oint_{\p \CD_{r_1, r_2}^{x_2}} \frac 1c e^{-ik ct}V_2  (k, c+U(x_2), x_2) dc \\
& \qquad \quad - i\pi \big( (1+ sgn(kt))V_{20} (k, U(x_2), x_2)   + (1-sgn(kt)) \overline{V_{20} (-k, U(x_2), x_2)}  \big)\Big|_{L_{x_2}^2 L_t^q ([-T, T]) }\\
\le & C  |k|^{-\frac 1q}\mu^{-\frac 12} \big( |\hat \eta_0(k) | + |k|^{-1} \mu^2 |\hat v_{20}' (k, 0)| + \mu^{1-\ep}|\hat \omega_0(k)|_{L_{x_2}^2} + \mu^{2-\ep}| \hat \omega_0'(k)|_{L_{x_2}^2}\big). 
\end{align*}
Taking $T \to \infty$, it follows from the above inequality and \eqref{E:decay-temp-1}
\be\label{E:decay-temp-2}  \begin{split} 
&\Big| 
\oint_{\p \CD_{r_1, r_2}} \frac {(-i)^m k^{m-1} c^m}{2\pi U'(x_2)} e^{-ik ct}V_2'' (k, c, x_2) dc + \frac {  (-ik)^{m-1} U(x_2)^m}{U'(x_2)} e^{-ik U(x_2) t} 
\hat \Omega^c(k, x_2) \Big|_{L_{x_2}^2 L_t^q (\R) }\\
\le & C  |k|^{m-1-\frac 1q} \mu^{-\frac 32} \big( |\hat \eta_0(k) | + |k|^{-1} \mu^2 |\hat v_{20}' (k, 0)| + \mu^{1-\ep}|\hat \omega_0(k)|_{L_{x_2}^2} + \mu^{2-\ep}| \hat \omega_0'(k)|_{L_{x_2}^2}\big). 
\end{split} \ee
Along with \eqref{E:decay-temp-0.3} and Lemma \ref{L:mode-decay-1} it implies the desired estimate of $\p_t^m (t \hat v_1^c)$. The estimates on $\p_t^m (\hat v_2^c)''$ and $\p_t^m \hat \omega^c$ are also obtained from the above inequality and Lemma \ref{L:mode-decay-1}. 
\end{proof}

Finally we consider $t^2 \hat v_2$. 

\begin{lemma} \label{L:mode-decay-3} 
Assume $U \in C^5$ and \eqref{E:F_0}, 
then, for any $\ep\in (0, 1)$, $q\in (2, \infty]$, and integer $m \ge 0$, there exists $C>0$ determined only by $\ep$, $q$, $m$, $F_0$, $|U'|_{C^4}$, and $|(U')^{-1}|_{C^0}$ (independent of $k \ne 0$) such that 
\begin{align*} 
& \Big| \p_t^m \Big( t^2 \hat v_2^c (t, k, x_2) -  \Big( -\frac {ie^{-ik U(x_2)t}}{k U'(x_2)^2} \hat \Omega^c (k, x_2) 
+ e^{-ik U(0)t}  \hat \Lambda_T(k, x_2) +  e^{-ik U(-h)t}  \hat \Lambda_B(k, x_2) \Big)\Big)\Big|_{L_{x_2}^2 L_t^q (\R) }\\
\le & C  |k|^{m-1-\frac 1q} \mu^{-\frac 32} \big( |\hat \eta_0(k) | + |k|^{-1} \mu^2 |\hat v_{20}' (k, 0)| + \mu^{1-\ep}|\hat \omega_0(k)|_{L_{x_2}^2} + \mu^{2-\ep}| \hat \omega_0'(k)|_{L_{x_2}^2}+ \mu^{3-\ep}| \hat \omega_0''(k)|_{L_{x_2}^2}\big), 
\end{align*}
where 
\be \label{E:Lambda_T}
\hat \Lambda_T( k, x_2) 
=- \frac ik  \frac {U''(0) \hat \eta_0(k) - \hat \omega_0(k, 0)}{U'(0)^2 y_{0-} (k, U(0), 0)} y_{0-} (k, U(0), x_2), 
\ee
\be \label{E:Lambda_B}
\hat \Lambda_B (k, x_2) =\frac {i \hat \omega_0(k, -h) y_{0+} (k, U(-h), x_2)}{kU'(-h)^2 y_{0+} (k, U(-h), -h)}. 
\ee
\end{lemma}


\begin{remark}
In the above lemmas, we also notice $\hat \Lambda_\dagger (-k, x_2) = \overline{\hat \Lambda_\dagger (k, x_2)}$, $\dagger=T, B$. The leading order terms $\hat \Lambda_B$ and $\hat \Lambda_T$ represent the contribution  from the rigid bottom and the water surface, while the asymptotic vorticity $\hat \Omega^c$ 
from the fluid interior. In the fixed boundary problem for $x_2 \in [-h, 0]$ with slip boundary condition on both horizontal boundaries, $\hat \Omega^c$ and $\hat \Lambda_B$ would take similar forms and $\hat \Lambda_T$ would be similar to $\hat \Lambda_B$. See Subsection \ref{SS:Euler-Channel}.  
\end{remark}

\begin{proof}
The definition of $\hat v_2^c$ implies, for each $x_2 \in [-h, 0]$ and $r_1, r_2 \in (0, \rho_0]$,  
\begin{align*}
& t^2 \p_t^m  \hat v_2^c (t, k, x_2) = \frac {-(-i)^m k^{m+1}}{2 \pi }  \oint_{\p \CD_{r_1, r_2}}   t^2 e^{-ik ct} c^mV_2 (k, c, x_2)  d c.
\end{align*}
Let $f=c^mV_2 $ and 
\[
f_1 = c^m \Big( \frac {V_2''(x_2)}{U'(x_2)^2} - \frac {g+\sigma k^2}{\BF(k, c)} \Big( \frac {V_2'' (-h)}{U'(-h)^2} y_+(x_2) - \frac {V_2''(0)}{U'(0)^2 } y_- (x_2)\Big) \Big) = c^m (\p_{c_R}^2 V_2 - \tilde V_2), 
\]
with $\tilde V_2$ defined in Lemma \ref{L:V2-esti-1}.
Applying Lemma \ref{L:auxL-1} with $n=2$ and Lemma \ref{L:V2-esti-1}, we obtain 
\begin{align*}
& \Big| t^2 \p_t^m \hat v_2^c (t, k, x_2) - \frac {(-i)^m k^{m-1} }{2\pi } \oint_{\p \CD_{r_1, r_2}} e^{-ik ct}f_1 (k, c, x_2) dc\Big|_{L_{x_2}^2 L_t^q (\R) } \\
\le & C |k|^{m-1-\frac 1q} \sup_{c_I \in (0, r_2]}\big( \big| \tilde V_2 |_{L_{c_R, x_2}^2}  + | \p_{c_R} V_2 |_{L_{c_R, x_2}^2} + | V_2 |_{L_{c_R, x_2}^2}\big) \\
\le & C  |k|^{m-1-\frac 1q} \mu^{-\frac 32} \big( |\hat \eta_0(k) | + |k|^{-1} \mu^2 |\hat v_{20}' (k, 0)| + \mu^{1-\ep}|\hat \omega_0(k)|_{L_{x_2}^2} + \mu^{2-\ep}| \hat \omega_0'(k)|_{L_{x_2}^2}+ \mu^{3-\ep}| \hat \omega_0''(k)|_{L_{x_2}^2}\big).
\end{align*}
Substituting $V_2''$ in $f_1$ by using the Rayleigh equation \eqref{E:Ray-1} yields 
\[
f_1 = c^m \Big( 
\frac {V_2''(x_2)}{U'(x_2)^2}+ \frac {f_{1B}}{U(-h)-c} + \frac {f_{1T}}{U(0)-c} 
+ \frac { ( g+ \sigma k^2) y_-(x_2)}{U'(0)^2 \BF(k, c)} k^2 V_2(0) \Big),
\]
where 
\[
f_{1B} = - \frac { ( g+ \sigma k^2)  \hat \omega_0 (-h) }{ 
U'(-h)^2\BF(k, c) } y_+(x_2) = - \frac {  \hat \omega_0 (-h) y_+(x_2)}{ U'(-h)^2 y_+(-h)}, 
\]
\[
f_{1T}= \frac {( g+ \sigma k^2)\big(U''(0) V_2(0) + \hat\omega_0 (0)\big)  }{
U'(0)^2 \BF(k, c) }y_-(x_2).
\]
Again the terms involving $k^2 V_2(0)$ not being divided by $U-c$ can be estimated by using assumption \eqref{E:F_0} and Lemmas \ref{L:auxL-1}, \ref{L:y-pm}, and \ref{L:V2-esti-1} and we have 
 \begin{align*}
& \Big| t^2 \p_t^m \hat v_2^c (t, k, x_2) - \frac {(-i)^m k^{m-1}}{2\pi } \oint_{\p \CD_{r_1, r_2}}  e^{-ik ct} c^m \Big( 
\frac {V_2''(x_2)}{U'(x_2)^2} + \frac {f_{1B}}{U(-h)-c} + \frac {f_{1T}}{U(0)-c} \Big) dc\Big|_{L_{x_2}^2 L_t^q (\R) } \\
\le & C  |k|^{m+1-\frac 1q} \mu^{\frac 12} \big( |\hat \eta_0(k) | + |k|^{-1} \mu^2 |\hat v_{20}' (k, 0)| + \mu^{1-\ep}|\hat \omega_0(k)|_{L_{x_2}^2} + \mu^{2-\ep}| \hat \omega_0'(k)|_{L_{x_2}^2}+ \mu^{3-\ep}| \hat \omega_0''(k)|_{L_{x_2}^2}\big).
\end{align*}
We shall identify the principle contributions from the terms $V_2''(x_2)$, $f_{1B}$, and $f_{1T}$ following a similar strategy and use the same notations $\p^\dagger \CD_{r_1, r_2}$, $\dagger=T, B, L, R$, as in the proof of Lemma \ref{L:mode-decay-2}, with necessary modifications to treat the contributions from the $x_2=0, -h$. 

Fix $T>0$. We start with $f_{1T}$ by letting 
\[
f_{1T}^0 (k, c_R, x_2)= \lim_{c_I \to 0+}   f_{1T} (k, c, x_2) = \frac {( g+ \sigma k^2)\big(U''(0) V_{20}(0) + \hat\omega_0 (0)\big)  }{U'(0)^2 \BF(k, c_R)}y_{0-} (x_2).  
\] 
From assumption \eqref{E:F_0}, Lemmas \ref{L:y0}, \ref{L:pcy-complex}(2b), \ref{L:ev-large-k}, and \ref{L:V2-esti-1}, 
we have, for any $q_1\in [1, 2)$,  
\[
|(f_{1T} -  f_{1T}^0)(k, \cdot+c_I, \cdot)|_{L_{x_2}^\infty W_{c_R}^{1, q_1}} \to 0, \; \text{ as } c_I \to 0+. 
\] 
The next step is the same argument via integrating by parts in $c_R$ as in the proof of Lemma \ref{L:mode-decay-2}, 
\begin{align*}
& \Big| \oint_{\p^T \CD_{r_1, r}}  e^{-ik ct} c^m \frac {f_{1T} (k, c, x_2) - f_{1T}^0 (k, c_R, x_2)}{ U(0)-c}  dc\Big|_{L_{x_2}^2 L_t^q ([-T,T]) } \\
= & \Big| \big(e^{-ik ct} c^m ( f_{1T}  - f_{1T}^0 ) \log (U(0)-c)\big) \big|_{U(0)+r_1+ir}^{U(-h)-r_1 +ir} \\
& \quad \; - \oint_{\p^T \CD_{r_1, r}}  e^{-ik ct} ( -ikt + \p_{c_R}) \big(c^m(f_{1T} - f_{1T}^0 )\big)  \log( U(0)-c)  dc\Big|_{L_{x_2}^2 L_t^q ([-T,T]) }
\to 0 \; \text{ as } r\to 0+.
\end{align*}
From \eqref{E:F_0} and Lemmas \ref{L:y-pm}, \ref{L:pcy1}--\ref{L:pcy3}, \ref{L:e-v-basic-1}(3), and \ref{L:ev-large-k}(1),  one may estimate,
\[
|y_{0-}/\BF |_{L_{x_2}^2 L_{c_R}^\infty} + |\p_{c_R} ( y_{0-}/\BF)|_{L_{x_2}^2 L_{c_R}^{q_1}} \le C \mu^{\frac 52}, \;\; \forall q_1 \in [1, \infty).
\]
Along with Lemma \ref{L:V2-esti-1}, it implies, for any $q_2 \in [1, 2)$,  
\[
|f_{1T}^0|_{L_{x_2}^2 L_{c_R}^\infty}+ 
|\p_{c_R} f_{1T}^0|_{L_{x_2}^2 L_{c_R}^{q_2}} \le C \mu^{-\frac 12} \big( |\hat \eta_0(k) | + |k|^{-1} \mu^2 |\hat v_{20}' (k, 0)| +  |\hat \omega_0(k)|_{L_{x_2}^2} + \mu^{2-\ep} | \hat \omega_0'(k)|_{L_{x_2}^2} |\big),
\]
where $|\hat \omega_0(0)|$ and $|V_2(0)|_{L_{c_R}^\infty}$ were bounded by the $L^2$ norms of $\hat \omega_0(k)$, $\hat \omega_0' (k)$, $|V_2(0)|_{L_{c_R}^2}$, and $|\p_{c_R}V_2(0)|_{L_{c_R}^2}$. Consequently, for any $r\in (0, r_2]$ 
\begin{align*}
&\Big| \oint_{\p^T \CD_{r_1, r}}  \frac {e^{-ik ct} c^m}{U(0)-c} f_{1T}^0 (k, c_R, x_2) dc - f_{1T}^0 (k, U(0), x_2)\oint_{\p^T \CD_{r_1, r}}  \frac {e^{-ik ct} c^m}{U(0)-c}  dc\Big|_{L_{x_2}^2 L_t^q ([-T,T]) } \\
\le & Ce^{r |k| T} |k|^{-\frac 1q}  \mu^{-\frac 12} \big( |\hat \eta_0(k) | + |k|^{-1} \mu^2 |\hat v_{20}' (k, 0)| + |\hat \omega_0(k)|_{L_{x_2}^2} + \mu^{2-\ep} | \hat \omega_0'(k)|_{L_{x_2}^2} \big).
\end{align*}
As in the proof of Lemma \ref{L:mode-decay-2}, by considering contour integrals, we have 
\[
\Big| \oint_{\p^T \CD_{r_1, r}}  \frac {e^{-ik ct}}{U(0)-c}  dc + i\pi (1+ sgn(kt)) e^{-ik U(0)t} \Big| \le \frac {Ce^{r |kt|}}{1+ |kt|}. 
\]
Again, since $\frac {c^m - U(0)^m}{c-U(0)}$ is bounded for $m\ge 1$, the above $|f_{1T}^0|_{L_{x_2}^2 L_{c_R}^\infty}$ estimate implies 
\begin{align*}
&\Big| f_{1T}^0 (k, U(0), x_2)\Big( \oint_{\p^T \CD_{r_1, r}}  \frac {e^{-ik ct} c^m}{U(0)-c}  dc + i\pi (1+ sgn(kt)) e^{-ik U(0)t} U(0)^m \Big) \Big|_{L_{x_2}^2 L_t^q ([-T,T]) } \\
\le & Ce^{r |k|T} |k|^{-\frac 1q}  \mu^{-\frac 12} \big( |\hat \eta_0(k) | + |k|^{-1} \mu^2 |\hat v_{20}' (k, 0)| + |\hat \omega_0(k)|_{L_{x_2}^2} + \mu^2 | \hat \omega_0'(k)|_{L_{x_2}^2} \big).
\end{align*}
The contributions from the integral along $\p^B \CD_{r_1, r_2}$ can be treated similarly and using the conjugacy relation, while the integrals along the vertical boundaries of $\p\CD_{r_1, r_2}$ vanish as $r \to 0+$. Using the Cauchy Integral Theorem, combining the above analysis, letting $r \to 0+$, and then $T\to 0+$, we obtain 
\begin{align*}
& \Big| \oint_{\p^T \CD_{r_1, r_2}}  e^{-ik ct} c^m f_{1T} (k, c, x_2) dc - 2\pi k e^{-ik U(0)t} U(0)^m \hat \Lambda_T(k, x_2)
\Big|_{L_{x_2}^2 L_t^q (\R) }\\
\le & C|k|^{-\frac 1q}  \mu^{-\frac 12} \big( |\hat \eta_0(k) | + |k|^{-1} \mu^2 |\hat v_{20}' (k, 0)| + |\hat \omega_0(k)|_{L_{x_2}^2} + \mu^{2-\ep} | \hat \omega_0'(k)|_{L_{x_2}^2} \big),
\end{align*}
where 
\[
\hat \Lambda_T(k, x_2) = - \tfrac i{2k} \big( (1+ sgn(kt))  f_{1T}^0 (k, U(0), x_2) +(1- sgn(kt))  \overline{f_{1T}^0 (-k, U(0), x_2)}\big).
\]
We give closer look at $\hat \Lambda_T$. From boundary condition \eqref{E:Ray-BC-1}, \eqref{E:Ray-BC-2}, and \eqref{E:Ray-data-1}, 
\[
V_2 (k, U(0), 0)= - \zeta_+ (U(0)) /(g+\sigma k^2) = -\hat \eta_0 (k),
\]
and thus  
\[
f_{1T}^0 (k, U(0), x_2) = \frac {U''(0) \hat \eta_0(k) - \hat \omega_0(k, 0)}{ U'(0)^2 y_{0-} (k, U(0), 0)} y_{0-} (k, U(0), x_2).
\]
Since $y_{0-} (k, U(0), x_2) \in \R$ for $x_2 \in [-h, 0]$, we obtain $\overline {f_{1T}^0 (k, U(0), x_2)}= f_{1T}^0 (-k, U(0), x_2)$ and hence leads to the desired form \eqref{E:Lambda_T} of $\Lambda_T$. The term involving $f_{1B}$ can be analyzed similarly (actually slightly simpler due to $V_2 (-h)=0$) using Lemmas \ref{L:pcy1}--\ref{L:pcy3} and \ref{L:pcy-complex}. The term involving $V_2''$ can be estimated much as in \eqref{E:decay-temp-2}.
Summarizing this estimates we obtain 
\begin{align*} 
& \Big| t^2 \p_t^m \hat v_2^c (t, k, x_2) - (-ik)^{m} \Big(-\frac {iU(x_2)^m}{k U'(x_2)^2}  e^{-ik U(x_2)t} \hat \Omega^c (k, x_2)  
\\
& \qquad \qquad + U(0)^m e^{-ik U(0)t} \hat \Lambda_T(k, x_2) + U(-h)^m e^{-ik U(-h)t}  \hat \Lambda_B(k, x_2) \Big)\Big|_{L_{x_2}^2 L_t^q (\R) }\\
\le & C  |k|^{m-1-\frac 1q} \mu^{-\frac 32} \big( |\hat \eta_0(k) | + |k|^{-1} \mu^2 |\hat v_{20}' (k, 0)| + \mu^{1-\ep}|\hat \omega_0(k)|_{L_{x_2}^2} + \mu^{2-\ep}| \hat \omega_0'(k)|_{L_{x_2}^2}+ \mu^{3-\ep}| \hat \omega_0''(k)|_{L_{x_2}^2}\big). 
\end{align*}
Combining it with Lemma \ref{L:mode-decay-1}, the desired estimate on $\p_t^m (t^2 \hat v_2^c)$  follows.
\end{proof}

\subsection{Linearized capillary gravity waves in the horizontally periodic case of $x_1 \in  \mathbb{T}_L$} \label{SS:per-case} 

In this subsection, we consider the case where the system is periodic in $x_1$ with wave length $L>0$. In this case 
\[
k \in \tfrac {2\pi}L \mathbb{Z}, 
\quad \hat v_{2} (t, k=0, x_2)=0, 
\]
where the latter properties is due to the divergence free condition on $v$. 
For $\dagger=c, p$, let  
\[
v_2^\dagger (t, x) = \sum_{|k| \in \tfrac {2\pi}L \mathbb{N}} \hat v_2^\dagger (t, k, x_2) e^{ik x_1}, \;\; \eta^c (t, x_1) = \sum_{|k| \in \tfrac {2\pi}L \mathbb{N}} \hat \eta^c (t, k) e^{ik x_1}, \;\; v_1^c (t, x) = \sum_{|k| \in \tfrac {2\pi}L \mathbb{N}} \hat v_1^c (t, k, x_2) e^{ik x_1}, 
\]
\[
\eta^p (t, x_1) = \hat \eta_0(0) + \sum_{|k| \in \tfrac {2\pi}L \mathbb{N}} \hat \eta^p (t, k) e^{ik x_1}, \;\; v_1^p (t, x) = \hat v_1(0, x_2) +  \sum_{|k| \in \tfrac {2\pi}L \mathbb{N}} \hat v_1^p (t, k, x_2) e^{ik x_1},
\quad v^\dagger = (v_1^\dagger, v_2^\dagger), 
\]
where $\hat v_1^\dagger$, $\hat v_2^\dagger$, and $\hat \eta^\dagger$ are defined in Lemma \ref{L:integralF} and Corollary \ref{C:integralF}. 
Here we used \eqref{E:0-mean} that the zeroth modes $\hat v_1(k=0)$ and $\hat \eta(0)$ are invariant in $t$. 
Throughout this subsection, we assume \eqref{E:no-S-M} holds for $\BK=\frac {2\pi}L\mathbb{N}$. 

We first give the decay estimates of $(v^c, \eta^c)$ based on Lemma \ref{L:mode-decay-1}--\ref{L:mode-decay-3}. In particular, for the estimates of $t v_1^c$ and $t^2 v_2^c$, recall $\hat \Omega^c (k, x_2)$ and $\hat \Lambda_\dagger (k, x_2)$, $\dagger=B, T$ defined in \eqref{E:Omega-1}, \eqref{E:Lambda_B}, and \eqref{E:Lambda_T}, respective. Let 
\be \label{E:Lambda-def}
\Omega^c (x_1, x_2) = \sum_{|k| \in \tfrac {2\pi}L \mathbb{N}} \hat \Omega^c (k, x_2) e^{ikx_1}, \quad \Lambda_\dagger (x_1, x_2) = \sum_{|k| \in \tfrac {2\pi}L \mathbb{N}} \hat \Lambda_\dagger (k, x_2) e^{ikx_1}.
\ee

\noindent {\it Proof of Theorem \ref{T:decay-per}(1--2).} 
The assumption of the non-existence of singular modes is given in the form of \eqref{E:no-S-M}. 
According to Proposition \ref{P:e-v-0}, \eqref{E:no-S-M} for $\BK=\frac {2\pi}L \mathbb{N}$ implies \eqref{E:F_0} holds for all $k$ with constants $\rho_0$ and $F_0$ uniform in $k$. Therefore from the definition of $v_2^c$ and Lemma \ref{L:mode-decay-1}, it is straight forward to estimate 
\begin{align*}
|\p_t^{n_0} v^c|_{H_{x_1}^{n_1} L_{x_2}^2 L_t^{q_1} (\R)}^2 \le & C \sum_{|k| \in \tfrac {2\pi}L \mathbb{N}} \mu^{-2n_1}  |k|^{2n_0+2 -\frac 2{q_1}}  \big( \mu^{\frac 12} |\hat \eta_0(k)| + |k|^{-1} \mu^{\frac 52} |\hat v_{20}' (k, 0)| + \mu^{\frac 32-\ep} |\hat \omega_0(k)|_{L_{x_2}^2} \big)^2\\
\le & C \sum_{|k| \in \tfrac {2\pi}L \mathbb{N}}  |k|^{2(n_0+n_1+1 -\frac 1{q_1})}  \big( |k|^{-1} |\hat \eta_0(k)|^2 + |k|^{-7} |\hat v_{20}' (k, 0)|^2 + |k|^{2\ep- 3} |\hat \omega_0(k)|_{L_{x_2}^2}^2 \big)\\
\le & C  \big( |\eta_0|_{H_{x_1}^{n_0+n_1+\frac 12 -\frac 1{q_1}}}^2 + |\p_{x_2} v_{20} (\cdot, 0)|_{H_{x_1}^{n_0+n_1-\frac 52-\frac 1{q_1}}}^2 + |\omega_0|_{H_{x_1}^{n_0+n_1-\frac 12 -\frac 1{q_1} +\ep} L_{x_2}^2}^2 \big).
\end{align*} 
The desired inequality follows from $\p_{x_2} v_{20} = -\p_{x_1} v_{10}$.  
The estimates on $\p_t^{n_0} \eta^c$, $t \p_t^{n_0} v_2^c$ and $t \p_t^{n_0} \eta^c$ are obtained similarly. The inequalities on $\p_t^{n_0} (t v_1^c)$ and $\p_t^{n_0} (t^2 v_2^c)$ are obtained by applying Lemma \ref{L:mode-decay-2} and \ref{L:mode-decay-3} through a similar procedure. The estimates on $\Omega^c$ and $\Lambda_\dagger$, $\dagger=B, T$, follow directly from the formula and estimates of their each Fourier modes given in those lemmas and 
\be \label{E:y_--Hs} \begin{split}
&|y_{0\pm} (k, c, \cdot)/y_{0\pm} (k, c, 0)|_{L_{x_2}^q} \le C\mu^{\frac 1q}, \; \forall q\in [1, \infty]; \\
& |\p_{x_2} y_{0\pm} (k, c, \cdot)/y_{0\pm} (k, c, 0)|_{L_{x_2}^q} \le C\mu^{\frac 1q-1}, \; \forall q\in [1, \infty); 
\end{split} \ee
which is obtained using 
\eqref{E:F_0} and Lemma \ref{L:y-pm}. The singular elliptic equations in \eqref{E:Lambda-BVP} are simply from the homogeneous Rayleigh equation with $c=U(-h), U(0)$, satisfied by $y_{0\pm}$ in $(-h, 0)$. The boundary conditions of $\Lambda_B$ and $\Lambda_T$ are direct corollaries of their definitions and the boundary conditions \eqref{E:y-pm} of $y_{0\pm}$.  
\hfill $\square$ \\

Next we consider the $(v^p(t, x), \eta^p(t, x_1))$ part of the linear solution $(v, \eta)$. Let 
\be \label{E:lambda&N} \begin{split}
&\lambda_0 = \max\{\RP\, (-ic_*k) \mid k \in \tfrac {2\pi}L \mathbb{N}, \ c_* \in R(k)\} \ge 0,  \\
&N = \max \{ \text{degree of root } c_* \text{ of } \BF(k, \cdot)  \mid k \in \tfrac {2\pi}L \mathbb{N}, \ c_* \in R(k), \ \RP \, (-ikc_*) =\lambda_0\} \ge 1, 
\end{split} \ee 
where the lower bounds are obtained due to the roots $c^\pm (k)$ for large $k$ (Lemma \ref{L:ev-large-k}(3)). \\


\noindent {\it Proof of Theorem \ref{T:decay-per}(3).} 
On the one hand, according to Lemma \ref{L:ev-large-k}(3), there exists $k_0>0$ such that $R(k) = \{c^\pm (k)\}$ with simple roots $c^\pm(k)$ for all $|k|\ge k_0$. On the other hand, \eqref{E:no-S-M} and Proposition \ref{P:e-v-0} imply that \eqref{E:F_0} holds for all $k \in \frac {2\pi}L \mathbb{N}$. Along with Lemma \ref{L:ev-large-k}(2), we obtain that, for all $k \in \frac {2\pi}L \mathbb{N}$ with $|k|< k_0$, the set of roots $R(k)$ is contained in a subset in the domain of analyticity of $\BF(k, \cdot)$ uniformly in such $k$. Hence $R(k)$ is a discrete set and the total algebraic multiplicity of $c_*\in R(k)$ for all $k \in \frac {2\pi}L \mathbb{N}$ with $|k|< k_0$ is finite. This proves $\lambda_0, N< \infty$. 

For any $k \in \frac {2\pi}L \mathbb{N}$ and $c_*\in R(k)$, let $n$ denote the degree of $c_*$ as a root of $\BF(k, \cdot)$, then $\Bb$ and $\Bb_S$ are polynomials of $t$ of degree $n-1$ (Lemma \ref{L:Bb}). Hence to prove the regularity estimates, we only need to consider $k \in \frac {2\pi}L \mathbb{N}$ with $|k|\ge k_0$ where all roots of $\BF(k, \cdot)$ are simple. For such $k$, $R(k) = \{c^\pm (k)\}$ and Lemma \ref{L:ev-large-k}(3) implies that there exists $C>0$ such that 
\[
|c_*|\ge \tfrac 1C|k|^{\frac 12}, \;\; |\p_c F(k, c_*)| \ge \tfrac 1C |k|^{\frac 32}, \quad \forall c_*\in R(k), \; k_0 \le |k| \in \tfrac {2\pi}L \mathbb{N}.
\] 
From the homogeneous Rayleigh equation \eqref{E:Ray-H1-1}, \eqref{E:F_0}, and Lemma \ref{L:y-pm}, it holds, 
\be \label{E:y_--Hs-reg}
|\p_{x_2}^s y_{-} (k, c_*, \cdot)|_{L_{x_2}^2} \le C\mu^{\frac 32-s}, \quad  \forall \, s\in [0, l_0], \ k \in \R, \ c_* \in R(k). 
\ee
Hence Lemmas \ref{L:integralF} and \ref{L:Bb} and the definition of $v_2^p$ imply, for any $n_1 \in \R$ and $n_2\in [0, l_0]$,  
\begin{align*} 
&\sum_{k_0\le |k| \in \tfrac {2\pi}L \mathbb{N}} \mu^{-2n_1}|\hat v_2^p (t, k, \cdot)|_{H_{x_2}^{n_2}}^2
\le C\sum_{k_0\le |k| \in \tfrac {2\pi}  L \mathbb{N}} \sum_{c_*= c^\pm (k)}\mu^{-2n_1} |\Bb(k, c_*, \cdot)|_{H_{x_2}^{n_2}}^2 \\
\le & C\sum_{k_0\le |k| \in \tfrac {2\pi} L \mathbb{N}} |k|^{2(n_1+n_2)-4} \big(|k|^3 |\hat \eta_0(k)| + |k|^{\frac 12} |\hat v_{20}' (k, 0)| + |k| |\hat \omega_0(k, \cdot)|_{L_{x_2}^2}\big)^2 \\
\le & C
\big(|\eta_0|_{H_{x_1}^{n_1+n_2+1}}^2 + |\p_{x_2} v_{20} (\cdot, 0)|_{H_{x_1}^{n_1+n_2-\frac 32}}^2 + |\omega_0|_{H_{x_1}^{n_1+n_2-1}L_{x_2}^2}^2\big).    
\end{align*}
The desired inequality follows from the divergence free condition. 
The expression of $v_1^p$ involves $y_-'$ and thus it can be differentiated in $x_2$ at most $l_0-1$ times. The procedure to obtain the estimates of $v_1^p$ and $\eta^p$ are similar and we skip the details. 
\hfill $\square$

Finally we give the invariant decomposition of the phase space which proves Theorem \ref{T:decay-per}(4). 

\begin{lemma} \label{L:inv-decomp-per} 
Let 
\[
\BX^p = \overline{ span\{ range (e^{ikx_1} \BP(k, c_*)) \mid c_* \in R(k), \, k \in \tfrac {2\pi}L \mathbb{Z}\}} \subset H^1 (\mathbb{T}_L \times (-h, 0)) \times H^2(\mathbb{T}_L), 
\]
\[
\BP (v, \eta)  = \oplus_{c_* \in R(k), k \in \tfrac {2\pi}L \mathbb{Z}} e^{ikx_1}\BP(k, c_*) \big(\hat v(k), \hat \eta(k)\big),  \quad 
\BX^c = \ker \BP \subset H^1 (\mathbb{T}_L \times (-h, 0)) \times H^2(\mathbb{T}_L).  
\]
where $\BP(k, c_*)$ was defined in \eqref{E:BX-k}, then the following hold. 
\begin{enumerate}
\item $\BP$ is a bounded projection operator from $H^{n} (\mathbb{T}_L \times (-h, 0)) \times H^{n+1}(\mathbb{T}_L)$ to $\BX^p \cap \big(H^{n} (\mathbb{T}_L \times (-h, 0)) \times H^{n+1}(\mathbb{T}_L)\big)$ for any integer $n \in [1, l_0-1]$. 
\item $\BX^p$ and $\BX^c$ are both invariant subspaces of \eqref{E:LE-F}. 
\item Moreover \eqref{E:LE-F} is also well-posed on the $L^2 \times H^1$ completion of $\BX^p$
and is a (possibly unstable) dispersive equation with the (multi-branches of) dispersion relation given by $k \to -k c_*$ including all $c_* \in R(k)$. 
\end{enumerate}\end{lemma} 

The boundedness of $\BP$ follows from the estimates in Theorem \ref{T:decay-per} at $t=0$. The invariance of $\BX^p$ and $\BX^c$ is due to Lemma \ref{L:Bb} and Corollary \ref{C:v^c}. The well-posedness of \eqref{E:LE-F} on the $L^2 \times H^1$ completion of $\BX^p$
is due to the fact that $R(k) = \{c^\pm (k)\} \subset \R\setminus U([-h, 0])$ except for finitely many $k \in \tfrac {2\pi}L \mathbb{Z}$. Here we did not set $\BX^p$ and $\BX^c$ in $L^2 \times H^1$ is due to the issue that we can not ensure $v_1 (\cdot, 0) \in H_{x_1}^{-\frac 12}$ for $v \in L^2$.

\subsection{Linearized capillary gravity waves in the horizontally infinite case of $x_1 \in \R$} \label{SS:R-case}

In this subsection, we consider the case where $x_1 \in \R$ and thus $k \in \R$. Throughout this subsection, we assume \eqref{E:no-S-M} for $\BK=\R$. 
For $\dagger=c, p$, let  
\be \label{E:component-R}
v^\dagger (t, x) = \int_{\R} \hat v^\dagger (t, k, x_2) e^{ik x_1} dk, \;\; \eta^\dagger (t, x_1) =  \int_{\R}  \hat \eta^\dagger (t, k) e^{ik x_1} dk, \quad v^\dagger = (v_1^\dagger, v_2^\dagger), 
\ee
where $\hat v_1^\dagger$, $\hat v_2^\dagger$, and $\hat \eta^\dagger$ are defined in Lemma \ref{L:integralF} and Corollary \ref{C:integralF}. 

Again we first carry out the decay estimates of $(v^c, \eta^c)$ based on Lemma \ref{L:mode-decay-1}--\ref{L:mode-decay-3}. 
Let 
\be \label{E:Lambda-def-R}
\Omega^c (x_1, x_2) = \int_{\R} \hat \Omega^c (k, x_2) e^{ikx_1}dk, \quad \Lambda_\dagger (x_1, x_2) = \int_{\R} \hat \Lambda_\dagger (k, x_2) e^{ikx_1} dk.
\ee



\noindent {\it Proof of Theorem \ref{T:decay-R}(1--3).} 
Again the assumption of the non-existence of singular modes is given in the form of \eqref{E:no-S-M}. 
According to Proposition \ref{P:e-v-0}, assumption \eqref{E:no-S-M} for $\BK =\R$ implies that \eqref{E:F_0} holds and $R(k) = \{c^\pm (k)\}$ with all these simple roots $c^\pm (k)$ of $\BF(k, \cdot)$ away from $U([-h, 0])$ for all $k \in \R$. Moreover, Lemma \ref{L:ev-large-k} yields  
\[
|dist(c^\pm(k), U([-h, 0]))|\ge \tfrac 1C \mu^{-\frac 12}, \;\; |\p_c F(k, c^\pm (k))| \ge \tfrac 1C \mu^{-\frac 32}, \quad \forall k\in \R.
\] 
Like in the periodic-in-$x_1$ case, the proof of the decay of $(v^c, \eta^c)$ is also a direct verification using Lemmas \ref{L:mode-decay-1}--\ref{L:mode-decay-3} along with \eqref{E:y_--Hs} and the divergence free condition. We omit the details. 

From Lemmas \ref{L:integralF} and \ref{L:Bb}(3), we obtain $\Bb$ and $\Bb_S$ are independent of $t$ and satisfy, for any $n_2 \in [0, l_0]$,  
\[
|\p_{x_2}^{n_2} \Bb(k, c^\pm(k), x_2)| \le C\big(|k|\mu^{-\frac 12} |\hat \eta_0(k)| + \mu |\hat v_{20}' (k, 0)| + |k|\mu^{\frac 32} |\hat \omega_0(k)|_{L_{x_2}^2}\big) \big|\tfrac {\mu^{1-n_2} e^{\mu^{-1} (x_2+h)}
}{y_-(k, c^\pm (k), 0)}\big|,
\]
\[
|\Bb_S (k, c^\pm(k))| \le C\big( |\hat \eta_0(k)| + |k|^{-1} \mu^{\frac 32} |\hat v_{20}' (k, 0)| + \mu^{2} |\hat \omega_0(k)|_{L_{x_2}^2} \big).
\]
The desired estimates follow from \eqref{E:y_--Hs-reg}, $ik \hat v_1 = - \hat v_2$, and direct computations. 
\hfill $\square$ \\

Similar to the periodic case, we also have the decomposition by invariant subspaces. 

\begin{lemma} \label{L:inv-decomp-R} 
Let 
\[
\BP (v, \eta)  = \int_\R \BP (k, c^+ (k)) (v, \eta) dk + \int_\R \BP (k, c^- (k)) (v, \eta) dk,  
\]
\[
\BX^p = range(\BP) \subset H^1 (\R \times (-h, 0)) \times H^2(\R), \quad 
\BX^c = \ker \BP \subset H^1 (\R \times (-h, 0)) \times H^2(\R),  
\]
where $\BP(k, c^\pm(k))$ was defined in \eqref{E:BX-k}, then the following hold. 
\begin{enumerate}
\item $\BP$ is a bounded projection operator from $H^{n} (\R \times (-h, 0)) \times H^{n+1}(\R)$ to $\BX^p \cap \big(H^{n} (\R \times (-h, 0)) \times H^{n+1}(\R)\big)$ for any integer $n \in [1, l_0-1]$.
\item $\BX^p$ and $\BX^c$ are both invariant subspaces of \eqref{E:LE-F}. 
\item In fact \eqref{E:LE-F} is also well-posed on the $L^2 \times H^1$ completion of $\BX^p$
and is a  dispersive equation with the dispersion relation given by $k\to - k c^\pm(k)$. 
\end{enumerate}\end{lemma} 

To end this subsection we show that, under assumptions \eqref{E:no-S-M} for $\BK =\R$ and \eqref{E:c(k)-mono-1} for $c^\pm(k)$, due to the monotonicity of $c^\pm(k)$ in $k>0$ (Lemma \ref{L:c(k)-mono}) and the asymptotics of $c^\pm (k)$ for $|k|\gg1$ (Lemma \ref{L:ev-large-k}(3)), the dynamics of the non-singular modes is conjugate to that of linear irrotational capillary gravity waves.  

For $k \in \R$, let 
\[
e^\pm (k, x_2) = (v_1, v_2, \eta) = e^{- |k|h } \Big( \mu^{-\frac 12} y_-' (k, c^\pm (k), x_2),  -i k \mu^{-\frac 12} y_-(k, c^\pm (k), x_2),  -\frac {  y_-(k, c^\pm (k), 0)}{\mu^{\frac 12}(U(0) - c^\pm (k))} \Big), 
\]
\[
e_{ir}^\pm (k, x_2) = (v_1, v_2, \eta) = e^{-|k|h} \Big( \mu^{-\frac 12} \cosh k( x_2+h), -i \mu^{-\frac 12} \sinh k(x_2+h),  \frac { \sinh kh}{k\mu^{\frac 12} c_{ir}^\pm (k)} \Big), \; 
\]
where $c_{ir}^\pm (k)$ 
is the wave speed of the free linear capillary gravity wave (system \eqref{E:LEuler} with $U\equiv 0$ and $\nabla \times v\equiv 0$) given in \eqref{E:dispersion-F}. 
Here $e^\pm (k)$ correspond to the two non-singular modes in the $k$-th Fourier modes in $x_1$, while $e_{ir}^\pm$ the modes of irrotational linear capillary gravity waters waves. Define 
\[
\CE^\pm \big(f \big) = \int_\R f(k) e^{ikx_1} e^{\pm} (k) dk, \quad \CE_{ir}^\pm \big(f \big) = \int_\R f(k) e^{ikx_1} e_{ir}^{\pm} (k) dk, 
\]
\[
\BX^\pm = \{ \CE^\pm (f) \mid  f \in L^2(\R)\}, \quad \BX_{ir}^\pm = \{ \CE_{ir}^\pm (f) \mid  f \in L^2(\R)\}. 
\]
Clearly $\BX^+ \oplus \BX^-$ is equal to the $L^2 \times H^1$ completion of $\BX^p$ and $\CE^\pm : L^2 (\R) \to \BX^\pm$ and $\CE_{ir}^\pm : L^2 (\R) \to \BX_{ir}^\pm$ parametrize $\BX^\pm$ and $\BX_{ir}^\pm$ by $L^2$.  
The following proposition finishes the proof of Theorem \ref{T:e-values}(2b) and Theorem \ref{T:decay-R}(4). 

\begin{proposition} \label{P:conjugacy}
Assume $U\in C^3$ and \eqref{E:no-S-M} for $\BK =\R$, then the following hold. 
\begin{enumerate}
\item The mappings $\CE^\pm$ and $\CE_{ir}^\pm$ are isomorphisms. Moreover there exists $C>0$ depending only on $U$ such that 
\[
C^{-1}\le |e^\pm (k)|_{L^2},  |e_{ir}^\pm (k)|_{L^2}\le C, \;\;\ C^{-1}|f|_{L^2} \le |\CE^\pm (f)|_{L^2}, |\CE_{ir}^\pm (f)| \le C|f|_{L^2}, \quad \forall k \in \R, \, f \in L^2(\R).   
\]
\item For any solution $(v(t, x), \eta(t, x_1))$ to the capillary gravity wave linearized at the shear flow $U(x_2)$, if its component $(v^p, \eta^p)$ as defined in \eqref{E:component-R} belongs to $\BX^+\oplus \BX^-$, then it takes the form 
\be \label{E:p-comp-R}
(v^p, \eta^p) = \CE^+ (e^{-i k c^+(k)t} f_+ (k)) + \CE^- (e^{-i k c^-(k)t} f_- (k)), 
\ee
for some unique $f_\pm \in L^2(\R)$. Similarly, any solution $(v(t, x), \eta(t, x_1)) \in L^2$ to the free linear capillary gravity wave (system \eqref{E:LEuler} with $U\equiv 0$), then it takes the form 
\be \label{E:free-CGWW-R}
(v, \eta) = \CE_{ir}^+ (e^{-i k c_{ir}^+(k)t} f_+ (k)) + \CE_{ir}^- (e^{-i k c_{ir}^-(k)t} f_- (k)), \quad f_\pm \in L^2.
\ee
\item In addition, assume \eqref{E:c(k)-mono-1} for $c^\pm(k)$ and $0\in U\big([-h, 0]\big)$, then there exist odd $C^1$ functions $\varphi_\pm (k)$ and $C>0$ depending only on $U$ such that 
\[
\varphi^\pm (k)  c^\pm (\varphi^\pm (k)) =  k c_{ir}^\pm (k), \quad C^{-1}\le |k|^{-1} |\varphi^\pm(k)|, (\varphi^\pm)' (k) \le C, \quad  
\forall k \in \R.
\]
Define $\Phi^\pm: \BX^\pm \to \BX_{ir}^\pm$ as  
\[
\Phi^\pm \big(\CE^\pm (f)\big) = \CE_{ir}^\pm ( f \circ \varphi^\pm ) 
\]
 for any $\CE^\pm (f) \in \BX^\pm$, then $\Phi^+ + \Phi^-$ is an isomorphism from $(\BX^+ \oplus \BX^-) \cap (H^n \times H^{n+1})$ to $(\BX_{ir}^+ \oplus \BX_{ir}^-) \cap (H^n \times H^{n+1})$ for any $n \in [0, l_0-1]$. Moreover flows \eqref{E:p-comp-R} and \eqref{E:free-CGWW-R} are conjugate through $\Phi^+ + \Phi^-$. Namely, for any $f_\pm \in L^2$, it holds 
\be \label{E:conj} \begin{split}
&\Phi^+  \big(\CE^+ (e^{-i k c^+(k)t} f_+ (k))\big) +  \Phi^-  \big(\CE^- (e^{-i k c^-(k)t} f_- (k))\big) \\
= & \CE_{ir}^+(e^{-i k c_{ir}^+(k)t} f_+ (\varphi^+(k))) + \CE_{ir}^- (e^{-i k c_{ir}^-(k)t} f_- (\varphi^-(k))). 
\end{split}\ee   
\end{enumerate}
\end{proposition} 

\begin{proof}
The estimates on $|e^\pm (k)|_{L^2}$ and $|e_{ir}^\pm (k)|_{L^2}$ are derived from direct computations based on Lemma \ref{L:y-pm}. In particular, since $c^\pm (k) \in \R \setminus U([-h, 0])$, formula \eqref{E:y_--0} of $y_-$ for $k=0$ and the bound on $\p_k y_-$ are used in obtaining the lower bounds of $|e^\pm (k)|_{L^2}$ for $|k|$ close to $0$. The estimates of $|\CE^\pm (f)|_{L^2}$ and $|\CE_{ir}^\pm (f)|_{L^2}$ follow from those of $e^\pm (f)$ and $e_{ir}^\pm (f)$ and the Parseval's identity. Statement (2) is a direct consequence of Lemma \ref{L:e-value} and the definition of $c^\pm (k)$ and $c_{ir}^\pm(k)$. 

Since $c_{ir}^\pm (0) =\sqrt{gh} \ne 0$ and $c^\pm(0) \notin U([-h, 0])$, under the additional assumptions \eqref{E:c(k)-mono-1} and $0\in U\big((-h, 0)\big)$, Proposition \ref{P:e-v-0} and Lemma \ref{L:c(k)-mono} imply that a.) both $k c^\pm(k)$ and $k c_{ir}^\pm (k)$ are odd in $k$, b.) both $\pm k c^\pm(k)$ and $\pm k c_{ir}^\pm (k)$ have positive derivative for $k >0$, and c.) both are of the order $O(|k|^{\frac 32})$ for $|k|\gg1$ and of the order $O(|k|)$ for $|k|\ll1$. 
Hence $\varphi^\pm$ exist and satisfy the estimates, which implies the boundedness of $\Phi$.   
The conjugacy relation \eqref{E:conj} can be verified directly using \eqref{E:p-comp-R}, \eqref{E:free-CGWW-R}, and the definition of $\varphi^\pm$. 
\end{proof}

\begin{remark} \label{R:conjugacy} 
Under \eqref{E:c(k)-mono-1}, $0\in U([-h, 0])$, and $F(k, U(-h))\ne 0$ for all $k \in \R$, without assuming  \eqref{E:no-S-M}, $\BX^+\oplus \BX^-$ may only be a closed subspace of $\BX^p$, but $c^\pm(k) \in \R \setminus U([-h, 0])$ are still monotonic and isolated from the rest of the singular or non-singular modes. The exactly same argument implies that the conclusions of the above proposition still hold on $\BX^+\oplus \BX^-$. 
\end{remark}

\subsection{A remark on the linearized Euler equation on a fixed 2-d channel} \label{SS:Euler-Channel}

We briefly comment on the 2-d Euler equation on a fixed channel $x_2 \in (-h, 0)$ with slip boundary condition $v_2 =0$ at $x_2 =-h, 0$. Let $U(x_2)$ be a shear flow and we assume 
\be \tag{$\mathbf{H}$}
U'>0 \text{ and there are no singular modes.}  
\ee
As in the literatures, singular modes mean linearized solutions in the form of  $e^{ik (x_1 -ct)} v(x_2)$ with $v \in H_{x_2}^1$ and $c \in U([-h, 0])$.

The approach in this paper can be easily adapted to analyze this problem. While the non-homogeneous term in the Rayleigh equation \eqref{E:Ray-1} is still $-\frac {\hat \omega_0(k, x_2)}{U(x_2)-c}$, the main modifications are: a.) replacing $y_+(k, c, x_2)$ and $V_2(k, c, x_2)$ by $\tilde y_+(k, c, x_2)$ and $y_E (k, c, x_2)$ which solve the homogeneous and non-homogeneous Rayleigh equations satisfying boundary conditions 
\[
\tilde y_+(0) = y_E(0) = y_E(-h)=0, \quad \tilde y_+'(0)=1,
\] 
respectively, 
and b.) replacing $\BF(k,c)$  by $y_-(k, c, 0)$. For the simplification of notations, we also use $y_-$, $\tilde y_+$, and $y_E$ to denote their limits as $c_I \to 0+$. In this case of channel flow with fixed boundary, obviously the set of non-singular modes (roots of $y_-(k, c, 0)$ outside $U([-h, 0])$) for all $k\in \R$ is finite, actually empty if $U''\ne 0$. Assuming ($\mathbf{H}$), 
through the same procedure as in Lemma \ref{L:integralF}, the solution $v(t, x)$ to the linearized Euler equation at the shear flow $U(x_2)$ can also be split into 
\[
v(t, x)= v^c (t, x) + v^p (t, x)
\]
associated to the continuous spectra and point spectra. Under assumption ($\mathbf{H}$), $v^p(t, \cdot)$ belongs to the eigenspace of unstable modes which is finite dimensional if $x_1 \in \mathbb{Z}_L$. 
Let 
\begin{align*}
\hat \Omega^c(k, x_2) =& \hat \omega_0 (k, x_2) \\
&+ \tfrac 12 U''(x_2)  \big( (1+ sgn(kt))y_E (k, U(x_2), x_2)  + (1-sgn(kt)) \overline{y_E (-k, U(x_2), x_2)}  \big),  \\
\hat \Lambda_T( k, x_2) 
=& \frac { i\hat \omega_0(k, 0)y_{-} (k, U(0), x_2)}{k U'(0)^2 y_{-} (k, U(0), 0)}, \quad \hat \Lambda_B (k, x_2) = \frac {i \hat \omega_0(k, -h) \tilde y_{+} (k, U(-h), x_2)}{kU'(-h)^2 \tilde y_{+} (k, U(-h), -h)},  
\end{align*}
and $\Omega^c$ and $\Lambda_\dagger$, $\dagger =B, T$, be defined as in \eqref{E:Lambda-def} for the $L$-periodic-in-$x_1$ case and in \eqref{E:Lambda-def-R} for the case of $x_1 \in \R$. 

\begin{theorem} \label{T:decay-channel}
Assume $U \in C^{l_0}$, $l_0\ge 3$, and ($\mathbf{H}$) holds for all $k \in K$ where $K=\frac {2\pi}L \mathbb{N}$ or $K =\R$,  then, for any $q_1\in [2, \infty]$, $q_2\in (2, \infty]$, $\ep>0$, $n_1\in \R$, and integer $n_0\ge 0$, there exists $C>0$ depending only on $q_1$, $q_2$, $\ep$, and $U$ such that any solution with $\hat v_{10} (0, x_2)=0$ satisfy  
\begin{align*}
|\p_t^{n_0} \p_{x_1}^{n_1} v_1^c|_{L_x^2 L_t^{q_1} (\R)} + |\p_t^{n_0} \p_{x_1}^{n_1-1}  (1-\p_{x_1}^2)^{\frac 12}& v_2^c|_{L_x^2 L_t^{q_1} (\R)}  \le C   \big| |\p_{x_1}|^{n_0+n_1-\frac 1{q_1}}\omega_0\big|_{H_{x_1}^{\ep-\frac 12} L_{x_2}^2};  
\end{align*}
if $l_0\ge 4$, then 
\begin{align*}
& \big| t\p_t^{n_0} \p_{x_1}^{n_1} (1-\p_{x_1}^2)^{\frac 12} v_2^c \big|_{L_{x}^2  L_t^{q_1} (\R)} +  \big|\p_t^{n_0} \p_{x_1}^{n_1+1} \big( t v_1^c - U'(x_2)^{-1} \p_{x_1}^{-1}\Omega^c (x_1- U(x_2)t, x_2) \big) \big|_{L_{x}^2  L_t^{q_2} (\R)}\\
&+ \big|\p_t^{n_0} \p_{x_1}^{n_1} \big( \omega^c - \Omega^c (x_1- U(x_2)t, x_2) \big) \big|_{L_{x}^2  L_t^{q_2} (\R)} \\
&+  \big|\p_t^{n_0}\p_{x_1}^{n_1-1} \big( \p_{x_2}^2 v_2^c - \p_{x_1} \Omega^c (x_1- U(x_2)t, x_2) \big) \big|_{L_{x}^2  L_t^{q_2} (\R)} \\
\le & C  \big(  \big| |\p_{x_1}|^{n_0+n_1-\frac 1{q_1}}\omega_0\big|_{H_{x_1}^{\ep+\frac 12} L_{x_2}^2} +  \big| |\p_{x_1}|^{n_0+n_1-\frac 1{q_1}} \p_{x_2} \omega_0\big|_{H_{x_1}^{\ep-\frac 12} L_{x_2}^2}\big);
\end{align*}
and if $U \in C^5$, then 
\begin{align*}
&\big|\p_t^{n_0}\p_{x_1}^{n_1+1} \big( t^2 v_2^c - U'(x_2)^{-2} \p_{x_1}^{-1} \Omega^c (x_1- U(x_2)t, x_2) 
 - \Lambda_B (x_1- U(-h)t, x_2)  \\
 & \qquad \qquad \qquad \qquad \qquad \qquad\qquad  \qquad \qquad\qquad  \qquad 
- \Lambda_T (x_1- U(0)t, x_2) \big) \big|_{L_{x}^2  L_t^{q_2} (\R)}  \\ 
\le & C  \big(  \big| |\p_{x_1}|^{n_0+n_1-\frac 1{q_1}}\omega_0\big|_{H_{x_1}^{\ep+\frac 12} L_{x_2}^2}  +  \big| |\p_{x_1}|^{n_0+n_1-\frac 1{q_1}} \p_{x_2} \omega_0\big|_{H_{x_1}^{\ep-\frac 12} L_{x_2}^2}+  \big| |\p_{x_1}|^{n_0+n_1-\frac 1{q_1}} \p_{x_2}^2 \omega_0\big|_{H_{x_1}^{\ep-\frac 32} L_{x_2}^2} \big).
\end{align*}
Moreover, 
\[
|\Omega^c -\omega_0|_{H_{x_1}^{n_1} L_{x_2}^2} \le C  |\omega_0|_{H_{x_1}^{n_1-1+\ep}L_{x_2}^2},
\quad 
|\p_{x_2} \Omega^c -  \p_{x_2} \omega_0|_{H_{x_1}^{n_1} L_{x_2}^2} 
\le C  \big(  |\omega_0|_{H_{x_1}^{n_1 +\ep} L_{x_2}^2}+ |\p_{x_2} \omega_0|_{H_{x_1}^{n_1-1 +\ep} L_{x_2}^2} \big). 
\]
\[
|k \hat \Lambda_B(k, \cdot)|_{L_{x_2}^q} \le C \langle k\rangle^{-\frac 1q} |\hat \omega_0(\cdot, -h)|, 
\quad |k \hat \Lambda_T(k, \cdot)|_{L_{x_2}^q} \le C \langle k\rangle^{-\frac 1q} |\hat \omega_0(k, 0)|, \quad \; \forall q\in [1, \infty],
\]
\[
|k\p_{x_2}  \hat \Lambda_B(k, \cdot)|_{L_{x_2}^q} \le C \langle k\rangle^{1-\frac 1q} |\hat \omega_0(\cdot, -h)|, 
\;\; |k \p_{x_2} \hat \Lambda_T(k, \cdot)|_{L_{x_2}^q} \le C \langle k\rangle^{1-\frac 1q}  |\hat \omega_0(k, 0)|, \;\; \forall q\in [1, \infty).
\]
Finally, $\Lambda_\dagger$, $\dagger=B, T$, satisfy $\hat \Lambda_\dagger (k=0, x_2) =0$ and  
\[\begin{cases}  
- (U- U(0)) \Delta \Lambda_T + U'' \Lambda_T =0, \qquad \qquad \qquad \qquad \qquad \qquad \qquad x_2 \in (-h, 0), \\ 
\Lambda_T (x_1, -h) = 0, \qquad \p_{x_1} \Lambda_T (x_1, 0) = - U'(0)^{-2}  \omega_0(x_1, 0); 
\end{cases}\]
\[ \begin{cases}  
- (U- U(-h)) \Delta \Lambda_B + U'' \Lambda_B =0, \qquad \qquad \qquad \qquad \qquad \qquad \qquad x_2 \in (-h, 0), \\ 
\p_{x_1} \Lambda_B (\cdot, -h) = - U'(-h)^{-2} \omega_0(x_1, -h), \qquad \Lambda_B (x_1, 0) =0. 
\end{cases}\]
\end{theorem}

\begin{remark} \label{R:boundryC}
In the case of the Couette flow $U(x_2)=x_2$, assumption ($\mathbf{H}$) is satisfied. Obviously $\Omega^c = \omega_0$, which in fact gives the whole linearized vorticity $\omega(t, x) = \omega_0(x_1 - x_2t, x_2)$ and the leading asymptotic terms of $tv_1$ and $\p_{x_2}^2 v_2$. However, $t^2 v_2$ does also include contributions $\Lambda_T$ and $\Lambda_B$ from the top and bottom boundaries. These  asymptotic leading order  terms are essentially same as those obtained in \cite{Jia20S} (after simplifications of (5.1) in Lemma 5.1 there), see also Lemma 3 in \cite{Zi16}. 
\end{remark}

\begin{center} Acknowledgement \end{center} 

The second author would like to thank Zhiwu Lin, Hao Jia, and Zhifei Zhang for helpful discussions during the completion of the paper.

\bibliographystyle{abbrv}
\bibliography{bibliography}
\endgroup
\end{document}